\documentclass{article}
\usepackage{latexsym}
\usepackage{amscd}
\usepackage{graphics}
\usepackage{amsmath}
\usepackage{amssymb}
\usepackage{amsthm}
\usepackage{mathrsfs}
\usepackage{bbold}
\input xy
\xyoption{all}
\usepackage{graphicx}
\usepackage{epstopdf}
\newcommand{\N}{\mathbb{N}}                     % the natural numbers
\newcommand{\Z}{\mathbb{Z}}                     % the integer numbers
\newcommand{\R}{\mathbb{R}}                     % the real line
\newcommand{\C}{\mathbb{C}}                     % the complex plane
\newcommand{\T}{\mathbb{T}}                     % the torus
\newcommand{\set}[2]{\left\{{#1}\mid{#2}\right\}}       % the set
\newcommand{\im}{\mathrm{Im\,}}                 % imaginary part
\newcommand{\re}{\mathrm{Re\,}}                 % real part
\newcommand{\coker}{\mathrm{coker\,}}           % Cokernel
\newcommand{\ind}{\mathrm{ind\,}}               % Fredholm index
\newcommand{\codim}{\mathrm{codim}}           % codimension
\newcommand{\supp}{\mathrm{supp\,}}             % support
\newcommand{\graf}{\mathrm{graph\,}}       % graph
\newcommand{\rank}{\mathrm{rank\,}}     % rank
\newcommand{\dom}{\mathrm{dom}\,}       % domain
\newcommand{\crit}{\mathrm{crit}}

\newcommand{\grad}{\mathrm{grad\,}}

\newcommand{\hess}{\mathrm{Hess\,}}

\newcommand{\delbar}{\overline{\partial}}               % delbar operator

\swapnumbers
\newtheorem{thm}{\sc Theorem}[section]      % numbered within each section

\newtheorem{cor}[thm]{\sc Corollary}        % numbered along with Theorem
\newtheorem{lem}[thm]{\sc Lemma}            % numbered along with Theorem
\newtheorem{prop}[thm]{\sc Proposition}     % numbered along with Theorem
\newtheorem{defn}[thm]{\sc Definition}      
\newtheorem{rem}[thm]{\sc Remark}

\usepackage{epsfig}

\newcommand{\loopprod}{\,\mathsf{o}\, }

\author{
Alberto Abbondandolo\\
{Universit\`a di Pisa}\\
{ Dipartimento di Matematica}\\
    { Largo Bruno Pontecorvo, 5}\\
    { 56127 Pisa, Italy}\\
abbondandolo@dm.unipi.it
\and 
Matthias Schwarz\\
Universit\"at Leipzig\\
Mathematisches Institut\\
Postfach 10 09 20\\
D-04009 Leipzig, Germany\\
mschwarz@math.uni-leipzig.de}

\title{Floer homology of cotangent bundles and the loop product}

\date{December 6, 2009}

\begin{document}

\maketitle

\begin{abstract}
We prove that the pair-of-pants product on the Floer homology of the cotangent bundle of a compact manifold $M$ corresponds to the Chas-Sullivan loop product on the singular homology of the loop space of $M$. We also prove related results concerning the Floer homological interpretation of the Pontrjagin product and of the Serre fibration. The techniques include a Fredholm theory for Cauchy-Riemann operators with jumping Lagrangian boundary conditions of conormal type, and a new cobordism argument replacing the standard gluing technique.
\end{abstract}

\tableofcontents

\section*{Introduction}
\addcontentsline{toc}{section}{\numberline{}Introduction}

Let $M$ be a closed manifold, and let $H$ be a time-dependent smooth Hamiltonian on $T^*M$, the cotangent bundle of $M$. We assume that $H$ is $1$-periodic in time and grows asymptotically  quadratically on each fiber. Generically, the corresponding Hamiltonian system
\begin{equation}
\label{introuno}
x'(t) = X_H(t,x(t)), 
\end{equation}
has a discrete set $\mathscr{P}(H)$ of $1$-periodic orbits. The free Abelian group $F_*(H)$ generated by the elements in $\mathscr{P}(H)$, graded by their Conley-Zehnder index,  
 supports a chain complex, the {\em Floer complex} $(F_\ast(H),\partial)$. The boundary operator $\partial$ is defined by an algebraic count of the maps $u$ from the cylinder $\R \times \T$ to $T^*M$, solving the Cauchy-Riemann type equation
 \begin{equation}
 \label{introdue}
 \partial_s u(s,t) + J(u(s,t)) \bigl(\partial_t u(s,t) - X_H(t,u(s,t))\bigr) = 0, \quad \forall (s,t) \in \R \times \T,
 \end{equation}
 and converging to two $1$-periodic orbits of (\ref{introuno}) for $s\rightarrow -\infty$ and $s\rightarrow +\infty$. Here $J$ is the almost-complex structure on $T^*M$ induced by a Riemannian metric on $M$, and (\ref{introdue}) can be seen as the negative $L^2$-gradient equation for the Hamiltonian action functional. 
 
This construction is due to A.\ Floer (see e.g. \cite{flo88a,flo88d,flo89b,flo89a}) in the case of a closed symplectic manifold $P$, in order to prove a conjecture of Arnold on the number of periodic Hamiltonian orbits. The extension to non-compact symplectic manifolds, such as the cotangent bundles we consider here, requires suitable growth conditions on the Hamiltonian, such as the convexity assumption used in \cite{vit96} or the asymptotic quadratic-growth assumption used in \cite{as06}. The Floer complex obviously depends on the Hamiltonian $H$, but its homology often does not, so it makes sense to call this homology the {\em Floer homology} of the underlying symplectic manifold $P$, and to denote it by $HF_*(P)$. The Floer homology of a compact symplectic manifold $P$ without boundary is isomorphic to the singular homology of $P$, as proved by A.\ Floer for special classes of symplectic manifolds, and later extended to larger and larger classes by several authors (the general case requiring special coefficient rings, see \cite{hs95,lt98,fo99}). Unlike the compact case, the Floer homology of a cotangent bundle $T^*M$ is a truly infinite dimensional homology theory, being isomorphic to the singular homology of the free loop space $\Lambda(M)$ of $M$. This fact was first proved by C.\ Viterbo (see \cite{vit96}) using a generating functions approach, later by D.\ Salamon and J.\ Weber using the heat flow for curves on a Riemannian manifold (see \cite{sw06}), and then by the authors in \cite{as06}. In particular, our proof reduces the general case to the case of a Hamiltonian which is uniformly convex in the momenta, and for such a Hamiltonian it constructs an explicit isomorphism between the Floer complex of $H$ and the {\em Morse complex} of the action functional
\[
\mathbb{S}_L(\gamma) = \int_{\T} L(t,\gamma(t),\gamma'(t)) \, dt, \quad \gamma \in W^{1,2}(\T,M),
\]
associated to the Lagrangian  $L$ which is the Fenchel dual of $H$. 
The latter complex is the standard chain complex associated to the Lagrangian action functional $\mathbb{S}_L$. The domain of such a functional is the
infinite dimensional Hilbert manifold $W^{1,2}(\T,M)$ 
consisting of closed loops of Sobolev class $W^{1,2}$ on $M$, and an 
important fact is that the functional $\mathbb{S}_L$ is  
bounded from below, has critical points with finite Morse index, satisfies the Palais-Smale condition, and, although in general it is not $C^2$, it admits a smooth Morse-Smale pseudo-gradient flow. The construction of the
Morse complex in this infinite dimensional setting and the proof that
its homology is isomorphic to the singular homology of the ambient manifold are described in \cite{ama06m}.  The isomorphism between the Floer and the Morse complex is obtained by coupling the Cauchy-Riemann type equation on half-cylinders with the gradient flow equation for the Lagrangian action. We call this the {\em hybrid method}.

Since the space $W^{1,2}(\T,M)$ is homotopy
equivalent to $\Lambda(M)$, we get the required isomorphism 
\begin{equation}
\label{introtre}
\Phi^{\Lambda} : H_*(\Lambda(M)) \stackrel{\cong}{\longrightarrow} HF_*(T^*M),
\end{equation}
from the singular homology of the free loop space of $M$ to the Floer homology of $T^*M$.  

Additional interesting algebraic structures on the Floer homology of a
symplectic manifold are obtained by considering other Riemann surfaces
than the cylinder as domain for the Cauchy-Riemann type equation
(\ref{introdue}). By considering the pair-of-pants surface, a
non-compact Riemann surface with three cylindrical ends, one obtains
the {\em pair-of-pants product} in Floer homology (see \cite{sch95}
and \cite{ms04}). When the symplectic manifold $P$ is closed and
symplectically aspherical, this product corresponds to the standard
cup product from topology, after identifying the Floer homology of $P$
with its singular cohomology by Poincar\'e duality, while when the
manifold $P$ can carry $J$-holomorphic spheres, the pair-of-pants
product corresponds to the {\em quantum cup product} of $P$ (see
\cite{pss96} and \cite{lt99}). 

The main result of this paper is that in the case of cotangent
bundles, the pair-of-pants product is also equivalent to a product on
$H_*(\Lambda(M))$ coming from topology, but a more interesting one
than the simple cup product:

\bigskip

\noindent {\sc Theorem A.} {\em
Let $M$ be a closed oriented manifold. Then the
isomorphism $\Phi^{\Lambda}$ in (\ref{introtre}) is a ring isomorphism when the
Floer homology of $T^*M$ is endowed with its pair-of-pants product,
and the homology of the space of free parametrized loops of $M$ is endowed with its Chas-Sullivan loop
product.}

\bigskip

The latter is an algebraic structure which was recently discovered by 
M.\ Chas and D.\ Sullivan \cite{csu99}, and which is currently having a strong
impact in string topology (see e.g.\ \cite{chv06} and \cite{sul07}).   
It is the free loop space version of the classical
Pontrjagin product 
\[
\# : H_j(\Omega(M,q_0)) \otimes H_k(\Omega(M,q_0)) \rightarrow H_{j+k} (\Omega(M,q_0))
\]
on the singular homology of the space $\Omega(M,q_0)$ of loops based at $q_0$.
As the Pontrjagin product, it is induced by concatenation and it can be described in the following way. Let $\Theta(M)$ be the
subspace of $\Lambda(M)\times \Lambda(M)$ consisting of pairs of parametrized loops with identical initial point. 
If $M$ is oriented and $n$-dimensional,
$\Theta(M)$ is both a co-oriented $n$-codimensional submanifold of the 
Banach manifold $\Lambda(M)\times \Lambda(M)$, as well as of $\Lambda(M)$ itself via the concatenation map  $\Gamma\colon \Theta(M) \rightarrow \Lambda(M)$,
\begin{equation}
\Lambda(M)\times\Lambda(M)\,\stackrel{e}{\hookleftarrow}\, \Theta(M)\,\stackrel{\Gamma}{\hookrightarrow}\,\Lambda(M)\,.
\end{equation}
Seen as continuous maps, $e$ and $\Gamma$ induce homomorphisms $e_\ast$, $\Gamma_\ast$ in homology. Seen as  $n$-codimensional co-oriented embeddings, they induce Umkehr maps 
\begin{eqnarray*}
e_{!}\colon H_j(\Lambda(M) \times \Lambda(M)) & \rightarrow &
H_{j-n}(\Theta(M)),\\
\Gamma_{!}\colon H_j(\Lambda(M)) & \rightarrow & H_{j-n}(\Theta(M)).
\end{eqnarray*}
The loop product is the degree $-n$ product on the homology of the
free loop space of $M$,
\[
\loopprod: H_j(\Lambda(M)) \otimes H_k(\Lambda(M)) \rightarrow
H_{j+k-n}(\Lambda(M)),
\]
defined as the composition
\begin{equation*}\begin{split}
H_j(\Lambda(M)) \otimes H_k(\Lambda(M))
&\stackrel{\times}{\longrightarrow} H_{j+k} (\Lambda(M) \times
\Lambda(M)) \stackrel{e_!}{\longrightarrow} H_{j+k-n} (\Theta(M))\\
&\stackrel{\Gamma_*}{\longrightarrow} H_{j+k-n}(\Lambda(M)),
\end{split}\end{equation*}
where $\times$ is the exterior homology product. The loop product
turns out to be associative, commutative, and to have a unit, namely
the image of the fundamental class of $M$ by the embedding of $M$ into $\Lambda(M)$ as the space of constant loops. More information about the loop product and about its relationship with the Pontrjagin and the intersection product on $M$ are recalled in Section \ref{altopo}. Similarly, the composition  $e_\ast\circ\Gamma_{!}$ gives a coproduct of degree $-n$ (for coefficients in a field), corresponding to the pair-of-pants coproduct on Floer homology. However, it is easy to see that this coproduct is almost entirely trivial, except for homology classes of dimension $n$, so we shall not consider it in this paper.

Coming back to Theorem A, it is worth noticing that the analogy between the pair-of-pants product and the loop
product is even deeper. Indeed, we may look at the solutions $(x_1,x_2) : [0,1]
\rightarrow T^*M \times T^*M$ of the following pair of Hamiltonian systems 
\begin{equation}
\label{introquattro}
x_1'(t) = X_{H_1}(t,x_1(t)), \quad x_2'(t) = X_{H_2}(t,x_1(t)),
\end{equation}
coupled by the nonlocal boundary condition
\begin{equation}
\label{introcinque}
\begin{split}
q_1(0) = q_1(1) & =  q_2(0) = q_2(1), \\
p_1(1) - p_1(0)  & =  p_2(0) - p_2(1).
\end{split}
\end{equation}
Here we are using the notation $x_j(t) = (q_j(t),p_j(t))$, with
$q_j(t)\in M$ and $p_j(t) \in T_{q_j(t)}^* M$, for $j=1,2$. By studying the
corresponding Lagrangian boundary value Cauchy-Riemann type problem on
the strip $\R \times [0,1]$, we obtain a chain complex, the  {\em Floer complex
for figure-8 loops} $(F^{\Theta}(H_1 \oplus H_2),\partial)$ on the graded free
Abelian group generated by solutions of
(\ref{introquattro})-(\ref{introcinque}). Then we can show that:
\begin{enumerate} 
\item The homology of the chain complex $(F^{\Theta}(H_1 \oplus H_2),\partial)$ is 
isomorphic to the singular homology of $\Theta(M)$.
\item The pair of pants product factors through the homology of this
  chain complex.
\item The first homomorphism in this factorization corresponds to the
  homomorphism $e_{!} \circ \times$, while the
  second one corresponds to the homomorphism $\Gamma_*$.
\end{enumerate}

We also show that similar results hold for the space  $\Omega(M,q_0)$ of loops which are based at $q_0\in M$. The Hamiltonian problem in this case is the equation
$(\ref{introuno})$ for $x=(q,p) : [0,1] \rightarrow T^* M$ with boundary
conditions
\[
q(0) = q(1) = q_0.
\]
Since the fiber $T_{q_0}^* M$ is a Lagrangian submanifold of $T^*M$, this is a Lagrangian intersection problem, and one can associate to it a  
Floer homology, that we denote by $HF^{\Omega}_*(T^*M)$. On such a Floer homology there is a product 
\[
\Upsilon^{\Omega}_* : HF_j^{\Omega} (T^*M) \otimes HF_k^{\Omega}(T^*M) \rightarrow HF_{j+k}^{\Omega} (T^*M), 
\]
which is called the {\em triangle product}. Then we can prove the following:

\bigskip

\noindent {\sc Theorem B.} {\em
Let $M$ be a closed manifold. Then there is a ring isomorphism
\[
\Phi^{\Omega} : H_*(\Omega(M,q_0)) \stackrel{\cong}{\longrightarrow} HF_*^{\Omega} (T^*M),
\]
where the singular homology of the based loop space $\Omega(M,q_0)$ is endowed with the Pontrjagin product $\#$ and the Floer homology $HF^{\Omega}_*(T^*M)$ is equipped with the triangle product.}

\bigskip

Actually, every arrow in the commutative diagram from topology
\begin{equation}
\label{introsei}
\begin{CD}
H_j(M) \otimes H_k(M) @>{\bullet}>> H_{j+k-n}(M) \\
@V{\mathrm{c}_* \otimes \mathrm{c}_*}VV @VV{\mathrm{c}_*}V \\
H_j(\Lambda(M)) \otimes H_k(\Lambda(M)) @>{\loopprod}>>
H_{j+k-n}(\Lambda(M)) \\
@V{\mathrm{i}_! \otimes \mathrm{i}_!}VV @VV{\mathrm{i}_!}V \\
H_{j-n}(\Omega(M)) \otimes H_{k-n}(\Omega(M)) @>{\#}>>
H_{j+k-2n} (\Omega(M,q_0)),
\end{CD}
\end{equation}  
has an equivalent homomorphism in Floer homology. Here $\bullet$ is
the intersection product in singular homology, $\mathrm{c}$ is the embedding of $M$ into $\Lambda(M)$ by constant loops,
and $\mathrm{i}_!$ denotes the Umkehr map induced by the $n$-codimensional
co-oriented embedding $\mathrm{i}: \Omega(M) \hookrightarrow \Lambda(M)$.

The first step in the proof of the main statements of this paper is to describe objects and morphisms from algebraic topology in a Morse theoretical way.  The way this translation is performed is well known in the case of finite dimensional manifolds (see e.g.\ \cite{fuk93}, \cite{sch93}, \cite{bc94}, \cite{vit95f}, \cite{fuk97}). In Appendix A we outline how these results extend to infinite dimensional Hilbert manifolds, paying particular attention to the transversality conditions required for each construction, In Section \ref{mtdsec}, we specialize the analysis to the action functional associated to Lagrangians which have quadratic growth in the velocities. See also \cite{coh06,chv06,cs08}.

The core of the paper consists of Sections \ref{fhrssec} and \ref{imfsec}. In the former we define the Floer complexes we are dealing with and the products on their homology. All the Floer homologies we consider here -- for free loops, based loops, or figure-8 loops -- are special cases of Floer homology for {\em nonlocal conormal} boundary conditions. Therefore, we unify the presentation by using this level of generality (see \cite{aps08}): Given a closed manifold $Q$, we fix a closed submanifold $R$ of $Q\times Q$ and consider the Hamiltonian orbits $x: [0,1] \rightarrow T^*Q$ such that $(x(0),-x(1))$ belongs to the conormal bundle $N^\ast R$ of $R$, that is to the set of covectors in $T^* (Q\times Q)$ which are based at $R$ and annihilate every vector which is tangent to $R$.  The Floer homology on $T^*Q$ associated to $N^\ast R$ is denoted by $HF^R_\ast(T^*Q)$.
The standard Floer homology on $T^*M$ with periodic boundary conditions corresponds to the choice $Q=M$ and $R=\Delta_M$, the diagonal in $M\times M$. Floer homology for based loops (or Dirichlet boundary conditions) corresponds to choosing $Q=M$ and $R$ to consist of the point $(q_0,q_0)$. Finally, in Floer homology for figure-8 loops we choose $Q=M\times M$ and $R=\Delta_M^{(4)}$, the set of quadruples $(q,q,q,q)$ in $M^4$. 

In Section \ref{imfsec}, we start by recalling the construction of the isomorphism 
$$
\Phi^R_* \colon H_\ast(P_R(Q)) \,\stackrel{\cong}{\longrightarrow}\,  HF^R_\ast(T^\ast Q)
$$
where $P_R(Q))$ is the space of continuous paths $\gamma\colon [0,1]\to Q$ such that $(\gamma(0),\gamma(1))\in R$. Then we prove Theorems A and B, as corollaries of chain level results (Theorems \ref{chainA} and \ref{Bchain}) involving the Morse complex of the Lagrangian action functional. In Section \ref{bah}, we complete the picture, by showing how the other homomorphisms which appear in diagram (\ref{introsei}) can be described in a Floer theoretical way.

The linear Fredholm theory used in these sections is described in Section \ref{lineartheory}, whereas Section \ref{cocosec} contains compactness and removal of singularities results, together with the proofs of three cobordism statements from Sections \ref{fhrssec} and \ref{imfsec}.

Some of the proofs are based on standard techniques in Floer homology, and in this case we just refer to the literature. However, there are a few key points where we need to introduce some new ideas. We conclude this introduction by briefly describing these ideas.

\paragraph{Riemann surfaces as quotients of strips with slits.} The definition of the pair-of-pants product requires extending the Cauchy-Riemann type equation (\ref{introdue}) to the pair-of-pants surface. The Cauchy-Riemann operator $\partial_s + J \partial_t$ naturally extends to any Riemann surface, by letting it take value into the vector bundle of anti-linear one-forms. The zero-order term $-J X_H(t,u)$ instead does not have a natural extension when the Riemann surface does not have a global coordinate $z=s+it$. The standard way to overcome this difficulty is to make this zero-order term act only on the cylindrical ends of the pair-of-pants surface -- which do have a global coordinate $z = s+it$ -- by multiplying the Hamiltonian by a cut-off
function making it vanish far from the cylindrical ends (see 
\cite{sch95}, \cite[Section 12.2]{ms04}, but see also \cite{sei06b} for a different approach). This construction does not cause problems when dealing with compact symplectic manifolds as in the above mentioned reference, but in the case of the cotangent bundle it would create problems with compactness of the spaces of solutions. In fact, on one hand cutting off the Hamiltonian destroys  the identity relating the energy of the solution with the oscillation of the action functional, on the other hand our $C^0$-estimate for the solutions requires coercive Hamiltonians. 

We overcome this difficulty by a different -- and we believe more natural -- way of extending the zero-order term. We describe the pair-of-pants surface -- as well as the other Riemann surfaces we need to deal with -- as the quotient of an infinite strip with a slit -- or more slits in the case of more general Riemann surfaces. At the end of the slit we use a chart given by the square root map. In this way, the Riemann surface is still seen as a smooth object, but it carries a global coordinate $z=s+it$ with singularities. This global coordinate allows to extend the zero-order term without cutting off the Hamiltonian, and preserving the energy identity. See Section \ref{Fetpp} below.

\paragraph{Cauchy-Riemann operators on strips with jumping boundary conditions.} When using the above description for the Riemann surfaces, the problems we are looking at can be described in a unified way as Cauchy-Riemann type equations on a strip, with Lagrangian boundary conditions presenting a finite number of jumps. In Section \ref{lineartheory} we develop a complete linear theory for such problems, in the case of Lagrangian boundary conditions of conormal type. This is the kind of conditions which occur naturally on cotangent bundles. Once the proper Sobolev setting has been chosen, the proof of the Fredholm property for such operators is standard. The computation of the index instead is reduced to a Liouville type statement, proved in Section  \ref{liouvsec}

These linear results have the following consequence.
Let $R_0,\dots,R_k$ be submanifolds of $Q\times Q$, such that $R_{j-1}$ and $R_j$ intersect cleanly, for every $j=1,\dots,k$. Let $-\infty = s_0 < s_1 < \dots < s_k < s_{k+1} = +\infty$, and consider the space $\mathscr{M}$ consisting of the maps $u: \R \times [0,1] \rightarrow T^*Q$ solving the Cauchy-Riemann type equation (\ref{introdue}), satisfying the boundary conditions
\[
(u(s,0),-u(s,1)) \in N^* R_j  \quad \forall s\in [s_j,s_{j+1}],  \; \forall j=0,\dots,k,        
\]
and converging to Hamiltonian orbits $x^-$ and $x^+$ for $s\rightarrow -\infty$ and $s\rightarrow +\infty$. The results of Section \ref{lineartheory} imply that for a generic choice of the Hamiltonian $H$ the space $\mathscr{M}$ is a manifold of dimension
\[
\dim \mathscr{M} = \mu^{R_0}(x^-) - \mu^{R_k}(x^+) - \sum_{j=1}^k ( \dim R_{j-1} -  \dim R_{j-1} \cap R_j).
\]
Here ${\mu}^{R_0}(x^-)$ and ${\mu}^{R_k}(x^+)$ are the Maslov indices of the Hamiltonian orbits $x^-$ and $x^+$, with boundary conditions $(x^-(0),-x^-(1)) \in N^* R_0$, $(x^+(0),-x^+(1)) \in N^* R_k$, suitably shifted so that in the case of a fiberwise convex Hamiltonian they coincide with the Morse indices of the corresponding critical points $\gamma^-$ and $\gamma^+$ of the Lagrangian action functional on the spaces of paths satisfying $(\gamma^-(0),\gamma^-(1)) \in R_0$ and $(\gamma^+(0),\gamma^+(1))\in R_k$, respectively.  
Similar formulas hold for problems on the half-strip. See Section \ref{lin} for precise statements. Different approaches to jumping Lagrangian boundary conditions can be found in \cite{is02}, in \cite{ww07, ww09, ww09a, ww09b}, and in \cite{cel09}.

\paragraph{Cobordism arguments.} The main results of this paper always reduce to the fact that certain diagrams involving homomorphisms defined either in a Floer or in a Morse theoretical way should commute up to a chain homotopy. The proof of such a commutativity is based on cobordism arguments, saying that a given solution of a certain Problem 1 can be ``continued'' by a unique one-parameter family of solutions of a certain Problem 2, and that this family of solutions converges to a solution of a certain Problem 3. In many situations such a statement can be proved by the classical gluing argument in Floer theory: One finds the one-parameter family of solutions of Problem 2 by using the given solution of Problem 1 to construct an approximate solution, to be used as the starting point of a Newton iteration scheme which converges to a true solution. When this is the case, we just refer to the literature. However, we encounter three situations in which the standard arguments do not apply, one reason being that we face a Problem 2 involving a Riemann surface whose conformal structure is varying with the parameter: this occurs when proving that the pair-of-pants product factorizes through the figure-8 Floer homology (Section \ref{fpop}), that the Pontrjagin product corresponds to the triangle product (Section \ref{ori}), and that the homomorphism 
$e_{!} \circ \times$ corresponds to its Floer homological counterpart (Section \ref{chlhs}). We manage to reduce the former two statements to the standard implicit function theorem (see Sections \ref{factsec} and \ref{omegasec}). The proof of the latter statement is more involved, because in this case the solution of Problem 2 we are looking for cannot be expected to be even $C^0$-close to the solution of Problem 1 we start with. We overcome this difficulty by the following algebraic observation: In order to prove that two chain maps $\varphi,\psi:C \rightarrow C'$ are chain homotopic, it suffices to find a chain homotopy between the chain maps $\varphi \otimes \psi$ and $\psi \otimes \varphi$, and to find an element $\epsilon\in C_0$ and a chain map $\delta$ from the complex $C'$ to the trivial complex $(\Z,0)$ such that $\delta(\varphi(\epsilon)) = \delta(\psi(\epsilon)) = 1$ (see Lemma \ref{alge} below). In our situation, the chain homotopy between   $\varphi \otimes \psi$ and $\psi \otimes \varphi$ is easier to find, by using a localization argument and the implicit function theorem (see Section \ref{coupros}). This argument is somehow reminiscent of an alternative way suggested by H.\ Hofer to prove standard gluing results in Floer homology. The construction of the element $\epsilon$ and of the chain map $\delta$ is presented in Section \ref{chlhs}, together with the proof of the required algebraic identity. This is done by considering special Hamiltonian systems, having a hyperbolic equilibrium point.

\bigskip

The main results of this paper were announced in \cite{as06m}. Related results concerning the equivariant loop product and its interpretation in the symplectic field theory of unitary cotangent bundles have been announced in \cite{cl07}. 

\paragraph{Acknowledgements.}

We wish to thank the Max Planck Institute for Mathematics in the
Sciences of Leipzig and the Department of Mathematics at Stanford
University, and in particular Yasha Eliashberg, for their kind
hospitality. We are also indebted with Ralph Cohen and Helmut Hofer for
many fruitful discussions. The first author thanks the Humboldt Foundation for financial support in the form of a {\em Humboldt Research Fellowship for Experienced Researchers}.
The second author thanks the Deutsche Forschungsgemeinschaft
for the support by the grant DFG SCHW 892/2-3.

\numberwithin{equation}{section}

\section{The Pontrjagin and the loop products}
\label{altopo}

\subsection{The Pontrjagin product}
\label{tpp}

Given a topological space $M$ and a point $q_0\in M$, we denote by
$\Omega(M,q_0)$ the space of loops on $M$ based at $q_0$, that is
\[
\Omega(M,q_0) := \set{\gamma\in C^0(\T,M)}{\gamma(0)=q_0},
\]
endowed with the compact-open topology. Here $\T = \R/\Z$ is the
circle parameterized by the interval $[0,1]$. The concatenation
\[
\Gamma(\gamma_1,\gamma_2)(t) := \left\{ \begin{array}{ll}
\gamma_1(2t)
  & \mbox{for } 0\leq t \leq 1/2, \\ \gamma_2(2t-1) & \mbox{for } 1/2
  \leq t \leq 1, \end{array} \right.
\]
maps $\Omega(M,q_0) \times \Omega(M,q_0)$ continuously into
$\Omega (M,q_0)$. The constant loop $q_0$ is a homotopy unit for
$\Gamma$, meaning that the maps $\gamma \mapsto
\Gamma(q_0,\gamma)$ and $\gamma \mapsto \Gamma (\gamma,q_0)$ are
homotopic to the identity map. Moreover, $\Gamma$ is homotopy
associative, meaning that $\Gamma \circ (\Gamma \times
\mathrm{id})$ and $\Gamma \circ (\mathrm{id} \times \Gamma)$ are
homotopic. Therefore, $\Gamma$ defines the structure of an
$H$-space on $\Omega(M,q_0)$.

We denote by $H_*$ the singular homology functor with integer
coefficients. The composition
\[
H_j(\Omega(M,q_0)) \otimes H_k(\Omega(M,q_0))
\stackrel{\times}{\longrightarrow} H_{j+k}(\Omega(M,q_0)\times
\Omega(M,q_0)) \stackrel{\Gamma_*}{\longrightarrow}
H_{j+k}(\Omega(M,q_0))
\]
where the first arrow is the exterior homology product, is by
definition the {\em Pontrjagin product}
\[
\#: H_j(\Omega(M,q_0)) \otimes H_k(\Omega(M,q_0)) \rightarrow
H_{j+k}(\Omega(M,q_0)).
\]
The fact that $q_0$ is a homotopy unit for $\Gamma$ implies that
$[q_0]\in H_0(\Omega(M,q_0))$ is the identity element for the
Pontrjagin product. The fact that $\Gamma$ is homotopy associative
implies that the Pontrjagin product is associative. Therefore, the
product $\#$ makes the
singular homology of $\Omega(M,q_0)$ a graded ring. In general,
it is a non-commutative graded ring. See for instance \cite{dkp70} for more
information on $H$-spaces and the Pontrjagin product.

\subsection{The Chas-Sullivan loop product}
\label{tcslp}

We denote by $\Lambda(M):=C^0(\T,M)$ the space of free loops on $M$.
Under the assumption that $M$ is an oriented $n$-dimensional
manifold, it is possible to use the concatenation map $\Gamma$ to
define a product of degree $-n$ on $H_*(\Lambda(M))$. In order to
describe the construction, we need to recall the definition of the
Umkehr map.

Let $\mathcal{M}$ be a (possibly infinite-dimensional) smooth
Banach manifold, and let $e:\mathcal{M}_0 \hookrightarrow
\mathcal{M}$ be a smooth closed embedding, which we assume to be
$n$-co-dimensional and co-oriented. In other words, $\mathcal{M}_0$
is a closed submanifold of $\mathcal{M}$ whose normal bundle
$N\mathcal{M}_0:= T\mathcal{M}|_{\mathcal{M}_0} / T\mathcal{M}_0$
has dimension $n$ and is oriented. The tubular neighborhood
theorem provides us with a homeomorphism
$u : \mathcal{U} \rightarrow N\mathcal{M}_0$, uniquely determined
up to isotopy, of an open neighborhood of $\mathcal{M}_0$ onto
$N\mathcal{M}_0$, mapping $\mathcal{M}_0$ identically onto the
zero section of $N\mathcal{M}_0$, that we also denote by
$\mathcal{M}_0$ (see \cite{lan99}, IV.\S 5-6, if $\mathcal{M}$
admits
  smooth partitions of unity -- for instance, if it is a Hilbert
  manifold -- then $u$ can be chosen to be a smooth diffeomorphism). 
The {\em Umkehr map} $e_{!}$ associated to the embedding $e$ is defined to be the composition
\begin{equation*}\begin{split}
H_j(\mathcal{M}) &\longrightarrow H_j (\mathcal{M},\mathcal{M}
\setminus \mathcal{M}_0) \stackrel{\cong}{\longrightarrow}
H_j(\mathcal{U},\mathcal{U}
\setminus \mathcal{M}_0) \stackrel{u_*}{\longrightarrow}
H_j(N\mathcal{M}_0, N\mathcal{M}_0 \setminus \mathcal{M}_0)\\
&\stackrel{\tau}{\longrightarrow} H_{j-n} (\mathcal{M}_0),
\end{split}\end{equation*}
where the first arrow is induced by the inclusion, the second one
is the isomorphism given by excision, and the last one is the Thom
isomorphism associated to the $n$-dimensional oriented vector
bundle $N\mathcal{M}_0$, that is, the cap product with the Thom
class $\tau_{N\mathcal{M}_0} \in H^n(N\mathcal{M}_0,N\mathcal{M}_0
\setminus \mathcal{M}_0)$. 

We recall that if $M$ is an $n$-dimensional manifold, $\Lambda(M)$ is
an infinite dimensional smooth manifold modeled on the Banach space
$C^0(\T,\R^n)$. The set $\Theta(M)$ of pairs of loops with the same
initial point (figure-8 loops),
\[
\Theta (M) := \set{(\gamma_1,\gamma_2)\in \Lambda(M) \times
  \Lambda(M)}{\gamma_1(0) = \gamma_2(0)},
\]
is the inverse image of the diagonal $\Delta_M$ of $M\times M$ by the
smooth submersion
\[
\mathrm{ev} \times \mathrm{ev} : \Lambda(M) \times \Lambda(M)
\rightarrow M \times M, \quad (\gamma_1,\gamma_2) \mapsto
(\gamma_1(0),\gamma_2(0)).
\]
Therefore, $\Theta(M)$ is a closed smooth submanifold of
$\Lambda(M)\times \Lambda(M)$, and its normal bundle  $N\Theta(M)$ is
$n$-dimensional, being isomorphic to the pull-back of the normal
bundle $N\Delta_M$ of $\Delta_M$ in $M\times M$ by the map $\mathrm{ev} \times
\mathrm{ev}$. The
  Banach manifold $\Lambda(M)$ does not admit smooth partitions of
  unity (actually, the Banach space $C^0(\T,\R^n)$ does not admit
  non-zero functions of class $C^1$ with bounded support). So in
  general a closed submanifold of $\Lambda(M)$, or of $\Lambda(M)
  \times \Lambda(M)$, will not have a smooth tubular
  neighborhood. However, it would not be difficult to show that the
  submanifold $\Theta(M)$ and all the submanifolds we consider
  in this paper do have a smooth tubular neighborhood, which can be
  constructed explicitly by using the exponential map and the tubular
  neighborhood theorem on finite-dimensional manifolds. 
If moreover $M$ is oriented, so is $N\Delta_M$ and thus
also $N\Theta(M)$. Notice also that the concatenation map $\Gamma$ is
well-defined and smooth from $\Theta(M)$ into $\Lambda(M)$. If we
denote by $e$ the inclusion of $\Theta(M)$ into $\Lambda(M) \times
\Lambda(M)$, the Chas-Sullivan {\em loop product} (see \cite{csu99})
is defined by the composition
\begin{equation*}\begin{split}
H_j(\Lambda(M)) \otimes H_k(\Lambda(M))
&\stackrel{\times}{\longrightarrow} H_{j+k} (\Lambda(M) \times
\Lambda(M)) \stackrel{e_!}{\longrightarrow} H_{j+k-n} (\Theta(M))\\
&\stackrel{\Gamma_*}{\longrightarrow} H_{j+k-n}(\Lambda(M)),
\end{split}\end{equation*}
and it is denoted by
\[
\loopprod: H_j(\Lambda(M)) \otimes H_k(\Lambda(M)) \rightarrow
H_{j+k-n}(\Lambda(M)).
\]
We denote by $\mathrm{c}:M \rightarrow \Lambda(M)$ the map which associates to
every $q\in M$ the constant loop $q$ in $\Lambda(M)$. A simple
homotopy argument shows that the image of the fundamental class
$[M]\in H_n(M)$ under the homomorphism $\mathrm{c}_*$ is a unit for the
loop product: $\alpha \loopprod \mathrm{c}_*[M] = \mathrm{c}_*[M] \loopprod \alpha =
\alpha$ for every $\alpha\in H_*(\Lambda(M))$. Since $\Gamma \circ
(\mathrm{id} \times \Gamma)$ and $\Gamma \circ (\Gamma \times
\mathrm{id})$ are homotopic on the space of triplets of loops with
the same initial points, the loop product turns out to be
associative. Finally, notice that the maps $(\gamma_1,\gamma_2)
\mapsto \Gamma(\gamma_1,\gamma_2)$ and $(\gamma_1,\gamma_2)
\mapsto \Gamma(\gamma_2,\gamma_1)$ are homotopic on $\Theta(M)$,
by the homotopy
\[
\Gamma_s(\gamma_1,\gamma_2) (t) := \left\{ \begin{array}{ll}
\gamma_2(2t-s) &
  \mbox{if } 0 \leq  t \leq s/2 \mbox{, or } (s+1)/2 \leq t
  \leq 1,\\ \gamma_1(2t-s) & \mbox{if } s/2
  \leq t \leq (s+1)/2. \end{array} \right.
\]
This fact implies the following commutation rule
\[
\beta \loopprod \alpha = (-1)^{(|\alpha|-n)(|\beta|-n)} \alpha \loopprod
\beta,
\]
for every $\alpha,\beta\in H_*(\Lambda(M))$.

In order to get a product of degree zero, it is convenient to shift
the grading by $n$, obtaining the graded group
\[
\mathbb{H}_j(\Lambda(M)) := H_{j+n}(\Lambda(M)),
\]
which becomes a graded commutative ring with respect to the loop
product (commutativity has to be understood in the graded sense, that is
$\beta \loopprod \alpha = (-1)^{|\alpha| |\beta|} \alpha \loopprod \beta$).

\subsection{Relationship between the two products}
\label{rbttp}

If $M$ is an oriented $n$-dimensional manifold, we denote by
\[
\bullet: H_j(M) \otimes H_k(M) \rightarrow H_{j+k-n} (M),
\]
the intersection product on the singular homology of $M$ (which is
obtained by composing the exterior homology product with the Umkehr map associated to the embedding of the diagonal into
$M\times M$). Shifting again the grading by $n$, we see that the
product $\bullet$ makes
\[
\mathbb{H}_j(M) := H_{j+n}(M)
\]
a commutative graded ring.

Being the inverse image of $q_0$ by the
submersion $\mathrm{ev} :\Lambda(M) \rightarrow M$,
$\mathrm{ev}(\gamma) = \gamma(0)$, $\Omega(M,q_0)$ is a closed
submanifold of $\Lambda(M)$, and its normal bundle is
$n$-dimensional and oriented. If $\mathrm{i}:\Omega(M,q_0) \hookrightarrow
\Lambda(M)$ is the inclusion map, we find that the following
diagram
\begin{equation}
\label{diaga}
\begin{CD}
H_j(M) \otimes H_k(M) @>{\bullet}>> H_{j+k-n}(M) \\
@V{\mathrm{c}_* \otimes \mathrm{c}_*}VV @VV{\mathrm{c}_*}V \\
H_j(\Lambda(M)) \otimes H_k(\Lambda(M)) @>{\loopprod}>>
H_{j+k-n}(\Lambda(M)) \\
@V{\mathrm{i}_! \otimes \mathrm{i}_!}VV @VV{\mathrm{i}_!}V \\
H_{j-n}(\Omega(M,q_0)) \otimes H_{k-n}(\Omega(M,q_0)) @>{\#}>>
H_{j+k-2n} (\Omega(M,q_0))
\end{CD}
\end{equation}
commutes. In other words, the maps
\[
\{\mathbb{H}_*(M), \bullet\} \stackrel{c_*}{\longrightarrow}
\{\mathbb{H}_*(\Lambda(M)), \loopprod\} \stackrel{\mathrm{i}_!}{\longrightarrow}
\{H_*(\Omega(M,q_0), \#\}
\]
are graded ring homomorphisms. Notice that the homomorphism $\mathrm{c}_*$ is
always injective onto a direct summand, the map $\mathrm{ev}$
being a left inverse of $\mathrm{c}$. Using the spectral sequence associated to the Serre fibration
\[
\Omega(M) \hookrightarrow \Lambda(M) \rightarrow M,
\]
it is possible to compute the ring $\{\mathbb{H}_*(\Lambda(M)), \loopprod\}$ from the intersection product on $M$ and the Pontrjagin product on $\Omega(M)$, when $M$ is a sphere or a projective space, see \cite{cjy03}.

The aim of this paper is to show how the homomorphisms
appearing in the diagram above can be described
symplectically, in terms of different Floer homologies of the cotangent
bundle of $M$, when the manifold $M$ is closed. The first step is to obtain a Morse theoretical chain level description of diagram (\ref{diaga}), by using suitable Morse functions on some infinite dimensional Hilbert manifolds having the homotopy type of
$\Omega(M,q_0)$, $\Lambda(M)$, and $\Theta(M)$..

\begin{rem}
The loop product was defined by Chas and Sullivan in \cite{csu99}, by
using intersection theory for transversal chains. The definition we
use here is due to Cohen and Jones \cite{cj02}. See also
\cite{csu99b,sul04,bcr06,coh06,chv06,ram06,sul07,cks08,cs08} for more information and for other interpretations of this product.
\end{rem}

\section{Morse chain level descriptions}
\label{mtdsec}

The aim of this section is to describe the Pontrjagin product $\#$, the loop product $\loopprod$, and the other homomorphisms appearing in diagram (\ref{diaga}) in a Morse theoretical way. As it is well known, the singular homology of a (possibly infinite dimensional) manifold $\mathcal{M}$ is isomorphic to the homology of the {\em Morse chain complex} associated to a suitable Morse function on $\mathcal{M}$, and most homomorphisms between singular homology groups can be read at the chain level using these Morse complexes. The constructions of the Morse complex and of various homomorphisms between them, in the infinite dimensional setting needed in this paper, are described in Appendix A.  

Because of technical reasons, Hilbert manifolds are easier to deal with than Banach manifolds. Therefore, the first thing to do is to replace the Banach manifolds $\Lambda(M)$ (continuous free loops on $M$),
$\Omega(M,q_0)$ (continuous loops based at $q_0$),  and $\Theta(M)$ (continuous figure-8 loops) by the Hilbert manifolds
\begin{eqnarray*}
\Lambda^1(M) := W^{1,2}(\T,M), \quad
\Omega^1(M,q_0) := \set{\gamma \in \Lambda^1(M)}{\gamma(0)=q_0},
\\
\Theta^1(M) := \set{(\gamma_1,\gamma_2) \in \Lambda^1(M) \times
\Lambda^1(M)}{\gamma_1(0)=\gamma_2(0)},
\end{eqnarray*}
where $W^{1,2}$ denotes the class of absolutely continuous curves whose
derivative is square integrable.
The inclusions
\[
\Lambda^1(M) \hookrightarrow \Lambda(M), \quad 
\Omega^1 (M,q_0) \hookrightarrow \Omega(M,q_0), \quad
 \Theta^1(M) \hookrightarrow
\Theta(M),
\]
are homotopy equivalences. Therefore, we can replace $\Lambda(M)$,
$\Omega(M,q_0)$,  and $\Theta(M)$ by $\Lambda^1(M)$,
$\Omega^1(M,q_0)$,  and $\Theta^1(M)$ in the constructions of
Section \ref{altopo} (notice that the concatenation of two curves of class $W^{1,2}$ is still of class $W^{1,2}$).

\subsection{The Morse complex of the Lagrangian action functional} 
\label{mclaf}

Applying the results of Appendix A, one could find Morse chain level descriptions of all the homomorphisms of diagram (\ref{diaga}) by using quite a general class of abstract Morse functions on $M$, $\Lambda^1(M)$, $\Omega^1(M,q_0)$, and $\Theta^1(M)$. However, in order to find a link with Floer theory, we wish to consider a special class of functions on the three latter manifolds, namely the action functionals associated to a (possibly time-dependent) Lagrangian $L$ on $TM$.

Since we wish to consider different boundary conditions, it is useful to unify the presentation by working with general {\em nonlocal conormal boundary conditions}. Let $Q$ be a closed manifold (in our applications, $Q$ is either $M$ or $M\times M$), and let $R$ be a closed submanifold of $Q \times Q$. Let $L: [0,1] \times TQ
\rightarrow \R$ be a smooth function such that there exist real numbers $\ell_1>0$ and $\ell_2> 0$ such that
\begin{enumerate}
\item[(L1)] $\nabla_{vv} L(t,q,v) \geq \ell_1 I$,
\item[(L2)] $|\nabla_{qq} L(t,q,v)|\leq \ell_2 (1+|v|^2)$, $|\nabla_{qv} L(t,q,v)|\leq \ell_2 (1+|v|)$,  $|\nabla_{vv} L(t,q,v)|\leq \ell_2$,
\end{enumerate}
for every $(t,q,v)$ in $[0,1] \times TQ$. Here we have fixed a Riemannian metric on $Q$, with associated norm $|\cdot|$, and $\nabla_v$, $\nabla_q$ denote the vertical and horizontal components of the gradient with respect to this metric and to the induced metric on $TQ$. The fact that $Q$ is compact implies that these conditions do not depend on the choice of the metric. Conditions (L1) implies that $L$ is strictly fiberwise convex and grows at least quadratically in $v$. Condition (L2) implies that $L$ grows at most quadratically in $v$.

The {\em Euler-Lagrange equation} for a curve $\gamma:[0,1] \mapsto Q$ can be written in local coordinates as
\begin{equation}
\label{lageq}
\frac{d}{dt} \partial_v L(t,\gamma(t),\gamma'(t)) = \partial_q
L(t,\gamma(t), \gamma(t)),
\end{equation}
and by (L1) the corresponding second order Cauchy problem is locally well-posed. By (L2), the Lagrangian action functional 
\[
\mathbb{S}_L(\gamma) := \int_0^1 L(t,\gamma(t),\gamma'(t))\, dt
\]
is continuously differentiable on the Hilbert manifold $W^{1,2}([0,1],Q)$. Moreover, it is twice Gateaux differentiable at every curve $\gamma$, but it is everywhere twice Fr\'ech\'et differentiable if and only if for every $(t,q)$ the function $v\mapsto L(t,q,v)$ is a polynomial of degree at most two on $T_q Q$ (see \cite{as08b}). In the latter case, if we also assume (L1), the homogeneous part of degree two in $v\mapsto L(t,q,v)$ should be a positive quadratic form, so $L$ is an {\em electro-magnetic} Lagrangian, i.e.\ it has the form
\[
L(t,q,v) = \frac{1}{2} \langle A(t,q) v,v \rangle + \langle \alpha(t,q), v \rangle - V(t,q),
\]
 where $\langle \cdot , \cdot \rangle$ denotes the duality pairing, $A(t,q): T_q Q \rightarrow T_q^* Q$ is a positive symmetric linear mapping smoothly depending on $(t,q)$ (the kinetic energy), $\alpha$ is a smoothly time dependent one-form (the magnetic potential), and $V$ is a smooth function (the scalar potential). In this case, the action functional $\mathbb{S}_L$ is actually  smooth. In some situations, one can restrict the attention to electromagnetic Lagrangians, as in \cite{aps08}. However, in view of the transversality issues coming from Floer homology, it is convenient to work with the more general class of Lagrangians which satisfy (L1) and (L2). The reason is that some issues concerning $J$-holomorphic curves on cotangent bundles become much simpler when $J$ is the {\em Levi-Civita} almost complex structure (see Section \ref{fhrssec}, and in particular Remark \ref{3.3}). Therefore, in this paper we prefer not to perturb the Levi-Civita $J$, and to achieve the transversality needed in Floer theory by perturbing the Hamiltonian, or equivalently, the Lagrangian. The class of electro-magnetic Lagrangians is too rigid for this purpose, whereas the fact that conditions (L1) and (L2) are stable with respect to $C^2$-small perturbations of $L$ allows to achieve transversality within this class.   
 
The closed submanifold $R$ of $Q\times Q$ defines the following boundary conditions for solutions $\gamma: [0,1]\times Q$ of the Euler-Lagrange equation (\ref{lageq}):
\begin{eqnarray}
\label{lagcon1}
\bigl(\gamma(0),\gamma(1)\bigr) \in Q, \\
\label{lagcon2}
D_v L(0,\gamma(0),\gamma'(0))[\xi_0] = D_v L(1,\gamma(1),\gamma'(1))[\xi_1], \quad \forall (\xi_0,\xi_1) \in T_{(\gamma(0),\gamma(1))} R.
\end{eqnarray}
These conditions take a simpler form when read in the Hamiltonian formulation, see Section \ref{fhpfo}. The set of solutions of (\ref{lageq})-(\ref{lagcon1})-(\ref{lagcon2}) is denoted by $\mathscr{P}^R(L)$. The elements of $\mathscr{P}^R(L)$ are precisely the critical points of $\mathbb{S}_L^R$, that is the restriction of the Lagrangian action functional $\mathbb{S}_L$ to the Hilbert submanifold
\[
W^{1,2}_R ([0,1],Q) := \set{\gamma\in W^{1,2}([0,1],Q)}{(\gamma(0),\gamma(1)) \in R}.
\]
A critical point $\gamma$ of $\mathbb{S}_L^R$ has always finite Morse index 
\[
i \bigl(\gamma; \mathbb{S}_L^R \bigr)< +\infty,
\]
and it is non-degenerate if and only if there are no non-zero Jacobi vector fields along $\gamma$ (i.e.\ solutions of the linear system which is obtained by linearizing  (\ref{lageq}) along $\gamma$) which satisfy the linear boundary conditions that one finds by linearizing (\ref{lagcon1})-(\ref{lagcon2}). 

Let us assume that all the elements of $\mathscr{P}^R(L)$ are non-degenerate. This assumption holds for a generic choice of $L$, in several reasonable topologies (a subset of a topological space is said to be {\em generic} if it is a countable intersection of open and dense sets). Although the functional $\mathbb{S}_L^R$ is in general not twice Fr\'ech\'et differentiable, it admits a smooth pseudo-gradient $X$, in the sense of Appendix \ref{mc}, which satisfies the conditions (X1)-(X2)-(X3)-(X4) of \ref{mc}. This fact is proved in \cite{as08b}. Hence, such a functional has a well-defined Morse complex
$M_*( \mathbb{S}_L^R )$, 
whose homology $HM_*( \mathbb{S}_L^R )$ is isomorphic to the singular homology of $W^{1,2}_R ([0,1],Q)$, as explained in Appendix \ref{mc}.

In this paper, we are interested in the following three boundary conditions.

\paragraph{Periodic boundary conditions.} Here $Q=M$ and $R=\Delta_M$ is the diagonal in $M\times M$. Under the extra assumption that $L\in C^{\infty}([0,1]\times TM)$ extends as a smooth 1-periodic function on $\R\times TM$, the set $\mathscr{P}^{\Delta_M}(L)$ is precisely the set 1-periodic solutions of the Euler-Lagrange equation (\ref{lageq}). We denote such a set also by $\mathscr{P}^{\Lambda}(L)$, and the restricted functional $\mathbb{S}_L^{\Delta_M}$ by $\mathbb{S}_L^{\Lambda} : \Lambda^1(M) \rightarrow \R$, noticing that $W^{1,2}_{\Delta_M}([0,1],M) = \Lambda^1(M)$. The non-degeneracy condition is just:
\begin{description}
\item[(L0)$^{\Lambda}$] every solution $\gamma\in
  \mathscr{P}^{\Lambda}(L)$ is non-degenerate, meaning that there are
  no non-zero periodic Jacobi vector fields along $\gamma$.
\end{description}
The Morse index of a critical point $\gamma$ of $\mathbb{S}_L^{\Lambda}$ is denoted by $i^{\Lambda}(\gamma)$, or by $i^{\Lambda}(\gamma;L)$ when the choice of the Lagrangian $L$ is not clear from the context. The homology of the Morse complex $M_*(\mathbb{S}^{\Lambda}_L)$ is isomorphic to the singular homology of $\Lambda^1(M)$, hence to the singular homology of $\Lambda(M)$.

\paragraph{Dirichlet boundary conditions.} Here $Q=M$ and $R$ consists of the single point $(q_0,q_0)$, for some fixed $q_0\in M$. The set $\mathscr{P}^{(q_0,q_0)}(L)$, that we denote also by $\mathscr{P}^{\Omega}(L)$,  is precisely the set of solutions $\gamma:[0,1] \rightarrow M$ of the Euler-Lagrange equation (\ref{lageq}) such that $\gamma(0)=\gamma(1)=q_0$. The space of curves $W^{1,2}_{(q_0,q_0)}([0,1],M)$ is just $\Omega^1(M,q_0)$, and the functional $\mathbb{S}_L^{(q_0,q_0)}$ is denoted also by
$\mathbb{S}_L^{\Omega}: \Omega^1(M,q_0) \rightarrow \R$. The non-degeneracy condition is now:
\begin{description}
\item[(L0)$^{\Omega}$] every solution $\gamma\in
  \mathscr{P}^{\Omega}(L)$ is non-degenerate, meaning that there are no
  non-zero Jacobi vector fields along $\gamma$ which vanish for $t=0$
  and for $t=1$.
\end{description}
The non-degeneracy condition (L0)$^{\Omega}$ is generic also in the smaller class of autonomous Lagrangians $L$, whereas (L0)$^{\Lambda}$ requires $L$ to be explicitly time-dependent, because if $L$ is autonomous then $\gamma'$ is always a Jacobi vector field along $\gamma$. The Morse index of a critical point $\gamma$ of $\mathbb{S}_L^{\Omega}$ is denoted by $i^{\Omega}(\gamma)=i^{\Omega}(\gamma;L)$. The homology of the Morse complex $M_*(\mathbb{S}^{\Omega}_L)$ is isomorphic to the singular homology of $\Omega^1(M,q_0)$, hence to the singular homology of $\Omega(M,q_0)$.

\paragraph{Figure-8 boundary conditions.} Here $Q=M\times M$ and $R=
\Delta_M^{(4)}$ is the fourth diagonal in $M^4$:
\[
\Delta_M^{(4)} := \set{(q,q,q,q)}{q\in M}.
\]
In this case, it is convenient (although not necessary) to chose $L$ of the form $L=L_1 \oplus L_2$, where
\[
L_1 \oplus L_2 \, (t,q,v) := L_1(t,q_1,v_1) + L_2(t,q_2,v_2), \quad \begin{array}{l} \forall t\in [0,1], \\ \forall q=(q_1,q_2)\in M\times M, \\\forall v=(v_1,v_2)\in T_{(q_1,q_2)} M\times M, \end{array}
\]
so that the elements of $\mathscr{P}^{\Delta^{(4)}_M} (L)$ are the pairs of $M$-valued curves $(\gamma_1,\gamma_2)$ such that each $\gamma_j$ solves the Euler-Lagrange equation induced by $L_j$, and such that the boundary conditions 
\begin{equation}
\label{thetabdry}
\gamma_1(0)=\gamma_1(1)=\gamma_2(0)=\gamma_2(1), \quad
\sum_{i=0}^1 \sum_{j=1}^2 (-1)^i D_v L_j
(i,\gamma_j(i),\gamma'_j(i)) = 0
\end{equation}
hold. The set of such solutions is denoted also by $\mathscr{P}^{\Theta}(L_1\oplus L_2)$. The space of curves $W^{1,2}_{\Delta^{(4)}_M}([0,1],M\times M)$ coincides with $\Theta^1(M)$, and the functional $\mathbb{S}_L^{\Delta^{(4)}_M}$ is denoted simply by $\mathbb{S}_L^{\Theta}:\Theta^1(M) \rightarrow \R$.
 The corresponding non-degeneracy condition is:
\begin{description}
\item[(L0)$^{\Theta}$] every solution $(\gamma_1,\gamma_2) \in
\mathscr{P}^{\Theta}(L_1\oplus L_2)$ is non-degenerate, meaning that
there are no non-zero pairs of Jacobi vector fields $(\xi_1,\xi_2)$
along $(\gamma_1,\gamma_2)$ such that
\begin{eqnarray*}
\xi_1(0) = \xi_1(1) = \xi_2(0) = \xi_2(1) , \\
\sum_{i=0}^1 \sum_{j=1}^2 \left( D_{qv} L_j (i,\gamma_j(i),
\gamma'_j(i)) \xi_j(i) + D_{vv} L_j (i,\gamma_j(i),
\gamma'_j(i)) \xi'_j(i) \right) = 0.
\end{eqnarray*}
\end{description}
This condition allows both $L_1$ and $L_2$ to be autonomous. It
also allows $L_1=L_2$, but this excludes the autonomous case
(otherwise pairs $(\gamma,\gamma)$ with $\gamma\in
\mathscr{P}^{\Lambda}(L_1) = \mathscr{P}^{\Lambda}(L_2)$
non-constant would violate (L0)$^{\Theta}$).
The Morse index of a critical point $\gamma$ of $\mathbb{S}_{L_1\oplus L_2}^{\Theta}$ is denoted by $i^{\Theta}(\gamma) = i^{\Theta}(\gamma;L_1\oplus L_2)$. The homology of the Morse complex $M_*(\mathbb{S}^{\Theta}_{L_1\oplus L_2})$ is isomorphic to the singular homology of $\Theta^1(M)$, hence to the singular homology of $\Theta(M)$.

\subsection{Morse description of the homomorphisms
  $\mathbf{c_*}$, $\mathrm{\mathbf{{ev}_*}}$, and $\mathbf{i_!}$}
\label{laf}

The aim of this section is to describe the chain maps between Morse complexes that are induced by the smooth maps
\begin{eqnarray*}
& \mathrm{c} : M \rightarrow \Lambda^1(M), \; c(q)(t) \equiv q, \quad
\mathrm{ev} : \Lambda^1(M) \rightarrow M, \; \mathrm{ev}(\gamma) =
\gamma(0), & \\ & \mathrm{i} : \Omega^1(M,q_0) \hookrightarrow \Lambda^1(M). &
\end{eqnarray*} 
Everything follows from the abstract results of Appendix \ref{funct} and \ref{inter}, provided that we check the transversality conditions which are required there.

It is straightforward to apply the results of Section
\ref{funct} to describe the homomorphisms
\[
\mathrm{c}_* : H_k(M) \rightarrow H_k(\Lambda^1(M)), \quad \mbox{and} \quad
\mathrm{ev}_* :  H_k(\Lambda^1(M)) \rightarrow H_k(M),
\]
in terms of the Morse complexes of $\mathbb{S}_L^{\Lambda}$ 
and of a smooth Morse function $f$ on $M$. Indeed, in the case of $\mathrm{c}_*$ one imposes
the condition
\begin{equation}
\label{noco5}
q \in \crit(f) \quad \implies \quad \mathrm{c}(q) \notin \mathscr{P}^{\Lambda}(L),
\end{equation}
which guarantees (\ref{noco1}) and (\ref{noco2}) (given any
Lagrangian $L$ satisfying (L0)$^{\Lambda}$, one can always find a
Morse function $f$ on $M$ such that (\ref{noco5}) holds, simply
because (L0)$^{\Lambda}$ implies that the Euler-Lagrange equation has
finitely many constant solutions). Then one can find a metric
$g$ on $M$ and a smooth pseudo-gradient $X^{\Lambda}_L$ for $\mathbb{S}_L^{\Lambda}$ such
that the vector fields $X^{\Lambda}_L$ and $-\mathrm{grad}_g f$
satisfy the Morse-Smale condition, and the restriction of $c$ to
the unstable manifold of each $q\in \crit(f)$ is transverse to the
stable manifold of each $\gamma \in \crit(\mathbb{S}_L^{\Lambda})$. When
$i(q;f)=i^{\Lambda}(\gamma)$, the intersection
\[
W^u(q;-\mathrm{grad}_g f) \cap \mathrm{c}^{-1} \left( W^s(\gamma;X^{\Lambda}_L)
\right)
\]
consists of finitely many points, each of which comes with an
orientation sign $\pm 1$. The algebraic sums $n_\mathrm{c}(q,\gamma)$ of these
signs provide us with the coefficients of a chain map
\[
M_k \mathrm{c} : M_k(f) \rightarrow M_k(\mathbb{S}_L^{\Lambda}),
\]
which in homology induces the homomorphism $\mathrm{c}_*$.

The map $\mathrm{ev}:\Lambda^1(M) \rightarrow M$ is a submersion, so
(\ref{noco1}) holds automatically. Condition (\ref{noco2}) instead is
implied by
\begin{equation}
\label{noco6}
\gamma \in \mathscr{P}^{\Lambda}(L) \quad \implies \quad \gamma(0)
\notin \crit (f),
\end{equation}
which again holds for a generic $f$, given $L$. Then, for a
generic metric $g$ on $M$ and a generic pseudo-gradient $X^{\Lambda}_L$ for $\mathbb{S}_L^{\Lambda}$, the intersections
\[
 W^u(\gamma;X^{\Lambda}_L) \cap \mathrm{ev}^{-1} \left(
W^s(q;-\mathrm{grad}_g f) \right), 
\]
consist of finitely many oriented points, which add up to the
integer $n_{\mathrm{ev}}(\gamma,q)$, for every $\gamma\in \mathscr{P}^{\Lambda}(L)$ and $q\in \crit(f),$ such that $i^{\Lambda}(\gamma) = i(q;f)$. 
These integers are the
coefficients of a chain map 
\[ 
M_k \mathrm{ev} : M_k
(\mathbb{S}_L^{\Lambda}) \rightarrow M_k (f),
\]
which in homology induces the homomorphism $\mathrm{ev}_*$.

\begin{rem}
The fact that $\mathrm{ev}_* \circ \mathrm{c}_* = \mathrm{id}_{H_*(M)}$,
together with the fact that the Morse complex is free, implies
that $M_* \mathrm{ev} \circ M_* \mathrm{c}$ is chain homotopic to the
identity on $M_*(f)$.
\end{rem}

We conclude this section by using the results of Section \ref{inter}
to describe the homomorphism
\[
\mathrm{i}_! : H_k(\Lambda^1(M)) \rightarrow H_{k-n} (\Omega^1(M,q_0)),
\]
working with the same action functional $\mathbb{S}_L$ both on
$\Lambda^1(M)$ and on $\Omega^1(M,q_0)$. Conditions (\ref{noco3a}) and
(\ref{noco3b}) are implied by the following assumption,
\begin{equation}
\label{noperq0}
\gamma \in \mathscr{P}^{\Lambda} (L) \quad \implies \quad \gamma(0)
\neq q_0,
\end{equation}
a condition holding for all but a countable set of $q_0$'s.
Under this assumption, up to the perturbation of the Morse-Smale pseudo-gradient vector fields $X^{\Lambda}_L$ and $X^{\Omega}_L$ for $\mathbb{S}_L^{\Lambda}$ and $\mathbb{S}_L^{\Omega}$, we may assume that the
unstable manifold of each $\gamma_1\in \mathscr{P}^{\Lambda}(L)$ is
transverse to the stable manifold of each $\gamma_2\in
\mathscr{P}^{\Omega}(L)$ in $\Lambda^1(M)$. When
$i^{\Omega}(\gamma_2) = i^{\Lambda}(\gamma_1)-n$, the set
\[
W^u(\gamma_1;X^{\Lambda}_L) \cap W^s(\gamma_2 ; X^{\Omega}_L)
\]
consists of finitely many oriented points, which determine the
integer $n_{\mathrm{i}_!}(\gamma_1,\gamma_2)$. These integers are the
coefficients of a chain map
\[
M{\mathrm{i}_!} : M_*(\mathbb{S}_L^{\Lambda}) \rightarrow
M_{*-n}(\mathbb{S}_L^{\Omega}),
\]
which in homology induces the homomorphism $\mathrm{i}_!$.

\subsection{Morse description of the Pontrjagin product}
\label{mtipp}

In the last section we have described the vertical arrows of
diagram (\ref{diaga}), as well as a preferred left inverse of the
top-right vertical arrow. The top horizontal arrow is
described at the end of Appendix \ref{inter}. It remains to
describe the middle and the bottom horizontal arrows, that is the
loop product and the Pontrjagin product. This section is devoted
to the description of the Pontrjagin product, while in the next one we will deal with the loop product.
The following Propositions \ref{prop MS}, \ref{prop e!} and \ref{prop MG} are consequences of the general statements in Appendix A. 

Given two Lagrangians $L_1,L_2\in C^{\infty}([0,1] \times TM)$ which satisfy (L0)$^{\Omega}$, (L1), (L2), and are
such that $L_1(1,\cdot)=L_2(0,\cdot)$ with all the time derivatives,
we define a Lagrangian $L_1 \# L_2 \in C^{\infty}([0,1] \times
TM)$ in the same class by setting
\begin{equation}
\label{ide} L_1 \# L_2 \, (t,q,v) = \left\{
\begin{array}{ll} 2 L_1(2t,q,v/2) & \mbox{if } 0\leq t \leq 1/2, \\
2 L_2 (2t-1,q,v/2) & \mbox{if } 1/2 \leq t \leq 1. \end{array}
\right.
\end{equation}
The curve $\gamma:[0,1] \rightarrow M$ is a solution of the
Euler-Lagrange equation (\ref{lageq}) with $L=L_1 \# L_2$ if and only
if the rescaled curves $t\mapsto \gamma(t/2)$ and $t\mapsto
\gamma((t+1)/2)$ solve the corresponding equation given by the
Lagrangians $L_1$ and $L_2$, on $[0,1]$. We assume that also $L_1 \# L_2$ satisfies (L0)$^{\Omega}$.

In view of the results of Section \ref{exte},
we wish to consider the functional $\mathbb{S}_{L_1}^{\Omega} \oplus
\mathbb{S}_{L_2}^{\Omega}$ on $\Omega^1(M,q_0) \times \Omega^1(M,q_0)$,
\[
\mathbb{S}_{L_1}^{\Omega} \oplus \mathbb{S}_{L_2}^{\Omega} 
(\gamma_1,\gamma_2) = \mathbb{S}_{L_1}^{\Omega}(\gamma_1) + 
\mathbb{S}_{L_2}^{\Omega}(\gamma_2),
\]
and the functional $\mathbb{S}_{L_1 \# L_2}^{\Omega}$ on
$\Omega^1(M,q_0)$. The concatenation map
\[
\Gamma: \Omega^1(M,q_0) \times \Omega^1(M,q_0) \rightarrow \Omega^1(M,q_0)
\]
is nowhere a submersion, so condition (\ref{noco1}) for the
triplet $(\Gamma,\mathbb{S}_{L_1}^{\Omega} \oplus \mathbb{S}_{L_2}^{\Omega},
\mathbb{S}_{L_1 \# L_2}^{\Omega})$ requires that the image of $\Gamma$
does not meet the critical set of $\mathbb{S}_{L_1 \# L_2}^{\Omega}$, that
is
\begin{equation}
\label{nofig8} \gamma \in \mathscr{P}^{\Omega}(L_1 \# L_2) \quad
\implies \quad \gamma (1/2) \neq q_0.
\end{equation}
Notice that (\ref{nofig8}) allows $L_1$ and $L_2$ to be equal, and
actually it allows them to be also autonomous (however, it implies
that $q_0$ is not a stationary solution, so they cannot be the
Lagrangian associated to a geodesic flow).

Assuming the generic condition (\ref{nofig8}), condition (\ref{noco2}) is automatically
fulfilled. Moreover, if $X^{\Omega}_{L_1}$ and $X^{\Omega}_{L_2}$ are pseudo-gradients for $\mathbb{S}_{L_1}^{\Omega}$ and $\mathbb{S}_{L_2}^{\Omega}$, 
we have that for every $\gamma_1\in \mathscr{P}^{\Omega}(L_1)$ and $\gamma_2\in
\mathscr{P}^{\Omega}(L_2)$,
\[
\Gamma \Bigl(W^u\bigl((\gamma_1,\gamma_2); X_{L_1}^{\Omega} \oplus X_{L_2}^{\Omega}\bigr)\Bigr) 
\cap \crit
(\mathbb{S}_{L_1 \# L_2}^{\Omega}) = \emptyset.
\]
By Remark \ref{nntp}, there is no need to perturb the
vector field $X^{\Omega}_{L_1} \oplus X^{\Omega}_{L_2}$ 
on $\Omega^1(M,q_0) \times \Omega^1(M,q_0)$ to
achieve transversality, and we arrive at the following description of the
Pontrjagin product.

Let $X^{\Omega}_{L_1}$, $X^{\Omega}_{L_2}$, and $X_{L_1 \# L_2}^{\Omega}$ be Morse-Smale pseudo-gradients for $\mathbb{S}_{L_1}^{\Omega}$, $\mathbb{S}_{L_1}^{\Omega}$, and $\mathbb{S}_{L_1 \# L_2}^{\Omega}$, respectively. Fix an arbitrary
orientation of the unstable manifolds of each critical point of
$\mathbb{S}_{L_1}^{\Omega}$, $\mathbb{S}_{L_2}^{\Omega}$, 
$\mathbb{S}_{L_1 \#
L_2}^{\Omega}$. Up to the perturbation of $X_{L_1 \# L_2}^{\Omega}$, we get that the
restriction of $\Gamma$ to the unstable manifold
\[
W^u((\gamma_1,\gamma_2);X_{L_1}^{\Omega} \oplus X_{L_2}^{\Omega}) = 
W^u(\gamma_1;X_{L_1}^{\Omega})
\times W^u(\gamma_2;X_{L_2}^{\Omega})
\]
of every critical point $(\gamma_1,\gamma_2) \in \mathscr{P}^{\Omega}
(L_1) \times \mathscr{P}^{\Omega}(L_2)$ is transverse to the stable manifold
\[
W^s(\gamma;X_{L_1 \# L_2}^{\Omega})
\]
of each critical point of $\mathbb{S}_{L_1 \# L_2}^{\Omega}$.
When $i^{\Omega}(\gamma) =
i^{\Omega}(\gamma_1) + i^{\Omega}(\gamma_2)$, the
corresponding intersections
\[
\set{ (\alpha_1,\alpha_2)\in W^u(\gamma_1;;X_{L_1}^{\Omega})
\times W^u(\gamma_2;X_{L_2}^{\Omega})}{\Gamma(\alpha_1,\alpha_2)
  \in W^s(\gamma;X_{L_1 \# L_2}^{\Omega} )},
\]
is a finite set of oriented points. Let
$n_{\#}(\gamma_1,\gamma_2;\gamma)$ be the algebraic sum of these
orientation signs.

\begin{prop}\label{prop MS}
The homomorphism
\[
M_{\#} : M_j(\mathbb{S}_{L_1}^{\Omega}) \otimes
M_k(\mathbb{S}_{L_2}^{\Omega})
\rightarrow M_{j+k} (\mathbb{S}_{L_1 \# L_2}^{\Omega}) , \;
\gamma_1\otimes \gamma_2
\mapsto \sum_{\substack{\gamma \in \mathscr{P}^{\Omega}(L_1 \# L_2) \\
    i^{\Omega}(\gamma) = j+k}}
n_{\#}(\gamma_1,\gamma_2;\gamma) \, \gamma,
\]
is a chain map, and it induces the Pontrjagin product $\#$ in
homology.
\end{prop}

\begin{rem}
It is not necessary to consider the Lagrangian $L_1 \# L_2$ on the
target space of this homomorphism. One could actually work with
any three Lagrangians (with the suitable non-degeneracy condition
replacing (\ref{nofig8})). The choice of dealing with two
Lagrangians $L_1,L_2$ and their concatenation $L_1 \# L_2$ is
important to get energy estimates in Floer homology. We have made
this choice also here mainly to see which kind of non-degeneracy
condition one needs.
\end{rem}

\subsection{Morse description of the loop product}
\label{mtilp}

The loop product is slightly more complicated than the other
homomorphisms considered so far, because it consists of a
composition where two homomorphisms are non-trivial (that is, not
just identifications) when read on the Morse homology groups,
namely the Umkehr map associated to the submanifold
$\Theta^1(M)$ of figure-8 loops, and the homomorphism induced by
the concatenation map $\Gamma: \Theta^1(M) \rightarrow
\Lambda^1(M)$. We shall describe these homomorphisms separately,
and then we will show a compact description of their composition. See also \cite{cs08}, where these Morse theoretical descriptions are extended to more general graphs than the figure-8.

Let us start by describing the Umkehr map
\[
e_! : H_k(\Lambda^1(M) \times \Lambda^1(M)) \rightarrow
H_{k-n}(\Theta^1(M)).
\]
Let $L_1,L_2\in C^{\infty}(\T\times TM)$ be Lagrangians which satisfy (L0)$^{\Lambda}$, (L1), (L2), and are such that $L_1 \oplus L_2$
satisfies (L0)$^{\Theta}$. Assume also
\begin{equation}
\label{nofig8s}
\gamma_1\in \mathscr{P}^{\Lambda} (L_1), \; \gamma_2\in
\mathscr{P}^{\Lambda} (L_2) \quad \implies \quad \gamma_1(0) \neq
\gamma_2(0).
\end{equation}
Notice that this condition prevents $L_1$ from coinciding with $L_2$, but it holds for a generic pair $(L_1,L_2)$.
We shall consider the functional
$\mathbb{S}_{L_1}^{\Lambda} \oplus \mathbb{S}_{L_2}^{\Lambda}$ on 
$\Lambda^1(M) \times \Lambda^1(M)$, and the functional
$\mathbb{S}_{L_1 \oplus L_2}^{\Theta}$ on $\Theta^1(M)$. Condition
(\ref{nofig8s}) implies that the unconstrained functional $\mathbb{S}_{L_1}^{\Lambda} \oplus \mathbb{S}_{L_2}^{\Lambda}$ has no
critical points on $\Theta^1(M)$, so conditions (\ref{noco3a}) and
(\ref{noco3b}) hold.

By the discussion of Appendix \ref{inter}, we can find smooth Morse-Smale pseudo-gradients $X_{L_1\oplus L_2}^{\Lambda}$ and $X_{L_1 \oplus L_2}^{\Theta}$ for $\mathbb{S}_{L_1}^{\Lambda} \oplus \mathbb{S}_{L_2}^{\Lambda}$ and $\mathbb{S}_{L_1\oplus L_2}^{\Theta}$ respectively, such that the unstable manifold
$W^u(\gamma^-; X_{L_1 \oplus L_2}^{\Lambda})$ of every
  $\gamma^- = (\gamma_1^-,\gamma_2^-)\in \mathscr{P}^{\Lambda}(L_1) \times
  \mathscr{P}^{\Lambda} (L_2)$ is transverse to $\Theta^1(M)$ and to
  the stable manifold $W^s(\gamma^+;X_{L_1\oplus L_2}^{\Theta})$ of every
  $\gamma^+ \in
  \mathscr{P}^{\Theta}(L_1\oplus L_2)$. Actually, it is convenient to assume that $X_{L_1\oplus L_2}^{\Lambda}$ is so close to $X_{L_1}^{\Lambda} \oplus X_{L_2}^{\Lambda}$ that the Morse complex of $(\mathbb{S}_{L_1}^{\Lambda} \oplus \mathbb{S}_{L_2}^{\Lambda}, X_{L_1\oplus L_2}^{\Lambda})$ equals the Morse complex of $(\mathbb{S}_{L_1}^{\Lambda} \oplus \mathbb{S}_{L_2}^{\Lambda}, X_{L_1}^{\Lambda} \oplus X_{L_2}^{\Lambda})$, and that $X_{L_1\oplus L_2}^{\Lambda} = X_{L_1}^{\Lambda} \oplus X_{L_2}^{\Lambda}$ up to order one at the critical points.
  
Fix an arbitrary orientation of the unstable manifold of every critical point of $\mathbb{S}_{L_1}^{\Lambda}$, $\mathbb{S}_{L_2}^{\Lambda}$, and $\mathbb{S}_{L_1\oplus
L_2}^{\Theta}$. By our assumptions on $X_{L_1\oplus L_2}^{\Lambda}$, also the unstable manifolds $W^u(\gamma^-; X_{L_1 \oplus L_2}^{\Lambda})$ are oriented.
When $i^{\Theta}(\gamma^+) = i^{\Lambda}(\gamma_1^-) +
i^{\Lambda}(\gamma_2^-) - n$, the intersection
\[
W^u(\gamma^-; X_{L_1\oplus L_2}^{\Lambda}) \cap
W^s(\gamma^+;X_{L_1 \oplus L_2}^{\Theta})
\]
is a finite set of oriented points. If we denote by
$n_{e_!}(\gamma^-,\gamma^+)$ the algebraic sum of these orientation
signs, we have the following:

\begin{prop}\label{prop e!}
The homomorphism
\[
M_k(\mathbb{S}_{L_1}^{\Lambda} \oplus \mathbb{S}_{L_2}^{\Lambda}) \rightarrow M_{k-n}(\mathbb{S}_{L_1\oplus
  L_2}^{\Theta}), \quad
\gamma^- \mapsto \sum_{\substack{\gamma^+\in
  \mathscr{P}^{\Theta}(L_1\oplus L_2)\\ i^{\Theta}(\gamma^+) = k-n}} n_{e_!}
  (\gamma^-,\gamma^+) \gamma^+,
\]
is a chain map, and it induces the Umkehr map
$e_!$ in homology.
\end{prop}

By composing this homomorphism with the Morse theoretical version of
the exterior homology product described in Section \ref{exte}, that is
the isomorphism
\[
M_j (\mathbb{S}_{L_1}^{\Lambda}) \otimes M_h
(\mathbb{S}_{L_2}^{\Lambda}) \rightarrow M_{j+h} (
\mathbb{S}_{L_1}^{\Lambda} \oplus \mathbb{S}_{L_2}^{\Lambda}),
\]
we obtain the homomorphism
\[
M_{!} : M_j (\mathbb{S}_{L_1}^{\Lambda}) \otimes M_h
(\mathbb{S}_{L_2}^{\Lambda}) \rightarrow M_{j+h-n} (\mathbb{S}_{L_1\oplus
  L_2}^{\Theta}).
\]
Here we are using the fact that the Morse complexes of $(\mathbb{S}_{L_1}^{\Lambda} \oplus \mathbb{S}_{L_2}^{\Lambda}, X_{L_1\oplus L_2}^{\Lambda})$ and $(\mathbb{S}_{L_1}^{\Lambda} \oplus \mathbb{S}_{L_2}^{\Lambda}, X_{L_1}^{\Lambda} \oplus X_{L_2}^{\Lambda})$ coincide.

Let us describe the homomorphism
\[
\Gamma_* : H_k (\Theta^1(M)) \rightarrow H_k(\Lambda^1(M)),
\]
induced by the concatenation map $\Gamma$. Let $L_1,L_2\in C^{\infty}([0,1]\times TM)$ be Lagrangians which satisfy (L1), (L2), and are such that $L_1(1,\cdot)=L_2(0,\cdot)$ and $L_2(1,\cdot)=L_1(0,\cdot)$ with all the time
derivatives. We assume that $L_1 \oplus L_2$
satisfies (L0)$^{\Theta}$ and that the time periodic Lagrangian  
$L_1\# L_2$ satisfies (L0)$^{\Lambda}$. We would like to apply the results of Section
\ref{funct} to the functionals $\mathbb{S}_{L_1 \oplus
L_2}^{\Theta}$ on $\Theta^1(M)$ and
$\mathbb{S}_{L_1 \# L_2}^{\Lambda}$ on $\Lambda^1(M)$. The map
$\Gamma:\Theta^1(M) \rightarrow \Lambda^1(M)$ is nowhere a
submersion, so condition (\ref{noco1}) for the triplet
$(\Gamma,\mathbb{S}_{L_1 \oplus
L_2}^{\Theta}, \mathbb{S}_{L_1 \# L_2}^{\Lambda})$
requires that $\Gamma(\Theta^1(M))$ does not contain critical
points of $\mathbb{S}_{L_1 \# L_2}^{\Lambda}$. The latter fact is equivalent to the condition
\begin{equation}
\label{nofig8n} \gamma \in \mathscr{P}^{\Lambda}({L_1 \# L_2})
\quad \implies \quad \gamma (1/2) \neq \gamma(0),
\end{equation}
which holds for a generic pair $(L_1,L_2)$.
Assuming (\ref{nofig8n}), conditions (\ref{noco1}) and
(\ref{noco2}) are automatically fulfilled. Therefore, the
discussion of Section \ref{funct} implies that we can find
Morse-Smale pseudo-gradients $X_{L_1 \oplus L_2}^{\Theta}$ and $X_{L_1 \# L_2}^{\Lambda}$ for $\mathbb{S}_{L_1 \oplus L_2}^{\Theta}$ and $\mathbb{S}_{L_1 \# L_2}^{\Lambda}$, respectively, such that the
restriction of $\Gamma$ to the unstable manifold
\[
W^u(\gamma^-;X_{L_1 \oplus L_2}^{\Theta})
\]
of every critical point $\gamma^- = (\gamma_1^-,\gamma_2^-) \in
\mathscr{P}^{\Theta}(L_1\oplus L_2)$ is transverse to the stable manifold
\[
W^s(\gamma^+;X_{L_1 \# L_2}^{\Lambda})
\]
of every critical point $\gamma^+ \in \mathscr{P}^{\Lambda}({L_1
\# L_2})$. Fix arbitrary orientations for the unstable manifolds
of every critical point of $\mathbb{S}_{L_1 \oplus
L_2}^{\Theta}$ and $\mathbb{S}_{L_1 \# L_2}^{\Lambda}$. When
$i^{\Lambda}(\gamma^+) = i^{\Theta}(\gamma^-)$, the intersection
\[
\set{ (\alpha_1,\alpha_2)\in W^u(\gamma^-;X_{L_1 \oplus L_2}^{\Theta})}{\Gamma(\alpha_1,\alpha_2)
  \in W^s(\gamma^+;X_{L_1 \# L_2}^{\Lambda} )},
\]
is a finite set of oriented points. If we denote by
$n_{\Gamma}(\gamma^-,\gamma^+)$ the algebraic sum of these
orientation signs, we have the following:

\begin{prop}\label{prop MG}
The homomorphism
\[
M_{\Gamma}: M_k(\mathbb{S}_{L_1\oplus
L_2}^{\Theta}) \rightarrow M_k
(\mathbb{S}_{L_1 \# L_2}^{\Lambda}), \quad \gamma^-
\mapsto \sum_{\substack{\gamma^+ \in \mathscr{P}^{\Lambda}(L_1 \# L_2) \\
    i^{\Lambda}(\gamma^+) = k}}
n_{\Gamma}(\gamma^-,\gamma^+) \, \gamma^+,
\]
is a chain map, and it induces the homomorphism $\Gamma_* :
H_k(\Theta^1(M)) \rightarrow H_k(\Lambda^1(M))$ in
homology.
\end{prop}

\noindent Therefore, the composition $M_{\Gamma} \circ M_{!}$ induces the loop
product in homology.\\

We conclude this section by exhibiting a compact description of the
loop product
\[
\loopprod : H_j(\Lambda^1(M)) \otimes H_k(\Lambda^1(M)) \rightarrow
H_{j+k-n} (\Lambda^1(M)).
\]
Since we are not going to use this description, we omit the proof.
Let $L_1,L_2\in C^{\infty}(\T\times TM)$ be Lagrangians which satisfy (L0)$^{\Lambda}$, (L1), (L2), such that
$L_1(0,\cdot)=L_2(0,\cdot)$ with all time derivatives, and such
that the concatenated Lagrangian $L_1 \#L_2$ defined by
(\ref{ide}) satisfies (L0)$^{\Lambda}$. We also assume
(\ref{nofig8s}), noticing that this condition prevents $L_1$ from
coinciding with $L_2$. Let $X_{L_1 \oplus L_2}^{\Lambda}$ and $X_{L_1 \# L_2}^{\Lambda}$ be Morse-Smale pseudo-gradients for $\mathbb{S}_{L_1}^{\Lambda} \oplus \mathbb{S}_{L_2}^{\Lambda}$ and $\mathbb{S}_{L_1 \# L_2}^{\Lambda}$, respectively.
By (\ref{nofig8s}), the
functional $\mathbb{S}_{L_1 \oplus L_2}^{\Theta}$ has no
critical points on $\Theta^1(M)$, so up to the perturbation of $X_{L_1 \oplus L_2}^{\Lambda}$, we can assume that for every $\gamma_1\in
\mathscr{P}^{\Lambda}(L_1)$, $\gamma_2\in
\mathscr{P}^{\Lambda}(L_2)$, the unstable manifold
$W^u((\gamma_1,\gamma_2);X_{L_1 \oplus L_2}^{\Lambda} )$
is transverse to $\Theta^1(M)$. Moreover, assumption
(\ref{nofig8s}) implies that the image of $\Theta^1(M)$ by the
concatenation map $\Gamma$ does not contain any critical point of
$\mathbb{S}_{L_1 \# L_2}^{\Lambda}$. Therefore, up to the perturbation of $X_{L_1 \# L_2}^{\Lambda}$, we can
assume that for every $\gamma_1\in \mathscr{P}^{\Lambda}(L_1)$,
$\gamma_2\in \mathscr{P}^{\Lambda}(L_2)$, the restriction of
$\Gamma$ to the submanifold
\[
W^u((\gamma_1,\gamma_2); X_{L_1 \oplus L_2}^{\Lambda}) \cap \Theta^1(M)
\]
is transverse to the stable manifold $W^s(\gamma;X_{L_1 \# L_2}^{\Lambda})$ of each $\gamma \in
\mathscr{P}^{\Lambda}(L)$. In particular, when
$i^{\Lambda}(\gamma) = i^{\Lambda}(\gamma_1) +
i^{\Lambda}(\gamma_2) - n$, the submanifold
\[
\set{ (\alpha_1,\alpha_2) \in  W^u((\gamma_1,\gamma_2); X_{L_1 \oplus L_2}^{\Lambda}) \cap
\Theta^1(M)}{\Gamma(\alpha_1,\alpha_2) \in W^s(\gamma;X_{L_1 \# L_2}^{\Lambda})}
\]
is a finite set of oriented points. We contend that if
$n_{\loopprod}(\gamma_1,\gamma_2;\gamma)$ denotes the algebraic sum of
the corresponding orientation signs, the following holds:

\begin{prop}
\label{nonserve}
The homomorphism
\begin{equation}
\label{compdescr} M_j(\mathbb{S}_{L_1}^{\Lambda}) \otimes
M_k(\mathbb{S}_{L_2}^{\Lambda}) 
\rightarrow M_{j+k-n} (\mathbb{S}_{L_1 \# L_2}^{\Lambda}) , 
\quad \gamma_1 \otimes
\gamma_2 \mapsto \sum_{\substack{\gamma\in \mathscr{P}^{\Lambda}(L_1 \# L_2) \\
    i^{\Lambda} (\gamma) = j+k-n}}
n_{\loopprod}(\gamma_1,\gamma_2;\gamma) \, \gamma,
\end{equation}
is a chain map, and it induces the loop product in homology.
\end{prop}

\section{Floer homologies on cotangent bundles and their ring structures}
\label{fhrssec}

\subsection{Floer homology for nonlocal conormal boundary conditions}
\label{fhpfo}

In this section we recall the construction of Floer homology for Hamiltonian orbits on cotangent bundles with
{\em nonlocal conormal boundary conditions}. This is the Hamiltonian version of the setting described in Section \ref{mclaf}.
See \cite{as06} and \cite{aps08} for detailed proofs.

Let $Q$ be a closed manifold (in our applications, $Q$ is either $M$ or $M 
\times M$), and let $R$ be a closed submanifold of $Q\times Q$. We shall often denote the elements of the cotangent bundle $T^*Q$ as pairs $(q,p)$, where $q\in Q$ and $p\in T_q^*Q$. Let $\omega= dp\wedge dq$ be the standard symplectic form on the manifold $T^*Q$, that is the
differential of the Liouville form $\eta := p\, dq$. Equivalently, the Liouville form $\eta$ can be defined by
\[
\eta(\zeta) = x(D\pi(x)[\zeta]), \quad \mbox{for } \zeta\in
T_xT^*M, \; x\in T^*Q,
\]
where $\pi:T^*Q \rightarrow Q$ is the bundle projection.
Let $Y=p\, \partial_p$ be the standard Liouville vector field on $T^*Q$, which is
defined by the identity
\begin{equation}
\label{Liouv}
\omega(Y,\cdot) = \eta.
\end{equation}
Consider the class of smooth Hamiltonians $H$ on 
$[0,1]\times T^*Q$ such that:
\begin{enumerate}
\item[(H1)] $DH(t,q,p)[Y] - H(t,q,p) \geq h_0 |p|^2 - h_1$,
\item[(H2)] $|\nabla_q H(t,q,p)| \leq h_2 (1+|p|^2)$, $|\nabla_p
  H(t,q,p)| \leq h_2 (1+|p|)$,
\end{enumerate}
for every $(t,q,p)$, for some constants $h_0>0$, $h_1\in \R$, $h_2> 0$. Here $\nabla_q$ and $\nabla_p$ denote the horizontal and vertical components of the gradient, with respect to a Riemannian metric on $Q$ and to the induced metric on $T^*Q$, and $|\cdot|$ denotes the norm associated to such a metric. The fact that $Q$ is compact easily implies that (H1) and (H2) do not depend on the choice of this metric. Condition (H1) essentially says that $H$ grows at
least quadratically in $p$ on each fiber of $T^*M$, and that it is
radially convex for $|p|$ large. Condition (H2) implies that $H$
grows at most quadratically in $p$ on each fiber. Notice also that if
$H$ is the Fenchel transform of a fiber-wise strictly convex Lagrangian
$L$ in $C^{\infty}([0,1]\times TM)$ (see Section \ref{cci}), then the term 
$DH(t,q,p)[Y(q,p)] - H(t,q,p)$ appearing in (H1) coincides with 
$L(t,q,D_p H(t,q,p))$.

Let $X_H$ be the time-dependent Hamiltonian vector field
associated to $H$ by the formula $\omega(X_H,\cdot) = -D_x
H$. Condition (H2) implies the quadratic bound
\begin{equation}
\label{cresc}
|X_H(t,q,p)| \leq h_3 (1+|p|^2),
\end{equation}
for some $h_3> 0$. Let
$(t,x) \mapsto \phi^H(t,x)$ be the non-autonomous flow associated
to the vector field $X_H$.

The {\em conormal bundle} $N^*R$ of $R$ in $Q\times Q$ is the set of covectors $x=(q,p)$ in $T^* (Q\times Q)$ such that $q=(q_1,q_2)$ belongs to $R$ and $p=(p_1,p_2)\in T_{(q_1,q_2)}^* (Q\times Q)$ vanishes identically on $T_{(q_1,q_2)} R$. It is a vector bundle over $R$ and the dimension of its fibers equals the codimension of $R$ in $Q\times Q$. 
The Liouville form of $T^* (Q\times Q)$ vanishes on $N^* R$, so in particular $N^* R$ is a Lagrangian submanifold of $T^* (Q\times Q)$. Actually, conormal bundles can be characterized as those middle dimensional closed submanifolds of a cotangent bundle on which the Liouville form vanishes, see Proposition 2.1 in \cite{aps08}.    

We are interested in the set $\mathscr{P}^R(H)$ of solutions $x:[0,1]\rightarrow T^*Q$ of the Hamiltonian equation
\begin{equation}
\label{hameq}
x'(t) = X_H(t,x(t)),
\end{equation}
which satisfy the boundary conditions
\begin{equation}
\label{conbdry}
(x(0), \mathscr{C} x(1)) \in N^* R,
\end{equation}
where $\mathscr{C}$ is the anti-symplectic involution
\[
\mathscr{C} : T^* Q \rightarrow T^*Q, \quad (q,p) \mapsto (q,-p).
\]
Equivalently, we are looking at the intersections of the Lagrangian submanifold $N^* R$ with the Lagrangian submanifold given by the graph of $\mathscr{C} \circ\phi^H(1,\cdot)$  (we are always considering the standard symplectic form $\omega \oplus \omega$ on $T^* (Q\times Q)$, not the flipped one $\omega \oplus (-\omega)$, so a diffeomorphism $\varphi: T^* Q \rightarrow T^*Q$ is symplectic if and only if the graph of the composition $\mathscr{C} \circ \varphi$ is Lagrangian).

We assume that:
\begin{enumerate}
\item[(H0)] All the elements of $\mathscr{P}^R(H)$ are {\em non-degenerate}, meaning that the submanifolds $N^* R$ and $\graf \mathscr{C} \circ\phi^H(1,\cdot)$ intersect transversally. 
\end{enumerate}
Given $x\in \mathscr{P}^R(H)$, we can conjugate the differential of the flow $D_x \phi^H(t,x(t))$ by a symplectic trivialization $\Psi$ of $x^*(TT^*Q)$ and we obtain a path 
\[
G^{\Psi}: [0,1] \rightarrow \mathrm{Sp} (2m), \quad m:= \dim Q, 
\]
of symplectic automorphisms of $T^* \R^m$, endowed with its standard symplectic structure. 

We assume that the symplectic trivialization $\Psi$ is {\em vertical-preserving}, meaning that it maps the vertical subbundle $T^v T^*M := \ker D\pi$ into the vertical space $N^* (0) = (0) \times (\R^m)^* \subset T^* \R^m$, and has the property that the following linear subspace of $T^* \R^m \times T^* \R^m = T^* \R^{2m}$,
\[
\bigl(\Psi(0) \times C\Psi(1) D\mathscr{C}(\mathscr{C} x(1))\bigr) T_{(x(0), \mathscr{C} 
x(1))} N^* R,
\]
is the conormal space $N^* W^{\Psi}$ of some linear subspace $W^{\Psi} \subset \R^m \times \R^m$, where $C$ is the anti-symplectic involution of $T^* \R^m$ which maps $(q,p)$ into $(q,-p)$. Then the {\em Maslov index of $x$ with respect to the nonlocal boundary condition induced by $R$} is defined as
\begin{equation}
\label{mascon}
\mu^R(x) := \mu (\graf G^{\Psi} C, N^* W^{\Psi})  + \frac{1}{2} (\dim R -\dim Q),
\end{equation}
where $\mu$ denotes the relative Maslov index of two paths of Lagrangian subspaces of $T^* \R^{2m}$, in the sense of \cite{rs93}  (see also Section \ref{tmi} for sign conventions). The fact that $x$ is non-degenerate implies that $\mu^R(x)$ is an integer, and the
assumptions on the trivialization $\Psi$ imply that this integer does not depend on the choice of $\Psi$ (see Proposition 3.3 in \cite{aps08}). The normalizing constant in the definition above is chosen in such a way that, when $H$ is the Fenchel dual of a Lagrangian $L$, $\mu^R(x)$ coincides with the Morse index of the corresponding solution of the Euler-Lagrange equation associated to $L$ (see Section \ref{cci}). 

The elements of $\mathscr{P}^R(H)$ are the critical points of the
Hamiltonian action functional
\[
\mathbb{A}_H(x) = \int_{[0,1]} x^*(\eta - H dt) = \int_0^1 \bigl(
p (t) [q'(t)] - H(t,q(t),p(t))\bigr) \, dt,
\]
on the space of curves $x:[0,1]\rightarrow T^* Q$ which satisfy (\ref{conbdry}). Indeed, the differential of $\mathbb{A}_H$ on the
space of free paths on $T^*Q$ is
\begin{equation}
\label{dafp}
D\mathbb{A}_H(x) [\zeta] = \int_0^1 
\omega(\zeta,x' - X_H(t,x) )\, dt
+\eta(x(1))[\zeta(1)] - \eta(x(0))[\zeta(0)],
\end{equation}
and the boundary term vanishes when $x$ satisfies (\ref{conbdry}) and the variation  $\zeta$ satisfies 
\[
\bigl( \zeta(0),D\mathscr{C}(x(1))[\zeta(1)] \bigr) \in T_{(x(0),\mathscr{C} x(1))} N^* R,
\]
because the Liouville form is zero on conormal bundles. The non-degeneracy assumption, together with conditions (H1) and (H2), imply that the set of $x\in
\mathscr{P}^R(H)$ with $\mathbb{A}_H(x)\leq A$ is finite, for every
$A\in \R$. Indeed, this follows immediately from the following
general:

\begin{lem}
\label{compham} Let $H\in C^{\infty}([0,1] \times T^*Q)$ be a
Hamiltonian satisfying (H1) and (H2). For every $A\in \R$ there exists a
compact subset $K \subset T^*Q$ such that each orbit $x:[0,1]
\rightarrow T^*Q$ of $X_H$ with $\mathbb{A}_H(x)\leq A$ lies in
$K$.
\end{lem}

\begin{proof} Let $x=(q,p)$ be an orbit of $X_H$ such that
$\mathbb{A}_H(x)\leq A$. Since $x$ is an orbit of $X_H$,
by (\ref{Liouv}),
\[
\eta (x)[x'] - H(t,x) = \omega (Y(x), X_H(t,x)) - H(t,x) =
DH(t,x)[Y(x)] - H(t,x).
\]
Therefore (H1) implies that $|p|$ is uniformly bounded in
$L^2([0,1])$. By (\ref{cresc}), $|x'|$ is uniformly bounded in
$L^1([0,1])$. Therefore $x$ is uniformly bounded in $W^{1,1}$,
hence in $L^{\infty}$.
\end{proof}

\begin{rem}
Assume that the flow generated by a Hamiltonian $H\in
C^{\infty}([0,1]\times T^*Q)$ is globally defined (for instance, this
holds if $H$ is coercive and $|\partial_t H| \leq c(|H|+1)$). Then the
conclusion of Lemma \ref{compham} holds assuming just that the function
$DH[Y]-H$ is coercive (a much weaker assumption than (H1), still
implying that $H$ is coercive), without
any upper bound such as (H2).
\end{rem}

Let us fix a Riemannian metric $\langle \cdot ,\cdot \rangle$ on $Q$. This metric induces metrics on $TQ$ and on $T^*Q$, both denoted by $\langle \cdot ,\cdot \rangle$. It induces also an identification $T^*Q \cong TQ$,
horizontal-vertical splittings of both $TTQ$ and $TT^*Q$, and a
particular almost complex structure $J$
on $T^*M$, namely the one which in the horizontal-vertical
splitting takes the form
\begin{equation}
\label{acsj}
J = \left( \begin{array}{cc} 0 & -I \\ I & 0 \end{array}
\right).
\end{equation}
This almost complex structure is $\omega$-compatible, meaning that 
\[
\langle \xi , \eta \rangle = \omega(J \xi,\eta), \quad \forall \xi,\eta\in T_x T^*Q, \; \forall x\in T^*Q.
\]
Notice that our sign
convention here differs from the one used in \cite{as06}. The
reason is that here we prefer to see the leading term in the Floer
equation as a Cauchy-Riemann operator, and not as an
anti-Cauchy-Riemann operator.

The $L^2$-negative gradient equation for the Hamiltonian action
functional $\mathbb{A}_H$ is the Floer equation
\begin{equation}
\label{fleq}
\delbar_{J,H}(u) := \partial_s u + J(u) [\partial_t u - X_H(t,u)] = 0,
\end{equation}
for $u=u(s,t)$, $(s,t) \in \R\times [0,1]$.
A generic choice of the Hamiltonian $H$ 
makes the following space of solutions of the Floer equation (\ref{fleq}) with nonlocal conormal boundary conditions defined by $R$,
\begin{equation*}\begin{split}
\mathscr{M}_{\partial}^R(x,y) = &\Big\{u\in C^{\infty}(\R \times
  [0,1] ,T^*Q)\,\big|\, (u(s,0),\mathscr{C}u(s,1)) \in N^* R \; \forall s\in \R,\\ &\delbar_{J,H}(u)=0, \; \mbox{ and } \lim_{s\rightarrow-\infty} u(s,t) = x(t), \; \lim_{s\rightarrow +\infty} 
  u(s,t) = y(t)\Big\}
\end{split}\end{equation*}
a manifold of dimension $\mu^R(x) - \mu^R(y)$, for
every $x,y\in \mathscr{P}^R(H)$.  Here {\em generic} means for a
countable intersection of open and dense subsets of the space of smooth time-dependent Hamiltonians satisfying (H0), (H1), and (H2), with respect to suitable topologies (we refer to \cite{fhs96} for transversality issues). In particular, the perturbation of a given Hamiltonian $H$ satisfying (H0), (H1), (H2) can be chosen in such a way that the discrete set $\mathscr{P}^R(H)$ is unaffected. 

\begin{rem}
\label{3.3}
As it is well-known, transversality can also be achieved for a fixed Hamiltonian by perturbing the almost complex structure $J$ in a time-dependent way. In order to have good compactness properties for the spaces $\mathscr{M}_{\partial}^R$ one needs the perturbed almost complex structure $J_1$ to be $C^0$-close enough to the metric one $J$ defined by (\ref{acsj}) (see \cite[Theorem 1.14]{as06}). Other compactness issues in this paper would impose further restrictions on the distance between $J_1$ and $J$. For this reason here we prefer to work with the fixed almost complex structure $J$, and to achieve transversality by perturbing the Hamiltonian. A different approach would be to choose almost complex structures on $T^*Q$ which are of contact type on $T^* Q \setminus Q$, seen as the symplectization of the unit cotangent sphere bundle (see e.g.\ \cite{vit96}). In this case, compactness of the spaces $\mathscr{M}_{\partial}^R$ follows from the maximum principle, but one needs more restrictive assumptions on the behavior of the Hamiltonian $H$ for $|p|$ large.
\end{rem}  

The manifolds $\mathscr{M}_{\partial}^R(x,y)$ can be
oriented in a coherent way. Assumptions (H1) and (H2) imply that these manifolds
have nice compactifications. In particular, when $\mu^R(x) - \mu^R(y)=1$,
$\mathscr{M}_{\partial}^R(x,y)$ consists of finitely many one-parameter families
of solutions $\sigma \mapsto u(\cdot+\sigma,\cdot)$, each of which
comes with a sign $\pm 1$, depending whether its orientation agrees or
not with the orientation determined by letting $\sigma$ increase. The
algebraic sum of these numbers is an integer
$n_{\partial}^R(x,y)$. If we
let $F_k^R(H)$ denote the free Abelian group generated by the elements
$x\in \mathscr{P}^R(H)$ of index $\mu^R(x)=k$, the above coefficients
define the homomorphism
\[
\partial: F_k^R(H) \rightarrow F_{k-1}^R(H), \quad x \mapsto
\sum_{\substack{y\in \mathscr{P}^R(H) \\ \mu^R(y)=k-1}}
n_{\partial}(x,y)\, y,
\]
which turns out to be a boundary operator. The resulting chain
complex $F^R_*(H)$ is the {\em Floer
complex associated to the Hamiltonian $H$ and to the nonlocal conormal boundary conditions defined by $R$}. If we change the metric on $Q$ - hence the almost complex structure $J$ - and the orientation data, the Floer complex $F^R_*(H)$ 
changes by an isomorphism. If we change the Hamiltonian $H$, the new Floer
complex is homotopically equivalent to the old one. In particular, the homology of
the Floer complex does not depend on the metric, on $H$, and on the orientation
data. This fact allows us to denote this graded Abelian group as
$HF_*^R(T^*Q)$, and to call it the {\em Floer homology of $T^*Q$ with nonlocal conormal boundary conditions defined by $R$}. 

As in Section \ref{mclaf}, here we are interested in the following three boundary conditions.

\paragraph{Periodic boundary conditions.} Here $Q=M$ is closed and oriented,
and $R=\Delta_M$ is the diagonal in $M\times M$. Under the extra assumption that $H$ extends as a smooth 1-periodic function on $\R\times TM$, the set $\mathscr{P}^{\Delta_M}(H)$ is precisely the set of 1-periodic solutions of the Hamiltonian equation (\ref{hameq}). We also use the notation $\mathscr{P}^{\Lambda}(H) := \mathscr{P}^{\Delta_M}(H)$. The non-degeneracy condition (H0) is just:
\begin{description}
\item[(H0)$^{\Lambda}$] For every $x\in \mathscr{P}^{\Lambda}(H)$,
  the number 1 is not an eigenvalue of $D_x\phi^H(1,x(0)): T_{x(0)} T^*
  M \rightarrow T_{x(0)} T^*M$.
\end{description}
The Maslov index $\mu^{\Delta_M}(x)$ coincides with the Conley-Zehnder index of the periodic orbit $x$, that we denote also by $\mu^{\Lambda}(x)$ (notice that the normalizing constant in (\ref{mascon}) vanishes). In the definition of such an index, one can choose the trivialization of $x^*(TT^*M)$ to be any 1-periodic symplectic vertical-preserving trivialization. If $M$ is not orientable and the closed curve $x$ is orientation-reversing, there are no 1-periodic and vertical-preserving trivializations of $x^*(TT^*M)$, so one has either to give up the periodicity, as in definition (\ref{mascon}), or the preservation of the vertical subbundle, as in \cite{web02}. Floer homology with periodic boundary conditions is defined (with integer coefficients) also for non-orientable manifolds, but since in this paper we are interested in the loop product with integer coefficients, we assume orientability. 
The elements $u$ of $\mathscr{M}_{\partial}^{\Lambda}(x,y) := \mathscr{M}_{\partial}^{\Delta_M}(x,y)$ are actually smooth solutions of the Floer equation (\ref{fleq}) on the cylinder, that is
\[
u: \R \times \T \rightarrow T^*M.
\]
The corresponding Floer complex is also denoted by $F_*^{\Lambda}(H)$, and its homology by $HF^{\Lambda}_*(T^*M)$.

\paragraph{Dirichlet boundary conditions.} Here $Q=M$ and $R$ consists of the single point $(q_0,q_0)$, for some fixed $q_0\in M$. The set $\mathscr{P}^{(q_0,q_0)}(H)$ is precisely the set of solutions $x:[0,1] \rightarrow T^*M$ of the Hamiltonian equation (\ref{hameq}) such that $\pi\circ x(0) = \pi\circ x(1)=q_0$, and we denote such a set also by $\mathscr{P}^{\Omega}(H)$. The non-degeneracy condition is now:
\begin{description}
\item[(H0)$^{\Omega}$] For every $x\in \mathscr{P}^{\Omega}(H)$,
  the linear mapping $D_x\phi^H(1,x(0)): T_{x(0)} T^*
  M \rightarrow T_{x(1)} T^*M$ maps the vertical subspace $T^v_{x(0)}
  T^*M$ at $x(0)$ into a subspace having intersection $(0)$ with the
  vertical subspace $T^v_{x(1)} T^*M$ at $x(1)$.
\end{description}
The Maslov index $\mu^{(q_0,q_0)}(x)$, that we denote also by $\mu^{\Omega}(x)$,  is just the relative Maslov index of the path of Lagrangian subspaces $D\phi^H(t,x(0)) T_{x(0)}^v T^*M$ - transported into $T^* \R^n$ by means of a vertical-preserving symplectic trivialization - with respect to the vertical space $N^* (0) = (0) \times (\R^n)^*$.  
The boundary condition for the elements $u$ of $\mathscr{M}_{\partial}^{\Omega}(x,y) := \mathscr{M}_{\partial}^{(q_0,q_0)}(x,y)$ is the local condition
\[
u(s,0) \in T_{q_0}^* M, \quad u(s,1) \in T_{q_0}^* M, \quad \forall s\in \R.
\]
The corresponding Floer complex is also denoted by $F_*^{\Omega}(H)$, and its homology by $HF^{\Omega}_*(T^*M)$.

\paragraph{Figure-8 boundary conditions.} Here $Q=M\times M$ and $R=
\Delta^{(4)}_M$ is the fourth diagonal in $M^4$:
\[
\Delta^{(4)}_M := \set{(q,q,q,q)}{q\in M}.
\]
In this case, it is convenient (although not necessary) to chose $H$ of the form $H=H_1 \oplus H_2$, where
\[
H_1 \oplus H_2 \, (t,x) := H_1(t,x_1) + H_2(t,x_2), \quad \begin{array}{l} \forall t\in [0,1], \\ \forall x=(x_1,x_2)\in T^* M \times T^*M, \end{array}
\]
so that the elements of $\mathscr{P}^{\Delta^{(4)}_M}(H)$ are the pairs of $T^*M$-valued curves $(x_1,x_2)$ such that each $x_j$ solves the Hamiltonian equation induced by $H_j$, and such that the coupling boundary conditions 
\[
\pi\circ x_1(0) = \pi\circ x_2(0) = \pi \circ x_1(1) = \pi\circ x_2(1), \quad x_1(1) - x_1(0) + x_2(1) - x_2(0) = 0,
\]
hold. The set of $\mathscr{P}^{\Delta^{(4)}_M}(H_1 \oplus H_2)$ is also denoted by $\mathscr{P}^{\Theta}(H_1\oplus H_2)$. The corresponding non-degeneracy condition is:
\begin{description}
\item[(H0)$^{\Theta}$] Every solution $x=(x_1,x_2)\in
\mathscr{P}^{\Theta}(H_1 \oplus H_2)$ is non-degenerate, meaning
that the graph of the map $\mathscr{C} \circ \phi^{H_1 \oplus H_2}(1,\cdot)$ is
transverse to the submanifold $N^* \Delta^{(4)}_M$ at the point
$(x(0),\mathscr{C}x(1))$.
\end{description}
If $x=(x_1,x_2)\in \mathscr{P}^{\Theta}(H_1 \oplus H_2)$, the Maslov index $\mu^{\Delta^{(4)}_M}(x)$, that we denote simply by $\mu^{\Theta}(x)$,  is the integer
\begin{equation}
\label{sarlu}
\mu^{\Theta}(x) = \mu\bigl( \graf G^{\Psi} C, \Delta^{(4)}_{\R^n}\bigr) - \frac{n}{2},
\end{equation}
where the symplectic path $G^{\Psi}: [0,1] \rightarrow \mathrm{Sp} (4n)$ is obtained by conjugating the differential of the flow $\phi^{H_1 \oplus H_2}$ along $x$ by a trivialization $\Psi$ of $x^*(TT^*M^2)$ of the form $\Psi = \Psi_1 \oplus \Psi_2$, where each $\Psi_j$ is the canonical vertical-preserving symplectic trivialization of $x_j^*(TT^*M)$ induced by a trivialization of $(\pi\circ x_j)^*(TM)$ over the circle $\T$. The boundary condition for the elements $u$ of $\mathscr{M}_{\partial}^{\Theta}(x,y) := \mathscr{M}_{\partial}^{\Delta^{(4)}_M}(x,y)$ is the nonlocal Lagrangian boundary condition 
\[
\bigl(u(s,0), \mathscr{C} u(s,1)\bigr) \in N^* \Delta_M^{(4)}, \quad \forall s\in \R.
\]
The corresponding Floer complex is also denoted by $F_*^{\Theta}(H)$, and its homology by $HF^{\Theta}_*(T^*M)$.

\subsection{The Floer equation on triangles and pair-of-pants}
\label{Fetpp}

Additional algebraic structures on Floer homology are defined by
extending the Floer equation to more general Riemann surfaces than
the strip $\R \times [0,1]$ and the cylinder $\R \times \T$.

Let $(\Sigma,j)$ be a Riemann surface, possibly with boundary.  For
$u\in C^{\infty}(\Sigma,T^*M)$ consider the {\em nonlinear
Cauchy-Riemann operator}
\[
\overline{D}_J u= \frac{1}{2} ( Du + J(u) \circ Du \circ j),
\]
that is the complex anti-linear part of $Du$ with respect to the
almost-complex structure $J$. The operator $\overline{D}_J$ is a
section of the bundle over $C^{\infty}(\Sigma,T^*M)$ whose fiber
at $u$ is $\Omega^{0,1}(\Sigma,u^*(TT^*M))$, the space of
anti-linear one-forms on $\Sigma$ taking values in the vector
bundle $u^*(TT^*M)$. If we choose a holomorphic coordinate $z=s+it$ on
$\Sigma$, the operator $\overline{D}_J$ takes the form
\begin{equation}
\label{forma} \overline{D}_J u = \frac{1}{2} ( \partial_s u +
J(u) \partial_t u)\, ds - \frac{1}{2} J(u) (\partial_s u +
J(u) \partial_t u)\, dt.
\end{equation}
This expression shows that the leading term $\delbar_J := \partial_s +
J(\cdot)\partial_t$ in the Floer equation (\ref{fleq}) can be
extended to arbitrary Riemann surfaces, at the only cost of
considering an equation which does not take values on a space of
tangent vector fields, but on a space of anti-linear one-forms.

When $\Sigma$ has a global coordinate $z=s+it$, as in the case of the strip $\R
\times [0,1]$ or of the cylinder $\R \times \T$, we can associate
to the Hamiltonian term in the Floer equation the complex anti-linear
one-form
\begin{equation}
\label{hamter} F_{J,H}(u) = - \frac{1}{2}(J(u) X_H(t,u)\, ds +
X_H(t,u)\, dt ) \in \Omega^{0,1}(\Sigma,u^*(TT^*M)).
\end{equation}
Formula (\ref{forma}) shows that the Floer equation (\ref{fleq}) is equivalent to
\begin{equation}
\label{fleqr} \overline{D}_J u + F_{J,H}(u) = 0.
\end{equation}
If we wish to use the formulation (\ref{fleqr}) to extend the Floer
equation to more general Riemann surfaces, we encounter the
difficulty that -- unlike $\overline{D}_J$ -- the Hamiltonian term
$F_{J,H}$ is defined in terms of coordinates.

One way to get around this difficulty is to consider Riemann
surfaces with cylindrical or strip-like ends, each of which is
endowed with some fixed holomorphic coordinate $z=s+it$,
to define the operator
$F_{J,H}$ on such ends, and then to extend it to the whole
$\Sigma$ by considering a Hamiltonian $H$ which also depends on
$s$ and vanishes far from the ends. In this way, only the
Cauchy-Riemann operator acts in the region far from the ends. This
approach is adopted in \cite{sch95,pss96,ms04}.

A drawback of this method is that one loses sharp energy
identities relating some norm of $u$ to the jump of the
Hamiltonian action functional. Moreover, an $s$-dependent
Hamiltonian which vanishes for some values of $s$ cannot satisfy
assumptions (H1) and (H2). These facts lead to problems with
compactness when dealing -- as we are here -- with a non-compact
symplectic manifold.

Therefore, we shall use a different method to extend the
Hamiltonian term $F_{J,H}$. We shall describe this construction in
the case of the triangle and the pair-of-pants surface, although
the same idea could be generalized to any Riemann surface.

Let $\Sigma_{\Upsilon}^{\Omega}$ be the {\em holomorphic triangle},
that is the Riemann surface consisting of a closed triangle with
the three vertices removed (equivalently, a closed disk with three
boundary points removed). Let $\Sigma_{\Upsilon}^{\Lambda}$ be the
{\em pair-of-pants} Riemann surface, that is the sphere with three
points removed.

\begin{figure}[t]
 \begin{center}
  \setlength{\unitlength}{0.01cm}
  \begin{picture}(1200,300)(0,0)
   \put(200,20){\epsfig{file=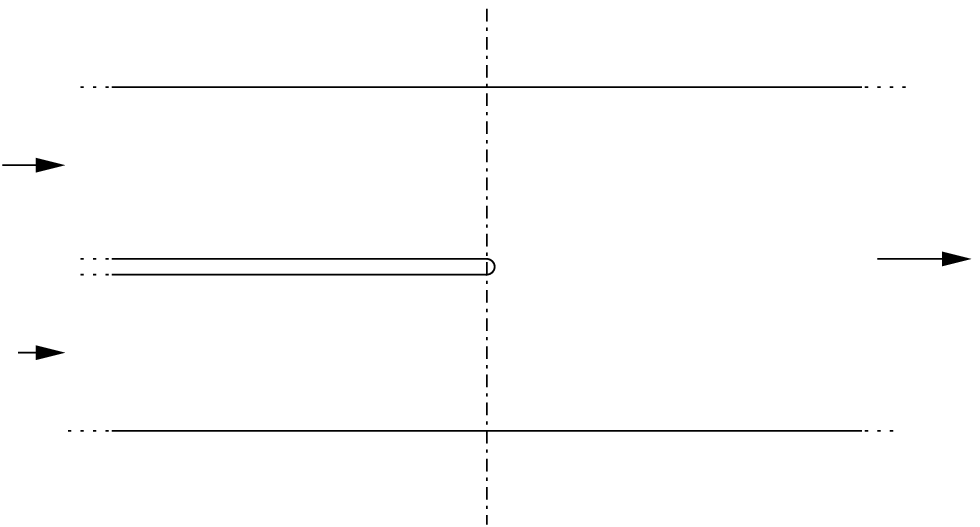,width=7cm}}
   \put(510,0){\makebox(0,0)[lb]{$s=0$}}
   \put(800,0){\makebox(0,0)[lb]{$s\to\infty$}}
  \end{picture}
 \end{center}
\caption{The strip with a slit $\Sigma_{\Upsilon}^{\Omega}$. 
\label{strip-slit}}
\end{figure}

The Riemann surface $\Sigma_{\Upsilon}^{\Omega}$ can be described
as a strip with a slit: One takes the disjoint union
\[
\R \times [-1,0] \sqcup \R \times [0,1]
\]
and identifies $(s,0^-)$ with $(s,0^+)$ for every $s\geq 0$. 
See Figure \ref{strip-slit}. The
resulting object is indeed a Riemann surface with interior
\[
\mathrm{Int}(\Sigma_{\Upsilon}) = (\R \times ]-1,1[ )
\setminus (]-\infty,0] \times \{0\})
\]
endowed with the complex structure of a subset of $\R^2\cong
\C$, $(s,t)\mapsto s+it$, and three boundary components
\[
\R \times \{-1\}, \quad \R \times \{1\}, \quad ]-\infty,0]\times
\{0^-,0^+\}.
\]
The complex structure at each boundary point other than
$0=(0,0)$ is induced by the inclusion in $\C$, whereas a
holomorphic coordinate at $0$ is given by the map
\begin{equation}
\label{holcor1} \set{\zeta\in \C}{\re \zeta\geq 0, \; |\zeta|<1}
\rightarrow \Sigma_{\Upsilon}^{\Omega} , \quad \zeta \mapsto
\zeta^2,
\end{equation}
which maps the boundary line $\{\re \zeta=0,\; |\zeta|<1\}$ into
the portion of the boundary $]-1,0] \times \{0^-,0^+\}$. 

Similarly, the pair-of-pants $\Sigma_{\Upsilon}^{\Lambda}$ can be
described as the following quotient of a strip with a slit: In the
disjoint union $\R \times [-1,0] \sqcup \R \times
[0,1]$ we consider the identifications
\[
\begin{array}{c} (s,-1) \sim (s,0^-) \\ (s,0^+) \sim (s,1)
\end{array} \quad \mbox{for } s\leq 0, \quad
\begin{array}{c} (s,0^-) \sim (s,0^+) \\ (s,-1) \sim (s,1)
\end{array} \quad \mbox{for } s\geq 0.
\]
See Figure \ref{pants}.

\begin{figure}[t]
 \begin{center}
  \setlength{\unitlength}{0.01cm}
  \begin{picture}(1200,300)(0,0)
   \put(0,0){\epsfig{file=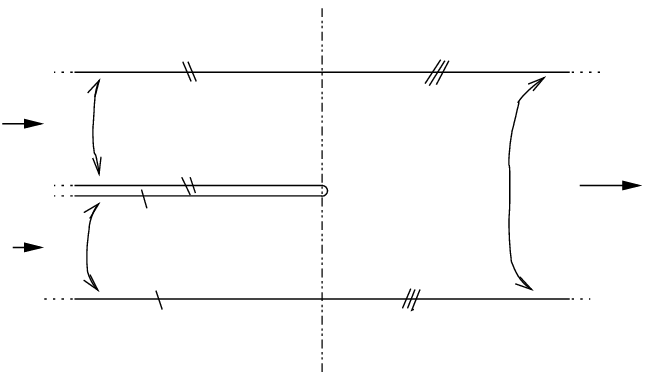,width=5cm}}
   \put(540,150){\makebox(0,0)[lb]{$\simeq$}}
   \put(620,60){\epsfig{file=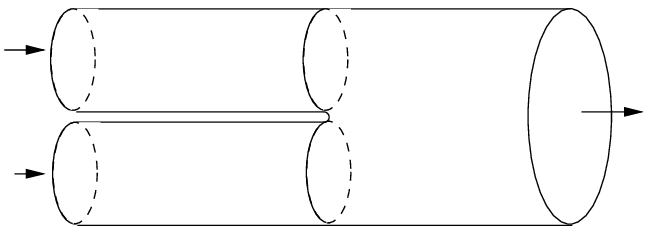,width=5cm}}
  \end{picture}
 \end{center}
\caption{The pair-of-pants $\Sigma_{\Upsilon}^{\Lambda}$. 
         \label{pants}}
\end{figure}

This object is a Riemann surface without boundary, by considering
the standard complex structure at every point other than $(0,0)
\sim (0,-1) \sim (0,1)$, and by choosing the holomorphic
coordinate
\begin{equation}
\label{holcor2} \set{\zeta\in \C}{|\zeta| <1/\sqrt{2}} \rightarrow
\Sigma_{\Upsilon}^{\Lambda}, \quad \zeta \mapsto \left\{
\begin{array}{ll} \zeta^2 & \mbox{if } \re \zeta\geq 0, \\
\zeta^2+i & \mbox{if } \re \zeta\leq 0, \; \im \zeta \geq 0, \\
\zeta^2-i & \mbox{if } \re \zeta\leq 0, \; \im \zeta\leq 0,
\end{array} \right.
\end{equation}
at this point.

The advantage of these representations is that now
$\Sigma_{\Upsilon}^{\Omega}$ and $\Sigma_{\Upsilon}^{\Lambda}$ are
endowed with a global coordinate $z=s+it$, which is holomorphic
everywhere except at the point $(0,0)$ (identified with $(0,-1)$
and $(0,1)$ in the $\Lambda$ case). We refer to such a point
as the {\em singular} point: it is a regular point for the
complex structure of $\Sigma_{\Upsilon}^{\Omega}$ or
$\Sigma_{\Upsilon}^{\Lambda}$, but it is singular for the global
coordinate $z=s+it$. In fact, the canonical map
\[
\Sigma_{\Upsilon}^{\Lambda} \rightarrow \R \times \T, \quad (s,t) \mapsto (s,t),
\]
is a $2:1$ branched covering of the cylinder.

Let $H\in C^{\infty}([-1,1]\times T^*M)$. If $u\in
C^{\infty}(\Sigma_{\Upsilon}^{\Omega},T^*M)$, the complex anti-linear
one-form $F_{J,H}(u)$ is everywhere defined by equation
(\ref{hamter}). We just need to check the regularity of
$F_{J,H}(u)$ at the singular point. Writing $F_{J,H}(u)$ in terms
of the holomorphic coordinate $\zeta=\sigma+i\tau$ by means of
(\ref{holcor1}), we find
\[
F_{J,H}(u) = (\tau I - \sigma J(u))
X_H(2\sigma\tau,u) \, d\sigma + (\sigma I +\tau J(u))
X_H(2\sigma \tau,u) \, d\tau.
\]
Therefore, $F_{J,H}(u)$ is smooth, and actually it vanishes at the
singular point.

Assume now that $H\in C^{\infty}(\R/2\Z \times T^*M)$ is such that
$H(-1,\cdot)=H(0,\cdot)=H(1,\cdot)$ with all the time derivatives. 
If $u\in C^{\infty}(\Sigma^{\Lambda}_{\Upsilon},T^*M)$, (\ref{hamter})
defines a smooth complex anti-linear one-form $F_{J,H}(u)\in
\Omega^{0,1}(\Sigma^{\Lambda}_{\Upsilon},u^*(TT^*M))$.

A map $u$ in $C^{\infty}(\Sigma_{\Upsilon}^{\Omega},T^*M)$ or in
$C^{\infty}(\Sigma^{\Lambda}_{\Upsilon},T^*M)$ solves equation
(\ref{fleqr}) if and only if it solves the equation
\[
\delbar_{J,H}(u) = \partial_s u + J(u) (\partial_t u - X_H(t,u)) = 0
\]
on $\mathrm{Int}(\Sigma_{\Upsilon})$. If $u$ solves the above
equation on $[s_0,s_1] \times [t_0,t_1]$, formula (\ref{dafp})
together with an integration by parts leads to the identity
\begin{eqnarray*}
\int_{s_0}^{s_1} \int_{t_0}^{t_1} |\partial_s
u(s,t)|^2 \, dt\, ds = \mathbb{A}_H^{[t_0,t_1]}
(u(s_0,\cdot)) - \mathbb{A}_H^{[t_0,t_1]} (u(s_1,\cdot)) \\ +
\int_{s_0}^{s_1} \bigl( \eta(u(s,t_1))[\partial_s u(s,t_1)] -
\eta(u(s,t_0))[\partial_s u(s,t_0)] \bigr)\, ds,
\end{eqnarray*}
where $\mathbb{A}_H^I(x)$ denotes the Hamiltonian
action of the path $x$ on the interval $I$. We conclude that a
solution $u$ of (\ref{fleqr}) on $\Sigma_{\Upsilon}^{\Lambda}$ or
on $\Sigma_{\Upsilon}^{\Omega}$ -- in the latter case with values
in $T^*_{q_0} M$ on the boundary -- satisfies the sharp energy
identity
\begin{equation}
\label{sei}
\begin{split}
\int_{\Sigma_{\Upsilon} \cap \{|s|\leq s_0\}} |\partial_s u(s,t)|^2 \, ds\,
dt \\ = \mathbb{A}_H^{[-1,0]}(u(-s_0,\cdot)) +
\mathbb{A}_{H}^{[0,1]}(u(-s_0,\cdot)) -
\mathbb{A}_H^{[-1,1]}(u(s_0,\cdot)).
\end{split}
\end{equation}

\subsection{The triangle and the pair-of-pants products}
\label{tppp}

Given $H_1,H_2\in C^{\infty}([0,1]\times T^*M)$ such that
$H_1(1,\cdot)=H_2(0,\cdot)$ with all time derivatives, we define
$H_1 \# H_2 \in C^{\infty}([0,1] \times T^*M)$ by
\begin{equation}
\label{canch}
H_1 \# H_2 (t,x) = \left\{ \begin{array}{ll} 2H_1 (2t,x) &
\mbox{for } 0\leq t \leq 1/2, \\ 2H_2(2t-1,x) & \mbox{for } 1/2
\leq t \leq 1. \end{array} \right.
\end{equation}
Let us assume that $H_1$, $H_2$, and $H_1 \# H_2$ satisfy
(H0)$^{\Omega}$. The {\em triangle product} on $HF^{\Omega}(T^*M)$
will be induced by a chain map
\[
\Upsilon^{\Omega} : F_h^{\Omega}(H_1) \otimes
F_k^{\Omega}(H_2) \rightarrow F_{h+k}^{\Omega}(H_1 \# H_2).
\]
In the periodic case, we consider Hamiltonians $H_1,H_2 \in
C^{\infty}(\T\times T^*M)$  such that $H_1(0,\cdot) =
H_2(0,\cdot)$ with all time
derivatives. Assuming that $H_1$, $H_2$, and $H_1 \# H_2$ satisfy
(H0)$^{\Lambda}$, the {\em pair-of-pants product} on
$HF^{\Lambda}(T^*M)$ will be induced by a chain map
\[
\Upsilon^{\Lambda} : F_h^{\Lambda}(H_1) \otimes
F_k^{\Lambda}(H_2) \rightarrow F_{h+k-n}^{\Lambda}(H_1 \#
H_2),
\]
where $n$ is the dimension of $M$.

Let $H\in C^{\infty}([-1,1]\times T^*M)$, respectively $H\in
C^{\infty}(\R/2\Z \times T^*M)$, be defined by
\begin{equation}
\label{hham}
H(t,x) = \frac{1}{2} H_1 \# H_2((t+1)/2,x) = \left\{
\begin{array}{ll} H_1(t+1,x) & \mbox{if } -1\leq t \leq 0, \\
H_2(t,x) & \mbox{if } 0\leq t \leq 1. \end{array} \right.
\end{equation}
Notice that $x: [-1,1] \rightarrow T^*M$ is an orbit of $X_H$ if
and only if the curve $t\mapsto x((t+1)/2)$ is an orbit of $X_{H_1
\# H_2}$.

Given $x_1\in \mathscr{P}^{\Omega}(H_1)$, $x_2\in
\mathscr{P}^{\Omega}(H_2)$, and $y\in \mathscr{P}^{\Omega}(H_1 \#
H_2)$, consider the following space of solutions of the Floer equation $\delbar_{J,H}(u)=0$ on the holomorphic triangle:
\begin{eqnarray*}
\mathscr{M}^{\Omega}_{\Upsilon}(x_1,x_2;y) := \Bigl\{ u \in
C^{\infty}(\Sigma_{\Upsilon}^{\Omega},T^*M) \, \Big| \, \delbar_{J,H}(u)=0, \;  u(z) \in T_{q_0}^* M \; \forall z\in
\partial \Sigma_{\Upsilon}^{\Omega}, \\ \lim_{s\rightarrow
-\infty} u(s,t-1) = x_1(t), \; \lim_{s\rightarrow -\infty} u(s,t)
= x_2(t), \; \lim_{s\rightarrow +\infty} u(s,2t-1) = y(t)\Bigr\}.
\end{eqnarray*}
Similarly, for $x_1\in \mathscr{P}^{\Lambda}(H_1)$, $x_2\in
\mathscr{P}^{\Lambda}(H_2)$, and $y\in \mathscr{P}^{\Lambda}(H_1
\# H_2)$, we consider the following space of solutions of the Floer equation on
the pair-of-pants surface:
\begin{eqnarray*}
\mathscr{M}^{\Lambda}_{\Upsilon}(x_1,x_2;y) := \Bigl\{ u \in
C^{\infty}(\Sigma_{\Upsilon}^{\Lambda},T^*M) \, \Big| \, \delbar_{J,H}(u)=0,\; \lim_{s\rightarrow -\infty} u(s,t-1) =
x_1(t), \\ \lim_{s\rightarrow -\infty} u(s,t) = x_2(t), \;
\lim_{s\rightarrow +\infty} u(s,2t-1) = y(t) \Bigr\}.
\end{eqnarray*}
The following result is proved in Section \ref{lin}.

\begin{prop}
\label{popss}
For a generic choice of $H_1$ and $H_2$ as above, the sets
$\mathscr{M}_{\Upsilon}^{\Omega}(x_1,x_2;y)$ and
$\mathscr{M}_{\Upsilon}^{\Lambda}(x_1,x_2;y)$ -- if non-empty --
are manifolds of dimension
\begin{equation*}\begin{split}
&\dim \mathscr{M}_{\Upsilon}^{\Omega}(x_1,x_2;y) =
\mu^{\Omega}(x_1) + \mu^{\Omega}(x_2) - \mu^{\Omega}(y), \\
&\dim \mathscr{M}_{\Upsilon}^{\Lambda}(x_1,x_2;y) =
\mu^{\Lambda}(x_1) + \mu^{\Lambda}(x_2) - \mu^{\Lambda}(y)-n.
\end{split}\end{equation*}
These manifolds carry coherent orientations.
\end{prop}

The energy identity (\ref{sei}) implies that every map $u$ in
$\mathscr{M}_{\Upsilon}^{\Omega}(x_1,x_2;y)$ or in
$\mathscr{M}_{\Upsilon}^{\Lambda}(x_1,x_2;y)$ satisfies
\begin{equation}
\label{seir} \int_{\Sigma_{\Upsilon}} |\partial_s u(s,t)|^2 \, ds \,dt =
\mathbb{A}_{H_1}(x_1) + \mathbb{A}_{H_2}(x_2) - \mathbb{A}_{H_1
\# H_2} (y).
\end{equation}
As a consequence, we obtain the following compactness result, which
is proved in Section \ref{compsec}.

\begin{prop}
\label{comprop1}
Assume that the Hamiltonians $H_1$ and $H_2$ satisfy (H1), (H2).  Then the spaces
$\mathscr{M}_{\Upsilon}^{\Omega}(x_1,x_2;y)$ and
$\mathscr{M}_{\Upsilon}^{\Lambda}(x_1,x_2;y)$ are pre-compact in
$C^{\infty}_{\mathrm{loc}}$.
\end{prop}

When $\mu^{\Omega}(y) = \mu^{\Omega}(x_1) + \mu^{\Omega}(x_2)$,
$\mathscr{M}_{\Upsilon}^{\Omega}(x_1,x_2;y)$ is a finite set of
oriented points, and we denote by
$n_{\Upsilon}^{\Omega}(x_1,x_2;y)$ the algebraic sum of the
corresponding orientation signs. Similarly, when $\mu^{\Lambda}(y)
= \mu^{\Lambda}(x_1) + \mu^{\Lambda}(x_2)-n$,
$\mathscr{M}_{\Upsilon}^{\Lambda}(x_1,x_2;y)$ is a finite set of
oriented points, and we denote by
$n_{\Upsilon}^{\Lambda}(x_1,x_2;y)$ the algebraic sum of the
corresponding orientation signs. These integers are the
coefficients of the homomorphisms
\begin{equation*}\begin{split}
\Upsilon^{\Omega} : &F_h^{\Omega}(H_1) \otimes F_k^{\Omega}(H_2)
\rightarrow F_{h+k}^{\Omega} (H_1 \# H_2), \quad \\
&x_1 \otimes x_2 \mapsto
\sum_{\substack{y\in \mathscr{P}^{\Omega}(H_1 \# H_2)\\
\mu^{\Omega}(y)=h+k}}
n_{\Upsilon}^{\Omega}(x_1,x_2;y)\, y, \\
\Upsilon^{\Lambda} : &F_h^{\Lambda} (H_1) \otimes
F_k^{\Lambda}(H_2) \rightarrow F_{h+k-n}^{\Lambda} (H_1 \# H_2),
\quad\\
 &x_1 \otimes x_2 \mapsto\sum_{\substack{y\in \mathscr{P}^{\Lambda}(H_1 \# H_2)\\
\mu^{\Lambda}(y)=h+k-n}} n_{\Upsilon}^{\Lambda}(x_1,x_2;y)\, y.
\end{split}\end{equation*}
A standard gluing argument shows that the homomorphisms
$\Upsilon^{\Omega}$ and $\Upsilon^{\Lambda}$ are chain maps.
Therefore, they define products
\begin{eqnarray*}
H_*\Upsilon^{\Omega} : HF_h^{\Omega}(T^*M) \otimes HF_k^{\Omega}
(T^*M) \rightarrow HF_{h+k}^{\Omega}(T^*M), \\
H_*\Upsilon^{\Lambda} : HF_h^{\Lambda}(T^*M) \otimes HF_k^{\Lambda}
(T^*M) \rightarrow HF_{h+k-n}^{\Lambda}(T^*M),
\end{eqnarray*}
in homology. Again by gluing arguments, it could be shown that
these products have a unit element, are associative, and the
second one is commutative. These facts will actually follow from
the fact that these products correspond to the Pontrjagin and the
loop products on $H_*(\Omega(M,q_0))$ and $H_*(\Lambda(M))$.

\subsection{Factorization of the pair-of-pants product}
\label{fpop}

Let $H_1,H_2\in C^{\infty}(\T\times T^*M)$ be two Hamiltonians
satisfying (H0)$^{\Lambda}$, (H1), and (H2). We assume that
$H_1(0,\cdot)=H_2(0,\cdot)$ with all time derivatives, so that
the Hamiltonian $H_1 \# H_2$ defined in (\ref{canch}) also belongs to $C^{\infty}(\T\times T^*M)$. We 
assume that $H_1 \# H_2$ satisfies (H0)$^{\Lambda}$, while $H_1
\oplus H_2$ satisfies (H0)$^{\Theta}$. The aim of this section is to construct two chain maps
\begin{eqnarray*}
E: F_h^{\Lambda}(H_1) \otimes F_k^{\Lambda}(H_2)
\rightarrow F_{h+k-n}^{\Theta}(H_1 \oplus H_2), \\
G: F_k^{\Theta}(H_1 \oplus H_2) \rightarrow
F_k^{\Lambda} (H_1 \# H_2),
\end{eqnarray*}
such that the composition $G\circ E$ is chain homotopic to the
pair-of-pants chain map $\Upsilon^{\Lambda}$.

The homomorphisms $E$ is defined by counting solutions of
the Floer equation on the Riemann surface $\Sigma_E$ which is the
disjoint union of two closed disks with an inner and a boundary point
removed. The homomorphism $G$ is defined by counting solutions of
the Floer equation on the Riemann surface $\Sigma_G$ obtained by 
removing one inner point and two boundary points from the closed
disk. 
Again, we find it useful to represent these Riemann surfaces as
suitable quotients of strips with slits.

\begin{figure}[htb]
 \begin{center}
  \setlength{\unitlength}{0.01cm}
  \begin{picture}(1200,400)(0,0)
   \put(0,190){\epsfig{file=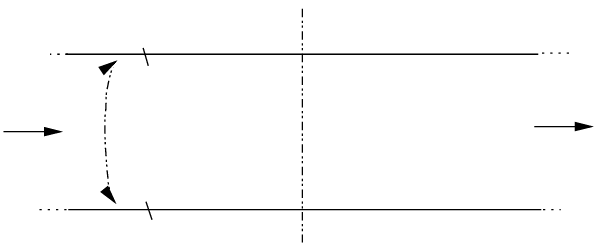,width=5cm}}
   \put(540,280){\makebox(0,0)[lb]{$\simeq$}}
   \put(540,50){\makebox(0,0)[lb]{$\simeq$}}
   \put(620,0){\epsfig{file=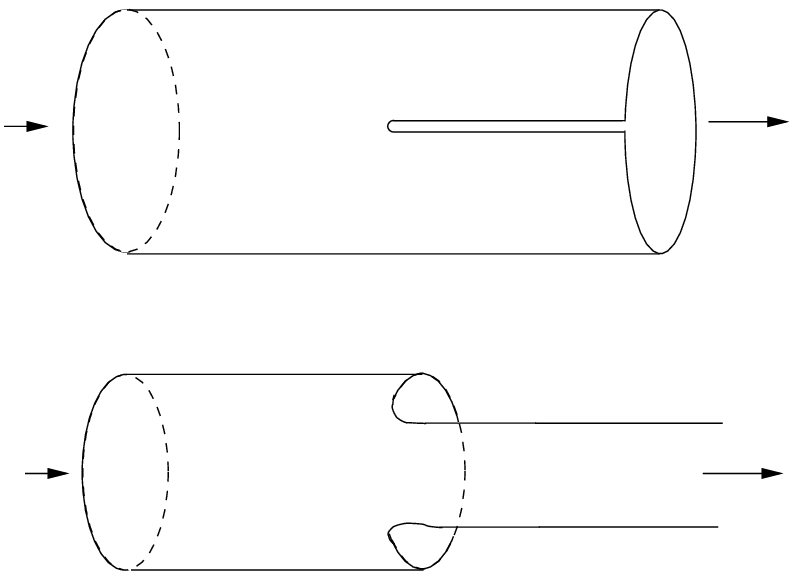,width=5cm}}
  \end{picture}
 \end{center}
\caption{A component of $\Sigma_E$: the cylinder with a slit.
         \label{cyl-slit}} 
\end{figure}

The surface $\Sigma_E$ can be described starting from the disjoint
union of two strips,
\[
\R \times [-1,0]\, \sqcup \, \R \times [0,1],
\]
by making the following identifications:
\[
(s,-1) \sim (s,0^-), \quad (s,0^+) \sim (s,1) \quad \mbox{for } s\leq 0.
\]
The complex structure of $\Sigma_E$ is constructed by
considering the holomorphic coordinate
\begin{equation}
\label{holcor3} 
\set{\zeta\in \C}{\im \zeta \geq 0, \; |\zeta| <
1/\sqrt{2}} \rightarrow \Sigma_E, \quad \zeta \mapsto \left\{ 
\begin{array}{ll} \zeta^2 - i & \mbox{if } \re \zeta \geq 0, \\
  \zeta^2 & \mbox{if } \re \zeta \leq 0, \end{array} \right.
\end{equation}
at $(0,-1)\sim (0,0^-)$, and the holomorphic coordinate
\begin{equation}
\label{holcor4} 
\set{\zeta\in \C}{\im \zeta \geq 0, \; |\zeta| <
1/\sqrt{2}} \rightarrow \Sigma_E, \quad \zeta \mapsto \left\{ 
\begin{array}{ll} \zeta^2 + i & \mbox{if } \re \zeta \geq 0, \\
  \zeta^2 & \mbox{if } \re \zeta \leq 0, \end{array} \right.
\end{equation}
at $(0,0^+) \sim (0,1)$. The resulting object is a Riemann surface
consisting of two disjoint components, each of which is a cylinder
with a slit: each component has one
cylindrical end (on the left-hand side), one strip-like
end and one boundary line (on the right-hand side). 
See Figure \ref{cyl-slit}. The global
holomorphic coordinate $z=s+it$ has two singular points, at
$(0,0^-)\sim (0,-1)$ and at $(0,0^+)\sim (0,1)$.

The Riemann surface $\Sigma_G$ is obtained from the disjoint union
of two strips $\R \times [-1,0] \sqcup \R
\times [0,1]$ by making the identifications:
\[
\left\{ \begin{array}{c} (s,0^-) \sim (s,0^+) \\ (s,-1) \sim (s,1)
\end{array} \right. \quad \mbox{for } s\geq 0.
\]
A holomorphic coordinate at $(0,0)$ is the one given by
(\ref{holcor1}), and a holomorphic coordinate at $(0,-1)\sim
(0,1)$ is:
\begin{equation}
\label{holcor5} \set{\zeta\in \C}{\re \zeta \geq 0, \; |\zeta| <
1} \rightarrow \Sigma_G, \quad \zeta \mapsto \left\{
\begin{array}{ll} \zeta^2 -i & \mbox{if } \im \zeta \geq 0,
\\ \zeta^2 + i & \mbox{if } \im \zeta\leq 0, \end{array}
\right.
\end{equation}
We obtain a Riemann surface with two boundary lines and two
strip-like ends (on the left-hand side), and a cylindrical end (on
the right-hand side). The global holomorphic coordinate $z=s+it$
has two singular points, at $(0,0)$, and at $(0,-1)\sim (0,1)$.

Let $H\in C^{\infty}(\R/2\Z \times T^*M)$ be defined by (\ref{hham}).
Given $x_1\in \mathscr{P}^{\Lambda}(H_1)$, $x_2\in
\mathscr{P}^{\Lambda}(H_2)$, $y=(y_1,y_2) \in
\mathscr{P}^{\Theta}(H_1 \oplus H_2)$, and $z\in
\mathscr{P}^{\Lambda}(H_1 \# H_2)$, we consider the following spaces of maps. The set $\mathscr{M}_E(x_1,x_2;y)$ is the space of solutions $u \in C^{\infty}(\Sigma_E,T^*M)$ of the Floer equation
\[
\delbar_{J,H}(u) = 0,
\]
satisfying the boundary conditions  
\[
\left\{ \begin{array}{l}
\pi \circ u(s,-1) = \pi \circ u(s,0^-) = \pi \circ u(s,0^+) = \pi \circ
u(s,1), \\ u(s,0^-) - u(s,-1) + u(s,1) - u(s,0^+) = 0, \end{array} \right.
\quad \forall s\geq 0, 
\]
and the asymptotic conditions
\begin{eqnarray*}
\lim_{s\rightarrow -\infty} u(s,t-1) =
x_1(t),\; \lim_{s\rightarrow -\infty} u(s,t) = x_2(t),\\
\lim_{s\rightarrow +\infty} u(s,t-1) = y_1(t), \;
\lim_{s\rightarrow +\infty} u(s,t) = y_2(t).
\end{eqnarray*}
The set $\mathscr{M}_G(y,z)$ is the set of solutions $u \in C^{\infty}(\Sigma_G,T^*M)$ of the same equation, the same boundary conditions but for $s\leq 0$, and the asymptotic conditions
\[
\lim_{s\rightarrow -\infty} u(s,t-1) = y_1(t), \;
\lim_{s\rightarrow -\infty} u(s,t) = y_2(t), \; \lim_{s\rightarrow
+\infty} u(s,2t-1) = z(t).
\]
The following result is proved in Section \ref{lin}.

\begin{prop}
\label{eg}
For a generic choice of $H_1$ and $H_2$, the spaces
$\mathscr{M}_E(x_1,x_2;y)$ and $\mathscr{M}_G(y,z)$ -- if
non-empty -- are manifolds of dimension
\[
\dim \mathscr{M}_E(x_1,x_2;y) = \mu^{\Lambda}(x_1) +
\mu^{\Lambda}(x_2) - \mu^{\Theta}(y) - n, \quad \dim
\mathscr{M}_G(y,z) = \mu^{\Theta}(y) - \mu^{\Lambda}(z).
\]
These manifolds carry coherent orientations.
\end{prop}

The energy identities are now
\begin{equation}
\label{seir1} \int_{\Sigma_E} |\partial_s
u(s,t)|^2 \, ds \,dt = \mathbb{A}_{H_1}(x_1) +
\mathbb{A}_{H_2}(x_2) - \mathbb{A}_{H_1 \oplus H_2} (y),
\end{equation}
for every $u\in \mathscr{M}_E(x_1,x_2;y)$, and
\begin{equation}
\label{seir2} \int_{\Sigma_G} |\partial_s
u(s,t)|^2 \, ds \,dt = \mathbb{A}_{H_1\oplus H_2}(y)
- \mathbb{A}_{H_1 \# H_2} (z),
\end{equation}
for every $u\in \mathscr{M}_G(y,z)$. As usual, they imply the
following compactness result (proved in Section
\ref{compsec}).

\begin{prop}
\label{comprop3}
Assume that $H_1,H_2$ satisfy (H1), (H2). Then the
spaces $\mathscr{M}_E(x_1,x_2;y)$ and $\mathscr{M}_G(y,z)$ are
pre-compact in $C^{\infty}_{\mathrm{loc}}$.
\end{prop}

When $\mu^{\Theta}(y) = \mu^{\Lambda}(x_1) + \mu^{\Lambda}(x_2) -
n$, $\mathscr{M}_E(x_1,x_2;y)$ is a finite set of oriented
points, and we denote by $n_E(x_1,x_2;y)$ the algebraic sum of the
corresponding orientation signs. Similarly, when $\mu^{\Lambda}(z)
= \mu^{\Theta}(y)$, $\mathscr{M}_G(y,z)$ is a finite set of
oriented points, and we denote by $n_G(y,z)$ the algebraic sum of
the corresponding orientation signs. These integers are the
coefficients of the homomorphisms
\begin{equation*}\begin{split}
E : &F_h^{\Lambda}(H_1) \otimes F_k^{\Lambda}(H_2) \rightarrow
F_{h+k-n}^{\Theta} (H_1 \oplus H_2),\\ &\quad x_1 \otimes x_2 \mapsto
\sum_{\substack{y\in \mathscr{P}^{\Theta}(H_1 \oplus H_2)\\
\mu^{\Theta}(y)=h+k-n}}
n_E(x_1,x_2;y)\, y, \\
G : &F_k^{\Theta} (H_1 \oplus H_2) \rightarrow F_k^{\Lambda} (H_1
\# H_2),\\ 
&\quad y \mapsto\sum_{\substack{z\in \mathscr{P}^{\Lambda}(H_1 \# H_2)\\
\mu^{\Lambda}(z)=k}} n_G(y,z)\, z.
\end{split}\end{equation*}

A standard gluing argument shows that these homomorphisms are
chain maps. The main result of this section states that the
pair-of-pants product on $T^*M$ factors through the Floer homology
of figure-8 loops. More precisely, the following chain level result holds:

\begin{thm}
\label{tge}
The chain maps
\begin{eqnarray*}
\Upsilon^{\Lambda}, 
\; G\circ E : \bigl( F^{\Lambda}(H_1) \otimes F^{\Lambda} 
(H_2)\bigr)_k = \bigoplus_{j+h=k} F_j^{\Lambda} (H_1) \otimes
\ F^{\Lambda}_h (H_2)\\
 \longrightarrow
F_{k-n}^{\Lambda}(H_1 \# H_2)
\end{eqnarray*}
are chain homotopic.
\end{thm}

In order to prove the above theorem, we must construct a homomorphism
\[
P_{GE}^{\Upsilon} : \bigl( F^{\Lambda}(H_1) \otimes F^{\Lambda} 
(H_2)\bigr)_k
\rightarrow F_{k-n+1}^{\Lambda}(H_1 \# H_2),
\]
such that
\begin{equation}\begin{split}
\label{chhom} 
&(\Upsilon^{\Lambda} - G \circ E) ( \alpha \otimes \beta) \\
=\quad &\partial^{\Lambda}_{H_1 \# H_2} \circ P_{GE}^{\Upsilon}
(\alpha \otimes \beta) +  P_{GE}^{\Upsilon}
\bigl( \partial^{\Lambda}_{H_1} \alpha \otimes \beta + (-1)^h
\alpha \otimes \partial^{\Lambda}_{H_2} \beta \bigr),
\end{split}\end{equation}
for every $\alpha\in F_h^{\Lambda}(H_1)$ and $\beta\in
F_j^{\Lambda}(H_2)$. The chain homotopy $P_{GE}^{\Upsilon}$ is 
defined by counting solutions of the Floer equation on a one-parameter 
family of Riemann surfaces with boundary $\Sigma_{GE}^{\Upsilon}(\alpha)$,
$\alpha\in ]0,+\infty[$, obtained by removing an open disks from
the pair-of-pants.

More precisely, given $\alpha\in ]0,+\infty[$, we define
$\Sigma_{GE}^{\Upsilon}(\alpha)$ as the quotient of the disjoint
union $\R \times [-1,0] \sqcup \R \times [0,1]$
under the identifications
\[
\left\{ \begin{array}{l} (s,-1) \sim (s,0^-) \\ (s,0^+) \sim (s,1)
\end{array} \right. \quad \mbox{if } s\leq 0, \quad \left\{
\begin{array}{l} (s,-1) \sim (s,1) \\ (s,0^-) \sim (s,0^+)
\end{array} \right. \quad \mbox{if } s \geq \alpha.
\]
This object is a Riemann surface with boundary, with the
holomorphic coordinates (\ref{holcor3}) and (\ref{holcor4})
 at $(0,-1)\sim (0,0^-)$ and at
$(0,0^+) \sim (0,1)$, and with the holomorphic coordinates
(\ref{holcor1}) and (\ref{holcor5}) (translated by $\alpha$) at
$(\alpha,0)$ and at $(\alpha,-1)\sim (\alpha,1)$. The resulting object 
is a Riemann surface with three cylindrical ends and one boundary circle.
Given $x_1\in \mathscr{P}^{\Lambda}(H_1)$, $x_2\in
\mathscr{P}^{\Lambda}(H_2)$, and $z\in \mathscr{P}^{\Lambda}(H_1
\# H_2)$, we define 
$\mathscr{M}_{GE}^{\Upsilon}(x_1,x_2;z)$ to be the space of pairs $(\alpha,u)$, with $\alpha>0$ and $u \in
C^{\infty}(\Sigma_{GE}^{\Upsilon}(\alpha),T^*M)$ the solution of
\[
\delbar_{J,H}(u) = 0,
\]
with boundary conditions
\[
\left\{ 
\begin{array}{l}  \pi \circ u(s,-1) = \pi \circ u(s,0^-) = \pi\circ u(s,0^+) = \pi\circ 
u(s,1), \\ u(s,0^-) - u(s,-1) + u(s,1) - u(s,0^+) =
0, \end{array} \right. \quad \forall s\in [0,\alpha], 
\]
and asymptotic conditions
\[
\lim_{s\rightarrow -\infty} u(s,t-1) = x_1(t), \quad
\lim_{s\rightarrow -\infty} u(s,t) = x_2(t), \quad \lim_{s\rightarrow
+\infty} u(s,2t-1) = z(t).
\]
The following result is proved in Section \ref{lin}.

\begin{prop}
\label{ge}
For a generic choice of $H_1$ and $H_2$,
$\mathscr{M}_{GE}^{\Upsilon}(x_1,x_2;z)$ -- if non-empty -- is a
manifold of dimension
\[
\dim \mathscr{M}_{GE}^{\Upsilon}(x_1,x_2;z) = \mu^{\Lambda}(x_1)
+ \mu^{\Lambda}(x_2) - \mu^{\Lambda}(z) - n+1.
\]
The projection $(\alpha,u)\mapsto \alpha$ is smooth on
$\mathscr{M}_{GE}^{\Upsilon}(x_1,x_2;z)$. These manifolds carry
coherent orientations.
\end{prop}

Energy estimates, together with (H1) and (H2), again imply compactness. When $\mu^{\Lambda}(z) =
\mu^{\Lambda}(x_1) + \mu^{\Lambda}(x_2) - n +1$,
$\mathscr{M}_{GE}^{\Upsilon}(x_1,x_2;z)$ is a finite set of
oriented points. Denoting by $n_{GE}^{\Upsilon}(x_1,x_2;z)$ the
algebraic sum of the corresponding orientation signs, we define
the homomorphism
\begin{equation*}\begin{split}
P_{GE}^{\Upsilon} : &F_h^{\Lambda}(H_1) \otimes F_k^{\Lambda}(H_2)
\rightarrow F_{h+k-n+1}^{\Lambda} (H_1 \# H_2),\\ \quad 
&x_1 \otimes x_2 \mapsto
\sum_{\substack{z\in \mathscr{P}^{\Lambda}(H_1 \# H_2)\\
\mu^{\Lambda}(z)=h+k-n+1}} n_{GE}^{\Upsilon}(x_1,x_2;z)\, z.
\end{split}\end{equation*}
Then Theorem \ref{tge} is a consequence of  the following:

\begin{prop}
\label{facthom}
The homomorphism $P_{GE}^{\Upsilon}$ is a chain homotopy between $
\Upsilon^{\Lambda}$ and $G\circ E$.
\end{prop}

The proof of the above result is contained in Section \ref{factsec}. 

\subsection{The homomorphisms $\mathbf{C}$, $\mathbf{Ev}$, and
$\mathbf{I_!}$}
\label{hcevi}

In this section we define the Floer homological counterparts of the the
homomorphisms
\begin{eqnarray*}
& \mathrm{c}_* : H_k(M) \rightarrow H_k(\Lambda(M)), \quad \mathrm{ev}_* :
H_k(\Lambda(M)) \rightarrow H_k(M), & \\ & \mathrm{i}_! : H_k(\Lambda(M)) \rightarrow H_{k-n}(\Omega(M,q_0)). &
\end{eqnarray*}
These are the homomorphism which will appear in the Floer homological counterpart of diagram (\ref{diaga}). The reader who is interested only in Theorems A and B of the Introduction may skip this section.

\paragraph{The homomorphisms $\mathbf{C}$ and $\mathbf{Ev}$.} Let $f$ be a smooth Morse function on $M$, and assume that the vector field $-\grad f$ satisfies the
Morse-Smale condition. Let $H\in C^{\infty}(\T \times T^*M)$ be a
Hamiltonian which satisfies (H0)$^{\Lambda}$, (H1), (H2). We shall
define two chain maps
\[
\mathrm{C}: M_k(f) \rightarrow F_k(H), \quad \mathrm{Ev}: F_k(H)
\rightarrow M_k(f).
\]
Given $x\in \mathrm{crit}(f)$ and $y\in \mathscr{P}^{\Lambda}(H)$,
consider the following spaces of maps
\begin{equation*}\begin{split}
\mathscr{M}_{\mathrm{C}}(x,y) = \Bigl\{ &u \in C^{\infty}([0,+\infty[ \times
\T, T^*M) \, \Big| \, \delbar_{J,H}(u) = 0,\\ &\pi\circ u
(0,t) \equiv q \in W^u(x) \; \forall t\in \T,\;
\lim_{s\rightarrow +\infty} u(s,t) = y(t) \Bigr\},
\end{split}\end{equation*}
and
\begin{eqnarray*}
\mathscr{M}_{\mathrm{Ev}}(y,x) = \Bigl\{ u \in
C^{\infty}(]-\infty,0] \times \T, T^*M) \, \Big| \, \delbar_{J,H}(u) = 0, \; u(0,t) \in \mathbb{O}_M \; \forall t\in \T, \\
u(0,0) \in W^s(x), \; \lim_{s\rightarrow -\infty} u(s,t) = y(t)
\Bigr\},
\end{eqnarray*}
where $\mathbb{O}_M$ denotes the image of the zero section in $T^*M$.
The following result is proved in Section \ref{lin}.

\begin{prop}
\label{cev}
For a generic choice of $H$ and $f$, $\mathscr{M}_C(x,y)$ and
$\mathscr{M}_{\mathrm{Ev}}(y,x)$ are manifolds with
\[
\dim \mathscr{M}_{\mathrm{C}}(x,y) = i(x) - \mu^{\Lambda}(y), \quad
\dim \mathscr{M}_{\mathrm{Ev}}(y,x) = \mu^{\Lambda}(y) - i(x).
\]
These manifolds carry coherent orientations.
\end{prop}

If $u$ belongs to $\mathscr{M}_{\mathrm{C}}(x,y)$ or
$\mathscr{M}_{\mathrm{Ev}}(y,x)$, the fact that $u(0,\cdot)$
takes value either on the fiber of some point $q\in M$ or
on the zero section of $T^*M$ implies that
\[
\mathbb{A}_H(u(0,\cdot)) = - \int_0^1 H(t,u(0,t)) \, dt.
\]
Therefore, we have the energy estimates
\[
\int_{[0,+\infty[ \times \T} |\partial_s u(s,t)
|^2 \, ds\,dt  \leq - \min H - \mathbb{A}_H(y),
\]
for every $u\in \mathscr{M}_{\mathrm{C}}(x,y)$, and
\[
\int_{]-\infty,0] \times \T} |\partial_s u(s,t)
|^2 \, ds\,dt  \leq \mathbb{A}_H(y) + \max_{(t,q)\in
\T \times M} H(t,q,0),
\]
for every $u\in \mathscr{M}_{\mathrm{Ev}}(y,x)$. These energy
estimates allow to prove the following compactness result:

\begin{prop}
The spaces $\mathscr{M}_{\mathrm{C}}(x,y)$ and
$\mathscr{M}_{\mathrm{Ev}}(y,x)$ are pre-compact in\\
$C^{\infty}_{\mathrm{loc}}([0,+\infty[\times \T,T^*M)$ and
$C^{\infty}_{\mathrm{loc}}(]-\infty,0]\times \T,T^*M)$.
\end{prop}

When $\mu^{\Lambda}(y)=i(x)$, $\mathscr{M}_{\mathrm{C}}(x,y)$ and
$\mathscr{M}_{\mathrm{Ev}}(y,x)$ consist of finitely many
oriented points. The algebraic sums of these orientation signs,
denoted by $n_{\mathrm{C}}(x,y)$ and $n_{\mathrm{Ev}}(y,x)$, define the
homomorphisms
\begin{eqnarray*}
\mathrm{C}:M_k(f) \rightarrow F_k^{\Lambda}(H), \quad x \mapsto
\sum_{\substack{y\in
\mathscr{P}^{\Lambda}(H) \\ \mu^{\Lambda}(y)=k}} n_{\mathrm{C}}(x,y)\, y, \\
\mathrm{Ev}:F_k^{\Lambda}(H) \rightarrow M_k(f), \quad y
\mapsto \sum_{\substack{x\in \crit (f) \\ i(x)=k}}
n_{\mathrm{Ev}}(y,x)\, x.
\end{eqnarray*}
A standard gluing argument shows that $\mathrm{C}$ and $\mathrm{Ev}$ are
chain maps. The induced homomorphisms in homology are denoted by $\mathrm{C}_*$ and $\mathrm{Ev}_*$. 

\paragraph{The homomorphism $\mathbf{I_!}$.} Let $\Sigma_{I_!}$ be a cylinder with a slit. More precisely,
$\Sigma_{I_!}$ is obtained from the strip $\R \times [0,1]$ by the
identifications $(s,0) \sim (s,1)$ for every $s\leq 0$. At the
point $(0,0)\sim(0,1)$ we have the holomorphic coordinate
\[
\set{\zeta\in \C}{\re \zeta\geq 0, \; |\zeta| <1/\sqrt{2}}
\rightarrow \Sigma_{I_!}, \quad \zeta \mapsto \left\{
\begin{array}{ll} \zeta^2 & \mbox{if } \im \zeta \geq 0, \\
\zeta^2 +i & \mbox{if } \im \zeta \leq 0. \end{array} \right.
\]
It is a Riemann surface with one cylindrical end (on the
left-hand side), one strip-like end, and one boundary line (on the
right-hand side). It is the  copy of one component of $\Sigma_E$,
see Figure \ref{cyl-slit}.

Consider now a Hamiltonian $H\in C^{\infty}(\T \times T^*M)$
which satisfies (H0)$^{\Lambda}$, (H0)$^{\Omega}$, (H1), (H2). 
We also assume the condition:
\begin{equation}
\label{puoi} x\in \mathscr{P}^{\Lambda}(H) \quad \implies \quad
x(0) \notin T_{q_0}^* M.
\end{equation}
Given $x\in \mathscr{P}^{\Lambda}(H)$ and $y\in
\mathscr{P}^{\Omega}(H)$, we introduce the space of maps
\begin{equation*}\begin{split}
\mathscr{M}_{I_!}(x,y) = \Bigl\{ u\in
C^{\infty}(\Sigma_{I_!},T^*M) \, \Big| \, \delbar_{J,H} (u) = 0, \;
u(s,0)\in T_{q_0}^*M \mbox{ and} \\ u(s,1) \in T_{q_0}^* M\; \forall s\geq 0, 
\; \lim_{s\rightarrow -\infty} u(s,t) = x(t), \; \lim_{s\rightarrow
+\infty} u(s,t) = y(t)  \Bigr\}.
\end{split}\end{equation*}
The following result is proved in Section \ref{lin}:

\begin{prop}
\label{i!p}
For a generic $H$ satisfying (\ref{puoi}), the space 
$\mathscr{M}_{I_!}(x,y)$ is a manifold, with
\[
\dim \mathscr{M}_{I_!}(x,y)= \mu^{\Lambda}(x) -
\mu^{\Omega}(y)-n.
\]
These manifolds carry coherent orientations.
\end{prop}

The following compactness statement follows from the general
discussion of Section \ref{compsec}.

\begin{prop}
\label{comprop4}
The space $\mathscr{M}_{I_!}(x,y)$ is pre-compact in
$C^{\infty}_{\mathrm{loc}}(\Sigma_{I_!},T^*M)$.
\end{prop}

When $\mu^{\Omega}(y)=\mu^{\Lambda}(x)-n$, the space
$\mathscr{M}_{I_!}(x,y)$ consists of finitely many oriented
points. The algebraic sum of these orientations is denoted by
$n_{I_!}(x,y)$, and defines the homomorphism
\[
I_!: F_k^{\Lambda}(H) \rightarrow F_{k-n}^{\Omega}(H), \quad x
\mapsto \sum_{\substack{y\in \mathscr{P}^{\Omega}(H) \\
\mu^{\Omega}(y) = k-n}} n_{I_!}(x,y)\, y.
\]
A standard gluing argument shows that $I_!$ is a chain map. The induced map in homology is denoted by the same symbol.

\section{Isomorphisms between Morse and Floer complexes}
\label{imfsec}

\subsection{The chain complex isomorphism}
\label{cci}

Let $Q$ be a closed manifold and let $R$ be a closed submanifold of $Q\times Q$, as in Sections \ref{mclaf} and \ref{fhpfo}. The aim of this section is to recall the construction of an isomorphism between $HF_*^R(T^*Q)$, the Floer homology of $T^*Q$ with nonlocal conormal boundary conditions defined by $R$ (see Section \ref{fhpfo}), and the singular homology of the path space
\[
P_R(Q) := \set{\gamma\in C^0([0,1],Q)}{(\gamma(0),\gamma(1))\in R}.
\]
The existence of such an isomorphism was first proved by C.\ Viterbo in \cite{vit96}, in the case of periodic boundary conditions (that is, when $R$ is the diagonal in $Q\times Q$). A different proof is due to D.\ Salamon and J.\ Weber, see \cite{sw06}. Here we adopt a third approach, which we have introduced in \cite{as06} for periodic and Dirichlet boundary conditions, and later extended to arbitrary nonlocal conormal boundary conditions with A.\ Portaluri in \cite{aps08}. See also \cite{web05} for a nice exposition comparing the three approaches.  

The strategy is to choose the Hamiltonian $H\in C^{\infty}([0,1]\times T^*Q)$ to be the Fenchel dual of a Lagrangian $L\in C^{\infty}([0,1]\times TQ)$, and to work at the chain level, by constructing a chain isomorphism
\begin{equation}
\label{theiso}
\Phi^R_L: M_*(\mathbb{S}_L^R) \rightarrow F^R_*(H)
\end{equation}
from the Morse complex of the Lagrangian action functional $\mathbb{S}_L^R$ introduced in Section \ref{mclaf} to the Floer complex of $H$. More precisely, we assume that the Lagrangian $L$ satisfies the assumptions (L1) and (L2) and that all the solutions $\gamma$ in $\mathscr{P}^R(L)$ are non-degenerate. It follows that the Fenchel dual Hamiltonian $H$, which is defined by
\[
H(t,q,p) := \max_{v\in T_q M} \bigl( \langle p,v\rangle - L(t,q,v) \bigr),
\]
is smooth and satisfies (H0), (H1) and (H2). If $v(t,q,p)\in T_q M$ is the (unique) vector where the above
maximum is achieved, the map
\[
[0,1] \times T^*M \rightarrow [0,1] \times TM, \quad (t,q,p)
\mapsto (t,q,v(t,q,p)),
\]
is a diffeomorphism, called the {\em Legendre transform}
associated to the Lagrangian $L$.
There is a one-to-one correspondence $x\mapsto \pi \circ x$ between the
orbits of the Hamiltonian vector field $X_H$ and the solutions of
the Euler-Lagrange equation (\ref{lageq}) associated to $L$, such that $(t,\pi\circ x(t), (\pi\circ x)'(t))$ is the Legendre transform of $(t,x(t))$. Therefore,
$x$ belongs to $\mathscr{P}^R(H)$ if and only if $\pi\circ x$ belongs to $\mathscr{P}^R(L)$, and the fact that $\pi\circ x$ is non-degenerate is equivalent to the fact that $x$ is non-degenerate. Therefore, both the Morse complex $M_*(\mathbb{S}_L^R)$ and the Floer complex $F_*^R(H)$ are well-defined. The existence of the isomorphism (\ref{theiso}) implies that the Floer homology $HF^R_*(T^*Q)$ is isomorphic to the singular homology of $P_Q(R)$, just because the Morse homology $HM_*(\mathbb{S}_L^R)$ of the functional $\mathbb{S}_L^R$ is isomorphic to the singular homology of its domain $W^{1,2}_R([0,1],Q)$, and because the latter space is homotopically equivalent to $P_Q(R)$.

Clearly, the identification between generators $\mathscr{P}^R(H) \rightarrow \mathscr{P}^R(L)$, $x\mapsto \pi\circ x$, need not produce a chain map, because the definitions of the boundary operator in the two complexes have little in common. The construction of the isomorphism $\Phi_L^R$ is based instead on counting solutions of a hybrid problem, that we now describe. 

Given $\gamma\in \mathscr{P}^R(L)$ and $x\in
\mathscr{P}^R(H)$, we denote by $\mathscr{M}_{\Phi}^R(\gamma,x)$ the space of maps $u: [0,+\infty[ \times [0,1] \rightarrow  T^*Q$ which 
solve the Floer equation
\begin{equation}
\label{floen}
\delbar_{J,H}(u) = 0,
\end{equation}
together with the asymptotic condition
\begin{equation}
\label{asym}
\lim_{s\rightarrow +\infty} u(s,t) = x(t),
\end{equation}
and the boundary conditions
\[
(u(s,0), \mathscr{C}u(s,1)) \in N^* R, \quad \forall s\geq 0, \quad
\pi \circ u(0,\cdot) \in W^u(\gamma),
\]
where $W^u(\gamma)\subset W^{1,2}_R([0,1],Q)$ denotes the unstable manifold of $\gamma$ with respect to the pseudo-gradient vector field $X$ for $\mathbb{S}_L^R$ used in the definition of the Morse complex of $\mathbb{S}_L^R$. 
For a generic choice of $L$ and of the pseudo-gradient $X$, these spaces of maps are manifolds of dimension
\[
\dim \mathscr{M}_{\Phi}^R(\gamma,x) = i(\gamma;\mathbb{S}_L^R) -
\mu^R(x),
\]
where $i(\gamma;\mathbb{S}_L^R)$ denotes the Morse index of $\gamma$, seen as a critical point of $\mathbb{S}_L^R$, and $\mu^R(x)$ is the Maslov index of $x$ defined in (\ref{mascon}). 
The fact that $H$ is the Fenchel dual of $L$ immediately implies the following important inequality between the Hamiltonian and the Lagrangian action functionals:
\begin{equation}
\label{fencsti} \mathbb{A}_H(x) \leq \mathbb{S}_L(\pi \circ x),
\quad \forall x\in C^{\infty}([0,1],T^*Q),
\end{equation}
with the equality holding if and only if $x$ is related to $(\pi\circ x,
(\pi\circ x)')$ by the Legendre transform. In particular, the equality
holds if $x$ is an orbit of the Hamiltonian
vector field $X_H$. The inequality (\ref{fencsti}) provides us with the energy
estimates which allow to prove suitable compactness properties for
the spaces $\mathscr{M}_{\Phi}^R(\gamma,x)$. When
$\mu^R(x)=i(\gamma;\mathbb{S}_L^R)$, the space
$\mathscr{M}_{\Phi}^R(\gamma,x)$ consists of finitely
many oriented points, which add up to the integers
$n_{\Phi}^R(\gamma,x)$. These integers are the
coefficients of the homomorphism
\[
\Phi^R_L : M_k(\mathbb{S}_L^R) \rightarrow
F^R_k (H,J), \quad \gamma \mapsto \sum_{\substack{x\in
\mathscr{P}^R(H) \\ \mu^R (x)=k}} n_{\Phi}^R
(\gamma,x)\, x, 
\]
which can be shown to be a chain map. The inequality
(\ref{fencsti})  implies
that $n_{\Phi}^R(\gamma,x)=0$ if $\mathbb{A}_H(x)\geq
\mathbb{S}_L(\gamma)$ and $\gamma \neq \pi \circ x$, while
$n_{\Phi}^R(\gamma,x)=\pm 1$ if $\gamma = \pi \circ x$. These facts
imply that $\Phi^R_L$ is an isomorphism.

As in Sections \ref{mclaf} and \ref{fhpfo}, we are interested in the boundary conditions given by the following choices for $R$: Periodic ($Q=M$ and $R=\Delta_M$), Dirichlet ($Q=M$ and $R=(q_0,q_0)$), Figure-8 ($Q=M\times M$, $R=\Delta_M^{(4)}$). We denote the corresponding isomorphisms also by the symbols
\[
\Phi_L^{\Lambda} := \Phi_L^{\Delta_M}, \quad  \Phi_L^{\Omega} := \Phi_L^{(q_0,q_0)}, \quad \Phi_L^{\Theta} := \Phi_L^{\Delta_M^{(4)}}.
\]

\subsection{A chain level proof of Theorem B}
\label{ori}

Theorem B of the Introduction says that if $M$ is a closed manifold and $q_0\in M$, then there is a graded ring isomorphism
\[
H_*(\Omega(M,q_0)) \cong HF_*^{\Omega}(T^*M)
\]
where the first graded group is endowed with the Pontrjagin product $\#$ (see Section \ref{tpp}), and the second one with the triangle product $\Upsilon^{\Omega}_*$ (see Section \ref{tppp}). By the results of Section \ref{mtipp}, the Pontrjagin product $\#$ can be read on the Morse complex of the Lagrangian action functional with Dirichlet boundary conditions as the chain map
\[
M_{\#} : M_j(\mathbb{S}_{L_1}^{\Omega}) \otimes
M_k(\mathbb{S}_{L_2}^{\Omega})
\rightarrow M_{j+k} (\mathbb{S}_{L_1 \# L_2}^{\Omega}),
\]
defined in Proposition \ref{prop MS}. Therefore, Theorem B is implied 
by the following chain level statement, which is the main result of this section:

\begin{thm}
\label{Bchain}
Let $L_1,L_2\in C^{\infty}([0,1] \times TM)$ be two 
Lagrangians which satisfy  (L0)$^{\Omega}$, (L1), (L2), are such that $L_1(1,\cdot)=L_2(0,\cdot)$ with all the time derivatives, and (\ref{nofig8}) holds. Assume also that the Lagrangian $L_1 \# L_2$ defined
by (\ref{ide}) satisfies (L0)$^{\Omega}$. Let $H_1$ and $H_2$ be the Fenchel
transforms of $L_1$ and $L_2$, so that $H_1 \# H_2$ is the Fenchel
transform of $L_1 \# L_2$, and the three Hamiltonians $H_1$, $H_2$, and
$H_1 \# H_2$ satisfy (H0)$^{\Omega}$, (H1), (H2). Then the diagram
\[
\xymatrix@!C{ \left(M(\mathbb{S}_{L_1}^{\Omega})
  \otimes M (\mathbb{S}_{L_2}^{\Omega})\right)_k \ar[r]^{M_{\#}}
  \ar[d]_{\Phi^{\Omega}_{L_1} \otimes \Phi^{\Omega}_{L_2}} & M_k
  (\mathbb{S}_{L_1 \# L_2}^{\Omega}) \ar[d]^{\Phi^{\Omega}_{L_1 \# L_2}} \\   
\left(F^{\Omega}(H_1) \otimes F^{\Omega} (H_2)\right)_k
\ar[r]^{\Upsilon^{\Omega}}  
& F_k^{\Omega}(H_1 \# H_2) }
\]
is chain-homotopy commutative. 
\end{thm}

Instead of constructing directly a 
homotopy between $\Phi_{L_1 \# L_2}^{\Omega} \circ M_{\#}$ and
$\Upsilon^\Omega \circ \Phi_{L_1}^{\Omega} \otimes
\Phi_{L_2}^{\Omega}$, we shall prove that both these chain maps are
homotopic to a third one, that we name $K^{\Omega}$, see Figure \ref{KOmega}.
\begin{figure}[htb]
 \begin{center}
  \setlength{\unitlength}{0.01cm}
  \begin{picture}(1200,700)(0,0)
   \put(0,350){\epsfig{file=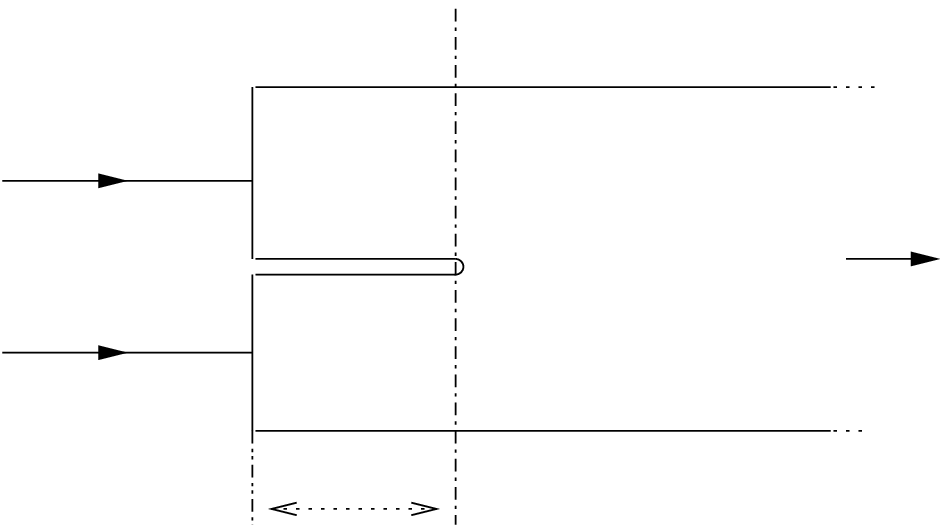,width=5cm}}
   \put(170,370){\makebox{$\alpha$}}
   \put(550,450){\makebox{$\simeq$}}
   \put(620,400){\epsfig{file=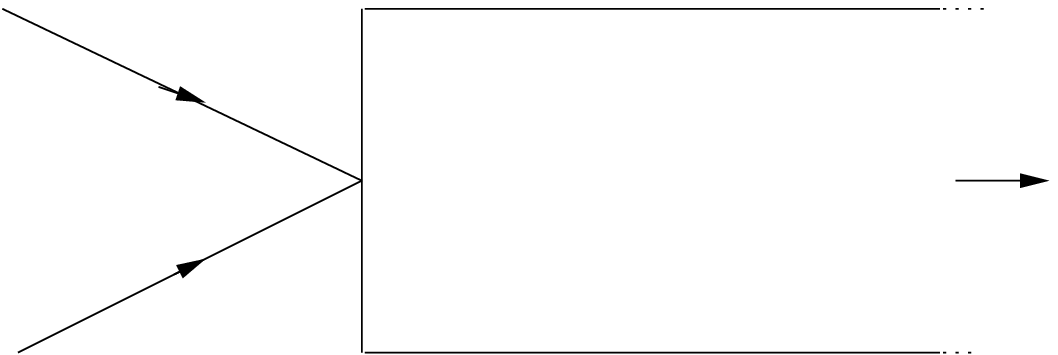,width=55mm}}
   \put(100,300){\makebox{$\Upsilon^\Omega \circ \Phi_{L_1}^{\Omega} \otimes
\Phi_{L_2}^{\Omega}$}}
   \put(900,320){\makebox{$K^{\Omega}$}}
   \put(400,130){\makebox{$\simeq$}}
   \put(470,30){\epsfig{file=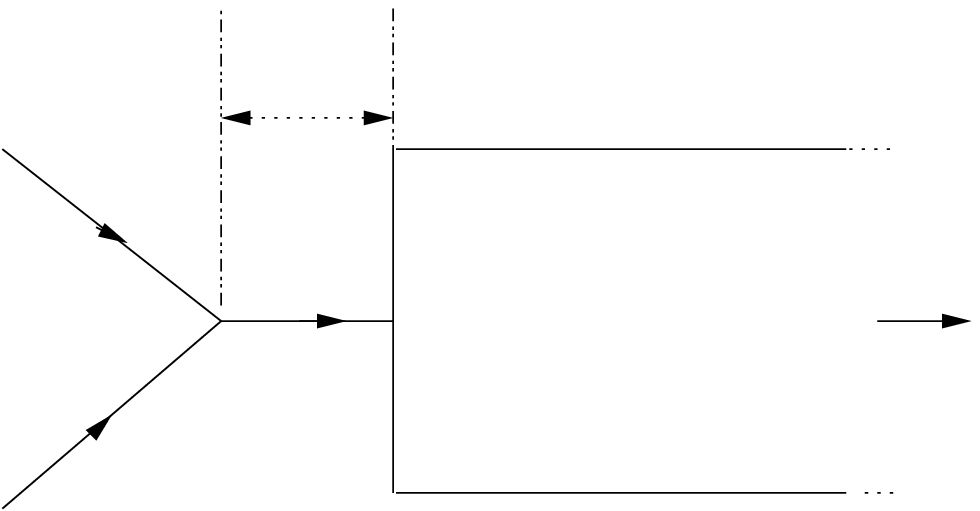,width=50mm}}
  \put(625,260){\makebox{$\alpha$}}
  \put(580,-20){\makebox{$\Phi_{L_1 \# L_2}^{\Omega} \circ M_{\#}$}}
  \end{picture}
 \end{center}
\caption{The homotopy through $K^{\Omega}$.
  \label{KOmega}}
\end{figure}

The definition of $K^{\Omega}$ is based on the following space of
solutions of the Floer equation for the Hamiltonian $H$ 
defined in (\ref{hham}): given $\gamma_1\in \mathscr{P}^{\Omega}(L_1)$, $\gamma_2\in
\mathscr{P}^{\Omega}(L_2)$, and $x\in \mathscr{P}^{\Omega}(H_1 \#
H_2)$, let $\mathscr{M}_K^{\Omega} (\gamma_1,\gamma_2;x)$ be the space of solutions $u:[0,+\infty[ \times [-1,1] \rightarrow T^*M$ of the Floer equation
\[
\delbar_{J,H}(u)=0,
\]
with boundary conditions
\begin{eqnarray*}
\pi\circ u (s,-1) = \pi \circ u(s,1) = q_0, \quad
  \forall s\geq 0, \\
\pi\circ u(0,\cdot-1) \in W^u(\gamma_1; X_{L_1}^{\Omega}), \;  \pi\circ u(0,\cdot) \in
W^u(\gamma_2; X_{L_2}^{\Omega}),
\end{eqnarray*}
and the asymptotic behavior  
\[
\lim_{s\rightarrow +\infty} u(s,2t-1) = x(t).
\]
Here, $X_{L_1}^{\Omega}$ and $X_{L_2}^{\Omega}$ are Morse-Smale pseudo-gradients for the functionals $\mathbb{S}_{L_1}^{\Omega}$ and $\mathbb{S}_{L_2}^{\Omega}$.
Theorem 3.2 in \cite{as06} (or the arguments of Section
\ref{lin}) implies that for a generic choice of the Lagrangians $L_1$, $L_2$, and of the  pseudo-gradients $X_{L_1}^{\Omega}$, $X_{L_2}^{\Omega}$, the space
$\mathscr{M}_K^{\Omega} (\gamma_1,\gamma_2;x)$ -- if non-empty -- is a
smooth manifold of dimension
\[
\dim \mathscr{M}_K^{\Omega} (\gamma_1,\gamma_2;x) =
i^{\Omega}(\gamma_1; L_1) + i^{\Omega}(\gamma_2;
L_2) - \mu^{\Omega}(x;H_1 \# H_2).
\]
These manifolds carry coherent orientations.
The energy identity is now
\[
\int_{[0,+\infty[ \times [-1,1]} 
|\partial_s u(s,t)|^2 \, ds\, dt =
\mathbb{A}_{H_1} (x_1) + \mathbb{A}_{H_2} (x_2) - \mathbb{A}_{H_1\#
  H_2}(x),
\]
where $x_1(t)=u(0,t-1)$ and $x_2(t)=u(0,t)$. Since $\pi\circ x_1$ is
in the unstable manifold of $\gamma_1$ and $\pi\circ x_2$ is
in the unstable manifold of $\gamma_2$, the inequality (\ref{fencsti})
implies that
\[
\mathbb{A}_{H_1} (x_1) \leq \mathbb{S}_{L_1} (\gamma_1), \quad  
\mathbb{A}_{H_2} (x_2) \leq \mathbb{S}_{L_2} (\gamma_2),
\]
so the elements $u$ of $\mathscr{M}_K^{\Omega}(\gamma_1,\gamma_2;x)$ 
satisfy the energy estimate
\begin{equation}
\label{enesti}
\int_{[0,+\infty[ \times [-1,1]} 
|\partial_s u(s,t)|^2 \, ds\, dt \leq \mathbb{S}_{L_1}
(\gamma_1) + \mathbb{S}_{L_2} (\gamma_2) - \mathbb{A}_{H_1\#
  H_2}(x).
\end{equation}
When $i^{\Omega}(\gamma_1;L_1) + i^{\Omega}(\gamma_2;L_2) = \mu^{\Omega}(x;H_1 \oplus H_2) $, the space 
$\mathscr{M}_K^{\Omega}(\gamma_1,\gamma_2;x)$ is a
compact zero-dimensional oriented manifold. If
$n_K^{\Omega}(\gamma_1,\gamma_2;x)$ is the algebraic sum of its  points, we can
define the homomorphism
\begin{equation*}\begin{split}
K^{\Omega} \colon\; &\bigl( M(\mathbb{S}^{\Omega}_{L_1}) \otimes
M(\mathbb{S}^{\Omega}_{L_2}) \bigr)_k \rightarrow
F_k^{\Omega}(H_1 \# H_2),\\
&\gamma_1\otimes \gamma_2 \mapsto \sum_{\substack{x\in \mathscr{P}^{\Omega}(H_1
    \# H_2)\\ \mu^{\Omega}(x) = k}}
n_K^{\Omega}(\gamma_1,\gamma_2;x)\, x.
\end{split}\end{equation*}
A standard gluing argument shows that $K^{\Omega}$ is a chain map. 

It is easy to construct a homotopy $P_K^{\#}$ between $\Phi_{L_1\#
  L_2}^{\Omega}\circ M_{\#}$ and $K^{\Omega}$. In fact, it is enough
to consider the space of pairs $(\alpha,u)$, where $\alpha$ is a
positive number and $u$ is a solution of the Floer equation on
$[0,+\infty[ \times [-1,1]$ which converges to $x$ for $s\rightarrow
+\infty$, and such that the curve $t\mapsto \pi\circ u(0,2t-1)$
belongs to the evolution at time $\alpha$ of 
\[
\Gamma \bigl( W^u(\gamma_1; X_{L_1}^{\Omega}) \times
W^u(\gamma_2; X_{L_2}^{\Omega}) \bigr),
\]
by flow of $X_{L_1 \# L_2}^{\Omega}$, a pseudo-gradient for $\mathbb{S}_{L_1 \# L_2}^{\Omega}$. Here $\Gamma$ is the concatenation map defined in Section \ref{tpp}. More precisely, set
\begin{equation*}\begin{split}
\mathscr{M}_K^{\#} (\gamma_1,\gamma_2;x)  := \Bigl\{ (\alpha,u) \,
\Big| \, \alpha>0, \; u: [0,+\infty[ \times [-1,1] \rightarrow T^*M 
\mbox{ solves } \overline{\partial}_{J,H}(u)=0, \\
\pi\circ u (s,-1) = \pi \circ u(s,1) = q_0 \;
  \forall s\geq 0, \;
\lim_{s\rightarrow +\infty} u(s,2t-1) = x(t),  \\
\pi\circ u(0,2 \cdot-1) \in \phi_{\alpha}^{\Omega}\bigl( \Gamma(W^u(\gamma_1;X_{L_1}^{\Omega}) \times 
W^u(\gamma_2;X_{L_2}^{\Omega})) \bigr) \Bigr\},
\end{split}\end{equation*} 
where $\phi_s^{\Omega}$ denotes the flow of $X_{L_1 \#
  L_2}^{\Omega}$. For a generic choice of $L_1$, $L_2$, $X_{L_1}^{\Omega}$, $X_{L_2}^{\Omega}$, and $X_{L_1 \# L_2}^{\Omega}$, the space
$\mathscr{M}_K^{\#} (\gamma_1,\gamma_2;x)$ -- if non-empty -- is a
smooth manifold of dimension
\[
\dim \mathscr{M}_K^{\#} (\gamma_1,\gamma_2;x) =
i^{\Omega}(\gamma_1; L_1) + i^{\Omega}(\gamma_2;
L_2) - \mu^{\Omega}(x;H_1 \# H_2) + 1,
\]
and these manifolds carry coherent orientations. The energy estimate
is again (\ref{enesti}). By counting the elements of the
zero-dimensional manifolds, we obtain a homomorphism
\[
P^{\#}_K : \bigl( M(\mathbb{S}^{\Omega}_{L_1}) \otimes
M(\mathbb{S}^{\Omega}_{L_2}) \bigr)_k \rightarrow
F_{k+1}^{\Omega}(H_1 \# H_2).
\]
A standard gluing argument shows that $P^{\#}_K$ is a chain
homotopy between $\Phi_{L_1\#
  L_2}^{\Omega}\circ M_{\#}$ and $K^{\Omega}$.
    
The homotopy $P^K_{\Upsilon}$ between $K^{\Omega}$ and $\Upsilon^{\Omega} \circ (
\Phi_{L_1}^{\Omega} \otimes \Phi_{L_2}^{\Omega})$ 
is defined by counting solutions of the
Floer equation on a one-parameter family of Riemann surfaces
$\Sigma^K_{\Upsilon}(\alpha)$, obtained by removing a point from the
closed disk. More precisely, given $\alpha>0$, we define
$\Sigma_{\Upsilon}^K(\alpha)$ as the quotient of the disjoint union
$[0,+\infty[ \times [-1,0] \sqcup [0,+\infty[ \times [0,1]$ under the
identification
\[
(s,0^-) \sim (s,0^+) \quad \mbox{for } s\geq \alpha.
\]
This object is a Riemann surface with boundary: Its complex
structure at each interior point and at each boundary point other than
$(\alpha,0)$ is induced by the inclusion, whereas the holomorphic
coordinate at $(\alpha,0)$ is given by the map
\[
\set{\zeta\in \C}{\re \zeta\geq 0, \; |\zeta|<\epsilon} \rightarrow
\Sigma_{\Upsilon}^K(\alpha), \quad \zeta \mapsto \alpha + \zeta^2,
\]
where the positive number $\epsilon$ is smaller than $1$ and $\sqrt{\alpha}$.
Given $\gamma_1\in \mathscr{P}(L_1)$, $\gamma_2\in \mathscr{P}(L_2)$,
and $x\in \mathscr{P}(H_1 \# H_2)$, we consider the space of pairs
$(\alpha,u)$ where $\alpha$ is a positive number, and $u(s,t)$ is a
solution of the Floer equation on $\Sigma^K_{\Upsilon}(\alpha)$ which converges
to $x$ for $s\rightarrow +\infty$, lies above some element in the
unstable manifold of $\gamma_1$ for $s=0$ and $-1\leq t \leq 0$,
lies above some element in the unstable manifold of $\gamma_2$ for
$s=0$ and $0 \leq t \leq 1$, and lies above $q_0$ at all the other
boundary points. More precisely, $\mathscr{M}_{\Upsilon}^K (\gamma_1,\gamma_2;x) $ is the set of pairs $(\alpha,u)$ where $\alpha$ is a positive number and $u: \Sigma_{\Upsilon}^K(\alpha) \rightarrow T^*M$ is a solution of 
\[
\delbar_{J,H}(u) = 0,
\]
which satisfies the boundary conditions
\begin{eqnarray*}
\pi\circ u (s,-1) = \pi \circ u(s,1) = q_0, \quad \forall s\geq 0, \\ 
\pi \circ u (s,0^-) = \pi \circ u (s,0^+) = q_0, \quad \forall s\in [0,\alpha], \\
\pi\circ u(0,\cdot-1) \in W^u(\gamma_1;X_{L_1}^{\Omega}), \quad  \pi\circ u(0,\cdot) \in
W^u(\gamma_2;X_{L_2}^{\Omega}),
\end{eqnarray*}
and the asymptotic condition
\[
\lim_{s\rightarrow +\infty} u(s,2t-1) = x(t).
\]
The following result is proved in Section \ref{lin}.

\begin{prop}
\label{POmega}
For a generic choice of $L_1$, $L_2$, $X_{L_1}^{\Omega}$, and $X_{L_2}^{\Omega}$,
$\mathscr{M}_{\Upsilon}^K (\gamma_1,\gamma_2;x)$ -- if non-empty -- is a
smooth manifold of dimension
\[
\dim \mathscr{M}_{\Upsilon}^K (\gamma_1,\gamma_2;x) =
i^{\Omega}(\gamma_1; L_1) + i^{\Omega}(\gamma_2;
L_2) - \mu^{\Omega}(x;H_1 \# H_2) + 1.
\]
These manifolds carry coherent orientations.
\end{prop}

As before, the elements $(\alpha,u)$ of
$\mathscr{M}_{\Upsilon}^K(\gamma_1,\gamma_2;x)$ satisfy the energy
estimate
\[
\int_{\Sigma^K_{\Upsilon} (\alpha)} 
|\partial_s u(s,t)|^2 \, ds\, dt \leq \mathbb{S}_{L_1}
(\gamma_1) + \mathbb{S}_{L_2} (\gamma_2) - \mathbb{A}_{H_1\#
  H_2}(x),
\]
which allows to prove compactness. By counting the zero-dimensional components, we define a homomorphism
\[
P^K_{\Upsilon}: \bigl( M(\mathbb{S}^{\Omega}_{L_1}) \otimes
M(\mathbb{S}^{\Omega}_{L_2}) \bigr)_k \rightarrow
F_{k+1}^{\Omega}(H_1 \# H_2,J).
\]
The conclusion arises from the following:

\begin{prop}
\label{omegahom}
The homomorphism $P^K_{\Upsilon}$ is a chain homotopy between $K^{\Omega}$ and  $\Upsilon^{\Omega} \circ (
\Phi_{L_1}^{\Omega} \otimes \Phi_{L_2}^{\Omega})$.
\end{prop}

The proof of the above proposition is contained in Section \ref{omegasec}. It is again a compactness-cobordism argument. The analytical tool is the implicit function theorem together with a suitable family of conformal transformations of the half-strip. This concludes the proof of Theorem \ref{Bchain}, hence of Theorem B.

\subsection{A chain level proof of Theorem A}
\label{oria}

Theorem A of the Introduction says that if $M$ is an oriented closed manifold, then there is a graded ring isomorphism
\[
H_*(\Lambda(M)) \cong HF_*^{\Lambda}(T^*M)
\]
where the first graded group is endowed with the loop product $\loopprod$ (see Section \ref{tcslp}), and the second one with the pair-of-pants product $H_* \Upsilon^{\Lambda}$ (see Section \ref{tppp}).

The loop product is the composition of two non-trivial homomorphisms:
The first one is the exterior homology product $\times$ followed by the  Umkehr map $e_!$, that is
\[
e_!\circ \times :H_h(\Lambda(M)) \otimes H_j(\Lambda(M)) \longrightarrow H_{h+j-n} (\Theta(M)),
\]
the second one is the homomorphism induced by concatenation,
\[
\Gamma_*: H_k (\Theta(M)) \longrightarrow H_k (\Lambda(M)).
\]
In Section \ref{mtilp}, we have shown how these two homomorphisms can be
read on the Morse complexes of the Lagrangian action functional with either periodic or figure-8 boundary conditions: $e_!\circ \times$ is induced by the chain map
\[
M_{!} : M_h(\mathbb{S}_{L_1}^{\Lambda}) \otimes M_j
(\mathbb{S}_{L_2}^{\Lambda})
\longrightarrow M_{h+j-n} (\mathbb{S}_{L_1  \oplus L_2}^{\Theta} ), 
\]
defined immediately after Proposition \ref{prop e!}, and $\Gamma_*$ by the chain map
\[
M_{\Gamma} :  M_k (\mathbb{S}_{L_1 \oplus L_2}^{\Theta}) \longrightarrow M_k(\mathbb{S}_{L_1 \# L_2}^{\Lambda}), 
\]
defined in Proposition \ref{prop MG}.
On the other hand, by Theorem \ref{tge} the pair-of-pants product $\Upsilon^{\Lambda}_*$ on the Floer homology of $T^*M$ with periodic boundary conditions is induced by the composition of the chain maps
\[
E : F^{\Lambda}_h(H_1) \otimes F^{\Lambda}_j (H_2) \longrightarrow F^{\Theta}_{h+j-n} (H_1 \oplus H_2)
\]
and
\[
G: F^{\Theta}_k (H_1 \oplus H_2) \longrightarrow F^{\Lambda}_k (H_1 \# H_2). 
\]
Therefore, Theorem A is an immediate consequence of the following chain level result:

\begin{thm}
\label{chainA}
Let $L_1,L_2\in C^{\infty}(\T \times TM)$ be two Lagrangians
which satisfy (L0)$^{\Lambda}$, (L1), (L2), are such that $L_1(0,\cdot)=L_2(0,\cdot)$ with all the time derivatives, and
satisfy (\ref{nofig8s}), or equivalently (\ref{nofig8n}). 
Assume also that the Lagrangian $L_1 \# L_2$ defined
by (\ref{ide}) satisfies (L0)$^{\Lambda}$. Let $H_1$ and $H_2$ be the Fenchel
transforms of $L_1$ and $L_2$, so that $H_1 \# H_2$ is the Fenchel
transform of $L_1 \# L_2$, and the three Hamiltonians $H_1$, $H_2$, and
$H_1 \# H_2$ satisfy (H0)$^{\Lambda}$, (H1), (H2). Then the two squares in the diagram
\begin{equation*}
\xymatrix@!C{ 
\scriptstyle{\bigl(M(\mathbb{S}_{L_1}^{\Lambda})
  \otimes M(\mathbb{S}_{L_2}^{\Lambda})\bigr)_k} \ar^{M_{!}}[r]
  \ar[d]_{\Phi^{\Lambda}_{L_1} \otimes \Phi^{\Lambda}_{L_2}}  &
  \scriptstyle{M_{k-n} 
  (\mathbb{S}_{L_1 \oplus L_2}^{\Theta})} \ar[d]^{\Phi^{\Theta}_{L_1 \oplus L_2}}
  \ar[r]^{M_{\Gamma}} & \scriptstyle{M_{k-n} (\mathbb{S}^{\Lambda}_{L_1 \#
    L_2})} \ar[d]^{\Phi_{L_1\# L_2}^{\Lambda}} \\   
\scriptstyle{\bigl(F^{\Lambda}(H_1) \otimes F^{\Lambda} (H_2)\bigr)_k}
\ar[r]^{E}  &  
\scriptstyle{F_{k-n}^{\Theta} (H_1 \oplus H_2)} \ar[r]^{G} &
\scriptstyle{F_{k-n}^{\Lambda} (H_1 \# H_2)} }
\end{equation*}
commute up to chain homotopies.
\end{thm}

The chain homotopy commutativity of the left-hand square is more delicate and  is proved in the next section. The second square is studied in Section \ref{chrhs}.

\subsection{The left-hand square is homotopy commutative}
\label{chlhs}

In this section we show that the chain maps $\Phi_{L_1
  \oplus L_2}^{\Theta} \circ M_{!}$ and $E \circ (\Phi_{L_1}^{\Lambda}
\otimes \Phi_{L_2}^{\Lambda})$ are homotopic. We start by
constructing a one-parameter family of chain maps
\[
K^{\Lambda}_{\alpha} : \bigl(M(\mathbb{S}_{L_1}^{\Lambda}) \otimes M(\mathbb{S}_{L_2}^{\Lambda})\bigr)_* \longrightarrow F^{\Theta}_{*-n} (H_1 \oplus H_2), 
\]
where $\alpha$ is a non-negative number. The definition of
$K^{\Lambda}_{\alpha}$ is based on the solution spaces of the Floer
equation on the Riemann surface $\Sigma^K_{\alpha}$ consisting of a
half-cylinder with a slit. 
More precisely, when $\alpha$ is positive $\Sigma^K_{\alpha}$ is the quotient of $[0,+\infty[ \times [0,1]$ by the identifications
\[
(s,0)\sim (s,1) \quad \forall s\in [0,\alpha].
\]
with the holomorphic coordinate at $(\alpha,0)\sim (\alpha,1)$ 
obtained from (\ref{holcor4}) by a 
translation by $\alpha$. When $\alpha=0$, $\Sigma^K_{\alpha} = \Sigma^K_0$ is just the half-strip $[0,+\infty[ \times [0,1]$.
Fix $\gamma_1\in \mathscr{P}^{\Lambda}(L_1)$, $\gamma_2\in\mathscr{P}^{\Lambda}(L_2)$, and $x\in \mathscr{P}^{\Theta}(H_1 \oplus H_2)$. Let $\mathscr{M}^K_{\alpha}(\gamma_1,\gamma_2;x)$ be the space of solutions $u: \Sigma_{\alpha}^K \rightarrow  T^*M^2$ of the equation
\[
\delbar_{J,H_1 \oplus H_2} (u) = 0,
\]
which satisfy the boundary conditions
\begin{eqnarray}
\label{bkl1}
& & \pi\circ u(0,\cdot) \in W^u((\gamma_1,\gamma_2); X_{L_1\oplus L_2}^{\Lambda}), \\ 
\label{bkl2}
& & (u(s,0), \mathscr{C} u(s,1)) \in N^* \Delta^{(4)}_M, \quad \forall s\geq
\alpha, \\ \label{bkl3} & & \lim_{s\rightarrow +\infty} u(s,t) = x(t),
\end{eqnarray}
where $X_{L_1\oplus L_2}^{\Lambda}$ is a pseudo-gradient for $\mathbb{S}_{L_1 \oplus L_2}^{\Lambda}$ on $\Lambda^1 (M \times M)$.
Let us fix some $\alpha_0\geq 0$. The following result is proved in Section \ref{lin}:

\begin{prop}
\label{Kappa}
For a generic choice of the Lagrangians $L_1$, $L_2$, and of the pseudo-gradient $X_{L_1 \oplus L_2}^{\Lambda}$, the space $\mathscr{M}^{K}_{\alpha_0} (\gamma_1,\gamma_2;x)$ -- if non-empty -- is a smooth manifold of dimension
\[
\dim \mathscr{M}^{K}_{\alpha_0}(\gamma_1,\gamma_2;x) =
i^{\Lambda}(\gamma_1;L_1) + i^{\Lambda}(\gamma_2;L_2) -
\mu^{\Theta}(x;H_1 \oplus H_2) - n.
\]
These manifolds carry coherent orientations.
\end{prop}

Compactness is again a consequence of the energy estimate
\begin{equation}
\label{enniste}
\int_{ \Sigma_{\alpha}^K} |\partial_s u(s,t)|^2 \, ds\,dt 
\leq \mathbb{S}_{L_1}(\gamma_1) +
\mathbb{S}_{L_2}(\gamma_2) - \mathbb{A}_{H_1 \oplus H_2} (x),
\end{equation}
implied by (\ref{fencsti}). When $i^{\Lambda}(\gamma_1;L_1)+ 
i^{\Lambda}(\gamma_2;L_2) = k$ and $\mu^{\Theta}(x;H_1 \oplus H_2) =
k -n$, the space $\mathscr{M}^{K}_{\alpha_0}(\gamma_1,\gamma_2;x)$ is a compact zero-dimensional oriented manifold. The usual counting process defines the homomorphism
\[
K^{\Lambda}_{\alpha_0} : \bigl( M(\mathbb{S}^{\Lambda}_{L_1}) \otimes
M(\mathbb{S}^{\Lambda}_{L_2}) \bigr)_k \rightarrow
F_{k-n}^{\Theta}(H_1 \oplus H_2), 
\]
and a standard gluing argument shows that $K^{\Lambda}_{\alpha_0}$ is a chain map.  

Now assume $\alpha_0>0$.
By standard compactness and gluing arguments, the family of solutions $\mathscr{M}^K_{\alpha}$ for $\alpha$ varying in the interval $[\alpha_0,+\infty[$ allows to define a chain homotopy between $K^{\Lambda}_{\alpha_0}$ and the composition $E \circ (\Phi_{L_1}^{\Lambda} \otimes \Phi_{L_2}^{\Lambda})$. 

Similarly, a compactness and cobordism argument on the Morse side
shows that $K^{\Lambda}_0$ is chain homotopic to the composition
$\Phi^{\Theta}_{L_1 \oplus L_2} \circ M_!$. See Figure \ref{EKLM}.

\bigskip
\bigskip

\begin{figure}[h]
 \begin{center}
  \setlength{\unitlength}{0.01cm}
  \begin{picture}(1400,300)(0,0)
   \put(0,50){\epsfig{file=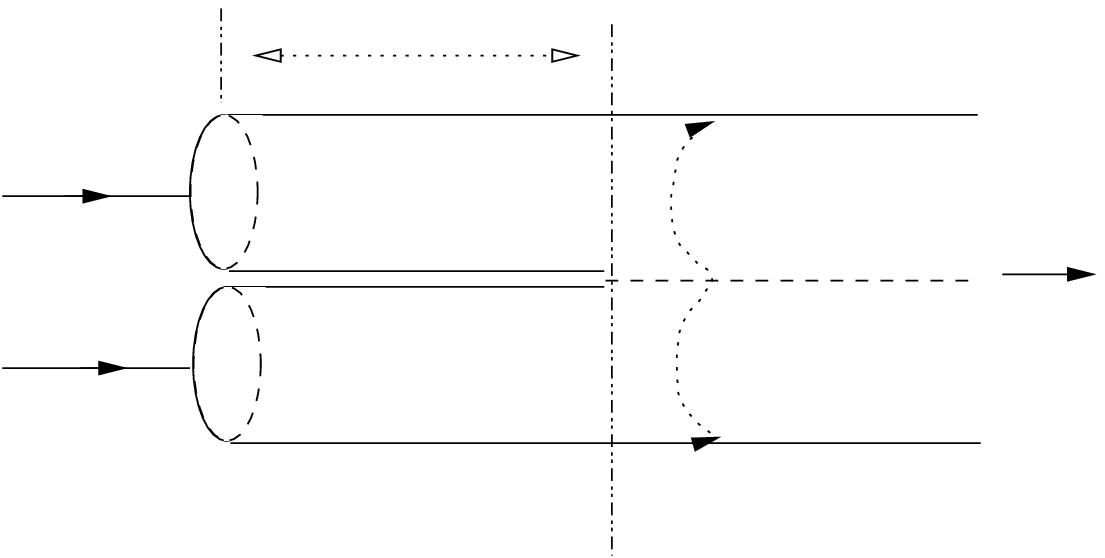,width=65mm}}
   \put(220,370){\makebox{$\alpha$}}
   \put(430,240){\makebox{$\Theta$}}
   \put(100,0){\makebox{$E \circ (\Phi_{L_1}^{\Lambda} \otimes \Phi_{L_2}^{\Lambda})\,\simeq\,K^\Lambda_\alpha$}}
   \put(600,0){\makebox{$\stackrel{!}{\simeq}$}}
   \put(680,120){\epsfig{file=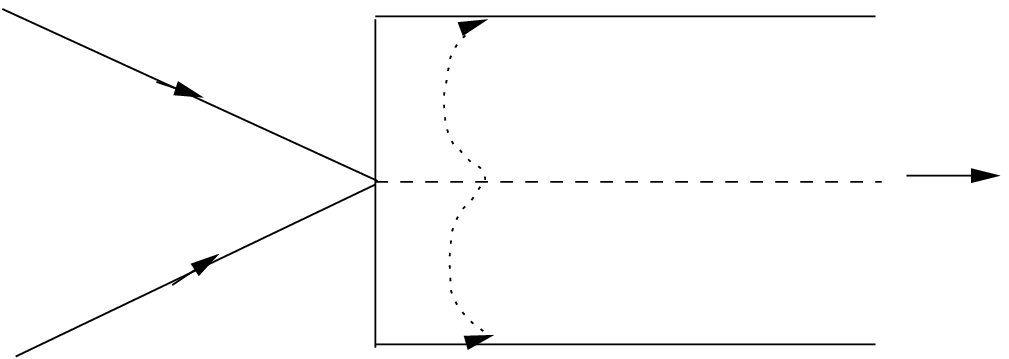,width=60mm}}
   \put(750,0){\makebox{$K^\Lambda_0\,\simeq\,\Phi^{\Theta}_{L_1 \oplus L_2} \circ M_!$}}
   \put(960,250){\makebox{$\Theta$}}
  \end{picture}
 \end{center}
\caption{The homotopy through $K^{\Lambda}_\alpha$ and $K^\Lambda_0$.
  \label{EKLM}}
\end{figure}

It remains to prove that $K^{\Lambda}_{\alpha_0}$ is homotopic to $K^{\Lambda}_0$. Constructing a homotopy between these chain maps by using the spaces of solutions $\mathscr{M}^K_{\alpha}$ for $\alpha\in [0,\alpha_0]$ presents analytical difficulties: If we are given a solution $u$ of the limiting problem $\mathscr{M}^K_0$, the existence of a (unique) one-parameter family of solutions ``converging'' to $u$ is problematic, because we do not expect $u$ to be $C^0$ close to the one-parameter family of solutions, due to the jump in the boundary conditions. 

Therefore, we use a detour, starting from the following algebraic observation. If two chain maps $\varphi,\psi: C \rightarrow C'$ are homotopic, so are their tensor products $\varphi \otimes \psi$ and $\psi \otimes \varphi$. The converse is obviously not true, as the example of $\varphi=0$ and $\psi$ non-contractible shows. However, it becomes true under suitable conditions on $\varphi$ and $\psi$.
Denote by $(\Z,0)$ the graded group which vanishes at every degree, except for degree zero, where it coincides with $\Z$. 
We see $(\Z,0)$ as a chain complex with the trivial boundary operator. Then we have the following:

\begin{lem}
\label{alge}
Let $(C,\partial)$ and $(C',\partial)$ be chain complexes, bounded from below. Let 
$\varphi,\psi: C \rightarrow C'$ be chain maps. Assume that there is an element $\epsilon\in C_0$ with $\partial \epsilon=0$ and a chain map $\delta: C' \rightarrow (\Z,0)$ such that
\[
\delta(\varphi(\epsilon)) = \delta(\psi(\epsilon)) = 1.
\]
If $\varphi\otimes \psi$ is homotopic to $\psi \otimes \varphi$, then 
$\varphi$ is homotopic to $\psi$.
\end{lem}

\begin{proof}
Let $\pi$ be the chain map
\[
\pi : C' \otimes C' \rightarrow C' \otimes (\Z,0) \cong C', \quad \pi = \mathrm{id} \otimes \delta.
\]
Let $H:C\otimes C \rightarrow C' \otimes C'$ be a chain homotopy between $\varphi\otimes \psi$ and $\psi \otimes \varphi$, that is
\[
\varphi \otimes \psi - \psi \otimes \varphi = \partial H + H \partial.
\]
If we define the homomorphism $h:C \rightarrow C'$ by 
\[
h(a) := \pi \circ H(a\otimes \epsilon), \quad \forall a\in C,
\]
we have
\begin{eqnarray*}
\partial h(a) + h\partial a = \partial \pi (H(a\otimes \epsilon)) + \pi(H\partial (a\otimes \epsilon)) = \pi ( \partial H(a\otimes \epsilon) + H\partial(a\otimes \epsilon))\\ = \pi ( \varphi(a) \otimes \psi(\epsilon) - \psi(a) \otimes \varphi(\epsilon)) = \varphi(a) \otimes \delta(\psi(\epsilon)) - \psi(a) \otimes \delta(\varphi(\epsilon)) = \varphi(a) - \psi(a).
\end{eqnarray*}
Hence $h$ is the required chain homotopy.
\end{proof}

We shall apply the above lemma to the complexes
\[
C_k = \bigl(M(\mathbb{S}^{\Lambda}_{L_1}) \otimes M(\mathbb{S}^{\Lambda}_{L_2})\bigr)_{k+n}, \quad
C_k' = F^{\Theta}_k(H_1\oplus H_2),
\]
and to the chain maps $K^{\Lambda}_{0}$ and $K^{\Lambda}_{\alpha_0}$.
The tensor products $K^{\Lambda}_0 \otimes K^{\Lambda}_{\alpha_0}$ and
$K^{\Lambda}_{\alpha_0} \otimes K^{\Lambda}_0$ are represented by the
coupling -- in two different orders -- of the corresponding elliptic
boundary value problems. 

\begin{prop}
\label{coupro}
The chain maps $K^{\Lambda}_0 \otimes K^{\Lambda}_{\alpha_0}$ and $K^{\Lambda}_{\alpha_0} \otimes K^{\Lambda}_0$ are homotopic.
\end{prop}

Constructing a homotopy between the coupled
problems is easier than dealing with the original ones: we can keep
$\alpha_0$ fixed and rotate the boundary condition on the initial part
of the half-strip. This argument is similar to an alternative way, due
to H.\ Hofer, to prove the gluing statements in standard Floer homology.  Details of the proof of Proposition \ref{coupro} are contained in Section \ref{coupros} below.

Here we just construct the cycle $\epsilon$ and the chain map
$\delta$ required in Lemma \ref{alge}. Let
\[
\tilde{\delta}: M(\mathbb{S}_{L_1 \oplus L_2}^{\Theta}) \rightarrow (\Z,0)
\]
be the standard augmentation on the Morse complex of the Lagrangian
action functional on the space of figure-8 loops, that is the
homomorphism mapping each generator $\gamma\in
\mathscr{P}^{\Theta}(L_1 \oplus L_2)$ of Morse index zero into
1. Since the unstable manifold of a critical point $\gamma$ of Morse
index 1 is one-dimensional, its boundary is of the form
$\gamma_1-\gamma_2$, where $\gamma_1$ and $\gamma_2$ are two relative
minimizers.  Hence, $\tilde{\delta}$ is a chain map. We choose the chain map
\[
\delta: F^{\Theta}(H_1 \oplus H_2) \rightarrow (\Z,0)
\]
to be the composition of $\tilde{\delta}$ with the isomorphism $(\Phi_{L_1 \oplus L_2}^{\Theta})^{-1}$ from
the Floer to the Morse complex. Since the latter isomorphism is the
identity on global minimizers, $\delta(x)=1$ if $x\in
\mathscr{P}^{\Theta}(H_1\oplus H_2)$ corresponds to a global minimizer of $\mathbb{S}_{L_1\oplus L_2}^{\Theta}$.

We now construct the cycle $\epsilon$ in
$\bigl(M(\mathbb{S}_{L_1}^{\Lambda}) \otimes M(\mathbb{S}_{L_2}^{\Lambda})\bigr)_n$. Since changing the Lagrangians
$L_1$ and $L_2$ (and the corresponding Hamiltonians) changes the chain maps appearing in the diagram of Theorem \ref{chainA} by a chain homotopy, we are free to choose the Lagrangians so to make the construction easier. 

We consider a Lagrangian of the form
\[
L_1(t,q,v) := \frac{1}{2} |v|^2 - V_1(t,q),
\]
where the potential $V_1\in C^{\infty}(\T \times M)$ satisfies
\begin{eqnarray}
\label{v1a}
V_1(t,q) < V_1(t,q_0) = 0, \quad \forall t\in \T, \; \forall q\in M \setminus \{q_0\}, \\
\label{v1b}
\hess V_1(t,q_0) < 0, \quad \forall t\in \T.
\end{eqnarray}
The corresponding Euler-Lagrange equation is
\begin{equation}
\label{eln1}
\nabla_t \gamma'(t) = - \grad V_1(t,\gamma(t)),
\end{equation}
where $\nabla_t$ denotes the covariant derivative along the curve $\gamma$.
By (\ref{v1a}) and (\ref{v1b}), the constant curve $q_0$ is a
non-degenerate minimizer for the action functional
$\mathbb{S}_{L_1}^{\Lambda}$ on the free loop space (actually, it is
the unique global minimizer), so
\[
i^{\Lambda}(q_0;L_1) = 0.
\]
Notice also that the equilibrium point $(q_0,0)$ is hyperbolic and unstable for the Hamiltonian dynamics on $T^*M$ induced by the Fenchel dual Hamiltonian $H_1$. We claim that there exists $\omega>0$ such that
\begin{equation}\begin{split}
\label{av1}
&\mbox{every solution }\gamma\mbox{ of (\ref{eln1}) such that
}\gamma(0)=\gamma(1)\mbox{, other than } \gamma(t)\equiv q_0,\\ &\mbox{satisfies } \mathbb{S}_{L_1}(\gamma)\geq \omega.
\end{split}\end{equation}
Assuming the contrary, there exists a sequence $(\gamma_h)$ of
solutions of (\ref{eln1}) with $\gamma_h(0)=\gamma_h(1)$ and
$0<\mathbb{S}_{L_1}(\gamma_h) \rightarrow 0$. The space of solutions of
(\ref{eln1}) with action bounded from above is compact -- for instance in
$C^{\infty}([0,1],M)$ -- so a subsequence of $(\gamma_h)$ converges to
a solution of (\ref{eln1}) with zero action. Since $q_0$ is the only
solution with zero action, we find non-constant solutions $\gamma$ of
(\ref{eln1}) with $\gamma(0)=\gamma(1)$ in any $C^{\infty}$-neighborhood of the constant curve $q_0$. But this is impossible: The fact that the local stable and unstable manifolds of the hyperbolic equilibrium point $(q_0,0)\in T^*M$ are transverse to the vertical foliation $\set{T_q^*M}{q\in M}$ easily implies that if $(x_h)$ is a sequence in the phase space $T^*M$ tending to $(q_0,0)$ such that the Hamiltonian orbit of 
$x_h$ at time $T_h$ is on the leaf $T_{\pi(x_h)}^* M$ containing $x_h$, then the sequence $(T_h)$ must diverge.

A generic choice of the potential $V_1$ satisfying (\ref{v1a}) and (\ref{v1b}) produces a Lagrangian $L_1$ whose associated action functional is Morse on $\Lambda^1(M)$.

Next we consider an autonomous Lagrangian of the form 
\[
\tilde{L}_2(q,v) := \frac{1}{2} |v|^2 - V_2(q),
\]
where
\begin{enumerate}
\item $V_2$ is a smooth Morse function on $M$;
\item $0= V_2(q_0) < V_2(q) <\omega/2$ for every $q\in M \setminus \{q_0\}$;
\item $V_2$ has no local minimizers other than $q_0$;
\item $\|V_2\|_{C^2(M)} < \rho$.
\end{enumerate}
Here $\rho$ is a small positive constant, whose size is to be specified. The critical points of $V_2$ are equilibrium solutions of the Euler-Lagrange equation associated to $\tilde{L}_2$. The second differential of the action at such an equilibrium solution $q$ is 
\[
d^2 \mathbb{S}_{\tilde{L}_2}^{\Lambda} (q)[\xi,\xi] = \int_0^1 \Bigl(
\langle \xi'(t), \xi'(t) \rangle - \langle \hess V_2(q) \,
\xi(t),\xi(t) \rangle \Bigr)\, dt.
\]
If $0< \rho < 2\pi$, (iv) implies that $q$ is a non-degenerate critical point of $\mathbb{S}_{\tilde{L}_2}^{\Lambda}$ with Morse index
\[
i^{\Lambda}(q;\tilde{L}_2)=n - i(q;V_2),
\]  
a maximal negative subspace being the space of constant vector fields at $q$ taking values into the positive eigenspace of $\hess V_2(q)$.

The infimum of the energy
\[
\frac{1}{2} \int_0^1 |\gamma'(t)|^2 \, dt
\]
over all non-constant closed geodesics is positive. It follows that if $\rho$ in (iv) is small enough, then
\begin{equation}
\label{disc}
\inf\set{\mathbb{S}_{\tilde{L}_2}(\gamma)}{\gamma\in \mathscr{P}^{\Lambda}(\tilde{L}_2), \; \gamma \mbox{ non-constant}} >0.
\end{equation}
Since the Lagrangian $\tilde{L}_2$ is autonomous, non-constant periodic orbits cannot be non-degenerate critical points of  $\mathbb{S}_{\tilde{L}_2}^{\Lambda}$.
Let $W\in C^{\infty}(\T \times M)$ be a $C^2$-small time-dependent potential satisfying:
\begin{enumerate}
\setcounter{enumi}{4}
\item $0 \leq W(t,q) < \omega/2$ for every $(t,q)\in \T \times M$;
\item $W(t,q)=0$, $\grad W(t,q)=0$, $\hess W(t,q)=0$ for every $t\in \T$ and every critical point $q$ of $V_2$.
\end{enumerate}
For a generic choice of such a $W$, the action functional associated to the Lagrangian
\[
L_2(t,q,v) := \frac{1}{2} \langle v ,v \rangle - V_2(q) - W(t,q),
\]
is Morse on $\Lambda^1(M)$. By (vi), the critical points of $V_2$ are still equilibrium solutions of the Euler-Lagrange equation
\begin{equation}
\label{eln2}
\nabla_t \gamma'(t) = - \grad \bigl(V_2(t,\gamma(t)) + W(t,\gamma(t)\bigr),
\end{equation}
and 
\begin{equation}
\label{llindices}
i^{\Lambda}(q;L_2) = n - i(q;V_2) \quad \forall q\in \crit V_2.
\end{equation}
Moreover, (\ref{disc}) implies that if the $C^2$ norm of $W$ is small enough, then
\begin{equation}
\label{disc2}
\inf\set{\mathbb{S}_{L_2}(\gamma)}{\gamma\in \mathscr{P}^{\Lambda}(L_2), \; \gamma \mbox{ non-constant}} >0.
\end{equation}
Up to a generic perturbation of the
potential $W$, we may also assume that the equilibrium solution $q_0$ is the only $1$-periodic solution of (\ref{eln2})
with $\gamma(0)=\gamma(1)=q_0$ (generically, the set of periodic
orbits is discrete, and so is the set of their initial points).

Since the inclusion $\mathrm{c}:M \hookrightarrow \Lambda^1(M)$ induces an injective homomorphism between the singular homology groups, the image $\mathrm{c}_*([M])$ of the fundamental class of the oriented closed manifold $M$ does not vanish in $\Lambda^1(M)$. By (ii) and (v), the action $\mathbb{S}_{L_2}$ of every constant curve in $M$ does not exceed $0$. So we can regard $\mathrm{c}_*([M])$ as a non vanishing element of the homology of the sublevel $\{\mathbb{S}_{L_2}^{\Lambda}\leq 0\}$. The singular homology of $\{\mathbb{S}_{L_2}^{\Lambda}\leq 0\}$ is isomorphic to the homology of the subcomplex  of the Morse complex $M_*(\mathbb{S}_{L_2}^{\Lambda})$ generated by the critical points of $\mathbb{S}_{L_2}^{\Lambda}$ whose action does not exceed $0$. By (\ref{disc2}), 
these critical points are the equilibrium solutions $q$, with $q\in \crit\, (V_2)$. By (ii), (iii), and (\ref{llindices}), the only critical point of index $n$ in this sublevel is $q_0$. It follows that the Morse homological counterpart of $\mathrm{c}_*([M])$ is $\pm q_0$. In particular, $q_0 \in M_n(\mathbb{S}_{L_2}^{\Lambda})$ is a cycle. Since $\mathbb{S}_{L_2}^{\Lambda}(q_0)=0$, 
\begin{equation}
\label{v2c}
\mathbb{S}_{L_2}^{\Lambda}(\gamma) \leq 0, 
\quad \forall \gamma\in W^u(q_0;X^{\Lambda}_{L_2}).
\end{equation}

We now regard the pair $(q_0,q_0)$ as an element of $\mathscr{P}^{\Theta}(L_1\oplus L_2)$. We claim that if $\rho$ is small enough,
(iv) implies that $(q_0,q_0)$ is a non-degenerate minimizer for
$\mathbb{S}_{L_1 \oplus L_2}$ on the space of figure-8 loops
$\Theta^1(M)$. The second differential of $\mathbb{S}_{L_1 \oplus L_2}^{\Theta}$ at $(q_0,q_0)$ is the quadratic form
\begin{equation*}\begin{split}
&d^2 \mathbb{S}_{L_1 \oplus L_2}^{\Theta} (q_0,q_0) [(\xi_1,\xi_2)]^2\\ =\quad& \int_0^1 \Bigl(
\langle \xi'_1, \xi'_1 \rangle - \langle \hess V_1(t,q_0) \,
\xi_1,\xi_1 \rangle + \langle \xi'_2, \xi'_2 \rangle - \langle \hess V_2(q_0) \, \xi_2,\xi_2 \rangle \Bigr)\, dt,
\end{split}\end{equation*}
on the space of curves $(\xi_1,\xi_2)$ in the Sobolev space $W^{1,2}([0,1],T_{q_0}M \times T_{q_0} M)$ which satisfy the boundary conditions
\[
\xi_1(0) = \xi_1(1) = \xi_2(0) = \xi_2(1).
\] 
By (\ref{v1b}), we can find $\alpha>0$ such that that
\[
\hess V_1(t,q_0) \leq - \alpha I.
\]
By comparison, it is enough to show that the quadratic form
\[
Q_{\rho}(u_1,u_2) := \int_0^1 ( u_1'(t)^2 + \alpha u_1(t)^2 + u_2'(t)^2
- \rho u_2(t)^2 )\, dt
\]
is coercive on the space 
\[
\set{ (u_1,u_2) \in W^{1,2}([0,1],\R^2) }{u_1(0) = u_1(1) = u_2(0) =
  u_2(1)}.
\]
When $\rho=0$, the quadratic form $Q_0$ is non-negative. An
isotropic element $(u_1,u_2)$ for $Q_0$ would solve the boundary value
problem
\begin{eqnarray}
\label{ubvp1}
- u_1'' (t) + \alpha u_1 (t) & = & 0, \\
\label{ubvp2}
- u_2''(t) & = & 0, \\
\label{ubvp3}
u_1(0) = u_1(1) & = &u_2(0) = u_2(1), \\
\label{ubvp4}
u_1'(1) - u_1'(0) & = &u_2'(0) - u_2'(1).
\end{eqnarray}
By (\ref{ubvp2}) and (\ref{ubvp3}), $u_2$ is constant, so by (\ref{ubvp3}) and
(\ref{ubvp4}) $u_1$ is a periodic solution of (\ref{ubvp1}). Since
$\alpha$ is positive, $u_1$ is zero and by (\ref{ubvp3}) so is
$u_2$. Since the bounded
self-adjoint operator associated to $Q_0$ is Fredholm, we deduce that
$Q_0$ is coercive. By continuity, $Q_{\rho}$ remains coercive for
$\rho$ small. This proves our claim.

Let $H_1$ and $H_2$ be the Hamiltonians which are Fenchel dual to
$L_1$ and $L_2$. In order to simplify the notation, let us denote by
$(q_0,q_0)$ also the constant curve in $T^* M^2$ identically equal to
$((q_0,0),(q_0,0))$. Then $(q_0,q_0)$ is a non-degenerate element of
$\mathscr{P}^{\Theta}(H_1\oplus H_2)$, and it has Maslov index
\[
\mu^{\Theta}(q_0,q_0) = i^{\Theta}(q_0,q_0)=0.
\]

Let $x$ be an element in
$\mathscr{P}^{\Theta}(H_1\oplus H_2)$, and let $\gamma$ be its
projection onto $M\times M$. By the definition of the Euler-Lagrange
problem for figure-8 loops (see in particular the boundary conditions (\ref{thetabdry})), 
$\gamma_1$ is a solution of 
(\ref{eln1}), $\gamma_2$ is a solution of (\ref{eln2}), and
\begin{equation}
\label{bdry8}
\gamma_1(0) = \gamma_1(1) = \gamma_2(0) = \gamma_2(1), \quad
\gamma_2'(1) - \gamma_2'(0) = \gamma_1'(0) - \gamma_1'(1).
\end{equation} 
If $\gamma_1$ is the constant orbit $q_0$, then (\ref{bdry8}) implies
that $\gamma_2$ is a $1$-periodic solution of (\ref{eln2}) such that
$\gamma_2(0)= \gamma_2(1) = q_0$, and we have assumed that the only curve
with these properties is $\gamma_2\equiv q_0$. 
If $\gamma_1$ is not the constant orbit $q_0$, (\ref{av1}) implies that
\[
\mathbb{S}_{L_1}(\gamma_1) \geq \omega.
\]
By (ii) and (v), the infimum of $\mathbb{S}_{L_2}$ is larger than $-\omega$, so we deduce that 
\begin{equation}\begin{split}
\label{elptheta}
&\mathbb{A}_{H_1 \oplus H_2} (x) = \mathbb{A}_{H_1}(x_1) +
\mathbb{A}_{H_2}(x_2) = \mathbb{S}_{L_1}(\gamma_1) + \mathbb{S}_{L_2}(\gamma_2)
>0,\\ &\forall x\in \mathscr{P}^{\Theta}(H_1 \oplus H_2) \setminus \{(q_0,q_0)\},
\end{split}\end{equation}  
so $(q_0,q_0)$ is the global minimizer of $\mathbb{S}_{L_1 \oplus
  L_2}^{\Theta}$. 

Now we choose $\epsilon$ in the $n$-th degree component of the chain complex $M(\mathbb{S}_{L_1}^{\Lambda}) \otimes M(\mathbb{S}_{L_2}^{\Lambda})$ to be the cycle
\[
\epsilon = q_0 \otimes q_0 \in M_0(\mathbb{S}_{L_1}^{\Lambda}) \otimes M_n(\mathbb{S}_{L_2}^{\Lambda}).
\]
We must show that
\begin{equation}
\label{dadim}
\delta(K^{\Lambda}_0(q_0\otimes q_0)) = \delta(K^{\Lambda}_{\alpha_0}(q_0\otimes q_0)) = 1.
\end{equation}
Let $x \in \mathscr{P}^{\Theta}(H_1 \oplus H_2)$, and let $u$ be an element of either 
\[
\mathscr{M}^K_0(q_0,q_0;x) \quad \mbox{or} \quad
\mathscr{M}^K_{\alpha_0}(q_0,q_0;x).
\]
By the boundary condition (\ref{bkl1}), the curve $u(0,\cdot)$ projects onto a closed curve in $M\times M$ whose first component is the constant $q_0$ and whose second component is in the unstable manifold of $q_0$ with respect to the negative pseudo-gradient flow of $\mathbb{S}_{L_2}^{\Lambda}$. By the fundamental inequality (\ref{fencsti}) between the Hamiltonian and the Lagrangian action and by (\ref{v2c}), we have 
\[
\mathbb{A}_{H_1\oplus H_2}(x) \leq \mathbb{A}_{H_1\oplus H_2} (u(0,\cdot)) \leq \mathbb{S}_{L_1 \oplus L_2} (\pi\circ u(0,\cdot)) = \mathbb{S}_{L_1}^{\Lambda}(q_0) + \mathbb{S}_{L_2}^{\Lambda}(\pi\circ u_2(0,\cdot)) \leq 0.
\]
By (\ref{elptheta}), $x$ must be the constant curve $(q_0,q_0)$, and
all the inequalities in the above estimate are equalities. It follows
that $u$ is constant, $u(s,t) \equiv (q_0,q_0)$. 

Therefore, the spaces $\mathscr{M}^K_0(q_0,q_0;x)$ and
$\mathscr{M}^K_{\alpha_0}(q_0,q_0;x)$ are non-empty if and only if
$x=(q_0,q_0)$, and in the latter situation they consist of the unique
constant solution $u\equiv (q_0,q_0)$. Automatic transversality holds
for such solutions (see \cite[Proposition 3.7]{as06}), so such a
picture survives to the generic perturbations of $L_1$, $L_2$, and of the pseudo-gradients which are necessary to achieve a Morse-Smale situation. 
Taking also the orientations into account, it follows that
\[
K^{\Lambda}_0 (q_0 \otimes q_0) = (q_0,q_0), \quad  K^{\Lambda}_{\alpha} 
(q_0 \otimes q_0) = (q_0,q_0).
\]
Since $(q_0,q_0)$ is the global minimizer of $\mathbb{S}_{L_1 \oplus
  L_2}^{\Theta}$, $\delta((q_0,q_0)) = 1$ as previously observed, 
so (\ref{dadim}) holds. 

This concludes the construction of a cycle $\epsilon$ and a chain map $\delta$ which satisfy the assumptions of Lemma \ref{alge}. Together with Proposition \ref{coupro}, this
proves that left-hand square in the diagram of Theorem \ref{chainA} commutes up to a chain homotopy.
  
\subsection{The right-hand square is homotopy commutative}
\label{chrhs}

In this section we prove that the chain maps $\Phi_{L_1
  \# L_2}^{\Lambda} \circ M_{\Gamma}$ and $G \circ \Phi_{L_1 \oplus L_2}^{\Theta}$
are both homotopic to a third chain map, 
named $K^{\Theta}$. This fact implies that the right-hand square in
the diagram of Theorem \ref{chainA} commutes up to chain homotopy.

The chain map $K^{\Theta}$ is defined by using the following 
spaces of solutions of the Floer equation on the half-cylinder 
for the Hamiltonian $H_1 \# H_2$: 
given $\gamma \in \mathscr{P}^{\Theta}(L_1 \oplus
L_2)$ and $x\in \mathscr{P}^{\Lambda}(H_1 \# H_2)$, set
\begin{eqnarray*}
\mathscr{M}_{K}^{\Theta}(\gamma;x) := \Bigl\{u :
[0,+\infty[ \times \T \rightarrow T^*M\, \Big| \, \delbar_{J,H_1 \# H_2} (u) = 0, \\ 
\pi\circ u(0,\cdot) \in \Gamma\bigl(W^u(\gamma; X_{L_1 \oplus L_2}^{\Theta}) \bigr), \;
\lim_{s\rightarrow + \infty} u(s,\cdot) =
x \Bigr\},
\end{eqnarray*}
where $X_{L_1 \oplus L_2}^{\Theta}$ is a pseudo-gradient for $\mathbb{S}_{L_1 \oplus L_2}^{\Theta}$ on $\Theta^1(M)$.
By Theorem 3.2 in \cite{as06} (or by the arguments of Section
\ref{lin}), the space $\mathscr{M}_{K}^{\Theta}(\gamma;x)$ is a
smooth manifold of dimension
\[
\dim \mathscr{M}_{K}^{\Theta}(\gamma;x) =
i^{\Theta}(\gamma) - \mu^{\Lambda}(x),
\]
for a generic choice of $L_1$, $L_2$, and $X_{L_1 \oplus L_2}^{\Theta}$. 
These manifolds carry coherent orientations.

Compactness follows from the energy estimate
\[
\int_{[0,+\infty[ \times \T} |\partial_s u(s,t)|^2 \, ds\,dt 
\leq \mathbb{S}_{L_1 \oplus L_2}(\gamma) - \mathbb{A}_{H_1 \# H_2} (x),
\]
which is implied by (\ref{fencsti}). By counting the elements of the
zero-dimensional spaces, we define a chain map
\[
K^{\Theta} : M_j(\mathbb{S}^{\Theta}_{L_1 \oplus L_2}) 
\rightarrow F_j^{\Lambda}(H_1 \# H_2).
\]
It is easy to construct a chain homotopy $P_K^{\Gamma}$ 
between $\Phi_{L_1 \# L_2}^{\Lambda} \circ M_{\Gamma}$ and
$K^{\Theta}$ by considering the space 
\begin{eqnarray*}
\mathscr{M}_K^{\Gamma} (\gamma;x) :=
\Bigl\{ (\alpha,u) \, \Big| \, \alpha>0, \;
u:[0,+\infty[ \times \T \rightarrow T^*M, \; 
\delbar_{J,H_1 \# H_2} (u) = 0,\\ 
\phi_{-\alpha}^{\Lambda}(\pi\circ u(0,\cdot)) \in 
\Gamma\bigl(W^u(\gamma; X_{L_1 \oplus L_2}^{\Theta}) \bigr), \;
\lim_{s\rightarrow + \infty} u(s,\cdot) =
x \Bigr\}.
\end{eqnarray*}
where $\phi^{\Lambda}_s$ denotes the flow of $X_{L_1
  \# L_2}^{\Lambda}$ on $\Lambda^1(M)$. As before, we find that generically
$\mathscr{M}_K^{\Gamma}(\gamma_1,\gamma_2;x)$ is a manifold of
dimension
\[
\dim \mathscr{M}_K^{\Gamma} (\gamma;x) =
i^{\Lambda}(\gamma) - \mu^{\Theta}(x) + 1.
\]
Compactness holds, so an algebraic count of the zero-dimensional
spaces produces the homomorphism 
\[
P_K^{\Gamma} : M_j(\mathbb{S}_{L_1\oplus L_2}^{\Theta}) 
\rightarrow F_{j+1}^{\Lambda}(H_1\# H_2).
\] 
A standard gluing argument shows that $P_K^{\Gamma}$ is the required homotopy. 

Finally, the construction of the chain homotopy $P^K_G$ between 
$K^{\Theta}$ and $G\circ \Phi_{L_1 \oplus L_2}^{\Theta}$ is based on
the one-parameter family of
Riemann surfaces $\Sigma_G^K(\alpha)$, $\alpha>0$, defined
as the quotient of the disjoint union $[0,+\infty[ \times [-1,0]
\sqcup [0,+\infty[ \times [0,1]$ under the identifications
\[
(s,0^-) \sim (s,0^+) \mbox{ and } (s,-1) \sim (s,1) \quad \mbox{for }
s\geq \alpha.
\]
This object is a Riemann surface with boundary, the holomorphic
structure at $(\alpha,0)$ being given by the map
\[
\set{\zeta \in \C}{\re \zeta\geq 0, \; |\zeta|<\epsilon} \rightarrow
\Sigma_G^K(\alpha), \quad \zeta \mapsto \alpha+ \zeta^2,
\]
and the holomorphic structure at $(\alpha,-1)\sim (\alpha,1)$ being
given by the map
\[
\set{\zeta \in \C}{\re \zeta\geq 0, \; |\zeta|<\epsilon} \rightarrow
\Sigma_G^K(\alpha), \quad \zeta \mapsto \left\{
  \begin{array}{ll} \alpha - i + \zeta^2 & \mbox{if } \im \zeta\geq 0,
    \\ \alpha + i + \zeta^2 & \mbox{if } \im \zeta\leq 0. \end{array}
\right.
\]
Here $\epsilon$ is a positive number smaller than $1$ and $\sqrt{\alpha}$. 

Given $\gamma\in \mathscr{P}^{\Theta}(L_1 \oplus L_2)$ and $x\in
\mathscr{P}^{\Lambda}(H_1 \# H_2)$, we consider the space $\mathscr{M}_G^K (\gamma,x)$ of pairs $(\alpha,u)$ where $\alpha$ is a positive number and $u:
\Sigma_G^K(\alpha) \rightarrow T^*M$ solves the equation
\[
\delbar_{J,H_1 \# H_2} (u) = 0,
\]
satisfies the boundary conditions
\begin{eqnarray*}
\left\{ \begin{array}{l} \pi\circ u (s,-1) = \pi \circ u(s,0^-) = \pi
\circ u(s,0^+) = \pi \circ u(s,1),  \\
u(s,0^-) - u(s,-1) + u(s,1) - u(s,0^+) = 0, \end{array} \right. \quad 
\forall s\in [0,\alpha], \\ 
(\pi\circ u(0,\cdot-1), \pi\circ u(0,\cdot)) \in W^u(\gamma, X_{L_1 \oplus L_2}^{\Theta}), 
\end{eqnarray*}
and the asymptotic condition
\[
\lim_{s\rightarrow +\infty} u(s,2t-1) = x(t).
\] 
The following result is proved in Section \ref{lin}.

\begin{prop}
\label{PRLambda}
For a generic choice of $L_1$, $L_2$, and $X_{L_1 \oplus L_2}^{\Theta}$, $\mathscr{M}_G^K (\gamma,x)$ -- if non-empty -- is a
smooth manifold of dimension
\[
\dim \mathscr{M}_G^K (\gamma,x) =
i^{\Theta}(\gamma;L_1 \oplus L_2) - \mu^{\Lambda}(x; H_1 \# H_2) + 1.
\]
The projection $(\alpha,u) \mapsto \alpha$ is smooth on
$\mathscr{M}_G^K (\gamma,x)$. These manifolds carry
coherent orientations.
\end{prop}

The elements $(\alpha,u)$ of $\mathscr{M}_G^K(\gamma,x)$ satisfy the energy estimate
\[
\int_{\Sigma^K_G(\alpha)} 
|\partial_s u(s,t)|^2 \, ds\, dt \leq \mathbb{S}_{L_1
  \oplus L_2} (\gamma) - \mathbb{A}_{H_1\# H_2}(x).
\]
This provides us with the compactness which is necessary to define the homomorphism 
\[
P^K_G : M_j(\mathbb{S}_{L_1 \oplus L_2}^{\Theta}) \longrightarrow F_{j+1}^{\Lambda} (H_1 \# H_2),
\]
by the usual counting procedure applied to the spaces $\mathscr{M}_G^K$.
A standard gluing argument shows that $P^K_G$ is a chain homotopy between 
$K^{\Theta}$ and $G\circ \Phi_{L_1 \oplus L_2}^{\Theta}$.

This concludes the proof of Theorem \ref{chainA}, hence of its corollary, Theorem A of the Introduction.

\subsection{Comparison between $\mathbf{C}$, $\mathbf{EV}$, $\mathbf{I_!}$ and $\mathbf{c}$,  $\mathbf{ev}$, $\mathbf{i_!}$} 
\label{bah}
 
The aim of this section is to prove that the homomorphisms
\begin{eqnarray*}
\mathrm{c}_* : H_j(M) & \rightarrow & H_j(\Lambda(M)), \\
\mathrm{ev}_* : H_j(\Lambda(M)) & \rightarrow & H_j(M), \\ \mathrm{i}_!: H_j(\Lambda(M)) & \rightarrow & H_{j-n}(\Omega(M,q_0)), 
\end{eqnarray*}
on the topological side (see Section \ref{rbttp}), correspond - via the isomorphisms of Section \ref{cci} - to the homomorphisms 
\begin{eqnarray*}
\mathrm{C}_* : HM_j(f) & \rightarrow & HF_j^{\Lambda}(T^* M), \\
\mathrm{Ev}_* : HF_j^{\Lambda}(T^* M)) & \rightarrow & HM_j(f), \\
\mathrm{I}_!: HF_j^{\Lambda}(T^*M) & \rightarrow & HF_{j-n}^{\Omega}(T^*M), 
\end{eqnarray*}
on the Floer side (see Section \ref{hcevi}). We start by comparing the first two pairs of homomorphisms.

\paragraph{Comparison between $\mathbf{c}, \mathbf{ev}$ and $\mathbf{C}, \mathbf{Ev}$.} In Section \ref{laf} we have shown that the homomorphisms $\mathrm{c}_*$ and $\mathrm{ev}_*$ are induced by chain maps
\[
M{\mathrm{c}} : M_j(f) \rightarrow M_j(\mathbb{S}_L^{\Lambda}), \quad
M{\mathrm{ev}} : M_j(\mathbb{S}_L^{\Lambda}) \rightarrow M_j(f),
\]
between the Morse complex of the Morse functions $f: M \rightarrow \R$ and $\mathbb{S}_L^{\Lambda}: \Lambda^1(M) \rightarrow \R$. Therefore, the fact that $\mathrm{c}_*$ and $\mathrm{ev}_*$ correspond to $\mathrm{C}_*$ and $\mathrm{Ev}_*$ is implied by the following chain level result:

\begin{thm}
\label{c=C,ev=Ev}
The triangles
\[
\xymatrix{M_j(f) \ar[r]^{M\mathrm{c}} \ar[rd]_{C} & M_j(\mathbb{S}_L^{\Lambda}) \ar[d]^{\Phi_L^{\Lambda}} \ar[r]^{M{\mathrm{ev}}} &
    M_j(f) \\ & F^{\Lambda}_j (H,J) \ar[ur]_{\mathrm{Ev}} & }
\]
are chain homotopy commutative.
\end{thm}

A homotopy $P^{\mathrm{C}}$ between $\mathrm{C}$ and $\Phi_L^{\Lambda} \circ M\mathrm{c}$ is
defined by the following spaces: Given $x\in \crit(f)$ and $y\in
\mathscr{P}^{\Lambda}(H)$, set
\begin{eqnarray*}
\mathscr{M}_P^\mathrm{C}(x,y) := \Bigl\{ (\alpha,u) \, \Big| \, \alpha>0, \;
u:[0,+\infty[ \times \T\rightarrow T^*M, \; \delbar_{J,H}(u) = 0, \\ \phi^{\Lambda}_{-\alpha}(\pi\circ u(0,\cdot))
\equiv q \in W^u(x) \Bigr\},
\end{eqnarray*}
where $\phi^{\Lambda}$ is the flow of $X_L^{\Lambda}$, a pseudo-gradient for $\mathbb{S}_L^{\Lambda}$ on $\Lambda^1(M)$, and $W^u(x)\subset M$ is the unstable manifold of $x$ with respect to the negative gradient flow of $f$.

Similarly, the definition of the homotopy $P^{\mathrm{Ev}}$ 
between $\mathrm{Ev} \circ \Phi_L^{\Lambda}$ and $M\mathrm{ev}$ is
obtained from the composition of three 
homotopies  based on the following spaces:
Given $\gamma\in \mathscr{P}^{\Lambda}(L)$ and $x\in \crit(f)$, set
\[
\begin{split}
  \mathscr{M}_{P_1}^{\mathrm{Ev}}(\gamma,x) := \Bigl\{ (\alpha,u) \, \Big|
  \alpha \in [1,+\infty[, \; u: [0,\alpha] \times \T \rightarrow T^*M
  \; \text{solves }\; \delbar_{J,H}(u) = 0,\\
  u(\alpha,t) \in \mathbb{O}_M \; \forall t\in \T, \; u(\alpha,0) \in W^s(x),\; \pi\circ u(0,\cdot) \in
  W^u(\gamma; X_L^{\Lambda}) \Bigr\};
 \end{split} \]

 \[
 \begin{split}
  \mathscr{M}_{P_2}^{\mathrm{Ev}}(\gamma,x) := \Bigl\{ (\alpha,u) \, \Big|
  \alpha \in [0,1], \; u: [0,1] \times \T \rightarrow T^*M \;
  \text{solves }\; \delbar_{J,H}(u) = 0,\\ u(1,t) \in
  \mathbb{O}_M \; \forall t\in \T, \;
  u(\alpha,0) \in W^s(x), \;
  \pi\circ u(0,\cdot) \in
  W^u(\gamma; X_L^{\Lambda})
  \Bigr\};
 \end{split}
 \]
 and
 \[
 \begin{split}
  \mathscr{M}_{P_3}^{\mathrm{Ev}}(\gamma,x) := \Bigl\{ (\alpha,u) \, \Big|
  \alpha \in ]0,1], \; u: [0,\alpha] \times \T \rightarrow T^*M \;
  \text{solves }\; \delbar_{J,H}(u) = 0,\\ u(\alpha,t) \in
  \mathbb{O}_M \; \forall t\in \T, \; u(0,0) \in W^s(x),\;
  \pi\circ u(0,\cdot) \in
  W^u(\gamma; X_L^{\Lambda}) \Bigr\} .
\end{split}
\]
Moreover, recalling that the definition of $M\mathrm{ev}$ is based on the space
\[
   \mathscr{M}_{M\mathrm{ev}}(\gamma,x)\,=\,
   W^u(\gamma;X_L^\Lambda)\cap\mathrm{ev}^{-1}\big(
   W^s(x;-\grad f)\big),
\]
we make the following observation:
\begin{prop} For every $\gamma\in \mathscr{P}^{\Lambda}(L)$ and $x\in \crit\, (f)$ with $i^{\Lambda}(\gamma) = i(x;f)$, there exists $\alpha_o>0$
  such that for each $c$ in the finite set $\mathscr{M}_{M\mathrm{ev}}(\gamma,x)$ and
  $\alpha\in (0,\alpha_o]$ the problem
\begin{equation}\label{eq ev}\begin{split}
    &u\in [0,\alpha]\times \T\to T^\ast M,\quad \delbar_{J,H} u=0,\\
    &u(\alpha,t)\in\mathbb{O}_M\;\forall t\in\T,\quad \pi\circ
    u(0,\cdot)=c,
  \end{split}
\end{equation}
has a unique solution with the same coherent orientation as $c$.
\end{prop}

\begin{proof}
We give a sketch of the proof, details are left to the reader.

First, given a sequence $\alpha_n\to 0$ and associated solutions
$u_n\colon [0,\alpha_n]\times \T\to T^\ast M$ of (\ref{eq ev}), one
can show that $u_n\to (c,0)\in\Lambda (T^\ast M)$ uniformly. Here, it is
important to make a case distinction for the three cases of
possible gradient blow-up: if we set $R_n:=\|\nabla u_n\|_{\infty} = |\nabla u_n(z_n)|$, up to a subsequence we may assume that the sequence
$(\alpha_n R_n)$ either is infinitesimal, or diverges, or converges to some $\kappa>0$.

The most interesting case is $\alpha_nR_n\to \kappa>0$ which is dealt
with by rescaling $v_n=u_n(\alpha_n\cdot,\alpha_n\cdot)$ as in
the proof of Lemma \ref{fp-cpt}.

For the converse, we need a Newton type method which is hard to
implement for the shrinking domains $[0,\alpha]\times \T$ with
$\alpha\to 0$. Instead, we consider the conformally rescaled equivalent
problem. Let $v(s,t)=u(\alpha s,\alpha t)$ and consider the
corresponding problem
for $\alpha\to 0$,
\begin{equation}\label{eq rescale}\begin{split}
    &v\colon [0,1]\times \T_{\alpha^{-1}}\to T^\ast M,\quad
    \delbar_{J,H_\alpha}v=0,\\
    &\pi(v(0,t))=c(\alpha t),\;v(1,t)\in\mathbb{O}_M\;\forall
    t\in\T_{\alpha^{-1}},
  \end{split}
\end{equation}
where $H_\alpha(t,\cdot)=\alpha H(\alpha t,\cdot)$ and
$\T_{\alpha^{-1}}=\R/\alpha^{-1}\Z$. The proof is now based on the
Newton method which requires to show that:
\begin{itemize}
\item[(a)] for $v_o(s,t)=0\in T^\ast_{c(\alpha t)} M$ we have
  $\delbar_{J,H}(v_o) \to 0$ as $\alpha\to 0$, which is obvious, and
\item[(b)] the linearization $D_\alpha$ of $\delbar_{J,H_\alpha}$ at
  $v_o$ is invertible for small $\alpha>0$ with uniform bound on
  $\|D_\alpha^{-1}\|$ as $\alpha\to 0$.
\end{itemize}
We sketch now the proof of this uniform bound.

After suitable trivializations, the linearization $D_\alpha$ of
$\delbar_{J,H_\alpha}$ at $v_o$ with the above Lagrangian boundary
conditions can be viewed as an operator $D_\alpha$ on
\[
   W^{1,p}_{i\R^n,\R^n}(\alpha)\,:=\,
   \big\{\,v\colon ]0,1[ \times \R/ \alpha^{-1} \Z \to\C^n\,|\,
   v(0,\cdot)\in i\R^n,\,v(1,\cdot)\in\R^n\,\big\},
\]
with norm $\|\cdot\|_{1,p;\alpha}$. Assuming that $D^{-1}_\alpha$ is
not uniformly bounded as $\alpha\to 0$ means that we would have
$\alpha_n\to 0$ and $v_n\in W^{1,p}_{i\R^n,\R^n}(\alpha_n)$ with
$\|v_n\|_{1,p;\alpha_n}=1$ such that
  $\|D_{\alpha_n}v_n\|_{0,p;\alpha_n}\to 0$. The limit operator to
  compare to is the standard $\delbar $-operator on maps
\[
   v\colon [0,1]\times\R\to\C^n,\,\text{s.t. }\;
   v(0,t)\in i\R^n,\,v(1,t)\in\R^n\;\forall t\in\R\,.
\]
This comparison operator is clearly an isomorphism, so one
easily shows that $D_{\alpha_n}$ has to be invertible for $\alpha_n$
small.

Let $\beta\in C^\infty(\R,[0,1])$ be a cut-off function such that
\[
   \beta(t)= \begin{cases} 1,\,& t\leq 0,\\ 0,\,& t\geq 1,
   \end{cases}\qquad
    \beta'\leq 0,
\]
and set $\beta_n(t)=\beta(\alpha_n t-1)\cdot \beta(-\alpha_n t)$, hence
\[
    \beta_n{}_{\textstyle| [0,\alpha_n^{-1}]}\equiv 1\quad\text{and}\quad
    \supp\beta_n\subset [-\alpha_n^{-1},2\alpha_n^{-1}]\,.
\]
We have
\[
   \|v_n\|_{1,p;\alpha_n}\,\leq\,
   \|\beta_n v_n\|_{1,p;\R}\,\leq\,
   c_1\|v_n\|_{1,p;\alpha_n}\,.
\]
Since the linearization  $D_\alpha$ is of the form
\[
  D_\alpha v\,=\,\partial_s v + i\partial_t v+ \alpha A(s,t) v
\]
with some matrix $A(s,t)$, we observe that
\[
   \|\delbar (\beta_n v_n)-D_{\alpha_n}(\beta_n v_n)\|_{0,p;\R}\,=\,
   \alpha_n\|A\beta_n v_n\|_{0,p;\R}\to 0\,.
\]
Moreover,
\begin{equation*}\begin{split}
   \|D_{\alpha_n}(\beta_n v_n)\|_{0,p;\R}\, &\leq
   \|i\beta'_n v_n\|_{0,p;\R}\,+\,\|\beta_n
   D_{\alpha_n}v_n\|_{0,p;\R}\\
   &\leq c_2\alpha_n\|v_n\|_{0,p;[-\alpha_n^{-1},0]\cup
     [\alpha_n^{-1},2\alpha_n^{-1}]}
   +\,3\|D_{\alpha_n}v_n\|_{0,p;\alpha_n}\\
   &\leq c_2\alpha_n
   2\|v_n\|_{1,p;\alpha_n}\,+\,3\|D_{\alpha_n}v_n\|_{0,p;\alpha_n}\to\,0\,.
 \end{split}Ä
\end{equation*}
Hence, we find a subsequence such that $\beta_{n_k}v_{n_k}\to
v_o\in\ker\delbar =\{0\}$ which means that $\|v_{n_k}\|_{1,p;\alpha_{n_k}}\to
0$ in contradiction to $\|v_n\|=1$.

Similarly, we see that the coherent orientation for the determinant of
$D_{\alpha_n}$ equals that of $\delbar $ which is canonically $1$. This
completes the proof of the proposition.
\end{proof}

From the cobordisms $\mathscr{M}_{P_i}^{\mathrm{Ev}}(\gamma,x)$,
$i=1,2,3$, we now obtain the chain homotopy between $M\mathrm{ev}$ and $\mathrm{Ev} \circ \Phi-L^{\Lambda}$. This concludes the proof of Theorem \ref{c=C,ev=Ev}.

\paragraph{Comparison between $\mathbf{i_!}$ and $\mathbf{I_!}$.}
In Section \ref{laf}, we have shown that the homomorphism $\mathrm{i}_!$ is induced by the chain map
\[
M\mathrm{i}_! : M_j(\mathbb{S}^{\Lambda}_L) \rightarrow M_{j-n}(\mathbb{S}^{\Omega}_L),
\]
between the Morse complexes of the Lagrangian action functional on the spaces $\Lambda^1(M)$ and $\Omega^1(M,q_0)$. Therefore, the fact that $\mathrm{i}_!$
corresponds to $I_!$ is implied by the following chain level result:

\begin{thm}
\label{i=I}
The diagram
\[
\xymatrix{M_j(\mathbb{S}_L^{\Lambda} ) \ar[r]^{\Phi_L^{\Lambda}}
  \ar[d]_{M\mathrm{i}_!} & F_j^{\Lambda}(H) \ar[d]^{I_!} \\
  M_{j-n}(\mathbb{S}_L^{\Omega}) \ar[r]^{\Phi^{\Omega}_L} &
  F^{\Omega}_{j-n}(H) }
\]
is chain homotopy commutative. 
\end{thm}

Indeed, one can show that both $I_! \circ
\Phi_L^{\Lambda}$ and $\Phi_L^{\Omega} \circ M\mathrm{i}_!$ are homotopic to
the same chain map $K^!$. The definition of $K^!$ makes use of the
following spaces: Given $\gamma\in \mathscr{P}^{\Lambda}(L)$ and $x\in
\mathscr{P}^{\Omega}(H)$, set 
\begin{eqnarray*}
\mathscr{M}_K^! (\gamma,x) := \Bigl\{ u: [0,+\infty[
\times [0,1] \rightarrow T^*M \, \Big| \, \delbar_{J,H}(u) = 0,\\
\pi\circ u(s,0) = \pi\circ u(s,1) = q_0 \; \forall s\geq 0, \;
 \pi\circ u(0,\cdot) \in W^u(\gamma; X_L^{\Lambda}),\;
 \lim_{s\rightarrow +\infty} u(s,\cdot) = x 
\Bigr\}.
\end{eqnarray*}
Again, details are left to the reader.

We conclude that all the homomorphisms which appear in diagram (\ref{diaga}) have their Floer homological counterpart: The vertical arrows have been treated in this section, whereas the horizontal ones are described in Sections \ref{inter}, \ref{oria}, and \ref{ori}.

\section{Linear theory}
\label{lineartheory}

The remaining part of this paper consists of technical tools. In this section we develop the linear theory which allows us to study the Floer problems on the various Riemann surfaces introduced in the previous sections. All these elliptic problems are treated in the unifying setting of Cauchy-Riemann operators on strips with jumping nonlocal conormal boundary conditions.  

\subsection{The Maslov index}
\label{tmi}

Let $\eta_0$ be the {\em Liouville one-form} on $T^* \R^n = \R^n \times
(\R^n)^*$, that is the tautological one-form $\eta_0 = p\, dq$, that is
\[
\eta_0(q,p)[(u,v)] := p[u], \quad \mbox{for } q,u\in \R^n, \; p,v \in
(\R^n)^*.
\]
Its differential $\omega_0 = d \eta_0 = dp \wedge dq$,
\[
\omega_0 [(q_1,p_1),(q_2,p_1)] = p_1[q_2] - p_2 [q_1], \quad \mbox{for }
q_1,q_2 \in \R^n, \; p_1,p_2 \in (\R^n)^*,
\]
is the {\em standard symplectic form} on $T^* \R^n$. 

The {\em symplectic group}, that is the group of linear automorphisms of $T^* \R^n$ preserving $\omega_0$, is denoted by $\mathrm{Sp}(2n)$.
Let $\mathscr{L}(n)$ be the Grassmannian of Lagrangian subspaces of $T^* \R^n$, that is the set of $n$-dimensional linear subspaces of $T^* \R^n$ on which $\omega_0$ vanishes. The {\em relative Maslov index} assigns to
every pair of Lagrangian paths $\lambda_1, \lambda_2: [a,b]
\rightarrow \mathscr{L}(n)$ a half integer
$\mu(\lambda_1,\lambda_2)$. We refer to \cite{rs93} for the definition and for the properties of the relative Maslov index.

Another useful invariant is the H\"ormander index of four Lagrangian subspaces (see \cite{hor71}, \cite{dui76}, or \cite{rs93}):

\begin{defn}
\label{horind}
Let $\lambda_0,\lambda_1,\nu_0,\nu_1$ be four Lagrangian subspaces of $T^* \R^n$. Their {\em H\"ormander index} is the half
integer
\[
h(\lambda_0,\lambda_1;\nu_0,\nu_1) := \mu(\nu,\lambda_1) - \mu(\nu,\lambda_0),
\]
where $\nu:[0,1] \rightarrow \mathscr{L}(n)$ is a
Lagrangian path such that $\nu(0)=\nu_0$ and $\nu(1)=\nu_1$.
\end{defn}

Indeed, the quantity defined above does not depend on the choice of
the Lagrangian path $\nu$ joining $\nu_0$ and $\nu_1$. 

If $V$ is a linear subspace of $\R^n$, $N^*V \subset T^* \R^n$ 
denotes its {\em conormal space}, that is
\[
N^* V := \set{(q,p)\in \R^n \times (\R^n)^*}{q\in V, \; V \subset \ker
  p} = V \times V^{\perp},
\]
where $V^{\perp}$ denotes the set of covectors in $(\R^n)^*$ which
vanish on $V$. Conormal spaces are Lagrangian subspaces of
$T^* \R^n$. 

Let $C:T^* \R^n \rightarrow T^* \R^n$ be the linear involution
\[
C(q,p) := (q,-p) \quad \forall (q,p) \in T^* \R^n.
\]
The involution $C$ is anti-symplectic, meaning that
\[
\omega_0 (C \xi, C \eta) = - \omega_0 (\xi,\eta) \quad \forall \xi,\eta\in T^* \R^n.
\]
In particular, $C$ maps Lagrangian subspaces into Lagrangian subspaces. Since the Maslov index is natural with respect to symplectic transformations and changes sign if we change the sign of the symplectic structure, we have the identity 
\begin{equation}
\label{conju}
\mu (C\lambda,C\nu) = - \mu (\lambda,\nu),
\end{equation}
for every pair of Lagrangian paths $\lambda,\nu :[a,b] \rightarrow \mathscr{L}(n)$. Since conormal subspaces are $C$-invariant, we deduce that
\begin{equation}
\label{muzero}
\mu (N^*V,N^*W) = 0,
\end{equation}
for every pair of paths $V,W$ into the Grassmannian of $\R^n$.
Let $V_0,V_1,W_0,W_1$ be four linear subspaces of $\R^n$, and let $\nu:[0,1] \rightarrow \mathscr{L}(n)$ be a Lagrangian path such that $\nu(0)=N^* W_0$ and $\nu(1) = N^* W_1$. By (\ref{conju}),
\begin{equation*}\begin{split}
&h(N^*V_0,N^*V_1;N^*W_0,N^*W_1)\\ =\,&\mu(\nu,N^*V_1) - \mu(\nu,N^*V_0) = - \mu(C\nu,N^* V_1) + \mu(C\nu,N^* V_0).
\end{split}\end{equation*}
But also the Lagrangian path $C\nu$ joins $N^*W_0$ and $N^* W_1$, so the latter quantity equals 
\[
-h(N^*V_0,N^*V_1;N^*W_0,N^*W_1).
\]
We deduce the following:

\begin{prop}
\label{lemhor}
Let $V_0,V_1,W_0,W_1$ be four linear subspaces of $\R^n$. Then
\[
h(N^* V_0, N^*V_1;N^* W_0,N^* W_1) = 0.
\]
\end{prop}

We identify the product $T^* \R^n \times T^* \R^n$ with $T^* \R^{2n}$, and we endow it with its standard symplectic structure. In other words, we consider the product symplectic form, not the twisted one used in
\cite{rs93}. Note that the conormal space of the diagonal $\Delta_{\R^n}$ in $\R^n \times \R^n$ is the graph of $C$,
\[
N^* \Delta_{\R^n} = \graf C \subset T^* \R^n \times T^*\R^n = T^* \R^{2n}.
\]
The linear endomorphism $\Psi$ of $T^* \R^n$ belongs to the symplectic group $\mathrm{Sp}(2n)$ if and only if the graph of the linear endomorphism $\Psi C$ is a Lagrangian subspace of $T^* \R^{2n}$, if and only if the graph of  $C \Psi$ is a Lagrangian subspace of $T^* \R^{2n}$.  If $\lambda_1,\lambda_2$ are paths of Lagrangian subspaces of $T^* \R^n$ and $\Psi$ is a path in $\mathrm{Sp}(2n)$, Theorem 3.2 of \cite{rs93} leads to the identities
\begin{equation}
\label{masprod}
\mu(\Psi \lambda_1, \lambda_2) = \mu (\graf \Psi C,
C \lambda_1 \times \lambda_2) = - \mu(\graf C\Psi ,
\lambda_1 \times C \lambda_2).
\end{equation}
The {\em Conley-Zehnder index} $\mu_{CZ}(\Psi)$ of a symplectic path
$\Psi: [0,1] \rightarrow \mathrm{Sp}(2n)$ is related to the relative
Maslov index by the formula
\begin{equation}
\label{czi}
\mu_{CZ}(\Psi) = \mu(\graf \Psi C, N^* \Delta_{\R^n}) = \mu (N^* \Delta_{\R^n}, \graf C\Psi).
\end{equation}

We conclude this section by fixing some standard identifications, which allow to see $T^* \R^n$ as a complex vector space.
By using the Euclidean inner product on $\R^n$, we can identify $T^*
\R^n$ with $\R^{2n}$. 
We also identify the latter space to $\C^n$, by means of the isomorphism $(q,p) \mapsto q+ip$. In other words, we consider the complex structure
\[
J_0 := \left( \begin{array}{cc} 0 & -I \\ I & 0 \end{array} \right)
\]
on $\R^{2n}$. With these identifications, the Euclidean inner product 
$u \cdot v$, respectively the symplectic product $\omega_0(u,v)$, of 
two vectors $u,v\in T^* \R^n \cong \R^{2n} \cong \C^n$ is the real
part, respectively the imaginary part, of their Hermitian product 
$\langle \cdot,\cdot \rangle$,
\[
\langle u ,v \rangle := \sum_{j=1}^n u_j \overline{v_j} =
u \cdot v  + i \, \omega_0(u,v).
\]
The involution $C$ is the complex conjugation.
By identifying $V^{\perp}$ with the Euclidean orthogonal complement, we have
\[
N^* V = V \oplus i V^{\perp} = \set{z\in \C^n}{\re z\in V, \; \im z
  \in V^{\perp}}.
\]
If $\lambda:[0,1] \rightarrow \mathscr{L}(1)$ is the path
\[
\lambda(t) = e^{i\alpha t} \R, \quad \alpha \in \R,
\]
the relative Maslov index of $\lambda$ with respect to $\R$ is the
half integer
\begin{equation}
\label{maslov}
\mu(\lambda,\R) = \left\{ \begin{array}{ll} - \frac{1}{2} - \lfloor
  \frac{\alpha}{\pi} \rfloor & \mbox{if } \alpha \in \R \setminus \pi
  \Z, \\ - \frac{\alpha}{\pi} & \mbox{if } \alpha \in \pi
  \Z. \end{array} \right.
\end{equation}
Notice that the sign is different from the one appearing in
\cite{rs93} (localization axiom in Theorem 2.3), due to the fact that
we are using the opposite symplectic form on $\R^{2n}$. Our sign
convention here also differs from the one used in \cite{as06}, because
we are using the opposite complex structure on $\R^{2n}$.

\subsection{Elliptic estimates on the quadrant}
\label{ellsec}

We recall that a real linear subspace $V$ of $\C^n$ is said to be {\em
totally real} if $V\cap i V =(0)$. Denote by $\mathbb{H}$ the upper half-plane $\set{z\in \C}{\im z>0}$, and by $\mathbb{H}^+$ the upper-right
quadrant $\set{z\in \mathbb{H}}{\re z>0}$. We shall make use of the
following Calderon-Zygmund estimates for the Cauchy-Riemann
operator $\delbar = \partial_s + i \partial_t$:

\begin{thm}
\label{czt}
Let $V$ be an $n$-dimensional totally real subspace of
$\C^n$. For every $p\in ]1,+\infty[$, there exists a constant
$c=c(p,n)$ such that
\[
\| D u\|_{L^p} \leq c \| \delbar  u \|_{L^p}
\]
for every $u\in C^{\infty}_c(\C,\C^n)$, and for every $u\in
C^{\infty}_c(\mathrm{Cl}(\mathbb{H}),\C^n)$ such that $u(s)\in V$
for every $s\in \R$.
\end{thm}

We shall also need the following regularity result for weak solutions
of $\delbar $. Denoting by $\partial := \partial_s - i \partial_t$ the anti-Cauchy-Riemann operator, we have:

\begin{thm} \label{rws}
{\em (Regularity of weak solutions of $\delbar $)}
Let $V$ be an $n$-dimensional totally real subspace of $\C^n$, and let
$1<p<\infty$, $k\in \N$.
\begin{enumerate}
\item Let $u\in L^p_{\mathrm{loc}}
  (\C,\C^n)$, $f\in W^{k,p}_{\mathrm{loc}} (\C,\C^n)$ be such that
\[
\re \int_{\C} \langle u, \partial \varphi \rangle
 \, ds\,dt = - \re \int_{\C} \langle f, \varphi \rangle \, ds\,dt,
\]
for every $\varphi\in C^{\infty}_c(\C,\C^n)$.
Then $u\in W^{k+1,p}_{\mathrm{loc}}(\C,\C^n)$ and
$\delbar  u=f$.
\item Let
  $u\in L^p(\mathbb{H},\C^n)$, $f\in W^{k,p}(\mathbb{H},\C^n)$ be such
  that
\[
\re \int_{\mathbb{H}} \langle u, \partial \varphi \rangle
 \, ds\,dt = - \re \int_{\mathbb{H}} \langle f, \varphi \rangle \, ds\,dt,
\]
for every $\varphi \in C^{\infty}_c(\C,\C^n)$
such that $\varphi(\R) \subset V$.
Then $u\in W^{k+1,p}(\mathbb{H},\C^n)$, $\delbar  u=f$, and the trace
of $u$ on $\R$ takes values into the $\omega_0$-orthogonal complement of
$V$,
\[
V^{\perp_{\omega_0}} := \set{\xi\in \C^n}{\omega_0(\xi,\eta)=0 \;
  \forall \eta \in V}.
\]
\end{enumerate}
\end{thm}

\begin{rem}
If we replace the upper half-plane $\mathbb{H}$ in (ii) by the right
half-plane $\{\re z>0\}$ and the test mappings $\varphi\in
C^{\infty}_c(\C,\C^n)$ satisfy $\varphi( i\R) \subset V$, then the
trace of $u$ on $i\R$ takes value into $V^{\perp}$, the Euclidean
orthogonal complement of $V$ in $\R^{2n}$.
\end{rem}

Two linear subspaces $V,W$ of $\R^n$ are said to be {\em partially
orthogonal} if the linear subspaces $V\cap (V \cap W)^{\perp}$ and
$W \cap (V\cap W)^{\perp}$ are orthogonal, that is if their
projections into the quotient $\R^n/V\cap W$ are orthogonal. 

\begin{lem}
\label{czl} Let $V$ and $W$ be partially orthogonal linear subspaces
of $\R^n$. For
every $p\in ]1,+\infty[$, there exists a constant $c=c(p,n)$
such that
\begin{equation}
\label{caldzyg} \| D u\|_{L^p} \leq c \| \delbar 
u \|_{L^p}
\end{equation}
for every $u\in C^{\infty}_c(\mathrm{Cl}(\mathbb{H}^+),\C^n)$ such that
\begin{equation}
\label{bdryc} u(s)\in N^*V\: \forall s\in [0,+\infty[, \quad
u(it) \in N^*W \: \forall t\in [0,+\infty[.
\end{equation}
\end{lem}

\begin{proof} Since $V$ and $W$ are partially orthogonal, $\R^n$ has an orthogonal splitting $\R^n = X_1 \oplus X_2 \oplus X_3 \oplus X_4$ such that
\[
V= X_1 \oplus X_2, \quad W = X_1 \oplus X_3.
\]
Therefore,
\[
N^*V = X_1 \oplus X_2 \oplus iX_3 \oplus iX_4, \quad N^*W = X_1
\oplus X_3 \oplus iX_2 \oplus iX_4.
\]
Let $U\in \mathrm{U}(n)$ be the identity on $(X_1 \oplus X_2)
\otimes \C$, and the multiplication by $i$ on $(X_3 \oplus X_4)
\otimes \C$. Then
\[
UN^* V = \R^n, \quad UN^*W = X_1 \oplus X_4 \oplus iX_2 \oplus i
X_3 = N^*(X_1 \oplus X_4).
\]
Up to multiplying $u$ by $U$, we can replace the boundary
conditions (\ref{bdryc}) by
\begin{equation} \label{bdry1}
u(s)\in \R^n \: \forall s\in [0,+\infty[, \quad u(it) \in Y\:
\forall t\in [0,+\infty[,
\end{equation}
where $Y$ is a totally real $n$-dimensional subspace of $\C^n$
such that $\overline{Y}=Y$. Define a $\C^n$-valued map $v$ on the
right half-plane $\{\re z\geq 0\}$ by Schwarz reflection,
\[
v(z) := \left\{ \begin{array}{ll} u(z) & \mbox{if } \im z\geq 0,
\\ \overline{u(\overline{z})} & \mbox{if } \im z\leq 0. \end{array} \right.
\]
By (\ref{bdry1}) and by the fact that $Y$ is self-conjugate,
$v$ belongs to $W^{1,p}(\{\re z>0 \},\C^n)$, and satisfies
\[
v(it) \in Y \quad \forall t\in \R.
\]
Moreover,
\[
\| \nabla v\|_{L^p(\{\re z>0\})}^p = 2 \| \nabla
v\|_{L^p(\mathbb{H}^+)}^p,
\]
and since $\delbar  v(z) =
\overline{\delbar  u(\overline{z})}$ for $\im z\leq 0$,
\[
\| \delbar  v\|_{L^p(\{\re z>0\})}^p = 2 \|
\delbar  v\|_{L^p(\mathbb{H}^+)}^p.
\]
Then (\ref{caldzyg}) follows from the Calderon-Zygmund estimate on
the half plane with totally real boundary conditions (Theorem
\ref{czt}). 
\end{proof}

Similarly, Theorem \ref{rws} has the following consequence about
regularity of weak solutions of $\delbar $ on the upper right quadrant $\mathbb{H}^+$:

\begin{lem}
\label{woq}
Let $V$ and $W$ be partially orthogonal linear subspaces of $\R^n$.
Let $u\in L^p(\mathbb{H}^+,\C^n)$, $f\in L^p(\mathbb{H}^+,\C^n)$, $1<p<\infty$, be such that
\[
\re \int_{\mathbb{H}^+} \langle u, \partial \varphi \rangle
 \, ds\,dt = - \re \int_{\mathbb{H}^+} \langle f, \varphi \rangle \, ds\,dt,
\]
for every $\varphi \in C^{\infty}_c(\C,\C^n)$
such that $\varphi(\R) \subset N^*V$, $\varphi(i\R) \subset N^* W$.
Then $u\in W^{1,p}(\mathbb{H}^+,\C^n)$, $\delbar  u=f$, the trace of $u$ on
$\R$ takes values into $N^*V$, and the trace of $u$ on $i\R$ takes
values into $(N^* W)^{\perp} = N^* (W^{\perp}) = i N^* W$.
\end{lem}

\begin{proof}
By means of a linear unitary transformation, as in the proof of Lemma \ref{czl}, we may assume that $V=N^* V = \R^n$. A Schwarz reflection then allows to extend $u$ to a map $v$ on the right half-plane $\{\re z>0\}$ which is in $L^p$ and is a weak solution of $\delbar  v = g \in L^p$,  with boundary condition in $i N^* W$ on $i\R$. The thesis  follows from Theorem \ref{rws}.
\end{proof}

We are now interested in studying the operator $\delbar $
on the half-plane $\mathbb{H}$, with boundary conditions
\[
u(s)\in N^* V, \quad u(-s) \in N^*W \quad \forall s>0,
\]
where $V$ and $W$ are partially orthogonal linear subspaces of $\R^n$. 
Taking Lemmas \ref{czl} and \ref{woq} into account, the natural idea is to obtain the required estimates by applying a conformal change of variable mapping the half-plane $\mathbb{H}$ onto the the upper right quadrant $\mathbb{H}^+$. More precisely,
let $\mathscr{R}$ and $\mathscr{T}$ be the transformations
\begin{eqnarray}
\label{transr}
\mathscr{R}: \mathrm{Map}(\mathbb{H},\C^n) \rightarrow \mathrm{Map}(\mathbb{H}^+,\C^n), \quad (\mathscr{R}u)(\zeta) = u(\zeta^2), \\
\label{transt}
\mathscr{T}: \mathrm{Map}(\mathbb{H},\C^n) \rightarrow \mathrm{Map}(\mathbb{H}^+,\C^n), \quad (\mathscr{T}u)(\zeta) = 2 \overline{\zeta} u(\zeta^2),
\end{eqnarray}
where $\mathrm{Map}$ denotes some space of maps. Then the diagram
\begin{equation}
\label{diagrt}
\begin{CD}
\mathrm{Map}(\mathbb{H},\C^n) @>{\delbar }>>  \mathrm{Map}(\mathbb{H},\C^n) \\ @VV{\mathscr{R}}V @V{\mathscr{T}}VV \\ 
\mathrm{Map}(\mathbb{H}^+,\C^n) @>{\delbar }>>  \mathrm{Map}(\mathbb{H}^+,\C^n)
\end{CD}
\end{equation}
commutes. By the elliptic estimates of Lemma \ref{czl}, suitable domain and codomain for the operator on the lower horizontal arrow are the standard $W^{1,p}$ and $L^p$ spaces, for $1<p<\infty$. Moreover, if  $u\in \mathrm{Map}(\mathbb{H},\C^n)$ then
\begin{eqnarray}
\label{Lpr}
\|\mathscr{R} u\|_{L^p(\mathbb{H}^+)}^p = \frac{1}{4} \int_{\mathbb{H}} \frac{1}{|z|} |u(z)|^p \, ds\,dt, \\ \label{Lpdr}
\|D (\mathscr{R} u)\|_{L^p(\mathbb{H}^+)}^p
=   2^{p-2} \int_{\mathbb{H}} |D u(z)|^p |z|^{p/2-1} \, ds\,dt, \\
\label{Lpt}
\|\mathscr{T} u\|_{L^p(\mathbb{H}^+)}^p = 2^{p-2} \int_{\mathbb{H}} |u(z)|^p |z|^{p/2-1}\, ds\,dt.
\end{eqnarray}
Note also that by the generalized Poincar\'e inequality, the $W^{1,p}$ norm on $\mathbb{H}^+ \cap \mathbb{D}_r$, where $\mathbb{D}_r$ denotes the open disk of radius $r$, is equivalent to the norm
\begin{equation}
\label{eqsobnorm}
\|v\|_{\tilde{W}^{1,p}(\mathbb{H}^+ \cap \mathbb{D}_r)}^p :=
\|D v\|_{L^p(\mathbb{H}^+ \cap \mathbb{D}_r)}^p + \int_{\mathbb{H}^+ \cap \mathbb{D}_r} |v(\zeta)|^p |\zeta|^p\, d\sigma d\tau,
\end{equation}
and the $\tilde{W}^{1,p}$ norm of $\mathscr{R} u$ is
\begin{equation}
\label{treqsobnorm}\begin{split}
\|\mathscr{R}u\|_{\tilde{W}^{1,p}(\mathbb{H}^+ \cap \mathbb{D}_r)}^p \,=\, &\frac{1}{4} \int_{\mathbb{H}\cap \mathbb{D}_{r^2}} |u(z)|^p |z|^{p/2-1}\, ds\,dt\\ &+\,2^{p-2} \int_{\mathbb{H}\cap \mathbb{D}_{r^2}} |D u(z)|^p |z|^{p/2-1} \, ds\,dt.\end{split}
\end{equation}
So when dealing with bounded domains, both the transformations $\mathscr{R}$ and $\mathscr{T}$ involve the appearance of the weight $|z|^{p/2-1}$ in the $L^p$ norms. Note also that when $p=2$, this weight is just $1$, reflecting the fact that the $L^2$ norm of the differential is a conformal invariant. 

By the commutativity of diagram (\ref{diagrt}) and by the identities (\ref{Lpdr}), (\ref{Lpt}), Lemma \ref{czl} applied to $\mathscr{R}u$ implies the following:

\begin{lem}
\label{czl2} Let $V$ and $W$ be partially orthogonal
linear subspaces of $\R^n$.  For
every $p\in ]1,+\infty[$, there exists a constant $c=c(p,n)$
such that
\[
\int_{\mathbb{H}} |\nabla u(z)|^p |z|^{p/2-1} \, ds
\, dt \leq c^p \int_{\mathbb{H}}
|\delbar  u(z)|^p |z|^{p/2-1} \, ds \, dt
\]
for every compactly supported map
$u: \mathrm{Cl}(\mathbb{H}) \rightarrow \C^n$ such that
$\zeta \mapsto u(\zeta^2)$ is smooth on $\mathrm{Cl}(\mathbb{H}^+)$, and
\[
\label{bdry} u(s)\in N^*V \; \forall s\in ]-\infty,0], \quad u(s) \in
N^*W \: \forall s\in [0,+\infty[.
\]
\end{lem}

\subsection{Strips with jumping conormal boundary conditions}
\label{sjlbc}

Let us consider the following data: two integers $k,k^{\prime}\geq
0$, $k+1$ linear subspaces $V_0,\dots, V_k$ of $\R^n$ such that
$V_{j-1}$ and $V_j$ are partially orthogonal, for every $j=1,\dots,k$,
$k^{\prime}+1$ linear subspaces $V_0^{\prime},\dots, V_k^{\prime}$ of
$\R^n$ such that
$V_{j-1}^{\prime}$ and $V_j^{\prime}$ are partially orthogonal,
for every $j=1,\dots,k^{\prime}$, and real numbers
\begin{equation*}\begin{split}
&-\infty = s_0 < s_1 < \dots < s_k < s_{k+1} = +\infty,\\
&-\infty = s_0^{\prime} < s_1^{\prime} < \dots < s_{k^{\prime}}^{\prime} <
s_{k^{\prime}+1}^{\prime} = +\infty.
\end{split}\end{equation*}
Denote by $\mathscr{V}$ the $(k+1)$-tuple $(V_0,\dots,V_k)$, by
$\mathscr{V}^{\prime}$ the $(k^{\prime}+1)$-tuple
$(V_0^{\prime},\dots,V_k^{\prime})$, and set
\[
\mathscr{S}:= \{s_1,\dots,s_k,s_1'+i, \dots,
s_{k^{\prime}}^{\prime} + i\}.
\]
Let $\Sigma$ be the closed strip
\[
\Sigma := \set{z\in \C}{0 \leq \im z \leq 1}.
\]
The space $C^{\infty}_{\mathscr{S}}
(\Sigma,\C^n)$ is the space of maps $u: \Sigma
\rightarrow \C^n$ which are smooth on $\Sigma \setminus
\mathscr{S}$, and such that the maps 
$\zeta \mapsto u (s_j + \zeta^2)$ and $\zeta \mapsto u(s_j' + i -
\zeta^2)$ are smooth in a neighborhood of $0$ in the  closed
upper-right quadrant
\[
\mathrm{Cl}(\mathbb{H}^+) = \set{\zeta\in \C}{\re \zeta\geq 0, \;
\im \zeta\geq 0}.
\]
The symbol $C^{\infty}_{\mathscr{S},c}$
indicates bounded support.

Given $p\in [1,+\infty[$, we define the $X^p$ norm of a map $u\in L^1_{\mathrm{loc}} (\Sigma,\C^n)$ by
\[
\|u\|_{X^p(\Sigma)}^p :=
     \|u\|^p_{L^p(\Sigma \setminus B_r(\mathscr{S}))} + \sum_{w\in \mathscr{S}} \int_{\Sigma \cap B_r(w)} |u(z)|^p
    |z-w|^{p/2-1}\, ds\,dt,
\]
where $r<1$ is less than half of the minimal distance between pairs of distinct points in $\mathscr{S}$. This is just a weighted $L^p$ norm, where the weight $|z-w|^{p/2-1}$ comes from the identities (\ref{Lpdr}), (\ref{Lpt}), and (\ref{treqsobnorm}) of the last section.
Note that when $p>2$ the $X^p$ norm is weaker than the
$L^p$ norm, when $p<2$ the $X^p$ norm is stronger than the $L^p$ norm, and when $p=2$ the two norms are equivalent.

The space $X^p_{\mathscr{S}}(\Sigma,\C^n)$ is the space of locally integrable $\C^n$-valued maps on $\Sigma$ whose
$X^p$ norm is finite. The $X^p$ norm makes it a Banach space. We
view it as a {\em real} Banach space.

The space $X^{1,p}_{\mathscr{S}}(\Sigma,\C^n)$
is defined as the completion of the space $C^{\infty}_{\mathscr{S},c}(\Sigma,\C^n)$ with respect to the norm
\[
\|u\|_{X^{1,p}(\Sigma)}^p := \|u\|_{X^p(\Sigma)}^p + \|Du\|_{X^p(\Sigma)}^p.
\]
It is a Banach space with the above norm. Equivalently, it is the
space of maps in $X^p(\Sigma,\C^n)$ whose distributional derivative is
also in $X^p$. 
The space $X^{1,p}_{\mathscr{S},\mathscr{V},\mathscr{V}^{\prime}} (\Sigma,\C^n)$ is defined as the closure in $X^{1,p}_{\mathscr{S}}
(\Sigma,\C^n)$ of the space of all $u\in C^{\infty}_{\mathscr{S},c}
(\Sigma,\C^n)$ such that
\begin{equation}
\label{conoboco}
\begin{array}{lll}
u(s)\in N^*V_j & \forall s\in [s_j,s_{j+1}], & j=0,\dots,k, \\
u(s+i)\in N^*V_j^{\prime} & \forall s\in
  [s_j^{\prime},s_{j+1}^{\prime}], & j=0,\dots,k^{\prime}.
\end{array} \end{equation}
Equivalently, it can be defined in terms of the trace of $u$ on the
boundary of $\Sigma$.

Let $A:\R \times [0,1]\rightarrow  \mathrm{L}(\R^{2n},\R^{2n})$ be continuous and bounded. For every $p \in [1,+\infty[$, the linear operator
\[
\delbar_A : X^{1,p}_{\mathscr{S}}(\Sigma,\C^n) \rightarrow X^p_{\mathscr{S}}(\Sigma,\C^n), \quad \delbar_A u := \delbar  u + A u,
\]
is bounded. Indeed, $\delbar $ is a bounded operator because of the inequality $|\delbar  u|\leq |Du|$, while the multiplication operator by $A$ is bounded because
\[
\|Au\|_{X^p(\Sigma)} \leq \|A\|_{\infty} \|u\|_{X^p(\Sigma)}.
\]
We wish to prove that if $p>1$ and $A(z)$ satisfies suitable asymptotics for $\re z \rightarrow \pm \infty$ the operator $\delbar_A$ restricted to the space $X^{1,p}_{\mathscr{S}, \mathscr{V}, \mathscr{V}'} (\Sigma,\C^n)$ of maps satisfying the boundary conditions (\ref{conoboco}) is Fredholm.

Assume that $A\in C^0(\overline{\R} \times [0,1], \mathrm{L}(\R^{2n},\R^{2n}))$ is
such that $A(\pm \infty,t) \in \mathrm{Sym}(2n,\R)$ for every $t\in
[0,1]$. Define $\Phi^+,\Phi^-:[0,1] \rightarrow \mathrm{Sp}(2n)$ to be
the solutions of the linear Hamiltonian systems
\begin{equation}
\label{wronsk}
\frac{d}{dt} \Phi^{\pm}(t) = i A(\pm \infty,t) \Phi^{\pm}(t), \quad
\Phi^{\pm}(0) = I.
\end{equation}
Then we have the following:

\begin{thm}
\label{fredholm}
Assume that $\Phi^-(1) N^*V_0 \cap N^*V_0' = (0)$ and $\Phi^+(1) N^*
V_k \cap N^* V_{k'}' = (0)$. Then the bounded $\R$-linear operator
\[
\delbar_A :
X^{1,p}_{\mathscr{S},\mathscr{V},\mathscr{V}^{\prime}}
(\Sigma,\C^n) \rightarrow X^p_{\mathscr{S}}
(\Sigma,\C^n), \quad
\delbar_A u = \delbar  u + A u,
\]
is Fredholm of index
\begin{equation}
\label{indform}
\begin{split}
\ind \delbar_A \,=\; &\mu(\Phi^- N^*V_0,N^* V_0^{\prime}) - \mu(\Phi^+ N^*
V_k,N^* V_{k^{\prime}}^{\prime}) \\ 
&- \frac{1}{2}\sum_{j=1}^k (\dim V_{j-1} + \dim V_j - 2 \dim V_{j-1} \cap V_j)\\
&- \frac{1}{2}\sum_{j=1}^{k^{\prime}} (\dim V_{j-1}^{\prime} + \dim V_j^{\prime} -
2 \dim V_{j-1}^{\prime} \cap V_j^{\prime}).
\end{split}
\end{equation}
\end{thm}

The proof of the Fredholm property for Cauchy-Riemann type operators is based on local estimates. By a partition of unity argument, 
the proof that $\delbar_A$ is Fredholm reduces to the Calderon-Zygmund estimates of Lemmas \ref{czl}, \ref{czl2}, and to the invertibility of $\delbar_A$ when $A$ does not depend on $\re z$ and there are no jumps in the boundary conditions. Details are contained in the next section. The index computation instead is based on homotopy arguments together with a Liouville type result stating that in a particular case with one jump the operator $\delbar_A$ is an isomorphism.

\subsection{The Fredholm property}
\label{tfps}

The elliptic estimates of Section \ref{ellsec} have the following consequence:

\begin{lem}
\label{czj}
For every $p\in ]1,+\infty[$, there exist constants
$c_0=c_0(p,n,\mathscr{S})$ and $c_1=c_1(p,n,k+k^{\prime})$ such that
\[
\| D u \|_{X^p} \leq  c_0
\|u\|_{X^p} + c_1 \|\delbar  u\|_{X^p},
\]
for every $u\in C^{\infty}_{\mathscr{S},c} (\Sigma,\C^n)$ such that
\[
u(s) \in N^*V_j \quad \forall s\in
[s_{j-1},s_j ], \quad u(s+i) \in N^*V_j^{\prime} \quad \forall s\in
[s_{j-1}^{\prime},s_j^{\prime} ]
\]
for every $j$.
\end{lem}

\begin{proof}
Let $\{\psi_1,\psi_2\} \cup \{\varphi_j\}_{j=1}^{k+k'}$ be a smooth
partition of unity on $\C$ satisfying: 
\begin{equation}
\label{partun}
\begin{array}{lcll}
\supp \psi_1 & \subset & \set{z\in \C}{\im z<2/3} \setminus B_{r/2} (\mathscr{S}), & \\
\supp \psi_2 & \subset & \set{z\in \C}{\im z>1/3} \setminus B_{r/2} (\mathscr{S}), & \\
\supp \varphi_j & \subset & B_r(s_j) & \forall j=1,\dots,k, \\
\supp \varphi_{k+j} & \subset & B_r(s_j'+i) & \forall j=1,\dots,k'.
\end{array} 
\end{equation}
By Lemma \ref{czl2},
\begin{equation*}\begin{split}
\|D(\varphi_j u)\|_{X^p(\Sigma)}\;&\leq\, c(p,n) \|\delbar 
(\varphi_j u)\|_{X^p(\Sigma)}\\ &\leq\, c(p,n)
( \|\delbar  \varphi_j \|_{\infty} \|u\|_{X^p(\Sigma)} +
\|\delbar  u\|_{X^p(\Sigma)} ), \quad 1\leq j \leq k+k'.
\end{split}\end{equation*}
Since the $X^p$ norm is equivalent to the $L^p$ norm on the subspace of maps whose support does not meet $B_{r/2}(\mathscr{S})$, the standard Calderon-Zygmund estimates on the half-plane (see Theorem \ref{czt}) imply 
\begin{equation*}\begin{split}
\|D(\psi_j u)\|_{X^p(\Sigma)}\; &\leq\, c(p,n)  \|\delbar 
(\psi_j u)\|_{X^p(\Sigma)}\\ &\leq\, c(p,n)
( \|\delbar  \psi_j \|_{\infty} \|u\|_{X^p(\Sigma)} +
\|\delbar  u\|_{X^p(\Sigma)} ), \quad \forall j=1,2.
\end{split}\end{equation*}
We conclude that
\begin{equation*}\begin{split}
\|Du\|_{X^p} \;&\leq\, \|D(\psi_1 u)\|_{X^p(\Sigma)} + \|D(\psi_2 u)\|_{X^p(\Sigma)} + \sum_{j=1}^{k+k^{\prime}} \|D(\varphi_j u)\|_{X^p(\Sigma)}\\ &\leq\,
c_0 \|u\|_{X^p(\Sigma)} + c_1 \|\delbar  u\|_{X^p(\Sigma)},
\end{split}\end{equation*}
with
\[
c_0 := c(p,n) \Bigl( \|\delbar 
\psi_1\|_{\infty} + \|\delbar 
\psi_2\|_{\infty} + \sum_{j=0}^{k+k^{\prime}+1} \|\delbar 
\varphi_j\|_{\infty}\Bigr), \quad c_1 := (k+k^{\prime}+2) c(p,n),
\]
as claimed.
\end{proof}

The next result we need is the following theorem, proved in
\cite[Theorem 7.1]{rs95}. Consider two continuously differentiable
Lagrangian paths $\lambda,\nu:\overline{\R} \rightarrow \mathscr{L}(n)$, assumed to be constant on $[-\infty,-s_0]$ and on
  $[s_0,+\infty]$, for some $s_0>0$.
Denote by $W^{1,p}_{\lambda,\nu}(\Sigma,\C^n)$ the
space of maps $u\in W^{1,p}(\Sigma,\C^n)$ such that
$u(s,0) \in \lambda(s)$ and $u(s,1)\in \nu(s)$, for every $s\in \R$ (in the
sense of traces).
Let $A\in C^0(\overline{\R} \times [0,1], \mathrm{L}(\R^{2n},\R^{2n}))$ be such that
$A(\pm \infty,t)\in \mathsf{Sym} (2n,\R)$ for any $t\in [0,1]$, and
define $\Phi^-,\Phi^+:[0,1] \rightarrow \mathsf{Sp}(2n)$ by (\ref{wronsk}).

\begin{thm} \label{strip}
{\em (Cauchy-Riemann operators on the strip)} Let $p\in ]1,+\infty[$,
and assume that
\[
\Phi^-(1) \lambda(-\infty) \cap \nu(-\infty) = (0), \quad \Phi^+(1) \lambda(+\infty) \cap \nu(+\infty) = (0).
\]
\begin{enumerate}
\item  The bounded $\R$-linear operator
\[
\delbar_A : W^{1,p}_{\lambda,\nu} (\Sigma,\C^n) \rightarrow
L^p (\Sigma,\C^n) , \quad
\delbar_A u = \delbar  u + A u,
\]
is Fredholm of index
\[
\ind \delbar_A = \mu(\Phi^- \lambda(-\infty),\nu(-\infty)) - \mu(\Phi^+\lambda(+\infty), \nu(+\infty)) + \mu(\lambda,\nu).
\]
\item If furthermore $A(s,t)=A(t)$, $\lambda(s)=\lambda$, and $\nu(s)=\nu$ do not depend on $s$, the operator $\delbar_A$ is an isomorphism.
\end{enumerate}
\end{thm}

Note that under the assumptions of (ii) above,
the equation $\delbar  u + Au$ can be rewritten as
$\partial_s u = -L_A u$, where
$L_A$ is the unbounded $\R$-linear operator on
$L^2([0,1],\C^n)$ defined by
\begin{equation*}\begin{split}
\dom &L_A\, =\, W^{1,2}_{\lambda,\nu}([0,1],\C^n) = \set{u\in
    W^{1,2}([0,1],\C^n)}{u(0)\in \lambda, \; u(1)\in \nu},\\ &L_A \,=\, i
    \frac{d}{dt} + A.
\end{split}\end{equation*}
The conditions on $A$ imply that $L_A$ is self-adjoint and invertible.
These facts lead to the following:

\begin{prop}
\label{dec}
Assume that $A$, $\lambda$, and $\nu$ satisfy the conditions of Theorem
 \ref{strip} (ii), and set $\delta:=\min \sigma(L_A)\cap
[0,+\infty[>0$. Then for every $k\in \N$ there exists $c_k$ such that
\[
\|u(s,\cdot)\|_{C^k([0,1])} \leq c_k \|u(0,\cdot)\|_{L^2([0,1])}
e^{-\delta s}, \quad \forall s\geq 0,
\]
for every $u\in W^{1,p}(]0,+\infty[ \times ]0,1[,\C^n)$, $p>1$, such
  that $u(s,0)\in \lambda$, $u(s,1)\in \nu$ for every $s\geq 0$, and
  $\delbar  u + A u=0$.
\end{prop}

Next we need the following easy consequence of the Sobolev embedding theorem:

\begin{prop}
\label{compemb}
Let $s>0$ and let $\chi_s$ be the characteristic function of the set $\set{z\in \Sigma}{|\re z|\leq s}$. Then the linear operator
\[
X^{1,p}_{\mathscr{S}} (\Sigma,\C^n) \rightarrow X^q_{\mathscr{S}} (\Sigma,\C^n), \quad u \mapsto \chi_s u,
\]
is compact for every $q<\infty$ if $p\geq 2$, and for every $q<2p/(2-p)$ if $1\leq p< 2$.
\end{prop}

\begin{proof}
Let $(u_h)$ be a bounded sequence in $X^{1,p}_{\mathscr{S}}(\Sigma,\C^n)$. Let $\{\psi_1,\psi_2\}\cup \{\varphi_j\}_{j=1}^{k+k'}$ be a smooth partition of unity of $\C$ satisfying (\ref{partun}). Then the sequences $(\psi_1 u_h)$, $(\psi_2 u_h)$ and $(\varphi_j u_h)$, for $1\leq j \leq k+k'$ are bounded in $X^{1,p}(\Sigma,\C^n)$. We must show that each of these sequences is compact in $X^q_{\mathscr{S}}(\Sigma,\C^n)$.

Since the $X^q$ and $X^{1,p}$ norms on the space of maps supported in $\Sigma\setminus B_{r/2}(\mathscr{S})$ are equivalent to the $L^q$ and $W^{1,p}$ norms, the Sobolev embedding theorem implies that the sequences $(\chi_s \psi_1 u_h)$ and $(\chi_s \psi_2 u_h)$ are compact in $X^q_{\mathscr{S}}(\Sigma,\C^n)$. 

Let $1\leq j \leq k$.  If $u$ is supported in $B_r(s_j)$, set $v(z) := u(s_j+z)$, so that by (\ref{Lpr})
\begin{equation}
\label{confro}\begin{split}
\|u\|_{X^q(\Sigma)}^q \,=\,& \int_{B_r(s_j) \cap \Sigma} |u(z)|^q |z|^{q/2-1}\, ds \,dt \,\leq\, \int_{B_r(s_j) \cap \Sigma} \frac{1}{|z|} |u(z)|^q \, ds \,dt\\ =\; &4 \|\mathscr{R} u\|_{L^q(\mathbb{H}^+ \cap \mathbb{D}_{\sqrt{r}})}^q.\end{split}
\end{equation}
Set $v_h (z) := \varphi_j(s_j + z) u_h(s_j + z)$.
By (\ref{treqsobnorm}), the sequence $(\mathscr{R} v_h)$ is bounded in $W^{1,p}(\mathbb{H}^+ \cap \mathbb{D}_{\sqrt{r}})$, hence it is compact in $L^q(\mathbb{H}^+ \cap \mathbb{D}_{\sqrt{r}})$ for every $q<\infty$ if $p\geq 2$, and for every $q<2p/(2-p)$ if $1\leq p<2$. Then (\ref{confro}) implies that $(\varphi_j u_h)$ is compact in $X^q_{\mathscr{S}}(\Sigma,\C^n)$.
A fortiori, so is $(\chi_s \varphi_j u_h)$.
 The same argument applies to $j\geq k+1$, concluding the proof.
\end{proof} 

\medskip 

Putting together Lemma \ref{czj}, statement (ii) in Theorem \ref{strip}, and the Proposition above we obtain the following:

\begin{prop}
\label{sfe}
Let $1<p<\infty$. Assume that the paths of symmetric matrices
$A(\pm \infty, \cdot)$ satisfy the assumptions of Theorem \ref{fredholm}.
Then 
\[
\delbar_A : X^{1,p}_{\mathscr{S},\mathscr{V},\mathscr{V}^{\prime}}(\Sigma,\C^n) \rightarrow X^p_{\mathscr{S}}(\Sigma,\C^n)
\] 
is semi-Fredholm with $\ind \delbar_A := \dim \ker \delbar_A - \dim \coker \delbar_A < +\infty$.
\end{prop}

\begin{proof}
We claim that there exist $c\geq 0$ and $s \geq 0$ such that, for any $u\in
X^{1,p}_{\mathscr{S},\mathscr{V},
\mathscr{V}^{\prime}}(\Sigma,\C^n)$, there holds
\begin{equation}
\label{sfineq}
\|u\|_{X^{1,p}(\Sigma)} \leq c \left(\|(\delbar  +
  A)u\|_{X^p(\Sigma)} +
\|\chi_s u\|_{X^p(\Sigma)} \right),
\end{equation}
where $\chi_s$ is the characteristic function of the set $\set{z\in \Sigma}{|\re z|\leq s}$. 

By Theorem \ref{strip} (ii), the asymptotic operators
\begin{eqnarray*}
\delbar  + A(-\infty,\cdot) : W^{1,p}_{N^*V_0,N^*V_0^{\prime}}
(\Sigma,\C^n) \rightarrow L^p(\Sigma,\C^n), \\
\delbar  + A(+\infty,\cdot) :
W^{1,p}_{N^*V_k,N^*V^{\prime}_{k^{\prime}}} (\Sigma,\C^n)
\rightarrow L^p(\Sigma,\C^n),
\end{eqnarray*}
are invertible. Since invertibility is an open condition in the
operator norm, there exist $s > \max |\re \mathscr{S}| + 2$ and $c_1>0$ such
that for any $u\in X^{1,p}_{\mathscr{S},
\mathscr{V},\mathscr{V}^{\prime}}(\Sigma,\C^n)$
with support disjoint from $\{|\re z|\leq s - 1\}$ there holds
\begin{equation}
\label{sta}
\|u\|_{X^{1,p}(\Sigma)} = \|u\|_{W^{1,p}(\Sigma)} \leq c_1 \|
(\delbar  + A)u\|_{L^p(\Sigma)} = c_1  \|
(\delbar  + A)u\|_{X^p(\Sigma)}.
\end{equation}
By Proposition \ref{czj}, there exists $c_2>0$ such that for every
$u\in X^{1,p}_{\mathscr{S},
\mathscr{V},\mathscr{V}^{\prime}}(\Sigma,\C^n)$ with support
in $\{|\re z|\leq s \}$ there holds
\begin{equation}
\label{stb}
\begin{split}
\|u\|_{X^{1,p}(\Sigma)} \leq c_2 (\|u\|_{X^p(\Sigma)}
+ \|\delbar  u\|_{X^p(\Sigma)} ) \\ \leq (c_2 +
\|A\|_{\infty}) \|u\|_{X^p(\Sigma)} + c_2
\|(\delbar +A) u\|_{X^p(\Sigma)}.
\end{split}
\end{equation}
The inequality (\ref{sfineq}) easily follows from (\ref{sta}) and (\ref{stb})
by writing any $u\in X^{1,p}_{\mathscr{S},\mathscr{V}, \mathscr{V}^{\prime}}(\Sigma,\C^n)$ as $u=(1-\varphi) u +
\varphi u$, for $\varphi$ a smooth real function on $\Sigma$
having support in $\{|\re z| < s\}$ and such that $\varphi=1$ on $\{
|\re z|\leq s-1\}$.

Finally, by Proposition \ref{compemb} the linear operator
\[
X^{1,p}_{\mathscr{S},\mathscr{V},\mathscr{V}^{\prime}}(\Sigma,\C^n) \rightarrow X^p_{\mathscr{S}}(\Sigma,\C^n), \quad u \mapsto \chi_s u,
\]
is compact. Therefore the estimate (\ref{sfineq}) implies that $\delbar_A$ has finite dimensional kernel and closed range, that is it is semi-Fredholm with index less than $+\infty$.
\end{proof}

It would not be difficult to use the regularity of weak solutions of the Cauchy-Riemann operator to prove that the cokernel of $\delbar_A$ is finite-dimensional, so that $\delbar_A$ is Fredholm. However, this will follow directly from the index computation presented in the next section.

\subsection{A Liouville type result}
\label{liouvsec}

Let us consider the following particular case in dimension $n=1$:
\[
k=1, \quad k^{\prime}=0, \quad \mathscr{S}=\{0\}, \quad 
V_0 = (0), \quad V_1 = \R , \quad V_0^{\prime} = \R, \quad A(z)=\alpha,
\]
with $\alpha$ a real number. In other words, we are looking at the operator $\delbar  + \alpha$ on a space of $\C$-valued maps $u$ on $\Sigma$ such that $u(s)$ is purely imaginary for $s\leq 0$, $u(s)$ is real for $s\geq 0$, and $u(s+i)$ is real for every $s\in \R$.
Notice that $\Phi^-(t) = \Phi^+(t) = e^{i\alpha t}$, so
\[
e^{i\alpha} i\R \cap \R = (0) \quad \forall \alpha \in \R \setminus ( \pi/2 + \pi \Z ), \quad e^{i \alpha} \R \cap \R = (0) \quad \forall \alpha \in \R \setminus \pi \Z,
\]
so the assumptions of Theorem \ref{fredholm} are satisfied whenever $\alpha$ is not an integer multiple of $\pi/2$. In order to simplify the notation, we set
\[
X^p(\Sigma) := X^p_{\{0\}}(\Sigma,\C), \quad X^{1,p}(\Sigma) := X^{1,p}_{\{0\},((0),\R),(\R)}(\Sigma,\C).
\]
We start by studying the regularity of the elements of the kernel of $\delbar_{\alpha}$:

\begin{lem}
\label{rker}
Let $p>1$ and $\alpha\in \R \setminus (\pi/2) \Z$. If $u$ belongs to the kernel of
\[
\delbar_{\alpha}: X^{1,p}(\Sigma)
\rightarrow X^p(\Sigma),
\]
then $u$ is smooth on $\Sigma\setminus \{0\}$, it satisfies the boundary conditions pointwise, and the function $(\mathscr{R} u)(\zeta) = 
u(\zeta^2)$ is smooth on $\mathrm{Cl}(\mathbb{H}^+) \cap \mathbb{D}_1$. In particular, $u$ is continuous at $0$, and $D u(z)=O(|z|^{-1/2})$
for $z\rightarrow 0$.
\end{lem}

\begin{proof}
The regularity theory for weak solutions of $\delbar $ on $\C$ and on the half-plane $\mathbb{H}$ (Theorem \ref{rws}) implies -- by a standard bootstrap argument -- that $u\in C^{\infty}(\Sigma \setminus \{0\})$. We just need to check the regularity of $u$ at $0$.

Consider the function $f(\zeta):= e^{\alpha \overline{\zeta}^2/2} u(\zeta^2)$ on $\mathbb{H}^+ \cap \mathbb{D}_1$. Since
\[
\delbar  f(\zeta) = 2 \overline{\zeta} e^{\alpha \overline{\zeta}^2/2} \bigl(\delbar  u(\zeta^2) + \alpha u(\zeta^2)\bigr) = 0,
\]
$f$ is holomorphic on $\mathbb{H}^+ \cap \mathbb{D}_1$.  
Moreover, by (\ref{treqsobnorm}) the function $f$ belongs to $W^{1,p}(\mathbb{H}^+ \cap \mathbb{D}_1)$, and in particular it is square integrable. The function $f$ is real on $\R^+$ and purely imaginary on $i\R^+$, so a double Schwarz reflection produces a holomorphic extension of $f$ to $\mathbb{D}_1\setminus \{0\}$. Such an extension of $f$ is still square integrable, so the singularity $0$ is removable and the function is holomorphic on the whole $\mathbb{D}_1$. It follows that
\[
(\mathscr{R} u)(\zeta) = u(\zeta^2) =  e^{-\alpha \overline{\zeta}^2/2} f(\zeta)
\]
is smooth on  $\mathrm{Cl}(\mathbb{H}^+) \cap \mathbb{D}_1$, as claimed.
\end{proof}

The real Banach space $X^p(\Sigma)$ is the space of $L^p$ functions with respect to the measure defined by the density
\[
\rho_p(z) := \left\{ \begin{array}{ll} 1 & \mbox{if } z\in \Sigma \setminus \mathbb{D}_r, \\ |z|^{p/2-1} & \mbox{if } z\in \Sigma \cap \mathbb{D}_r. \end{array} \right.
\]
So the dual of $X^p(\Sigma)$ can be identified with the real Banach space
\begin{equation}
\label{primodual}
\set{v\in L^1_{\mathrm{loc}}(\Sigma,\C)}{\int_{\Sigma} |v|^q \rho_p(z) \, ds\,dt <+\infty}, \quad \mbox{where } \frac{1}{p} + \frac{1}{q}=1,
\end{equation}
by using the duality paring
\[
\bigl(X^p(\Sigma)\bigr)^* \times X^p(\Sigma) \rightarrow \R, \quad
(v,u) \mapsto  \re \int_{\Sigma} \langle v,u\rangle \rho_p(z)\, ds\,dt.
\]
We prefer to use the standard duality pairing
\begin{equation}
\label{stdual}
\bigl(X^p(\Sigma)\bigr)^* \times X^p(\Sigma) \rightarrow \R, \quad
(w,u) \mapsto  \re \int_{\Sigma} \langle w,u\rangle \, ds\,dt.
\end{equation}
With the latter choice, the dual of $X^p(\Sigma)$ is identified with the space of functions $w=\rho_p(z) v$, where $v$ varies in the space (\ref{primodual}). From $1/p+1/q=1$ we get the identity
\begin{eqnarray*}
\|w\|_{X^q}^q = \int_{\Sigma\setminus \mathbb{D}_r} |w|^q \, ds \,dt + \int_{\Sigma \cap \mathbb{D}_r} |w|^q |z|^{q/2-1}\, ds\,dt \\= \int_{\Sigma\setminus \mathbb{D}_r} |v|^q \, ds \,dt + \int_{\Sigma \cap \mathbb{D}_r} |v|^q |z|^{(p/2-1)q} |z|^{q/2-1}\, ds\,dt \\= \int_{\Sigma\setminus \mathbb{D}_r} |v|^q \, ds \,dt + \int_{\Sigma \cap \mathbb{D}_r} |v|^q |z|^{p/2-1} \, ds\,dt
= \int_{\Sigma} |v|^q \rho_p(z) \, ds\,dt,
\end{eqnarray*}
which shows that the standard duality paring (\ref{stdual}) produces the identification
\[
\bigl( X^p(\Sigma) \bigr)^* \cong  X^q(\Sigma), \quad \mbox{for } \frac{1}{p} + \frac{1}{q} = 1.
\]
Therefore, we view the cokernel of $\delbar_{\alpha}: X^{1,p}(\Sigma) \rightarrow X^p(\Sigma)$ as a subspace of $X^q(\Sigma)$. Its elements are a priori less regular at $0$ than the elements of the kernel:

\begin{lem}
\label{cok}
Let $p>1$ and $\alpha \in \R \setminus (\pi/2) \Z$.  If $v \in X^q(\Sigma)$, $1/p+1/q=1$, belongs to the cokernel of
\[
\delbar_{\alpha}: X^{1,p}(\Sigma)
\rightarrow X^p(\Sigma),
\]
then $v$ is smooth on $\Sigma\setminus \{0\}$, it solves the equation $\partial v - \alpha v = 0$ with boundary conditions
\begin{equation}\label{bccok}
\begin{array}{ll}
v(s)\in \R, \quad v(-s) \in i\R & \forall s>0, \\
v(s+i) \in \R & \forall s\in \R,
\end{array}
\end{equation}
and the function $(\mathscr{T} v)(\zeta) = 2 \overline{\zeta} 
v(\zeta^2)$ is smooth on $\mathrm{Cl}(\mathbb{H}^+) \cap \mathbb{D}_1$. In particular, $v(z)=O(|z|^{-1/2})$ and $D v(z)=O(|z|^{-3/2})$
for $z\rightarrow 0$.
\end{lem}

\begin{proof}
Since $v\in X^q(\Sigma)$
annihilates the image of $\delbar_{\alpha}$, there holds
\begin{equation}
\label{ann}
\re \int_{\Sigma} \langle v(z) , \delbar  u(z) +
\alpha u(z) \rangle \, ds\,dt =0,
\end{equation}
for every $u\in
X^{1,p}(\Sigma)$. By letting $u$ vary among all smooth functions in $X^{1,p}(\Sigma)$ which are compactly supported in $\Sigma\setminus \{0\}$, the regularity theory for weak solutions of $\partial$ (the analogue of Theorem \ref{rws}) and a bootstrap argument show that $v$ is smooth on $\Sigma\setminus \{0\}$ and it solves the equation $\partial v - \alpha v=0$ with boundary conditions (\ref{bccok}). It remains to study the regularity of $v$ at $0$.

Set $w(\zeta):= (\mathscr{T} v)(\zeta) = 2 \overline{\zeta} v(\zeta^2)$. By (\ref{Lpt}), the function $w$ is in $L^q(\mathbb{H}^+\cap \mathbb{D}_1)$. Let $\varphi\in C^{\infty}_c (\mathrm{Cl}(\mathbb{H}^+) \cap \mathbb{D}_1)$ be real on $\R^+$ and purely imaginary on $i\R^+$. Then the function $u$ defined by $u(\zeta^2)=\varphi(\zeta)$ belongs to $X^{1,p}(\Sigma)$, and by (\ref{ann}) we have 
\begin{eqnarray*}
0 = \re \int_{\Sigma} \langle v , \delbar  u + \alpha u \rangle \, ds\,dt = 4 \re \int_{\mathbb{H}^+ \cap \mathbb{D}_1} |\zeta^2| \langle \frac{1}{2\overline{\zeta}} w(\zeta), \frac{1}{2\overline{\zeta}} \delbar  
\varphi(\zeta) + \alpha \varphi(\zeta) \rangle \, d\sigma d\tau \\ = \re \int_{\mathbb{H}^+ \cap \mathbb{D}_1}  \langle w(\zeta), \delbar  \varphi(\zeta) + 2\alpha \overline{\zeta} \varphi(\zeta) \rangle \, d\sigma d\tau.
\end{eqnarray*}
The above identity can be rewritten as
\[
\re \int_{\mathbb{H}^+ \cap \mathbb{D}_1} \langle w(\zeta), \delbar  \varphi(\zeta) \rangle \, d\sigma d\tau = - \re \int_{\mathbb{H}^+ \cap \mathbb{D}_1} \langle 2\alpha \zeta w(\zeta), \varphi(\zeta)\rangle \, d\sigma d\tau,
\]
so $w$ is a weak solution of $\partial w = 2\alpha \zeta w$ on $\mathrm{Cl}(\mathbb{H}^+) \cap \mathbb{D}_1$ with real boundary conditions.
Since $w$ is in $L^q( \mathbb{H}^+ \cap \mathbb{D}_1)$, Lemma \ref{woq} implies that $w$ is in $W^{1,q}( \mathbb{H}^+ \cap \mathbb{D}_1)$. In particular, $w$ is square integrable on $\mathbb{H}^+ \cap \mathbb{D}_1$, and so is the function
\[
f(\zeta) := e^{-\alpha \zeta^2/2} w(\zeta).
\]
The function $f$ is anti-holomorphic, it takes real values on $\R^+$ and on $i\R^+$, so by a double Schwarz reflection it can be extended to an anti-holomorphic function on $\mathbb{D}_1\setminus \{0\}$. Since $f$ is square integrable, the singularity $0$ is removable and $f$ is anti-holomorphic on $\mathbb{D}_1$. Therefore
\[
(\mathscr{T} v)(\zeta) = w(\zeta) = e^{\alpha \zeta^2/2} f(\zeta)
\]
is smooth on $\mathrm{Cl}(\mathbb{H}^+)\cap \mathbb{D}_1$, as claimed.
\end{proof}

We can finally prove the following Liouville type result:

\begin{prop}
\label{caspart1}
If $0<\alpha<\pi/2$, the operator
\[
\delbar_{\alpha} : X^{1,p}(\Sigma) \rightarrow X^p(\Sigma)
\]
is an isomorphism, for every $1<p<\infty$. 
\end{prop}

\begin{proof}
By Proposition \ref{sfe} the operator
$\delbar_{\alpha}$ is semi-Fredholm, so it is enough to prove
that its kernel and co-kernel are both $(0)$.

Let $u\in X^{1,p} (\Sigma)$ be
an element of the kernel of $\delbar_{\alpha}$. By Proposition
\ref{dec}, $u(z)$ has exponential decay for $|\re z|\rightarrow
+\infty$ together with all its derivatives. By Lemma \ref{rker}, $u$
is smooth on $\Sigma \setminus \{0\}$, it is continuous at $0$, and
$Du(z)=O(|z|^{-1/2})$ for $z\rightarrow 0$. Then the function $w:=
u^2$ belongs to $W^{1,q}(\Sigma,\C)$ for every $q<4$. Moreover, $w$ is
real on the boundary of $\Sigma$, and it satisfies the equation
\[
\delbar  w + 2 \alpha w = 0.
\]
Since $0<2\alpha<\pi$, $e^{2\alpha i} \R \cap \R = (0)$, so the
assumptions of Theorem \ref{strip} (ii) are satisfied, and the operator
\[
\delbar_{2\alpha} : W^{1,q}_{\R,\R} (\Sigma,\C)
\rightarrow L^q(\Sigma,\C)
\]
is an isomorphism. Therefore $w=0$, hence $u=0$, proving that the
operator $\delbar_{\alpha}$ has vanishing kernel.

Let $v\in X^q(\Sigma)$, $1/p+1/q=1$, be an
element of the cokernel of $\delbar_{\alpha}$.
By Lemma \ref{cok}, $v$ is smooth on $\Sigma \setminus \{0\}$,
$v(s)\in i\R$ for $s<0$, $v(s)\in \R$ for $s>0$, $v$
solves $\partial v - \alpha v=0$, and the function
\begin{equation}
\label{svil}
w(\zeta):= 2\overline{\zeta} v(\zeta^2)
\end{equation}
is smooth in $\mathrm{Cl}(\mathbb{H}^+)\cap \mathbb{D}_1$ and real on the boundary of $\mathbb{H}^+$. In particular,
$v(z) = O(|z|^{-1/2})$ and $Dv(z) = O(|z|^{-3/2})$ for $z\rightarrow
0$. Furthermore, by Proposition \ref{dec}, $v$ and $Dv$
decay exponentially for $|\re z|\rightarrow +\infty$. More precisely,
since the spectrum of the operator $L_{\alpha}$ on $L^2([0,1],\C)$,
\[
\dom L_{\alpha} = W^{1,2}_{\R,\R}([0,1],\C) = \set{u\in
    W^{1,2}([0,1],\C)}{u(0), u(1)\in \R}, \quad L_{\alpha} = i
    \frac{d}{dt} + \alpha,
\]
is $\alpha + \pi \Z$, we have $\min \sigma(L_{\alpha}) \cap
[0,+\infty)=\alpha$, hence
\begin{equation}
\label{ed}
|v(z)|\leq c e^{-\alpha |\re z|}, \quad \mbox{for } |\re z|\geq 1.
\end{equation}

If $w(0)=0$, the function $v$ vanishes at $0$, and
$Dv(z)=O(|z|^{-1})$ for $z\rightarrow
0$, so $v^2$ belongs to $W^{1,q}_{\R,\R}(\Sigma,\C)$ for any $q<2$,
it solves $\partial v^2 - 2\alpha v^2=0$, and
as before we deduce that $v=0$.
Therefore, we can assume that the real number $w(0)$ is not zero.

Consider the function
\[
f:\Sigma \setminus \{0\} \rightarrow \C, \quad
f(z) := e^{-\alpha \overline{z}/2} \overline{v}(z).
\]
Since $\partial v=\alpha v$,
\[
\delbar  f (z) = -\alpha e^{-\alpha \overline{z}/2}
\overline{v}(z) + e^{-\alpha\overline{z}/2} \overline{\partial v} (z)
= e^{-\alpha \overline{z}/2} ( - \alpha \overline{v}(z) +
\overline{\alpha v}(z)) =0,
\]
so $f$ is holomorphic on the interior of $\Sigma$.
Moreover, $f$ is smooth on $\Sigma \setminus \{0\}$, and
\begin{eqnarray}
\label{1a}
f(s) = e^{-\alpha s/2} \overline{v} (s)\in i\R \quad \mbox{for } s<0,
\quad\quad f(s)
= e^{-\alpha s/2} \overline{v}(s)\in \R \quad \mbox{for } s>0, \\
\label{1b} f(s+i) = e^{-\alpha s/2} \overline{v}(s+i) e^{\alpha i/2} \in
e^{\alpha i/2} \R \quad \mbox{for every } s\in \R.
\end{eqnarray}
Denote by $\sqrt{z}$ the determination of the
square root on $\C \setminus \R^-$ such that $\sqrt{z}$ is real and
positive for $z$ real and positive, so that $\overline{\sqrt{z}} =
\sqrt{\overline{z}}$. By (\ref{svil}),
\begin{equation}
\label{2a}
f(z) = -e^{-\alpha \overline{z}/2} \frac{1}{\sqrt{z}}
\overline{w} (\sqrt{z}) =\frac{\overline{w}(0)}{\sqrt{z}}
+ o(|z|^{-1/2}) \quad \mbox{for }
 z\rightarrow 0.
\end{equation}
Finally, by (\ref{ed}),
\begin{equation}
\label{3a}
\lim_{|\re z|\rightarrow +\infty} f(z) = 0.
\end{equation}
We claim that a holomorphic function with the properties listed above
is necessarily zero. By (\ref{2a}), setting $z=\rho
e^{\theta i}$ with $\rho>0$ and $0\leq \theta\leq \pi$,
\[
f(z) = \frac{\overline{w}(0)}{\sqrt{\rho}} e^{-\theta i/2} + o(|\rho|^{-1/2})
\quad \mbox{for } \rho\rightarrow 0.
\]
Since $\overline{w}(0)$ is real and not zero, the above expansion at
0 shows that there exists $\rho>0$ such that
\begin{equation}
\label{4a}
f(z) \in \bigcup_{\theta\in ]-\pi/2-\alpha/4,\alpha/4[}
e^{\theta i} \R, \quad \forall z\in (B_{\rho}(0) \cap \Sigma)
\setminus \{0\}.
\end{equation}
If $f=0$ on $\R+i$, then $f$ is identically zero (by reflection
and by analytic continuation), so we may assume that $f(\R+i) \neq
\{0\}$. By (\ref{1b}) the set $f(\R+i)$ is contained in
$\R e^{\alpha i/2}$. Since $f$ is holomorphic on $\mathrm{Int}(\Sigma)$,
it is open on such a domain, so we can find
$\gamma\in ]\alpha/4,\alpha/2[ \cup ]\alpha/2,3\alpha/4[$
such that $f(\mathrm{Int}(\Sigma)) \cap \R
e^{\gamma i} \neq \{0\}$. By (\ref{3a}) and
(\ref{4a}) there exists $z\in \Sigma \setminus
B_{\rho}(0)$ such that
\begin{equation}
\label{5a}
f(z) \in \R e^{\gamma i}, \quad |f(z)| = \sup |
f(\Sigma\setminus \{0\}) \cap \R e^{\gamma i} | >0.
\end{equation}
By (\ref{1a}) and (\ref{1b}), $z$ belongs to $\mathrm{Int}(\Sigma)$,
but since $f$ is open on $\mathrm{Int}(\Sigma)$ this fact
contradicts (\ref{5a}). Hence $f=0$. Therefore $v$ vanishes on
$\Sigma$, concluding the proof of the invertibility of the operator
$\delbar_{\alpha}$.
\end{proof}

If we change the sign of $\alpha$ and we invert the boundary conditions on $\R$ we still get an isomorphism. Indeed, if we set $v(s,t) := \overline{u}(-s,t)$ we have
\[
\delbar_{-\alpha} v (s,t) =
\delbar  v (s,t) - \alpha v (s,t) = - \overline{
  \delbar  u (-s,t) + \alpha u(-s,t) } = -
\overline{\partial_{\alpha} u (-s,t)},
\]
so the operators
\begin{eqnarray*}
\delbar_{\alpha} : X^{1,p}_{\{0\},((0),\R),(\R)}
(\Sigma,\C) & \rightarrow & X^p_{\{0\}} (\Sigma,\C), \\ 
\delbar_{-\alpha} : X^{1,p}_{\{0\},(\R,(0)),(\R)}
(\Sigma,\C) & \rightarrow & X^p_{\{0\}} (\Sigma,\C)
\end{eqnarray*}
are conjugated. Therefore Proposition \ref{caspart1} implies:

\begin{prop}
\label{caspart2}
If $0<\alpha<\pi/2$, the operator
\[
\delbar_{-\alpha} : X^{1,p}_{\{0\},(\R,(0)),(\R)}
(\Sigma,\C) \rightarrow X^p_{\{0\}} (\Sigma,\C)
\]
is an isomorphism.
\end{prop}

\subsection{Computation of the index}

The computation of the Fredholm index of $\delbar_A$ is
based on the Liouville type results proved in the previous section, together with the following additivity formula:

\begin{prop}
\label{add}
Assume that $A,A_1,A_2 \in C^0(\overline{\R} \times
[0,1],\mathrm{L}(\R^{2n},\R^{2n}))$ satisfy
\begin{equation*}\begin{split}
&A_1(+\infty,t) = A_2 (-\infty,t), \quad A(-\infty,t)=A_1(-\infty,t),\\
&A(+\infty,t) = A_2(+\infty,t),
\quad \forall t\in [0,1].
\end{split}\end{equation*}
Let $\mathscr{V}_1=(V_0,\dots,V_k)$, $\mathscr{V}_2 = (V_k,\dots,
V_{k+h})$,
$\mathscr{V}_1^{\prime}=(V_0^{\prime},\dots,V^{\prime}_{k^{\prime}})$,
$\mathscr{V}_2^{\prime} = (V_{k^{\prime}},\dots,
V_{k^{\prime}+h^{\prime}}^{\prime})$
 be finite ordered sets of linear subspaces of $\R^n$ such that
$V_j$ and $V_{j+1}$, $V_j^{\prime}$ and $V_{j+1}^{\prime}$ are
 partially orthogonal, for every $j$. Set
$\mathscr{V} = (V_0, \dots, V_k, V_{k+1}, \dots, V_{k+h})$ and
 $\mathscr{V}^{\prime}  = (V_0^{\prime}, \dots,
 V_{k^{\prime}}^{\prime},  V^{\prime}_{k^{\prime}+1}, \dots,
 V^{\prime}_{k^{\prime} +h^{\prime}})$.
Assume that $(A_1,\mathscr{V}_1,\mathscr{V}_1^{\prime})$ and
$(A_2,\mathscr{V}_2,\mathscr{V}_2^{\prime})$ satisfy the assumptions of Theorem
\ref{fredholm}. Let $\mathscr{S}_1$ be a set consisting of $k$ points in $\R$ and $k'$ points in $i+\R$, let $\mathscr{S}_2$ be a set consisting of $h$ points in $\R$ and $h'$ points in $i+\R$, and let $\mathscr{S}$ be a set consisting of $k+h$ points in $\R$ and $k'+h'$ points in $i+\R$.
For $p\in ]1,+\infty[$ consider the semi-Fredholm operators
\begin{eqnarray*}
& \delbar_{A_1} :
X^{1,p}_{\mathscr{S}_1, \mathscr{V}_1,
\mathscr{V}_1^{\prime}}(\Sigma,\C^n) \rightarrow
X^p_{\mathscr{S}_1}(\Sigma,\C^n), \quad
\delbar_{A_2} :
X^{1,p}_{\mathscr{S}_2,\mathscr{V}_2,
\mathscr{V}_2^{\prime}}(\Sigma,\C^n) \rightarrow
X^p_{\mathscr{S}_2}(\Sigma,\C^n) & \\ &
\delbar_A :
X^{1,p}_{\mathscr{S},\mathscr{V},\mathscr{V}^{\prime}}
(\Sigma,\C^n) \rightarrow X^p_{\mathscr{S}}(\Sigma,\C^n). &
\end{eqnarray*}
Then
\[
\ind \delbar_A = \ind \delbar_{A_1} + \ind
\delbar_{A_1}.
\]
\end{prop}

The proof is analogous to the proof of Theorem 3.2.12 in \cite{sch95},
and we omit it.
When there are no jumps, that is $\mathscr{S}=
\emptyset$ and $\mathscr{V}=(V)$, $\mathscr{V}^{\prime}=(V^{\prime})$,
Theorem \ref{strip} shows that the index of the operator
\[
\delbar_A : X^{1,p}_{\emptyset,(V),(V^{\prime})} (\Sigma,\C^n) =
W^{1,p}_{N^*V,N^*V^{\prime}} (\Sigma,\C^n)
\rightarrow L^p(\Sigma,\C^n) = X^p_{\emptyset} (\Sigma,\C^n)
\]
is
\[
\ind \delbar_A = \mu(\Phi^- N^*V,N^*V^{\prime}) - \mu(\Phi^+
N^*V,N^*V^{\prime}).
\]
In the general case, Proposition \ref{add} shows that
\begin{equation}
\label{corre}
\begin{split}
\ind (\delbar_A :
X^{1,p}_{\mathscr{S},\mathscr{V},\mathscr{V}^{\prime}}
(\Sigma,\C^n) \rightarrow X^p_{\mathscr{S}}
(\Sigma,\C^n) ) \\ =  \mu(\Phi^- N^*V_0,N^*V_0^{\prime}) - \mu(\Phi^+
N^*V_k,N^*V_{k^{\prime}}^{\prime}) +
c(V_0,\dots,V_k;V_0^{\prime},\dots, V_{k^{\prime}}^{\prime}),
\end{split} \end{equation}
where the correction term $c$ satisfies the additivity formula
\begin{equation}
\label{addf}\begin{split}
&c(V_0,\dots,
V_{k+h};V_0^{\prime},\dots,V_{k^{\prime}+h^{\prime}}^{\prime} )\\
=\; &c(V_0, \dots, V_k; V_0^{\prime},\dots,V_{k^{\prime}}^{\prime}) \,+\,
c(V_k, \dots, V_{k+h};V_{k^{\prime}}^{\prime},\dots,
V_{k^{\prime}+h^{\prime}}^{\prime}).\end{split}
\end{equation}
Since the Maslov index is in general a half-integer, and since
we have not proved that the cokernel of  $\delbar_A$ is finite dimensional, the correction term $c$ takes values in $(1/2)\Z \cup \{-\infty \}$. Actually, the analysis of this section shows that $c$ is always finite, proving that  $\delbar_A$ is Fredholm.

Clearly, we have the following direct sum formula
\begin{equation}
\label{dsf}
\begin{split}
&c(V_0 \oplus W_0, \dots, V_k \oplus W_k ; V_0^{\prime} \oplus
W_0^{\prime}, \dots, V_{k^{\prime}}^{\prime} \oplus W_{k^{\prime}}^{\prime}
) \\ =\, &c(V_0, \dots,V_k;V_0^{\prime},\dots, V_{k^{\prime}}^{\prime}) \,+\,
c(W_0,\dots,W_k ; W_0^{\prime}, \dots,W^{\prime}_{k^{\prime}}).
\end{split} \end{equation}
Note also that the index formula of Theorem \ref{strip} produces a correction term of the form
\begin{equation}
\label{nojumps}
c(\lambda;\nu) = \mu(\lambda,\nu),
\end{equation}
where $\lambda$ and $\nu$ are asymptotically constant paths of
Lagrangian subspaces on $\C^n$. The Liouville type results of the previous section imply that
\begin{equation}
\label{casopart}
c((0),\R^n;\R^n) = -\frac{n}{2} = c(\R^n,(0);\R^n).
\end{equation}
Indeed, by Proposition \ref{caspart1}  the operator 
\[
\delbar_{\alpha I} : X^{1,p}_{\{0\},  ((0),\R^n), (\R^n)} (\Sigma,\C^n) \rightarrow  X^p_{\{0\}}  (\Sigma,\C^n)
\]
is an isomorphism if $0<\alpha<\pi/2$. By (\ref{maslov}), the Maslov index of the path $e^{i\alpha t}\R^n$, $t\in [0,1]$, with respect to $\R^n$ is $-n/2$. On the other hand, the Maslov index of the path $e^{i\alpha t} i \R^n$, $t\in [0,1]$, with respect to $\R^n$ is $0$ because the intersection is $(0)$ for every $t\in [0,1]$. Inserting the information about the Fredholm and the Maslov index in (\ref{corre}), we find
\[
0 = \ind \delbar_{\alpha I} = \frac{n}{2} + c((0),\R^n;(0)),
\]
which implies the first identity in (\ref{casopart}). The second one is proved in the same way by using Proposition \ref{caspart2}.

\begin{lem}
\label{casopart3}
Let $(V_0,V_1,\dots, V_k)$ be a $(k+1)$-tuple of linear
subspaces of $\R^n$, with $V_{j-1}$ and $V_j$ partially orthogonal for
every $j=1,\dots,k$, and let $W$ be a linear subspace of $\R^n$. Then
\begin{equation*}\begin{split}
c(V_0, \dots, V_k ; W)\, &=\, c(W; V_0,\dots, V_k)\\ 
&=\,- \frac{1}{2}
\sum_{j=1}^k \left( \dim V_{j-1} + \dim V_j - 2 \dim V_{j-1} \cap V_j
\right).
\end{split}\end{equation*}
\end{lem}

\begin{proof}
Let us start by considering the case $W=\R^n$.
By the additivity formula (\ref{addf}),
\[
c(V_0, \dots, V_k ;\R^n) = \sum_{j=1}^k c(V_{j-1},V_j; \R^n).
\]
Since $V_{j-1}$ and $V_j$ are partially orthogonal, $\R^n$ has an
orthogonal splitting $\R^n = X_1^j \oplus X_2^j \oplus X_3^j \oplus
X_4^j$ where $V_{j-1}=X_1^j \oplus X_2^j$ and $V_j=X_1^j \oplus
X_3^j$. By the direct sum formula (\ref{dsf}) and by formula (\ref{casopart}),
\begin{eqnarray*}
c(V_{j-1},V_j; \R^n) = c(X_1^j,X_1^j;X_1^j) +  c(X_2^j,(0);X_2^j)
+ c((0),X_3^j;X_3^j) + c((0),(0);X_4^j) \\ = 0 - \frac{1}{2} \dim
X_2^j - \frac{1}{2} \dim X_3^j + 0 = - \frac{1}{2} \dim X_2^j \oplus X_3^j
\end{eqnarray*}
Since
\[
\dim X_2^j \oplus X_3^j = \dim V_{j-1} + \dim V_j - 2 \dim V_{j-1} \cap
V_j,
\]
the formula for $c(V_0,\dots,V_k;\R^n)$ follows.

Now let $\lambda:\R \rightarrow \mathscr{L}(n)$ be a continuous path of
Lagrangian subspaces such that $\lambda(s)=\R^n$ for $s\leq -1$ and
$\lambda(s)=N^*W$ for $s\geq 1$. By an easy generalization of the additivity
formula (\ref{addf}) to the case of non-constant Lagrangian boundary
conditions,
\begin{equation}
\label{above}
c(N^* V_0;\lambda) + c(V_0,\dots,V_k;W) = c(V_0, \dots, V_k;
\R^n) + c(N^* V_k;\lambda).
\end{equation}
By (\ref{nojumps}), $c(N^* V_0;\lambda)=-\mu(\lambda,N^* V_0)$ and
$c(N^* V_k;\lambda)=-\mu(\lambda,N^* V_k)$, so (\ref{above}) leads to
\begin{eqnarray*}
c(V_0, \dots, V^k;W) = c(V_0, \dots, V^k;\R^n) -
(\mu(\lambda,N^* V_k) - \mu(\lambda,N^* V_0)) \\=  c(V_0, \dots, V^k;\R^n) - h(N^* V_0,N^* V_k; \R^n, N^*W),
\end{eqnarray*}
where $h$ is the H\"ormander index. By Lemma \ref{lemhor}, the above
H\"ormander index vanishes, so we get the desired formula for
$c(V_0,\dots,V_k;W)$. The formula for $c(W;V_0,\dots,V_k)$ follows by
considering the change of variable $v(s,t) = \overline{u}(s,1-t)$.
\end{proof}

The additivity formula (\ref{addf}) leads to
\[
c(V_0,\dots,V_k;V_0^{\prime},\dots, V_{k^{\prime}}^{\prime}) =
c(V_0,\dots,V_k; V_0^{\prime}) + c(V_k; V_0^{\prime},\dots,
V_{k^{\prime}}^{\prime}),
\]
and the index formula in the general case follows from (\ref{corre}) and
the above lemma.
This concludes the proof of Theorem \ref{fredholm}.

\subsection{Half-strips with jumping conormal boundary conditions}

This section is devoted to the analogue of Theorem \ref{fredholm} on
the half-strips
\begin{equation*}\begin{split}
&\Sigma^+ := \set{z\in \C}{0\leq \im z \leq 1, \; \re
z\geq 0} \\
\mbox{and} \quad &\Sigma^- := \set{z\in \C}{0\leq \im z
  \leq 1, \; \re z\leq 0} .
\end{split}\end{equation*}
In the first case, we fix the following data. Let $k,k'\geq 0$ be integers, let
\[
0 = s_0 < s_1 < \dots < s_k < s_{k+1} = +\infty, \quad
0 = s_0' < s_1' < \dots < s_{k'}' < s_{k'+1}' = +\infty,
\]
be real numbers, and let $W$, $V_0,\dots,V_k$, $V_0', \dots ,V_{k'}'$
be linear subspaces of $\R^n$ such that $V_{j-1}$ and $V_j$,
$V_{j-1}'$ and $V_j'$, $W$ and $V_0$, $W$ and $V_0'$, are partially orthogonal.
We denote by $\mathscr{V}$ the $(k+1)$-tuple $(V_0,\dots,V_k)$, by
$\mathscr{V}'$ the $(k'+1)$-tuple $(V_0',\dots,V_{k'}')$, and
by $\mathscr{S}$ the set
$\{s_1,\dots,s_k,s_1'+i, \dots, s_{k'}'+i\}$.
The $X^p$ and $X^{1,p}$ norms on $\Sigma^+$ are defined as in Section \ref{sjlbc}, and so are the spaces $X^p_{\mathscr{S}}(\Sigma^+,\C^n)$ and $X^{1,p}_{\mathscr{S}}(\Sigma^+,\C^n)$.
Let $X^{1,p}_{\mathscr{S},W,\mathscr{V},\mathscr{V}'}
(\Sigma^+,\C^n)$ be the completion
of the space of maps $u\in C^{\infty}_{\mathscr{S},c}(\Sigma^+,\C^n)$ satisfying the boundary conditions
\begin{equation*}\begin{split}
&u(it) \in N^* W \;\; \forall t\in [0,1],\quad
u(s) \in N^* V_j \;\; \forall s\in [s_j,s_{j+1}],\\
&u(s+i) \in N^* V_j' \;\; \forall s\in [s_j',s_{j+1}'],
\end{split}\end{equation*}
with respect to the norm $\|u\|_{X^{1,p}(\Sigma^+)}$.

Let $A\in C^0([0,+\infty]\times [0,1],\mathrm{L}(\R^{2n},\R^{2n}))$ be such that $A(+\infty,t)$ is symmetric for every $t\in [0,1]$, and denote by
$\Phi^+:[0,1]\rightarrow \mathrm{Sp}(2n)$ the solutions of the linear
Hamiltonian system
\[
\frac{d}{dt} \Phi^+(t) = i A(+\infty,t) \Phi^+(t), \quad \Phi^+(0)=I.
\]
Then we have:

\begin{thm}
\label{fredhs+}
Assume that $\Phi^+(1) N^*V_k \cap N^* V_{k'}' = (0)$. Then the
$\R$-linear bounded operator
\[
\delbar_A :
X^{1,p}_{\mathscr{S},W,\mathscr{V},\mathscr{V}'}
(\Sigma^+,\C^n) \rightarrow X^p_{\mathscr{S}} (\Sigma^+,\C^n), \quad
\delbar_A u = \delbar  u + Au,
\]
is Fredholm of index
\begin{equation}
\label{inds+}
\begin{split}
\ind \delbar_A \,=\, &\frac{n}{2} -\mu(\Phi^+ N^*V_k, N^*
V_{k'}')\\ &- \frac{1}{2} ( \dim V_0 + \dim W - 2 \dim V_0 \cap W) \\
&-\frac{1}{2} ( \dim V_0' + \dim W - 2 \dim V_0' \cap W)\\
&-\frac{1}{2} \sum_{j=1}^k (\dim V_{j-1} + \dim V_j - 2 \dim V_{j-1}
\cap V_j)  \\ 
&-\frac{1}{2} \sum_{j=1}^{k'} (\dim V_{j-1}' + \dim V_j' - 2
\dim V_{j-1}' \cap V_j').
\end{split} \end{equation}
\end{thm}

\begin{proof}
The proof of the fact that $\delbar_A$ is
semi-Fredholm is analogous to the case of the full strip, treated in
Section \ref{tfps}. It remains to compute the index. By an
additivity formula analogous to (\ref{addf}), it is enough to prove
(\ref{inds+}) in the case with no jumps, that is $k=k'=0$,
$\mathscr{V}=(V_0)$, $\mathscr{V}' = (V_0')$. In this case, we
have a formula of the type
\[
\ind \delbar_A = - \mu(\Phi^+ N^* V_0,N^* V_0') +
c(W;V_0;V_0'),
\]
and we have to determine the correction term $c$.

Assume $W=(0)$, so that $N^* W=i\R^n$. Let us compute the correction
term $c$ when $V_0$ and $V_0'$ are either $(0)$ or $\R^n$.
We can choose the map $A$ to be the
constant map $A(s,t)=\alpha I$, for $\alpha\in ]0,\pi/2[$,
so that $\Phi^+(t)=e^{i\alpha t}$. The Kernel and co-kernel of
$\delbar_{\alpha I}$ are easy to determine explicitly, by separating
the variables in the corresponding boundary value PDE's:

\begin{enumerate}

\item If $V_0=V_0'=\R^n$, the kernel and co-kernel of
$\delbar_{\alpha I}$ are both $(0)$. Since $\mu(e^{i\alpha
  t} \R^n,\R^n) = -n/2$, we have $c((0);\R^n;\R^n)= -n/2$.

\item If $V_0=V_0'=(0)$, the kernel of $\delbar_{\alpha I}$
  is $i\R^n e^{-\alpha s}$, while its co-kernel
  is $(0)$. Since $\mu(e^{i\alpha
  t} i\R^n,i\R^n) = -n/2$, we have $c((0);(0);(0))= n/2$.

\item If either $V_0=\R^n$ and $V_0'=(0)$, or $V_0=(0)$ and
  $V_0'=\R^n$, the kernel and co-kernel of
$\delbar_{\alpha I}$ are both $(0)$. Since $\mu(e^{i\alpha
  t} \R^n,(0)) = \mu(e^{i\alpha t} (0),\R^n) = 0$, we have
$c((0);\R^n;(0)) = c((0);(0);\R^n) = 0$.

\end{enumerate}

Now let $W$, $V_0$, and $V_0'$ be arbitrary (with $W$ partially
orthogonal to both $V_0$ and $V_0'$).
Let $U\in \mathrm{U}(n)$ be such that $U N^*W=i\R^n$.
Then $UN^* V_0=N^*W_0$ and  $UN^* V_0'=N^*W_0'$, where
\[
W_0 = (V_0\cap W)^{\perp} \cap (V_0+W), \quad
W_0' = (V_0'\cap W)^{\perp} \cap (V_0'+W).
\]
By using the change of variable $v=Uu$, we find
\begin{equation}
\label{redu}
c(W;V_0;V_0') = c((0);W_0;W_0'),
\end{equation}
and we are reduced to compute the latter quantity.
By an easy homotopy argument, using the fact that the Fredholm index
is locally constant in the operator norm topology, we can assume that
$W_0$ and $W_0'$ are partially orthogonal. Then $\R^n$ has an
orthogonal splitting $\R^n = X_1 \oplus X_2 \oplus X_3 \oplus X_4$,
where
\[
W_0 = X_1 \oplus X_2, \quad W_0' = X_1 \oplus X_3,
\]
from which
\[
N^* W_0 = X_1 \oplus X_2 \oplus iX_3 \oplus iX_4, \quad
N^* W_0' = X_1 \oplus iX_2 \oplus X_3 \oplus iX_4.
\]
Then the operator $\delbar_{\alpha I}$ decomposes as the direct sum
of four operators, whose index is computed in cases (i), (ii), and
(iii) above. Indeed,
\begin{equation*}\begin{split}
c((0);W_0;W_0') \,=\,& \frac{1}{2} \dim X_4 - \frac{1}{2}
\dim X_1 \\ =\,& \frac{1}{2} \codim (W_0+W_0') - \frac{1}{2} \dim W_0 \cap
W_0'\\ =\,& \frac{1}{2} (n- \dim W_0 - \dim W_0').
\end{split}\end{equation*}
Since
\begin{eqnarray*}
\dim W_0 = \dim (V_0 + W) - \dim V_0 \cap W = \dim V_0 + \dim W - 2
\dim W_0 \cap W, \\
\dim W_0' = \dim (V_0' + W) - \dim V_0' \cap W = \dim V_0' + \dim W - 2
\dim W_0' \cap W,
\end{eqnarray*}
we find
\begin{equation*}\begin{split}
c((0);W_0;W_0') \,=\, &\frac{n}{2} - \frac{1}{2} (\dim V_0 + \dim W - 2
\dim W_0 \cap W)\\ &- \frac{1}{2} (\dim V_0' + \dim W - 2
\dim W_0' \cap W).
\end{split}\end{equation*}
Together with (\ref{redu}), this proves formula (\ref{inds+}).
\end{proof}

We conclude this section by considering the case of the left half-strip
$\Sigma^-$. Let $k,k'\geq 0$, $\mathscr{V}=(V_0,\dots,V_k)$,
and $\mathscr{V}'=(V_0',\dots,V_{k'}')$ be as above. Let
\[
-\infty = s_{k+1} < s_k < \dots < s_1 < s_0 = 0, \quad
-\infty = s_{k'+1}' < s_{k'}' < \dots < s_1' < s_0' = 0,
\]
be real numbers, and set $\mathscr{S}=\{s_1,\dots,s_k,s_1'+i,\dots,s_{k'}'+i\}$.

Let $X^{1,p}_{\mathscr{S},W,\mathscr{V},\mathscr{V}'}
(\Sigma^-,\C^n)$ be the completion
of the space of maps $u\in C^{\infty}_{\mathscr{S},c}(\Sigma^-,\C^n)$ satisfying the boundary conditions
\begin{equation*}\begin{split}
&u(it) \in N^* W \;\; \forall t\in [0,1],\quad
u(s) \in N^* V_j \;\; \forall s\in [s_{j+1},s_j],\\
&u(s+i) \in N^* V_j' \;\; \forall s\in [s_{j+1}',s_j'],
\end{split}\end{equation*}
with respect to the norm $\|u\|_{X^{1,p}(\Sigma^-)}$. 

Let $A\in C^0([-\infty,0]\times [0,1],\mathrm{L}(\R^{2n},\R^{2n}))$ be such
that $A(-\infty,t)$ is symmetric for every $t\in [0,1]$, and denote by
$\Phi^-:[0,1]\rightarrow \mathrm{Sp}(2n)$ the solutions of the linear
Hamiltonian system
\[
\frac{d}{dt} \Phi^-(t) = i A(-\infty,t) \Phi^-(t), \quad \Phi^-(0)=I.
\]
Then we have:

\begin{thm}
\label{fredhs-}
Assume that $\Phi^-(1) N^*V_k \cap N^* V_{k'}' = (0)$. Then the
$\R$-linear operator
\begin{equation}
\label{oper}
\delbar_A :
X^{1,p}_{\mathscr{S},W,\mathscr{V},\mathscr{V}'}
(\Sigma^-,\C^n) \rightarrow X^p_{\mathscr{S}} (\Sigma^-,\C^n), \quad
\delbar_A u = \delbar  u + Au,
\end{equation}
is bounded and Fredholm of index
\begin{equation}
\label{inds-}
\begin{split}
\ind \delbar_A \,=\, &\frac{n}{2} + \mu(\Phi^- N^*V_k, N^*
V_{k'}') - \frac{1}{2} ( \dim V_0 + \dim W - 2 \dim V_0 \cap W) \\
&-\frac{1}{2} ( \dim V_0' + \dim W - 2 \dim V_0' \cap W)\\
&-\frac{1}{2} \sum_{j=1}^k (\dim V_{j-1} + \dim V_j - 2 \dim V_{j-1}
\cap V_j)  \\ 
&-\frac{1}{2} \sum_{j=1}^{k'} (\dim V_{j-1}' + \dim V_j' - 2
\dim V_{j-1}' \cap V_j').
\end{split} \end{equation}
\end{thm}

Indeed, notice that if $u(s,t)=\overline{v}(-s,t)$, then
\[
- (\delbar  u(s,t) + A(s,t) u(s,t)) = C
  (\delbar  v(-s,t) - CA(s,t)C v(-s,t)),
\]
where $C$ is denotes complex conjugation. Then the operator (\ref{oper})
is obtained from the operator
\[
\delbar_B :
X^{1,p}_{-\mathscr{S},W,\mathscr{V},\mathscr{V}'}
(\Sigma^+,\C^n) \rightarrow
X^p_{-\mathscr{S}} (\Sigma^+,\C^n), \quad
\delbar_B v = \delbar  v + Bv,
\]
where $B(s,t) = - C A(-s,t) C$,
by left and right multiplication by isomorphisms. In particular, the
indices are the same. Then Theorem \ref{fredhs-} follows from Theorem
\ref{fredhs+}, taking into account the fact that the solution $\Phi^+$
of
\[
\frac{d}{dt} \Phi^+(t) = i B(+\infty,t) \Phi^+(t), \quad \Phi^+(0)=I,
\]
is $\Phi^+(t) = C \Phi^-(t) C$, so that
\[
\mu(\Phi^+ N^* V_k,N^* V_{k'}') = \mu(C\Phi^- C N^* V_k,N^* V_{k'}') =
- \mu(\Phi^- N^* V_k,N^* V_{k'}').
\]

\subsection{Nonlocal boundary conditions}
\label{nonlocsec}

It is useful to dispose of versions of Theorems \ref{fredholm},
\ref{fredhs+}, and \ref{fredhs-}, involving nonlocal boundary conditions.
In the case of the full strip $\Sigma$, let us fix the following data. Let $k\geq 0$ be an integer, let
\[
-\infty = s_0 < s_1 < \dots < s_k < s_{k+1} = +\infty
\]
be real numbers, and set $\mathscr{S}:=\{s_1,\dots,s_k, s_1+i,\dots, s_k + i\}$. Let $W_0,W_1,\dots,W_k$ be linear subspaces of $\R^n \times \R^n$ such that $W_{j-1}$ and $W_j$ are partially orthogonal, for $j=1,\dots,k$, and set $\mathscr{W}=(W_0,\dots,W_k)$.

The space $X^{1,p}_{\mathscr{S},\mathscr{W}}(\Sigma,\C^n)$ is defined
as the completion of the space of all $u\in C^{\infty}_{\mathscr{S},c}(\Sigma,\C^n)$ such that
\[
(u(s), \overline{u}(s+i)) \in N^*
    W_j, \quad \forall s\in [s_j,s_{j+1}], \quad j=0,\dots,k,
\]
with respect to the norm $\|u\|_{X^{1,p}(\Sigma)}$.

Let $A\in C^0(\overline{\R} \times [0,1],\mathrm{L}(\R^{2n},\R^{2n}))$ be such that $A(\pm \infty,t) \in \mathrm{Sym}(2n,\R)$ for every $t\in [0,1]$,
and define the symplectic paths $\Phi^+,\Phi^-: [0,1] \rightarrow
\mathrm{Sp}(2n)$ as the solutions of the linear Hamiltonian systems
\[
\frac{d}{dt} \Phi^{\pm} (t) = i A(\pm \infty,t) \Phi^{\pm}(t), \quad
\Phi^{\pm}(0) =I.
\]
Denote by $C$ the complex conjugation, and recall from Section \ref{tmi}
that $\Phi\in \mathrm{L}(\R^{2n},\R^{2n})$ is symplectic if and only if $\graf \Phi C$ is a Lagrangian subspace of $(\R^{2n}\times \R^{2n}, \omega_0 \times \omega_0)$. Then we have the following:

\begin{thm}
\label{fredholmnl}
Assume that $\graf C \Phi^-(1) \cap N^*W_0 = (0)$ and $\graf C \Phi^+(1)
\cap N^* W_k=(0)$. Then for every $p\in ]1,+\infty[$
the $\R$-linear operator
\[
\delbar_A : X^{1,p}_{\mathscr{S},\mathscr{W}}
(\Sigma,\C^n) \rightarrow X^p_{\mathscr{S}} (\Sigma,\C^n), \quad u
\mapsto \overline{\partial} u + A u,
\]
is bounded and Fredholm of index
\begin{equation*}\begin{split}
\ind \delbar_A \,=\, &\mu(\graf \Phi^-C, N^* W_0) - \mu(\graf \Phi^+C,N^*
W_k) \\ 
&-\frac{1}{2} \sum_{j=1}^k (\dim W_{j-1} + \dim W_j - 2 \dim W_{j-1}
\cap W_j).
\end{split}\end{equation*}
\end{thm}

\begin{proof}
Given $u:\Sigma \rightarrow \C^n$ define $\tilde{u}: \Sigma
\rightarrow \C^{2n}$ by
\[
\tilde{u} (z) := (u(z/2), \overline{u}(\overline{z}/2+i)).
\]
The map $u \mapsto \tilde{u}$ determines a linear isomorphism
\[
F: X^{1,p}_{\mathscr{S},\mathscr{W}} (\Sigma,\C^n)
\stackrel{\cong}{\longrightarrow}
X^{1,p}_{\mathscr{S}',\mathscr{W},\mathscr{W}'} (\Sigma,\C^{2n}),
\]
where $\mathscr{S}'=\{2s_1,\dots,2s_k, 2s_1+i,\dots,2 s_k+i\}$,
$\mathscr{W}'$ is the $(k+1)$-tuple $(\Delta_{\R^n},\dots,
\Delta_{\R^n})$, and we have used the identity
\[
N^* \Delta_{\R^n} = \graf C = \set{(w,\overline{w})}{w\in \C^n}.
\]
The map $v \mapsto \tilde{v}/2$ determines an isomorphism
\[
G: X^p_{\mathscr{S}} (\Sigma,\C^n)
\stackrel{\cong}{\longrightarrow} X^p_{\mathscr{S}'}
(\Sigma,\C^{2n}).
\]
The composition $G \circ \delbar_A \circ F^{-1}$ is the
operator
\[
\delbar_{\tilde{A}} :  X^{1,p}_{\mathscr{S}',
\mathscr{W},\mathscr{W}'} (\Sigma,\C^{2n}) \longrightarrow
  X^p_{\mathscr{S}'}(\Sigma,\C^{2n}), \quad u \mapsto
  \delbar  u + \tilde{A} u,
\]
where
\[
\tilde{A}(z) := \frac{1}{2} \left( A(z/2) \oplus C A(\overline{z}/2+i)
C \right).
\]
Since
\[
\tilde{A}(\pm \infty,t) = \frac{1}{2} \left( A(\pm \infty,t/2) \oplus
CA(\pm \infty,1-t/2)C \right),
\]
we easily see that the solutions $\tilde{\Phi}^{\pm}$ of
\[
\frac{d}{dt} \tilde{\Phi}^{\pm} (t) = i \tilde{A}(\pm \infty,t)
\tilde{\Phi}^{\pm} (t), \quad \tilde{\Phi}^{\pm} (0) = I,
\]
are given by
\[
\tilde{\Phi}^{\pm} (t) = \Phi^{\pm}(t/2) \oplus C \Phi^{\pm}(1-t/2)
\Phi(1)^{-1} C.
\]
The above formula implies
\begin{equation}
\label{grafi}
\tilde{\Phi}^{\pm}(t)^{-1} N^* \Delta_{\R^n} = \graf C \Phi^{\pm}(1)
\Phi^{\pm}(1-t/2)^{-1} \Phi^{\pm}(t/2).
\end{equation}
For $t=1$ we get
\begin{eqnarray*}
\tilde{\Phi}^-(1) N^* W_0 \cap N^* \Delta_{\R^n} = \tilde{\Phi}^-(1)
[ N^* W_0 \cap \graf C \Phi^-(1)] = (0), \\
\tilde{\Phi}^+(1) N^* W_k \cap N^* \Delta_{\R^n} = \tilde{\Phi}^+(1)
[ N^* W_k \cap \graf C \Phi^+(1)] = (0),
\end{eqnarray*}
so the transversality hypotheses of Theorem \ref{fredholm} are
fulfilled. By this theorem, the operator $\delbar_A = G^{-1}
\circ \delbar_{\tilde{A}} \circ F$ is Fredholm of index
\begin{equation}
\label{lafo}
\begin{split}
\ind \delbar_A = \ind \delbar_{\tilde{A}} =
\mu(\tilde{\Phi}^- N^* W_0,N^* \Delta_{\R^n}) - \mu(\tilde{\Phi}^+ N^*
W_k,N^* \Delta_{\R^n})
\\ - \frac{1}{2}
\sum_{j=1}^k ( \dim W_{j-1} + \dim W_j - 2 \dim W_{j-1} \cap W_j).
\end{split} \end{equation}
The symplectic paths $t\mapsto \Phi^{\pm}(1) \Phi^{\pm} (1-t/2)^{-1}
\Phi^{\pm} (t/2)$ and $t\mapsto \Phi^{\pm}(t)$ are homotopic by means
of the symplectic homotopy
\[
(\lambda,t) \mapsto \Phi^{\pm}(1) \, \Phi^{\pm} \left(\frac{1+\lambda}{2} -
  \frac{1-\lambda}{2} t \right)^{-1} \Phi^{\pm} \left(
  \frac{1+\lambda}{2} t \right),
\]
which fixes the end-points $I$ and $\Phi^{\pm}(1)$. By the symplectic
invariance and the homotopy invariance of the Maslov index we deduce
from (\ref{grafi}) that
\begin{equation}
\label{prim}
\begin{split}
\mu(\tilde{\Phi}^- N^* W_0,N^* \Delta_{\R^n}) =\, &\mu (N^* W_0,
\tilde{\Phi}^-(\cdot)^{-1}N^* \Delta_{\R^n}) \\ =\,& \mu (N^* W_0,
\graf C \Phi^-(1) \Phi^-(1-\cdot/2)^{-1} \Phi^-(\cdot/2))\\
=\,& \mu (N^*W_0, \graf C \Phi^-) = \mu (\graf \Phi^- C, N^* W_0),
\end{split}
\end{equation}
where the lest equality is obtained by applying the anti-symplectic involution $C$ to both arguments.
Similarly,
\begin{equation}
\label{second}
\mu(\tilde{\Phi}^+ N^* W_k,N^* \Delta_{\R^n}) = \mu (\graf 
\Phi^+C, N^*W_k).
\end{equation}
The conclusion follows from (\ref{lafo}), (\ref{prim}), and
(\ref{second}).
\end{proof}

In the case of the right half-strip $\Sigma^+$, we fix an integer
$k\geq 0$, real numbers
\[
0=s_0 < s_1 < \dots < s_k < s_{k+1}=+\infty,
\]
a linear subspace $V_0\subset \R^n$ and a $(k+1)$-tuple
$\mathscr{W}=(W_0,\dots,W_k)$ of linear subspaces of $\R^n \times
\R^n$, such that $W_0$ and $V_0 \times V_0$ are partially orthogonal,
and so are $W_{j-1}$ and $W_j$, for every $j=1,\dots,k$. Set
$\mathscr{S}=\{s_1,\dots,s_k,s_1+i,\dots,s_k+i\}$, and let
$X^{1,p}_{\mathscr{S},V_0,\mathscr{W}}(\Sigma^+,\C^n)$ be the completion
of the space of maps $u\in C^{\infty}_{\mathscr{S},c}(\Sigma^+,\C^n)$ such that
\[
u(it) \in V_0 \;\; \forall t\in [0,1], \quad
(u(s), \overline{u}(s+i)) \in N^*
    W_j, \quad \forall s\in [s_j,s_{j+1}], \quad j=0,\dots,k,
\]
with respect to the norm $\|u\|_{X^{1,p}(\Sigma^+)}$.

Let $A\in C^0([0,+\infty] \times [0,1],\mathrm{L}(\R^{2n},\R^{2n}))$ be such that $A(+ \infty,t) \in \mathrm{Sym}(2n,\R)$ for every $t\in [0,1]$,
and let $\Phi^+: [0,1] \rightarrow
\mathrm{Sp}(2n)$ be the solution of the linear Hamiltonian systems
\[
\frac{d}{dt} \Phi^+ (t) = i A(+ \infty,t) \Phi^+(t), \quad
\Phi^+(0) =I.
\]
Then we have:

\begin{thm}
\label{fredhs+nl}
Assume that $\graf C \Phi^+(1)
\cap N^* W_k=(0)$. Then for every $p\in ]1,+\infty[$
the $\R$-linear operator
\[
\delbar_A : X^{1,p}_{\mathscr{S},V_0,\mathscr{W}}
(\Sigma^+,\C^n) \rightarrow X^p_{\mathscr{S}} (\Sigma^+,\C^n), \quad u
\mapsto \overline{\partial} u + A u,
\]
is bounded and Fredholm of index
\begin{equation*}\begin{split}
\ind \delbar_A \,=\, &\frac{n}{2} - \mu(\graf 
\Phi^+C, N^* W_k) - \frac{1}{2} (\dim W_0 + 2 \dim V_0 - 2 \dim W_0 \cap
(V_0 \times V_0)) \\ &- \frac{1}{2} \sum_{j=1}^k (\dim W_{j-1} +
\dim W_j - 2 \dim W_{j-1} \cap W_j).
\end{split}\end{equation*}
\end{thm}

\begin{proof}
By the same argument used in the proof of Theorem \ref{fredholmnl},
the operator $\delbar_A$ is Fredholm and has the same
index as the operator
\[
\delbar_{\tilde{A}} : X^{1,p}_{\mathscr{S}',
  V_0 \times V_0, \mathscr{W},
  \mathscr{W}'}(\Sigma^+,\C^{2n}) \rightarrow X^p_{\mathscr{S}'} (\Sigma^+,\C^{2n}), \quad u \mapsto
  \delbar  u + \tilde{A} u,
\]
where $\mathscr{S}':= \{2s_1,\dots,2s_k,2s_1+i,\dots,2s_k+i\}$, $\mathscr{W}'$ is the $(k+1)$-tuple
$(\Delta_{\R^n},\dots,\Delta_{\R^n})$, and
\[
\tilde{A}(z) := \frac{1}{2} (A(z/2) \oplus C A (\overline{z}/2+i)C).
\]
By Theorem \ref{fredhs+} and by (\ref{second}), the index of this
operator is
\begin{equation*}\begin{split}
\ind \delbar_{\tilde{A}} \,=\, &n -  \mu(\graf 
\Phi^+ C, N^* W_k)\\
 & - \frac{1}{2} (\dim \Delta_{\R^n} + \dim
V_0 \times V_0 - 2 \dim \Delta_{\R^n} \cap (V_0 \times V_0))\\
 &- \frac{1}{2} (\dim W_0 + \dim V_0\times V_0 - 2 \dim W_0 \cap
(V_0 \times V_0)) \\ 
 &- \frac{1}{2} \sum_{j=1}^k (\dim W_{j-1} +
\dim W_j - 2 \dim W_{j-1} \cap W_j) \\ 
=\,& n - \mu(\graf 
\Phi^+ C, N^* W_k) - \frac{n}{2}\\
& - \frac{1}{2} (\dim W_0 + 2 \dim V_0 - 2
\dim W_0 \cap (V_0 \times V_0)) \\  
&- \frac{1}{2} \sum_{j=1}^k
(\dim W_{j-1} + \dim W_j - 2 \dim W_{j-1} \cap W_j).
\end{split}\end{equation*}
The desired formula follows.
\end{proof}

In the case of the left half-strip $\Sigma^-$, let $k$, $V_0$,
$\mathscr{W}$ be as above, and
let $\mathscr{S}=\{s_1,\dots,s_k,s_1+i,\dots,s_k+i\}$ with
\[
0=s_0 > s_1 > \dots > s_k > s_{k+1}=-\infty.
\]
Let $X^{1,p}_{\mathscr{S},V_0,\mathscr{W}}(\Sigma^-,\C^n)$ the completion
of the space of all maps $u\in C^{\infty}_{\mathscr{S},c}(\Sigma^-,\C^n)$ such that
\[
u(it) \in V_0 \;\; \forall t\in [0,1], \quad
(u(s), \overline{u}(s+i)) \in N^*
    W_j, \quad \forall s\in [s_{j+1},s_j], \quad j=0,\dots,k,
\]
with respect to the norm $\|u\|_{X^{1,p}(\Sigma^-)}$.

Let $A\in C^0([-\infty,0] \times [0,1],\mathrm{L}(\R^{2n},\R^{2n}))$ be such
that $A(- \infty,t)$ is symmetric for every $t\in [0,1]$,
and let $\Phi^-: [0,1] \rightarrow
\mathrm{Sp}(2n)$ be the solution of the linear Hamiltonian systems
\[
\frac{d}{dt} \Phi^- (t) = i A(- \infty,t) \Phi^+(t), \quad
\Phi^-(0) =I.
\]
Then we have:

\begin{thm}
\label{fredhs-nl}
Assume that $\graf C \Phi^-(1)
\cap N^* W_k=(0)$. Then for every $p\in ]1,+\infty[$
the $\R$-linear operator
\[
\delbar_A : X^{1,p}_{\mathscr{S},V_0,\mathscr{W}}
(\Sigma^-,\C^n) \rightarrow X^p_{\mathscr{S}} (\Sigma^-,\C^n), \quad u
\mapsto \overline{\partial} u + A u,
\]
is bounded and Fredholm of index
\begin{equation*}\begin{split}
\ind \delbar_A \,=\, & \frac{n}{2} + \mu(\graf 
\Phi^-C, N^* W_k)\\ &- \frac{1}{2} (\dim W_0 + 2 \dim V_0 - 2 \dim W_0 \cap
(V_0 \times V_0)) \\ &- \frac{1}{2} \sum_{j=1}^k (\dim W_{j-1} +
\dim W_j - 2 \dim W_{j-1} \cap W_j).
\end{split}\end{equation*}
\end{thm}

\subsection{Coherent orientations}
\label{cosec}

As noticed in \cite[Section 1.4]{as06}, the problem of giving coherent orientations for the spaces of maps arising in Floer homology on cotangent bundles is somehow simpler than in the case of a general symplectic manifolds, treated in \cite{fh93} for periodic orbits and in \cite{fooo09} for more general Lagrangian boundary conditions. This fact remains true if we deal with Cauchy-Riemann type operators on strips and half-strips with jumping conormal boundary conditions. We briefly discuss this issue in the general case of nonlocal boundary conditions on the strip, the case of the half-strip being similar (see \cite[Section 3.2]{as06}).

We recall that the space $\mathrm{Fred}(E,F)$ of Fredholm linear operators from the real Banach space $E$ to the real Banach space $F$ is the base space of a smooth real non-trivial line-bundle $\det(\mathrm{Fred}(E,F))$, with fibers
\[
\det (A) := \Lambda^{\max}(\ker A) \otimes ( \Lambda^{\max} (\coker A) )^*, \quad \forall A\in \mathrm{Fred}(E,F),
\]
where $\Lambda^{\max}(V)$ denotes the component of top degree in the exterior algebra of the finite-dimensional vector space $V$ (see \cite{qui85}). 

Let us recall the setting from Section \ref{nonlocsec}.
We fix the data $k\geq 0$, $\mathscr{S}=\{ s_1,\dots,s_k,s_1+i,\dots, s_k+i\}$, with $s_1 < \dots < s_k$,
and $\mathscr{W}=(W_0,\dots,W_k)$, where $W_0,\dots,W_k$ are linear subspaces of $\R^n \times \R^n$, such that $W_{j-1}$ is partially orthogonal to $W_j$, for $j=1,\dots,k$. Let $A^{\pm}:[0,1] \rightarrow \mathrm{Sym}(\R^n)$ be continuous paths of symmetric matrices such that the linear problems
\[
\left\{ \begin{array}{l} w'(t) = i A^-(t) w(t), \\ (w(0), C w(1)) \in N^* W_0, \end{array} \right. \quad \left\{ \begin{array}{l} w'(t) = i A^+(t) w(t), \\ (w(0), C w(1)) \in N^* W_k, \end{array} \right.
\]
have only the trivial solution $w=0$. Such paths are referred to as {\em non-degenerate} paths (with respect to $W_0$ and $W_k$, respectively). Fix some $p>1$, and let $\mathscr{D}_{\mathscr{S},\mathscr{W}} (A^-,A^+)$ be the space of operators of the form
\[
\overline{\partial}_A : X^{1,p}_{\mathscr{S},\mathscr{W}} (\Sigma,\C^n) \rightarrow X^p_{\mathscr{S}}(\Sigma,\C^n) , \quad u \mapsto \overline{\partial} u + A u,
\]
where $A\in C^0 (\overline{\R}\times [0,1],\mathrm{L} (\R^{2n},\R^{2n})) $ is such that $A(\pm \infty,t) = A^{\pm}(t)$ for every $t\in [0,1]$.  By Theorem \ref{fredholmnl}, 
$$\mathscr{D}_{\mathscr{S},\mathscr{W}} (A^-,A^+)\,\subset\,\mathrm{Fred}( X^{1,p}_{\mathscr{S},\mathscr{W}}(\Sigma,\C^n), X^p_{\mathscr{S}} (\Sigma,\C^n) )\,.
$$ 
It is actually a convex subset, so the restriction of the determinant bundle to  $\mathscr{D}_{\mathscr{S},\mathscr{W}} (A^-,A^+)$ -- that we denote by $\det(\mathscr{D}_{\mathscr{S},\mathscr{W}} (A^-,A^+))$ -- is trivial. 

Let $\mathfrak{S}$ be the family of all subsets of $\Sigma$ consisting of exactly $k$ pairs of opposite boundary points. It is a $k$-dimensional manifold, diffeomorphic to an open subsets of $\R^k$. An orientation of $\det(\mathscr{D}_{\mathscr{S},\mathscr{W}} (A^-,A^+))$ for a given $\mathscr{S}$ in $\mathfrak{S}$ uniquely determines an orientation for all choices of $\mathscr{S}'\in \mathfrak{S}$. Indeed, the disjoint unions
\[
\bigsqcup_{\mathscr{S}\in \mathfrak{S}}  X^{1,p}_{\mathscr{S},\mathscr{W}} (\Sigma,\C^n), \quad \bigsqcup_{\mathscr{S}\in \mathfrak{S}}  X^p_{\mathscr{S}} (\Sigma,\C^n),
\]
define locally trivial Banach bundles over $\mathfrak{S}$, and the operators $\overline{\partial}_A$ define a Fredholm bundle-morphism between them. Since $\mathfrak{S}$ is connected and simply connected, an orientation of the determinant space of this operator between the fibers of a given point $\mathscr{S}$ induces an orientation of the determinant spaces of the operators over each $\mathscr{S}'\in \mathfrak{S}$. 

The space of all Fredholm bundle-morphisms between the above Banach bundles induced by operators of the form $\overline{\partial}_A$ with fixed asymptotic paths $A^-$ and $A^+$
is denoted by $\mathscr{D}_{\mathscr{W}}(A^-,A^+)$.
An orientation of the determinant bundle over this space of Fredholm bundle-morphisms is denoted by $o_{\mathscr{W}}(A^-,A^+)$.

Let $\mathscr{W} = (W_0,\dots,W_k)$, $\mathscr{W}' = (W_k, \dots, W_{k+k'})$ be vectors consisting of consecutively partially orthogonal linear subspaces of $\R^n \times \R^n$, and set 
\[
\mathscr{W} \# \mathscr{W}' := (W_0,\dots,W_{k+k'}).
\]  
Let $A_0,A_1,A_2$ be non-degenerate paths with respect to $W_0$, $W_{k}$, and $W_{k+k'}$, respectively.
Then orientations $o_{\mathscr{W}}(A_0,A_1)$ and $o_{\mathscr{W}'}(A_1,A_2)$ of $\det(\mathscr{D}_{\mathscr{W}} (A_0,A_1))$ and $\det(\mathscr{D}_{\mathscr{W}'} (A_1,A_2))$, respectively, determine in a canonical way a {\em glued orientation}
\[
o_{\mathscr{W}}(A_0,A_1)\, \# \,o_{\mathscr{W}'}(A_1,A_2)
\]
of $\det(\mathscr{D}_{\mathscr{W}\# \mathscr{W}'} (A_0,A_2))$. The construction is analogous to the one described in \cite[Section 3]{fh93}.
This way of gluing orientations is associative. A {\em coherent orientation} is a set of orientations $o_{\mathscr{W}}(A^-,A^+)$ for each choice of compatible data such that
\[
o_{\mathscr{W}\# \mathscr{W}'}(A_0,A_1) = 
o_{\mathscr{W}}(A_0,A_1)\, \# \,o_{\mathscr{W}'}(A_1,A_2),
\]
whenever the latter glued orientation is well-defined. The proof of the existence of a coherent orientation is analogous to the proof of Theorem 12 in \cite{fh93}.

The choice of such a coherent orientation in this linear setting determines orientations for all the nonlinear objects we are interested in, and such orientations are compatible with gluing. As mentioned above, the fact that we are dealing with the cotangent bundle of an oriented manifold makes the step from the linear setting to the nonlinear one easier. The reason is that we can fix once for all special symplectic trivializations of the bundle $x^*(TT^*M)$, for every solution $x$ of our Hamiltonian problem. In fact, one starts by fixing an orthogonal and orientation preserving trivialization of $(\pi\circ x)^*(TM)$, and then considers the induced unitary trivialization of $x^*(TT^*M)$. Let $u$ be an element in some space $\mathscr{M}(x,y)$, consisting
of the solutions of a Floer equation on the strip $\Sigma$ which are asymptotic to two Hamiltonian orbits $x$ and $y$ and satisfy suitable jumping conormal boundary conditions. Then
we can find a unitary trivialization of $u^*(TT^*M)$ which converges to the given unitary trivializations of $x^*(TT^*M)$ and $y^*(TT^*M)$. We may use such a trivialization to linearize the problem, producing  a Fredholm operator in $\mathscr{D}_{\mathscr{S},\mathscr{W}}(A^-,A^+)$. Here $A^-,A^+$ are determined by the fixed unitary trivializations of $x^*(TT^*M)$ and $y^*(TT^*M)$. The orientation of the determinant bundle over $\mathscr{D}_{\mathscr{S},\mathscr{W}}(A^-,A^+)$ then induces an orientation of the tangent space of $\mathscr{M}(x,y)$ at $u$, that is an orientation of $\mathscr{M}(x,y)$.
See \cite[Section 1.4]{as06} for more details.

When the manifold $M$ is not orientable, one cannot fix once for all trivializations along the Hamiltonian orbits, and the construction of coherent orientations requires understanding the effect of changing the trivialization, as in \cite[Lemma 15]{fh93}. The Floer complex and the pair-of-pants product are still well-defined over integer coefficients, whereas the Chas-Sullivan loop product requires $\Z_2$ coefficients.  

\subsection{Nonlinear consequences}
\label{lin}

Let us derive the nonlinear consequences of Theorems \ref{fredholmnl}, \ref{fredhs+nl}, and \ref{fredhs-nl}. 
Let $Q$ be a finite dimensional Riemannian manifold, and let $R_0,\dots,R_k$ be submanifolds of $Q\times Q$, such that $R_{j-1}$ is partially orthogonal to $R_j$, for every $j=1,\dots,k$ (with respect to the product metric on $Q\times Q$). Partial orthogonality implies that the dimension of $R_{j-1}\cap R_j$ is locally constant, and we assume such a dimension is actually constant.
Let $k\geq 0$ be an integer, and let us fix numbers
\[
-\infty = s_0 < s_1 < \dots < s_k < s_{k+1} = +\infty.
\]
We recall from Section \ref{fhpfo} that if $H\in C^{\infty}([0,1],T^*Q)$ is a Hamiltonian, the symbol $\mathscr{P}^R(H)$ denotes the set of all the Hamiltonian orbits $x:[0,1]\rightarrow T^*Q$  which satisfy the boundary condition $(x(0),\mathscr{C}x(1))\in N^* R$, where $\mathscr{C}$ is the anti-symplectic involution on $T^*Q$ which maps $(q,p)$ into $(q,-p)$. When $x\in \mathscr{P}^R(H)$ is non-degenerate, its Maslov index $\mu^R(x)$ is defined in (\ref{mascon}).

We recall that a space $\mathscr{M}$ is said to have {\em virtual dimension $d$}, where $d\in \Z$, or briefly
\[
\mathrm{virdim}\, \mathscr{M} = d,
\]
if $\mathscr{M}$ can be seen as the set of zeroes of a smooth section of some Banach bundle, whose fiberwise derivative is Fredholm of index $d$. When such a section is transverse to the zero-section, the implicit function theorem implies that either $\mathscr{M}$ is empty, or $\mathscr{M}$ is a smooth manifold of dimension $d\geq 0$.  

We start by considering the case of the full strip $\Sigma = \set{z\in \C}{0\leq \im z\leq 1}$:

\begin{cor}
\label{corfred}
Let $x\in \mathscr{P}^{R_0}(H)$ and $y\in \mathscr{P}^{R_k}(H)$ be non-degenerate. Then the set $\mathscr{M}(x,y)$ of maps $u:\Sigma\rightarrow T^*Q$ which solve the equation $\delbar_{J,H} (u) = 0$ with boundary conditions
\[
(u(s),\mathscr{C} u(s+i)) \in N^* R_j, \quad \forall s\in [s_j,s_{j+1}],
\]
and asymptotic conditions
\[
\lim_{s\rightarrow -\infty} u(s+it) = x(t), \quad \lim_{s\rightarrow +\infty} u(s+it) = y(t),
\]
has virtual dimension
\[
\mathrm{virdim}\, \mathscr{M}(x,y) = \mu^{R_0}(x) - \mu^{R_k}(y) - \sum_{j=1}^k ( \dim R_{j-1} - \dim R_{j-1} \cap R_j ).
\]
\end{cor}

\begin{proof}
Standard arguments in Floer theory allow us to see $\mathscr{M}(x,y)$ as the set of zeroes of a smooth section of a suitable Banach bundle, the base of which is a Banach manifold modeled on $X^{1,p}_{\mathscr{S},\mathscr{W}}$, and the fibers are Banach spaces of sections of class $X^p_{\mathscr{S}}$, where $p>2$.
The fiberwise derivative of such a section at $u\in \mathscr{M}(x,y)$ is conjugated to a linear operator $\delbar_A$ of the form considered in Theorem \ref{fredholmnl}, where the linear subspaces $W_j$ of $\R^n \times \R^n$ are local models for the submanifolds $R_j$, and where 
\begin{eqnarray*}
\mu (\graf \Phi^- C,N^* W_0)  & = & \mu^{R_0} (x) - \frac{1}{2} ( \dim R_0 - \dim Q), \\
\mu (\graf \Phi^+ C,N^* W_k) & = &  \mu^{R_k} (y) - \frac{1}{2} ( \dim R_k - \dim Q),
\end{eqnarray*}
by the definition (\ref{mascon}) of the Maslov index. By Theorem \ref{fredholmnl}, such an operator is Fredholm of index
\begin{eqnarray*}
\ind \delbar_A = \mu^{R_0} (x) - \frac{1}{2} ( \dim R_0 - \dim Q) - \mu^{R_k} (y) + \frac{1}{2} ( \dim R_k - \dim Q) \\ - \frac{1}{2} \sum_{j=1}^k (\dim R_{j-1} + \dim R_j - 2 \dim R_{j-1}\cap R_j ).
\end{eqnarray*}
After simplification, this formula reduces to
\[
\ind \delbar_A = \mu^{R_0}(x) - \mu^{R_k}(y) - \sum_{j=1}^k ( \dim R_{j-1} - \dim R_{j-1} \cap R_j ),
\]
as claimed.
\end{proof}

\begin{rem}
By elliptic regularity, the maps $u\in \mathscr{M}(x,y)$ are smooth up to the boundary on $\Sigma\setminus \{s_1,s_1+i,\dots,s_k,s_k+i\}$. By Schwarz reflection, the maps
\[
\zeta \mapsto u(s_j+\zeta^2) \quad \mbox{and} \quad \zeta \mapsto u(s_j + i - \zeta^2)
\]
are smooth up to the boundary in a neighborhood of zero in the upper right quadrant $\mathrm{Cl} (\mathbb{H}^+) = \set{\zeta\in \C}{\re \zeta\geq 0, \; \im z\geq 0}$. Analogous regularity results hold for the maps which appear in the following two corollaries.
\end{rem}  

\begin{rem}
In this paper, all the pairs of submanifolds we need to consider are partially orthogonal. However, it might be useful to have a generalization of Corollary \ref{corfred} to the situation where the submanifolds $R_{j-1}$ and $R_j$ are only assumed to have a clean intersection. An easy way to deal with such a situation is the following. First, we transform the nonlocal boundary problem for $u:\Sigma \rightarrow T^*Q$ into a local one, by considering the maps $v: \Sigma \rightarrow T^* (Q\times Q)$ defined by
\[
v(z) := \bigl( \mathscr{C} u((i-z)/2), u((i+z)/2)\bigr).
\]
Then $v$ solves a Cauchy-Riemann type equation, together with local boundary conditions
\[
v(s) \in N^* \Delta_Q, \;\; \forall s\in \R, \quad v(s+i) \in N^* R_j, \;\; \forall s\in [s_j,s_{j+1}].
\]
Let $g_s$ be a smooth 1-parameter family of metrics on $Q\times Q$ such that $R_{j-1}$ is partially orthogonal to $R_j$ with respect to the metric  $g_{s_j}$ (when  the sets $R_{j-1}\cap R_j$ are pairwise disjoint, the family $g_s$ can be chosen to be independent of $s$). Then the analogue of Corollarly \ref{corfred} holds, where the the perturbed Cauchy-Riemann operator $\delbar_{J,H}$ is the one associated to the $s$-dependent family of Levi-Civita almost complex structures $J_s$ on $T^* (Q\times Q)$ induced by $g_s$. The same considerations apply to the next two corollaries.
\end{rem}  

In the case of the right half-strip $\Sigma^+ = \set{z\in \C}{\re z\geq 0, \; 0\leq \im z\leq 1}$, we fix the numbers
\[
0 = s_0 < s_1 <\dots< s_k <s_{k+1}=+\infty,
\]
and we have the following consequence of Theorem \ref{fredhs+nl}:

\begin{cor}
\label{corfred+}
Let $x\in \mathscr{P}^{R_k}(H)$ be non-degenerate and let $\gamma\in W^{1,2}([0,1],Q)$ be a curve such that $(\gamma(0),\gamma(1))\in R_0$.
Then the set $\mathscr{M}(\gamma,x)$ of maps $u:\Sigma^+ \rightarrow T^*Q$ which solve the equation $\delbar_{J,H} (u) = 0$ with boundary conditions
\[
(u(s),\mathscr{C} u(s+i)) \in N^* R_j, \;\; \forall s\in [s_j,s_{j+1}], \quad u(it) \in T_{\gamma(t)}^* Q, \;\; \forall t\in [0,1],
\]
and the asymptotic condition
\[
\lim_{s\rightarrow +\infty} u(s+it) = x(t),
\]
has virtual dimension 
\[
\mathrm{virdim}\, \mathscr{M}(\gamma,x) = - \mu^{R_k}(x) - \sum_{j=1}^k ( \dim R_{j-1} - \dim R_{j-1} \cap R_j ).
\]
\end{cor}

\begin{proof}
Arguing as in the proof of Corollary \ref{corfred}, we find that $\mathscr{M}(\gamma,x)$ has virtual dimension equal to the Fredholm index of an operator $\delbar_A$ of the form considered in Theorem \ref{fredhs+nl}, with $V_0=(0)$ and $W_j$ a local model for $R_j$ (see \cite{as06}, Section 3.1, for more details on how to deal with this kind of boundary data).

By Theorem \ref{fredhs+nl} and  (\ref{mascon}), we have
\begin{eqnarray*}
\ind \delbar_A = \frac{1}{2} \dim Q - \mu ( \graf \Phi^+ C,N^* W_k) - \frac{1}{2} \dim R_0 \\ - \frac{1}{2} \sum_{j=1}^k (\dim R_{j-1} + \dim R_j - 2 \dim R_{j-1} \cap R_j) \\
= - \mu^{R_k}(x) - \sum_{j=1}^k ( \dim R_{j-1} - \dim R_{j-1} \cap R_j ),
\end{eqnarray*}
concluding the proof.
\end{proof}

In the case of the left half-strip $\Sigma^- = \set{z\in \C}{\re z\leq 0, \; 0\leq \im z \leq 1}$, we fix the numbers
\[
0=s_0 > s_1 > \dots > s_k > s_{k+1}=-\infty.
\]
One could easily derive the analogue of Corollary \ref{corfred+} from Theorem \ref{fredhs-nl}. Instead, we prefer to derive the following variant, where $\mathbb{O}_Q$ denotes image of the zero-section in $T^*Q$:

\begin{cor}
\label{corfred-}
Let $x\in \mathscr{P}^{R_k}(H)$ be non-degenerate. Then the set $\mathscr{M}(x)$ of maps $u:\Sigma^- \rightarrow T^*Q$ which solve the equation $\delbar_{J,H} (u) = 0$ with boundary conditions
\[
(u(s),\mathscr{C} u(s+i)) \in N^* R_j, \;\; \forall s\in [s_{j+1},s_j], \quad u(it) \in \mathbb{O}_Q \;\; \forall t\in [0,1],
\]
and the asymptotic condition
\[
\lim_{s\rightarrow -\infty} u(s+it) = x(t),
\]
has virtual dimension 
\[
\mathrm{virdim}\, \mathscr{M}(x) = \mu^{R_k}(x) - \sum_{j=1}^k ( \dim R_j - \dim R_{j-1} \cap R_j ).
\]
\end{cor}

\begin{proof}
Arguing as in the proof of Corollary \ref{corfred}, we find that $\mathscr{M}(x)$ has virtual dimension equal to the Fredholm index of an operator $\delbar_A$ of the form considered in Theorem \ref{fredhs-nl}, with $V_0=\R^m$, $m=\dim Q$, and $W_j$ a local model for $R_j$. By Theorem \ref{fredhs-nl} and  (\ref{mascon}), we have
\begin{eqnarray*}
\ind \delbar_A = \frac{1}{2} \dim Q + \mu ( \graf \Phi^- C,N^* W_k) - \frac{1}{2} (2 \dim Q - \dim R_0) \\ - \frac{1}{2} \sum_{j=1}^k (\dim R_{j-1} + \dim R_j - 2 \dim R_{j-1} \cap R_j) \\
= \mu^{R_k}(x) - \sum_{j=1}^k ( \dim R_j - \dim R_{j-1} \cap R_j ),
\end{eqnarray*}
concluding the proof.
\end{proof}

\paragraph{Dimension computations.} We conclude this section by using the above corollaries to prove that all the spaces of solutions of the Cauchy-Riemann type problems considered in this paper are - generically - smooth manifolds, and to compute their dimension. Together with the results of Section \ref{cosec}, we deduce that these manifolds carry coherent orientations which are compatible with gluing. The transversality issues which lead to genericity in the space of the Hamiltonians are standard, see \cite{fhs96}. Here we compute the virtual dimensions, by making use of the following two lemmas.

\begin{lem}
\label{lemulo}
Let $R_1$ and $R_2$ be submanifolds of $Q\times Q$, let $H_1$, $H_2\in C^{\infty}([0,1]\times T^*Q)$ be two Hamiltonians, and set
\begin{eqnarray}
\label{lerre}
R:= \set{(q_1,q_2,q_3,q_4)\in Q^4}{(q_3,q_1)\in R_1, \; (q_2,q_4)\in R_2} \cong R_1 \times R_2, \\
\label{lakappa}
K\in C^{\infty}\bigl([0,1]\times T^* Q^2 \bigr), \quad K(t,x_1,x_2) := H_1(1-t,\mathscr{C}x_1) + H_2(t,x_2).
\end{eqnarray}
Then the curves $x_1$, $x_2\in C^{\infty}([0,1],T^*Q)$ belong to $\mathscr{P}^{R_1}(H_1)$ and $\mathscr{P}^{R_2}(H_2)$, respectively, if and only if the curve
\[
x: [0,1] \rightarrow T^* (Q \times Q) = T^* Q\times T^* Q, \quad x(t):= \bigl( \mathscr{C} x_1(1-t), x_2(t) \bigr),
\]
belongs to $\mathscr{P}^R(K)$. Furthermore,
\[
\mu^R(x) = \mu^{R_1}(x_1) + \mu^{R_2}(x_2).
\]
\end{lem}

\begin{proof}
It is easy to check that $x_1$ and $x_2$ are orbits of the Hamiltonian vector fields associated to $H_1$ and $H_2$ if and only if $y$ is an orbit of the vector field associated to $K$. Moreover,
\[
(y(0),\mathscr{C}y(1)) = \bigl( \mathscr{C}x_1(1),x_2(0),x_1(0),\mathscr{C}x_2(1) \bigr)
\]
belongs to $N^* R$ if and only if $(x_1(0),\mathscr{C}x_1(1))$ belongs to $N^* R_1$ and $(x_2(0),\mathscr{C}x_2(1))$ belongs to $N^* R_2$. We just have to check the identity involving the Maslov indices.

Let $j$ be either 1 or 2. Let $G_{H_j}:[0,1] \rightarrow \mathrm{Sp}(2m)$, $m=\dim Q$, be the symplectic path obtained by conjugating the differential of the Hamiltonian flow of $H_j$ along the orbit $x_j$ by a vertical-preserving trivialization $\Psi_j$ of $x^*_j(TT^*Q)$ such that
\[
\bigl( \Psi_j(0)\times C \Psi_j(1) D\mathscr{C}(\mathscr{C}x_j(1)) \bigr) T_{(x_j(0),\mathscr{C}x_j(1))} N^* R_j = N^* W_j,
\]
where $W_j$ is a linear subspace of $\R^m \times \R^m$ . Then, by the definition (\ref{mascon}) of the Maslov index,
\[
\mu^{R_j}(x_j) = \mu( \graf G_{H_j} C, N^* W_j) + \frac{1}{2} (\dim R_j - \dim Q_j).
\]
If $G_K: [0,1] \rightarrow \mathrm{Sp}(4m)$ is the symplectic path obtained by conjugating the differential of the Hamiltonian flow of $K$ along the orbit $y$ by the trivialization induced in the obvious way by $\Psi_1$ and $\Psi_2$, we have
\[
G_K(t) = C G_{H_1} ( 1-t ) G_{H_1} (1)^{-1}  C \times G_{H_2} (t),
\]
and
\[
\mu^R(y) = \mu ( \graf G_K C, N^* W) + \frac{1}{2} (\dim R - \dim Q^2),
\]
where 
\[
W := \set{(\xi_1,\xi_2,\xi_3,\xi_4)}{(\xi_3,\xi_1)\in W_1, \; (\xi_2,\xi_4)\in W_2}.
\]
Since $\dim R = \dim Q_1+\dim Q_2$, we must show that
\begin{equation}
\label{hkhk}
\mu ( \graf G_K C, N^* W) = \mu( \graf G_{H_1} C, N^* W_1) + \mu( \graf G_{H_2} C, N^* W_2).
\end{equation}
The linear mapping
\[
T: \R^{4m} \rightarrow \R^{4m}, \quad (\xi_1,\xi_2,\xi_3,\xi_4) \mapsto (\xi_3,\xi_1,\xi_2,\xi_4),
\]
maps $W$ onto $W_1 \times W_2$, hence the symplectic automorphism $T \oplus T^*$ maps $N^* W$ onto $N^* W_1 \times N^* W_2$. Moreover,
\[
(T\oplus T^*) \graf G_K (t) C = \Bigl( \graf G_{H_1}(1) G_{H_1}(1-t)^{-1} C\Bigr)\times \Bigl( \graf G_{H_2} (t) C \Bigr).
\]
By the symplectic invariance and the additivity of the Maslov index,
\begin{equation}
\label{hkhk1}
\begin{split}
\mu ( \graf G_K C, N^* W) = \mu ( (T\oplus T^*) \graf G_K (t) C, (T\oplus T^*)  N^* W ) \\ = \mu( \graf G_{H_1}(1) G_{H_1}(1-\cdot)^{-1}C , N^* W_1) + \mu( \graf G_{H_2} C, N^* W_2).
\end{split}
\end{equation}
The symplectic path $t\mapsto G_{H_1}(1) G_{H_1}(1-t)^{-1}$ is homotopic within the symplectic group to the path $G_{H_1}$, by the homotopy
\[
(\lambda,t) \mapsto G_{H_1} ( t+\lambda(1-t)) G_{H_1}(\lambda(1-t))^{-1},
\]
which fixes the end points $I$ and $G_{H_1}(1)$. By the homotopy invariance of the Maslov index,
\[
\mu( \graf G_{H_1}(1) G_{H_1}(1-\cdot)^{-1} C, N^* W_1) = \mu (\graf G_{H_1} C, N^* W_1),
\]
so (\ref{hkhk}) follows from (\ref{hkhk1}).
\end{proof}

\begin{lem}
\label{lemule}
Let $R$ be a submanifold of $Q\times Q$, let $H\in C^{\infty}([0,1]\times T^*Q)$, and define $K\in C^{\infty}([0,1]\times T^* Q^2)$ by
\[
K(t,x_1,x_2) := \frac{1}{2} H \left( \frac{1-t}{2}, \mathscr{C} x_1 \right) +\frac{1}{2} H \left( \frac{1+t}{2},  x_2 \right).
\]
Then the curve $x:[0,1]\rightarrow T^*Q$ belongs to $\mathscr{P}^R(H)$ if and only if the curve $y:[0,1]\rightarrow T^* Q^2$, defined by
\[
y(t) := \left( \mathscr{C} x \Bigl( \frac{1-t}{2} \Bigr), x \Bigl( \frac{1+t}{2} \Bigr) \right),
\]
belongs to $\mathscr{P}^{\Delta_Q \times R}(K)$. Furthermore,
\[
\mu^R(x) = \mu^{\Delta_Q \times R}(y).
\]
\end{lem}

\begin{proof}
A simple computation shows that $x$ is an orbit of the Hamiltonian vector field induced by $H$ if and only if $y$ is an orbit of the one induced by $K$.
Moreover,
\[
(y(0), \mathscr{C} y(1)) = \bigl( \mathscr{C} x (1/2), x(1/2), \mathscr{C} x(0), x(1) \bigr).
\]
so, by using the fact that conormals are $\mathscr{C}$-invariant, we deduce that  $x\in \mathscr{P}^R(H)$ if and only if $y\in \mathscr{P}^{\Delta_Q \times R}(K)$. Let $G_H:[0,1] \rightarrow \mathrm{Sp}(2m)$, $m=\dim Q$, be the symplectic path obtained by conjugating the differential of the Hamiltonian flow of $H$ along the orbit $x$ by a vertical-preserving trivialization $\Psi$ of $x^*(TT^*Q)$ such that
\[
\bigl( \Psi(0)\times C \Psi(1) D\mathscr{C}(\mathscr{C}x(1)) \bigr) T_{(x(0),\mathscr{C}x(1))} N^* R = N^* W,
\]
where $W$ is a linear subspace of $\R^m \times \R^m$. Then, by the definition (\ref{mascon}) of the Maslov index,
\[
\mu^R(x) = \mu( \graf G_H C, N^* W) + \frac{1}{2} (\dim R - \dim Q).
\]
If $G_K: [0,1] \rightarrow \mathrm{Sp}(4m)$ is the symplectic path obtained by conjugating the differential of the Hamiltonian flow of $K$ along the orbit $y$ by the trivialization induced in the obvious way by $\Psi$, we have
\[
G_K(t) = \left( C G_H \Bigl( \frac{1-t}{2} \Bigr) G_H \Bigl( \frac{1}{2} \Bigr)^{-1} C \right) \times \left( G_H \Bigl( \frac{1+t}{2} \Bigr) G_H \Bigl( \frac{1}{2} \Bigr)^{-1} \right).
\]
Moreover, using also (\ref{masprod}),
\begin{equation*}
\begin{split}
\mu^{\Delta_Q \times R}(y) = \mu( \graf G_K C, N^* (\Delta_{\R^m} \times W)) + \frac{1}{2} (\dim R \times \Delta_Q - \dim Q\times Q) \\ = \mu( \graf G_K C, N^* \Delta_{\R^m} \times N^* W )  + \frac{1}{2} (\dim R - \dim Q) \\
= \mu(G_K N^*_{\Delta_{\R^m}} , N^* W) + \frac{1}{2} (\dim R - \dim Q).
\end{split}
\end{equation*}
Since $N^* \Delta_{\R^m}=\graf C$, there holds
\[
G_K(t) N^*_{\Delta_{\R^m}} = \graf \Gamma(t) C,
\]
where $\Gamma:[0,1] \rightarrow \mathrm{Sp}(2m)$ is the path
\[
\Gamma(t) := G_H \left( \frac{1+t}{2} \right)G_H \left( \frac{1-t}{2} \right)^{-1}.
\]
The symplectic path $\Gamma$ is homotopic to the symplectic path $G_H$ by the symplectic homotopy 
\[
(\lambda,t) \mapsto G_H \left( t + \frac{\lambda}{2} (1-t) \right)G_H \left( \frac{\lambda}{2} (1-t) \right)^{-1},
\]
which fixes the end points $\Gamma(0)=G_H(0)=I$ and $\Gamma(1)=G_H(1)$. By the homotopy invariance of the Maslov index,
\[
\mu(G_K N^*_{\Delta_{\R^m}} , N^* W) = \mu ( \graf \Gamma C, N^* W) = \mu ( \graf G_H C, N^* W),
\]
and the conclusion follows from the above formulas for $\mu^R(x)$ and $\mu^{\Delta_Q \times R}(y)$. 
\end{proof}

We are finally ready to compute the virtual dimensions of the spaces of maps introduced in this paper.

\paragraph{The space $\mathscr{M}_{\Upsilon}^{\Lambda}$.} Let us
study the space of solutions
$\mathscr{M}_{\Upsilon}^{\Lambda}(x_1,x_2;y)$, where $x_1\in
\mathscr{P}^{\Lambda}(H_1)$, $x_2\in \mathscr{P}^{\Lambda}(H_2)$, and
$y\in \mathscr{P}^{\Lambda}(H_1 \# H_2)$ (see Section \ref{tppp}). It
is a space of solutions of the Floer equation on the pair-of-pants
Riemann surface $\Sigma_{\Upsilon}^{\Lambda}$.
The pair-of-pants Riemann surface $\Sigma_{\Upsilon}^{\Lambda}$ is described in Section \ref{Fetpp} as the quotient of the disjoint union of two strips $\R \cup [-1,0]$ and $\R \times [0,1]$ with respect to the identifications
\begin{eqnarray}
\label{sinistranew}
(s,-1) \sim (s,0-), \quad (s,0+) \sim (s,1) \quad \forall s\leq 0, \\
\label{destranew}
(s,-1) \sim (s,1), \quad (s,0-) \sim (s,0+) \quad \forall s\geq 0.
\end{eqnarray}
The space $\mathscr{M}^{\Lambda}_{\Upsilon}(x_1,x_2;y)$ consists of maps 
\[
u: \Sigma_{\Upsilon}^{\Lambda} \rightarrow T^* M,
\]
solving the Floer equation $\delbar_{J,H} (u) = 0$,
with asymptotics
\[
\lim_{s\rightarrow -\infty} u(s,t-1) = x_1(t), \quad
\lim_{s\rightarrow -\infty} u(s,t) = x_2(t), \quad
\lim_{s\rightarrow +\infty} u(s,2t-1) = y(t).
\]
We can associate to a map $u: \Sigma_{\Upsilon}^{\Lambda} \rightarrow T^*M$ the map $v: \Sigma  \rightarrow T^*M^2$ by setting
\begin{equation}
\label{lavu}
v(z) := \bigl( \mathscr{C} u(\overline{z}),u(z) \bigr).
\end{equation}
The identifications (\ref{sinistranew}) on the left-hand side of the domain of $u$ are translated into the fact that $v(s+it)$ is $1$ periodic in $t$ for $s\leq 0$, or equivalently into the nonlocal boundary condition
\begin{equation}
\label{bbb1}
(v(s),\mathscr{C} v(s+i)) \in N^* \Delta_{M^2}, \quad \forall s\leq 0,
\end{equation}
where $\Delta_{M^2}$ denotes the diagonal in $M^4 = M^2 \times M^2$. The identifications (\ref{destranew}) on the right-hand side of the domain of $u$ are translated into the local boundary conditions 
\begin{equation}
\label{bbb2}
v(s) \in N^* \Delta_M, \quad v(s+i) \in N^* \Delta_M, \quad \forall s\geq 0.
\end{equation}
The map $u$ solves the Floer equation $\delbar_{J,H}(u)=0$ if and only if $v$ solves the Floer equation $\delbar_{J,K}(v)=0$, where $K$ is the Hamiltonian defined in (\ref{lakappa}).
The asymptotic conditions for $u$ are equivalent to
\begin{equation}
\label{bbb3}
\begin{split}
x(t):= \lim_{s\rightarrow -\infty} v(s + it) = (\mathscr{C} x_1(1-t),x_2(t)), \\
z(t):= \lim_{s\rightarrow +\infty} v(s + it) = \bigl(\mathscr{C} y((1-t)/2) , y((1+t)/2) \bigr).
\end{split}
\end{equation}
We conclude that $\mathscr{M}_{\Upsilon}^{\Lambda}(x_1,x_2;y)$ can be identified with the space of maps $\mathscr{M}(x,z)$ of Corollary \ref{corfred}, where the underlying manifold is $Q=M\times M$, the Hamiltonian is $K$, the boundary conditions have a single jump at $s_1=0$ and are given by the following partially orthogonal submanifolds of $Q^2=M^4$:
\[
R_0 = \Delta_{M^2}, \quad R_1 = \Delta_M \times \Delta_M.
\]
By Corollary \ref{corfred}, 
 \begin{eqnarray*}
 \mathrm{virdim}\, \mathscr{M}_{\Upsilon}^{\Lambda}(x_1,x_2;y) = \mu^{\Delta_{M^2}} (x) - \mu^{\Delta_M \times \Delta_M} (z) \\- \bigl( \dim \Delta_{M^2} - \dim \Delta_{M^2} \cap (\Delta_M \times \Delta_M) \bigr) = \mu^{\Delta_{M^2}} (x) - \mu^{\Delta_M \times \Delta_M} (z) - n.
 \end{eqnarray*}
 By Lemma \ref{lemulo}, 
 \begin{equation}
 \label{gambe}
 \mu^{\Delta_{M^2}} (x) = \mu^{\Delta_M}(x_1) + \mu^{\Delta_M}(x_2) = \mu^{\Lambda}(x_1) + \mu^{\Lambda}(x_2).
 \end{equation}
 By Lemma \ref{lemule}, applied to the manifold $Q=M$, to the Hamiltonian $H_1\# H_2$, and to the submanifold $R=\Delta_M$,
 \[
 \mu^{\Delta_M \times \Delta_M} (z) = \mu^{\Delta_M}(y) = \mu^{\Lambda}(y).
 \]
We conclude that $\mathscr{M}_{\Upsilon}^{\Lambda}(x_1,x_2;y)$ has virtual dimension
 \[
 \mathrm{virdim}\, \mathscr{M}_{\Upsilon}^{\Lambda}(x_1,x_2;y) = \mu^{\Lambda}(x_1) + \mu^{\Lambda}(x_2) - \mu^{\Lambda}(y)- n,
 \]
proving the part of Proposition \ref{popss} which concerns the space $\mathscr{M}_{\Upsilon}^{\Lambda}$.

\paragraph{The space $\mathscr{M}_{\Upsilon}^{\Omega}$.}
Let us consider the space
$\mathscr{M}_{\Upsilon}^{\Omega}(x_1,x_2;y)$, where $x_1\in
\mathscr{P}^{\Omega}(H_1)$, $x_2\in \mathscr{P}^{\Omega}(H_2)$, and
$y\in \mathscr{P}^{\Omega}(H_1 \# H_2)$ (see Section \ref{tppp}). This
is the space of solutions $u$ of the Floer equation on the Riemann surface
with boundary $\Sigma^{\Omega}_{\Upsilon}$, described as a strip with
a slit in Section \ref{Fetpp}, which take values in $T_{q_0}^* M$ on the boundary, and converge to the orbits $x_1$, $x_2$, $y$ on the three ends. By defining the map $v: \Sigma \rightarrow T^* M^2$ as in (\ref{lavu}) and the curves $x$ and $z$ as in (\ref{bbb3}), we see that the space $\mathscr{M}_{\Upsilon}^{\Omega}(x_1,x_2;y)$ is in one-to-one correspondence with the space $\mathscr{M}(x,z)$ of Corollary \ref{corfred}, where $Q=M^2$, the Hamiltonian is the function $K$ defined in (\ref{lakappa}), the boundary conditions jump at $s_1=0$ and are given by the following partially orthogonal submanifolds of $Q^2=M^4$:
\[
R_0 = \{(q_0,q_0,q_0,q_0)\}, \quad R_1 = \Delta_M \times \{(q_0,q_0)\}.
\]
By Corollary \ref{corfred}, 
 \[
 \mathrm{virdim}\, \mathscr{M}_{\Upsilon}^{\Omega}(x_1,x_2;y) = \mu^{(q_0,q_0,q_0,q_0)} (x) - \mu^{\Delta_M \times (q_0,q_0)} (z). 
 \]
 By Lemma \ref{lemulo}, 
 \[
 \mu^{(q_0,q_0,q_0,q_0)} (x) = \mu^{(q_0,q_0)}(x_1) + \mu^{(q_0,q_0)}(x_2) = \mu^{\Omega}(x_1) + \mu^{\Omega}(x_2).
 \]
 By Lemma \ref{lemule},
 \[
 \mu^{\Delta_M \times (q_0,q_0)} (z) = \mu^{(q_0,q_0)}(y) = \mu^{\Omega}(y).
 \]
Therefore, $\mathscr{M}_{\Upsilon}^{\Omega}(x_1,x_2;y)$ has virtual dimension
 \[
 \mathrm{virdim}\, \mathscr{M}_{\Upsilon}^{\Omega}(x_1,x_2;y) = \mu^{\Omega}(x_1) + \mu^{\Omega}(x_2) - \mu^{\Omega}(y).
 \]
 This concludes the proof of Proposition \ref{popss}.

\paragraph{The space $\mathscr{M}_E$.} Let $(x_1,x_2)\in
\mathscr{P}^{\Lambda}(H_1) \times \mathscr{P}^{\Lambda}(H_2)$ and
$y\in \mathscr{P}^{\Theta}(H_1 \oplus H_2)$ (see Section
\ref{fpop}). We set $Q=M\times M$, we define the Hamiltonian $K\in C^{\infty}([0,1]\times T^*Q)$ by (\ref{lakappa}), and the $T^*Q$-valued curves $x$ and $z$ by 
\[
x(t) := \bigl( \mathscr{C} x_1(1-t), x_2(t) \bigr), \quad
z(t) := \bigl( \mathscr{C} y_1(1-t), y_2(t) \bigr),
\]
where $y_1$ and $y_2$ are the components of $y$.
Then $x\in \mathscr{P}^{\Delta_{M^2}}(K)$,  $z\in \mathscr{P}^{\Delta_M^{(4)}}(K)$, 
and one easily checks that
\[
\mu^{\Delta_M^{(4)}}(z) = \mu^{\Delta_M^{(4)}} (y) = \mu^{\Theta}(y),
\]
where the Maslov index of $z$ refers to the Hamiltonian $K$, and the Maslov index of $y$ to the Hamiltonian $H_1 \oplus H_2$. The  space of solutions $\mathscr{M}_E(x_1,x_2;y)$ is in one-to-one correspondence with the space $\mathscr{M}(x,z)$ of Corollary \ref{corfred}, where the boundary conditions switch at $s_1=0$ and are given by the submanifolds
\[
R_0 = \Delta_{M^2} , \quad R_1 = \Delta_M^{(4)}.
\]
Hence, using also (\ref{gambe}),
\begin{eqnarray*}
\mathrm{virdim}\, \mathscr{M}_E(x_1,x_2;y) = \mu^{\Delta_{M^2}}(x) - \mu^{\Delta_M^{(4)}} (z) \\ - \bigl( \dim \Delta_{M^2} - \dim \Delta_{M^2} \cap  \Delta_M^{(4)} \bigr) = \mu^{\Lambda}(x_1) + \mu^{\Lambda}(x_2) - \mu^{\Theta} (y) - n.
\end{eqnarray*}
This proves the part of Proposition \ref{eg} about $\mathscr{M}_E$.

\paragraph{The space $\mathscr{M}_G$.} Let
$y\in \mathscr{P}^{\Theta}(H_1 \oplus H_2)$ and $z\in
\mathscr{P}^{\Lambda}(H_1 \# H_2)$
(see Section \ref{fpop}). We set $Q=M\times M$, and
\begin{eqnarray*}
x: [0,1] \rightarrow T^* Q, \quad x(t) := \bigl( \mathscr{C} y_1(1-t), y_2(t) \bigr), \\
w: [0,1] \rightarrow T^* Q, \quad w(t):= \bigl(\mathscr{C} z((1-t)/2), z((1+t)/2) \bigr).
\end{eqnarray*}
As in the case of $\mathscr{M}_E$, the curve $x$ belongs to $\mathscr{P}^{\Delta_M^{(4)}}(K)$, where the Hamiltonian $K$ is defined by (\ref{lakappa}), and
\[
\mu^{\Delta_M^{(4)}}(x) = \mu^{\Delta_M^{(4)}} (y) = \mu^{\Theta}(y).
\]
On the other hand, by Lemma \ref{lemule}, 
$w$ belongs to $\mathscr{P}^{\Delta_M \times \Delta_M}(K)$ and
\[
\mu^{\Delta_M \times \Delta_M}(w) = \mu^{\Delta_M}(z) = \mu^{\Lambda}(z).
\]
Then Corollary \ref{corfred} implies that
\begin{eqnarray*}
\mathrm{virdim}\, \mathscr{M}_G (y,z) = \mu^{\Delta_M^{(4)}}(x) - \mu^{\Delta_M \times \Delta_M}(w)  \\ - \bigl( \dim \Delta_M^{(4)} - \dim \Delta_M^{(4)} \cap (\Delta_M \times \Delta_M) \bigr) = \mu^{\Theta}(y) - \mu^{\Lambda}(z).
\end{eqnarray*}
This concludes the proof of Proposition \ref{eg}.

\paragraph{The space $\mathscr{M}_{GE}^{\Upsilon}$.} Let $x_1\in
\mathscr{P}^{\Lambda}(H_1)$, $x_2\in \mathscr{P}^{\Lambda}(H_2)$, and $z\in
\mathscr{P}^{\Lambda}(H_1 \# H_2)$ (see the proof of Theorem
\ref{tge}). The space $\mathscr{M}_{GE}^{\Upsilon}(x_1,x_2;z)$ is the
set of pairs $(\alpha,u)$ where $\alpha$ is a real positive parameter
and $u$ is a solution of the Floer equation on the Riemann surface
$\Sigma^{\Upsilon}_{GE}(\alpha)$ with asymptotics $x_1$, $x_2$, $z$, and
suitable nonlocal boundary conditions. Let $Q=M^2$, $K$ be as in (\ref{lakappa}), $x$ be as in (\ref{bbb3}), and  
\[
y(t):= \bigl( \mathscr{C} z((1-t)/2), z((1+t)/2) \bigr),
\]
so that $x\in \mathscr{P}^{\Delta_{M^2}}(K)$ and $y\in \mathscr{P}^{\Delta_M \times \Delta_M}(K)$. For fixed $\alpha>0$, the set $\mathscr{M}_{GE}^{\Upsilon}(x_1,x_2;z)$ is in one-to-one correspondence with the space $\mathscr{M}(x,y)$ of Corollary \ref{corfred}, where the boundary conditions jump at $s_1=0$ and $s_2 = \alpha$, and are given by
\[
R_0 = \Delta_{M^2}, \quad R_1 = \Delta_M^{(4)}, \quad R_2 = \Delta_M \times \Delta_M.
\]
Considering also the parameter $\alpha$, we deduce that $\mathscr{M}_{GE}^{\Upsilon}(x_1,x_2;z)$ has virtual dimension 
\begin{eqnarray*}
\mathrm{virdim}\, \mathscr{M}_{GE}^{\Upsilon}(x_1,x_2;z) = 1 + \mathrm{virdim} \mathscr{M}(x,y) = 1 + \mu^{\Delta_{M^2}}(x) - \mu^{\Delta_M \times \Delta_M}(y) \\
- \bigl( \dim \Delta_{M^2} - \dim \Delta_{M^2} \cap \Delta_M^{(4)} + \dim \Delta_M^{(4)} - \dim \Delta_M^{(4)} \cap (\Delta_M \times \Delta_M)\bigr) \\
= 1 + \mu^{\Lambda}(x_1) + \mu^{\Lambda}(x_2) - \mu^{\Lambda}(z) - n.
\end{eqnarray*}
In the last identity we have used also Lemmas \ref{lemulo} and \ref{lemule}.
This proves Proposition \ref{ge}.

\paragraph{The spaces $\mathscr{M}_{\mathrm{C}}$ and
$\mathscr{M}_{\mathrm{Ev}}$.} Let $f$ be a Morse function on $M$,
let $x$ be a critical point of $f$, and let $y\in
\mathscr{P}^{\Lambda}(H)$. Given $q\in M$, let
\begin{eqnarray*}
\widetilde{\mathscr{M}}_{\mathrm{C}}(q,y) := \Bigl\{ u\in C^{\infty} ([0,+\infty[
\times \T, T^*M) \, \Big| \, \delbar_{J,H}(u)=0, \; \pi\circ
u(0,t) \equiv q \; \forall t\in \T, \\ \lim_{s\rightarrow +\infty} u(s,t) =
y(t)  \Big\}.
\end{eqnarray*}
For a fixed $q\in M$, the space $\widetilde{\mathscr{M}}_{\mathrm{C}}(q,y)$ coincides with $\mathscr{M}(q,y)$ from Corollary \ref{corfred+}, where $k=0$ and $R_0 = \Delta_M$. By Corollary \ref{corfred+}, $\widetilde{\mathscr{M}}_{\mathrm{C}}(q,y)$ has virtual dimension $- \mu^{\Delta_M}(y) = - \mu^{\Lambda} (y)$.
Therefore, the space
\[
\mathscr{M}_{\mathrm{C}}(x,y) = \bigcup_{q\in W^u(x)}
\widetilde{\mathscr{M}}_C(q,y),
\]
has virtual dimension
\[
\mathrm{virdim}\,  \mathscr{M}_{\mathrm{C}}(x,y) = \dim W^u(x) - \mu^{\Lambda}(y) = i(x) -
\mu^{\Lambda}(y), 
\]
proving the first part of Proposition \ref{cev}.

The space of maps
\begin{eqnarray*}
\widetilde{\mathscr{M}}_{\mathrm{Ev}} (y) := \Bigl\{ u\in
  C^{\infty}(]-\infty,0] \times \T,T^*M) \, \Big| \,\delbar_{J,H}(u)=0, \; u(0,t) \in \mathbb{O}_M \; \forall t\in \T, \\
\lim_{s\rightarrow -\infty} u(s,t) = y(t) \Bigr\}
\end{eqnarray*}
can be identified with $\mathscr{M}(y)$ from Corollary \ref{corfred-}, where the boundary condition has no jumps and is given by $R_0=\Delta_M$. Hence
\[
\mathrm{virdim}\,  \widetilde{\mathscr{M}}_{\mathrm{Ev}} (y) = \mu^{\Delta_M}(y) = \mu^{\Lambda}(y).
\]
Therefore, given $x$ a critical point of the Morse function $f$ on $M$, the space
\[ 
\mathscr{M}_{\mathrm{Ev}} (y,x) = \set{ u \in \widetilde{\mathscr{M}}_{\mathrm{Ev}} (y)}{u(0,0) \in W^s(x)}
\]
has virtual dimension
\[
\mathrm{virdim}\, \mathscr{M}_{\mathrm{Ev}} (y,x) = \mu^{\Lambda}(y) - \codim\, W^s(x) = \mu^{\Lambda}(y) - i(x).
\]
This concludes the proof of Proposition \ref{cev}.

\paragraph{The space $\mathscr{M}_{I_!}$.} Let $x\in
\mathscr{P}^{\Lambda}(H)$ and $y\in \mathscr{P}^{\Omega}(H)$ (see
Section \ref{hcevi}). The set $\mathscr{M}_{I_!}(x,y)$ is immediately seen to be in one-to-one correspondence with the space $\mathscr{M}(x,y)$ of Corollary \ref{corfred}, where the boundary conditions jump at $s_1=0$ from $R_0 = \Delta_M$ to $R_1=\{(q_0,q_0)\}$. Therefore,
\begin{eqnarray*}
\mathrm{virdim}\, \mathscr{M}_{I_!}(x,y) = \mu^{\Delta_M}(x) - \mu^{(q_0,q_0)}(y) - \bigl(\dim \Delta_M - \dim \Delta_M \cap \{(q_0,q_0)\} \bigr) \\ 
= \mu^{\Lambda}(x) - \mu^{\Omega}(y) - n.
\end{eqnarray*}
This proves Proposition \ref{i!p}.

\paragraph{The space $\mathscr{M}_{\Upsilon}^K$.} Let $\gamma_1\in
\mathscr{P}^{\Omega}(L_1)$, $\gamma_2 \in \mathscr{P}^{\Omega}(L_2)$, and $x\in
\mathscr{P}^{\Omega}(H_1 \# H_2)$ (see Section \ref{ori}). The space
$\mathscr{M}^K_{\Upsilon}(\gamma_1, \gamma_2;x)$ consists of pairs
$(\alpha,u)$ where $\alpha$ is a positive number and $u(s,t)$ is a solution
of the Floer equation on the Riemann surface
$\Sigma_{\Upsilon}^K(\alpha)$, 
which is asymptotic to $x$ for $s\rightarrow +\infty$, 
lies above some element in the unstable manifold of
$\gamma_1$ (resp.\ $\gamma_2$) for $s=0$ and $-1 \leq t \leq 0^-$
(resp.\ $0^+ \leq t \leq 1$), and lies above $q_0$ at the other
boundary points. Let us fix the two curves $q_1$ and $q_2$ in the unstable manifolds of $\gamma_1$ and $\gamma_2$ and the positive number $\alpha$. Set $Q=M^2$, let $K$ be the Hamiltonian on $[0,1]\times T^*Q$ defined by (\ref{lakappa}), and let $y:[0,1]\rightarrow T^* Q$, $\gamma:[0,1]\times Q$ be the curves
\begin{equation}
\label{ygamma}
y(t) := \left( \mathscr{C} x \Bigl( \frac{1-t}{2} \Bigr), x \Bigl( \frac{1+t}{2} \Bigr) \right), \quad \gamma(t) := \bigl(q_1(1-t),q_2(t)\bigr).
\end{equation}
Lemma \ref{lemule} implies that $y\in \mathscr{P}^{\Delta_M \times (q_0,q_0)}(K)$ and that
\[
\mu^{\Delta_M \times (q_0,q_0)} (y) = \mu^{(q_0,q_0)}(x) = \mu^{\Omega}(x).
\]
Then the set of elements $(\alpha,u)$ in $\mathscr{M}^K_{\Upsilon}(\gamma_1, \gamma_2;x)$ which lie above $q_1$ and $q_2$ for $s=0$ is in one-to-one correspondence with the space $\mathscr{M}(\gamma,y)$ of Corollary \ref{corfred+}, where the boundary conditions have a jump at $s_1=\alpha$ and are defined by
\[
R_0 = \{(q_0,q_0,q_0,q_0)\}, \quad R_1 = \Delta_M \times \{(q_0,q_0)\}.
\]
Such a space has virtual dimension
\begin{eqnarray*}
\mathrm{virdim}\, \mathscr{M}(\gamma,y) = -\mu^{\Delta_M \times (q_0,q_0)} ( y)= - \mu^{\Omega}(x).
\end{eqnarray*}
Letting the elements $q_1$ and $q_2$ of the unstable manifolds of
$\gamma_1$ and $\gamma_2$ vary, we increase the virtual dimension by
$i^{\Omega}(\gamma_1;L_1)+
i^{\Omega}(\gamma_2;L_2)$. Letting also $\alpha$ vary we
further increase the virtual dimension by $1$, and we find the formula 
\[
\mathrm{virdim}\,  \mathscr{M}_{\Upsilon}^K (\gamma_1,\gamma_2;x) =
i^{\Omega}(\gamma_1;L_1)+
i^{\Omega}(\gamma_2;L_2) - \mu^{\Omega}(x;H_1 \# H_2) + 1.
\]
This proves Proposition \ref{POmega}.

\paragraph{The space $\mathscr{M}^K_{\alpha_0}$.} Let $\gamma_1\in
\mathscr{P}^{\Lambda}(L_1)$, $\gamma_2\in
\mathscr{P}^{\Lambda}(L_2)$, and $x\in
\mathscr{P}^{\Theta}(H_1 \oplus H_2)$ (see Section \ref{chlhs}). The space $\mathscr{M}^K_{\alpha_0}(\gamma_1,\gamma_2;x)$ consists of solutions $u=(u_1,u_2)$ 
of the Floer equation on the Riemann surface $\Sigma^K_{\alpha_0}$, 
which is asymptotic to $x$ for $s\rightarrow +\infty$, $u_1$ and $u_2$
lie above some elements $q_1$ and $q_2$ in the unstable manifolds of
$\gamma_1$ and $\gamma_2$ for $s=0$, and $u$ satisfies the figure-8 
boundary condition for $s\geq \alpha_0$. Set $Q=M^2$, let $K$ be the Hamiltonian defined by (\ref{lakappa}), and let $y$ and $\gamma$ be as in (\ref{ygamma}). Then $y$ belongs to $\mathscr{P}^{\Delta_M^{(4)}}(K)$, and 
\[
\mu^{\Delta_M^{(4)}}(y) = \mu^{\Delta_M^{(4)}}(x) = \mu^{\Theta}(x).
\]
The space of $u_0\in \mathscr{M}^K_{\alpha_0}(\gamma_1,\gamma_2;x)$ which lie above $q_1$ and $q_2$ for $s=0$ is in one-to-one correspondence with the space $\mathscr{M}(\gamma,y)$ of Corollary \ref{corfred+}, where the boundary conditions jump at $s_1=\alpha_0$ and are given by
\[
R_0 = \Delta_{M^2}, \quad R_1 = \Delta_M^{(4)}.
\]
Such a space has virtual dimension
\[
\mathrm{virdim}\, \mathscr{M}(\gamma,y) = - \mu^{\Delta_M^{(4)}} (y) - ( \dim \Delta_{M^2} - \dim \Delta_{M^2} \cap \Delta_M^{(4)}) = - \mu^{\Theta}(x) - n.
\]
Letting the elements $q_1$ and $q_2$ of the unstable manifolds of
$\gamma_1$ and $\gamma_2$ vary, we increase the virtual dimension by
$i^{\Lambda}(\gamma_1;L_1)+i^{\Lambda}(\gamma_2;L_2)$, and
we find the formula  
\[
\mathrm{virdim}\,  \mathscr{M}^K_{\alpha_0} (\gamma_1,\gamma_2;x) =
i^{\Lambda}(\gamma_1;L_1)+i^{\Lambda}(\gamma_2;L_2) - \mu^{\Theta}(x)
- n.
\]
This proves Proposition \ref{Kappa}.

\paragraph{The space $\mathscr{M}_G^K$.} Let $\gamma\in
\mathscr{P}^{\Theta}(L_1\oplus L_2)$ and $x\in
\mathscr{P}^{\Lambda}(H_1 \# H_2)$ (see Section \ref{chrhs}). The space
$\mathscr{M}_G^K(\gamma,x)$ consists of pairs
$(\alpha,u)$ where $\alpha$ is a positive number and $u(s,t)$ is a solution
of the Floer equation on the Riemann surface
$\Sigma_G^K(\alpha)$, 
which is asymptotic to $x$ for $s\rightarrow +\infty$, 
lies above some element $q=(q_1,q_2)$ in the unstable manifold of
$\gamma$ for $s=0$, and satisfies the figure-8 boundary condition for
$s\in [0,\alpha]$. Set $Q=M^2$, let $K$ be as in (\ref{lakappa}), let $y$ be as in (\ref{ygamma}), and let $\tilde{\gamma}:[0,1]\rightarrow Q$ be the curve
\[
\tilde{\gamma}(t) := \bigl( q_1(1-t), q_2(t) \bigr).
\]
Then the space of elements $(\alpha,u)$ in $\mathscr{M}_G^K(\gamma,x)$ with $u$ above $q$ at $s=0$ is in one-to-one correspondence with the space $\mathscr{M}(\tilde{\gamma},y)$ of Corollary \ref{corfred+}, where the boundary conditions jump at $s_1=\alpha$ from $R_0=\Delta_M^{(4)}$ to $R_1= \Delta_M \times \Delta_M$. Such a space has virtual dimension
\begin{eqnarray*}
\mathrm{virdim}\, \mathscr{M}(\tilde{\gamma},y) = - \mu^{\Delta_M \times \Delta_M}(y) - \bigl( \dim \Delta_M^{(4)} - \dim \Delta_M^{(4)} \cap (\Delta_M \times \Delta_M) \bigr) \\ = - \mu^{\Lambda}(x).
\end{eqnarray*}
where we have used Lemma \ref{lemule}. Letting the elements $q$ of the unstable manifold of $\gamma$ vary, we increase the virtual dimension by
$i^{\Theta}(\gamma;L_1 \oplus L_2)$. Letting also $\alpha$ vary we
further increase it by $1$, and we find the formula 
\[
\mathrm{virdim}\,  \mathscr{M}_G^K (\gamma,x) =
i^{\Theta}(\gamma; L_1 \oplus L_2)- \mu^{\Lambda}(x;H_1 \# H_2) + 1.
\]
This proves Proposition \ref{PRLambda}.

\section{Compactness and cobordism}
\label{cocosec}

The first aim of this section is to explain how compactness and removal of singularities for Cauchy-Riemann problems can be obtained in the framework of cotangent bundles and conormal boundary conditions. The second aim is to prove the three already stated cobordism results (Propositions \ref{facthom}, \ref{omegahom}, and \ref{coupro}) which do not follow form standard arguments in Floer theory.

\subsection{Compactness in the case of jumping conormal boundary conditions}
\label{compsec}

Compactness in the $C^{\infty}_{\mathrm{loc}}$ topology of all the spaces of solutions of the Floer equation considered in this paper can be proved within the following general setting. Let $Q$ be a closed Riemannian manifold, and let $R_0,R_1,\dots,R_k$ be submanifolds of $Q \times Q$.  We assume that there is an isometric embedding $Q\hookrightarrow \R^N$ and linear subspaces $V_0,V_1,\dots,V_k$
of $\R^N \times \R^N$, such that $V_{j-1}$ is partially orthogonal to $V_j$, for every $j=1,\dots,k$, and
\[
R_j = V_j \cap (Q \times Q).
\]
The embedding $Q\hookrightarrow \R^N$ induces an embedding $T^* Q
\hookrightarrow T^* \R^N \cong \R^{2N} \cong \C^N$. Since the embedding $Q\hookrightarrow \R^N$ is isometric, the standard complex structure $J_0$ of $\R^{2N}$ restricts to the metric almost complex structure $J$ on $T^*Q$.
Let $H\in C^{\infty}([0,1]\times T^*Q)$ be a Hamiltonian
satisfying (H1) and (H2). Fix real numbers
\[
-\infty = s_0 < s_1 < \dots < s_k < s_{k+1} = +\infty,
\]
and let $u: \Sigma = \set{z\in \C}{0\leq \im z \leq 1} \rightarrow T^*Q$ be a solution of the Floer equation $\delbar_{J,H}(u)=0$ which
satisfies the nonlocal boundary conditions
\begin{equation}
\label{bdry3} (u(s),\mathscr{C} u(s+i)) \in N^* R_j \quad \forall
s\in [s_{j-1},s_j],
\end{equation}
for every $j=0,\dots,k$.

The map $u$ satisfies the energy identity
\begin{equation}
\label{enid}
\begin{split}
\int_a^b \int_0^1 |\partial_s u(s,t)|^2 \, dt\, ds \,=\,
&\mathbb{A}_H(u(a,\cdot)) - \mathbb{A}_H(u(b,\cdot)) \\ +
\int_{[a,b]} (u(\cdot,1)^* \eta - u(\cdot,0)^* \eta) 
= &\mathbb{A}_H(u(a,\cdot)) - \mathbb{A}_H(u(b,\cdot)),
\end{split}
\end{equation}
for every $a<b$, where the integral over $[a,b]$ vanishes because of the boundary conditions (\ref{bdry3}), thanks to the fact that
$\eta \oplus \eta$, that is the Liouville form on $T^* Q^2$, vanishes on $N^* R_j$. The following
result is proven in \cite[Lemma 1.12]{as06}  (in that lemma different
boundary conditions are considered, but the proof makes use only of
the energy identity (\ref{enid}) coming from those boundary conditions).

\begin{lem}
\label{illem}
For every $a>0$ there exists $c>0$ such that for every solution $u:\Sigma
\rightarrow T^*Q$ of $\delbar_{J,H} (u)=0$, with boundary conditions (\ref{bdry3}), and energy bound
\[
\int_{\Sigma} |\partial_s u (s,t)|^2 \, ds\, dt
\leq a,
\]
we have the following estimates:
\[
\|u\|_{L^2(I\times ]0,1[)} \leq c |I|^{1/2}, \quad \|\nabla
u\|_{L^2(I\times ]0,1[)} \leq c (1+|I|^{1/2}),
\]
for every interval $I$.
\end{lem}

The proof of the following result follows the argument of  
\cite[Theorem 1.14]{as06}, using the above lemma together with the elliptic estimates of Proposition \ref{czj}. 

\begin{prop}
\label{gencomp}
For every $a>0$ there exists $c>0$ such that for every solution $u:\Sigma
\rightarrow T^*Q$ of $\delbar_{J,H} (u) =0$, with boundary conditions (\ref{bdry3}), and energy bound
\[
\int_{\Sigma} |\partial_s u (s,t)|^2 \, ds\, dt
\leq a,
\]
we have the following uniform estimate:
\[
\|u\|_{L^{\infty}(\R\times ]0,1[)} \leq c.
\]
\end{prop}

\begin{proof} 
By using the above embedding, the equation $\delbar_{J,H} (u) =0$ can be rewritten as
\begin{equation}
\label{pg1}
\delbar  u = J_0 X_H(t,u).
\end{equation}
We can pass to local boundary conditions by considering the map $v: \Sigma \rightarrow T^*Q^2 \subset \C^{2N}$ defined by
\[
v(z) :=  \bigl( \mathscr{C} u((i-z)/2), u((i+z)/2)\bigr) = \bigl( \overline{u} ((i-z)/2), u((i+z)/2)\bigr).
\]
The map $v$ satisfies the boundary conditions
\begin{equation}
\label{pg2}
\begin{split}
v(s) \in N^* \Delta_Q \subset N^* \Delta_{\R^N}, \quad & \forall
s\in \R, \\
v(s+i) \in N^* R_j \subset N^* V_j, \quad & \mbox{ if } s\in [2s_{j-1},2s_j].
\end{split}
\end{equation}
Moreover,
\[
\delbar  v (z) = \frac{1}{2} \left( \overline{\delbar  u}
((i-z)/2) , \delbar u((i+z)/2) \right),
\]
so by (\ref{pg1}) and by the fact that $X_H(t,q,p)$ has quadratic
growth in $|p|$ by (\ref{cresc}), there is a constant $c$ such that
\begin{equation}
\label{pg3}
|\delbar  v (z)| \leq c (1+|v(z)|^2).
\end{equation}
Let $\chi$ be a smooth function such that $\chi(s)=1$ for $s\in
[0,1]$, $\chi(s)=0$ outside $[-1,2]$, and $0\leq \chi \leq 1$. Given
$h\in \Z$ set
\[
w(s+it) := \chi(s-h) v(s+it).
\]
Fix some $p>2$, and consider the norm $\|\cdot\|_{X^p}$ introduced in
Section \ref{sjlbc}, with $\mathscr{S}=\{2s_1,\dots, 2s_k\}$. The map $w$ has compact support and 
satisfies the boundary conditions (\ref{pg2}), so by Proposition
\ref{czj} we have the elliptic estimate
\[
\|\nabla w\|_{X^p} \leq c_0 \|w\|_{X^p} + c_1 \|\delbar  w
\|_{X^p}.
\]
Since
\[
\delbar  w = \chi'(s-h) v + \chi(s-h)\delbar  v
= \frac{\chi'}{\chi} (s-h) w  + \chi(s-h)\delbar  v,
\]
we obtain, together with (\ref{pg3}),
\begin{eqnarray*}
\|\nabla w\|_{X^p} \leq \left( c_0 + c_1 \|\chi'/\chi\|_{\infty}
\right) \|w\|_{X^p} + c_1 \| \chi(\cdot-h) \delbar  v
\|_{X^p} \\ 
\leq  \left( c_0 + c_1 \|\chi'/\chi\|_{\infty}
\right) \|w\|_{X^p} +
c_1 c \|\chi(\cdot-h) (1+|v|^2)\|_{X^p}.
\end{eqnarray*}
Therefore, we have an estimate of the form
\begin{equation}
\label{pg4}
\|\nabla w\|_{X^p} \leq a \|w\|_{X^p} + b \|\chi(\cdot-h)
(1+|v|^2)\|_{X^p}.
\end{equation}
Since $w$ has support in the set $[h-1,h+2]\times [0,1]$, we can
estimate its $X^p$ norm in terms of its $X^{1,2}$ norm, by Proposition
\ref{compemb}. The $X^{1,2}$ norm is equivalent to the $W^{1,2}$
norm, and the latter norm is bounded by Lemma \ref{illem}. We conclude
that $\|w\|_{X^p}$ is uniformly bounded. Similarly, the $X^p$ norm of $ 
\chi(\cdot-h)(1+|v|^2)$ is controlled by its $W^{1,2}$ norm, which is
also bounded because of Lemma \ref{illem}. Therefore, (\ref{pg4})
implies that $w$ is uniformly bounded in $X^{1,p}$. Since $p>2$, we deduce that $w$ is uniformly bounded in
$L^{\infty}$. The integer $h$ was arbitrary, hence we conclude that
$v$ is uniformly bounded in $L^{\infty}$, and so is $u$. 
\end{proof}

We conclude this section by discussing how the above result leads to $C^{\infty}_{\mathrm{loc}}$ compactness for the spaces of maps considered in this paper. We consider the model case of $\mathscr{M}_{\Upsilon}^{\Lambda}(x_1,x_2;y)$, the other cases being analogous. 

By the equivalent description of $\mathscr{M}_{\Upsilon}^{\Lambda}(x_1,x_2;y)$ of Section (\ref{lin}) (see in particular (\ref{lavu})), the space of maps we are considering fits into the above setting. Indeed, an isometric embedding of $M$ into $\R^N$ induces isometric embeddings of $Q=M^2$ into $\R^{2N}$ and of $Q^2$ into $\R^{4N}$, such that $\Delta_{M^2}$ and $\Delta_M \times \Delta_M$ are mapped into $Q^2 \cap \Delta_{\R^{2N}}$ and $Q^2 \cap (\Delta_{\R^N} \times \Delta_{\R^N})$, where the linear subspaces $\Delta_{\R^{2N}}$ and  $\Delta_{\R^N} \times \Delta_{\R^N}$ are partially orthogonal. Therefore, the energy estimate (\ref{seir}) and Proposition \ref{gencomp} imply that the elements of $\mathscr{M}_{\Upsilon}^{\Lambda}(x_1,x_2;y)$ have a uniform $L^{\infty}$ bound. 

For the remaining part of the argument leading to the $C^{\infty}_{\mathrm{loc}}$ compactness of $\mathscr{M}^{\Lambda}_{\Upsilon}(x_1,x_2;y)$ it is more convenient to use the original definition of this solutions space and the smooth structure of $\Sigma_{\Upsilon}^{\Lambda}$. Then the argument is absolutely standard: If by
contradiction there is no uniform $C^1$ bound, a concentration argument (see e.g. \cite[Theorem 6.8]{hz94}) produces a non-constant $J$-holomorphic sphere. However, there are no non-constant $J$-holomorphic spheres on cotangent bundles, because the symplectic form $\omega$ is exact. This contradiction proves the $C^1$ bound. Then the $C^k$ bounds for arbitrary $k$ follow from elliptic bootstrap, as in \cite[Section 6.4]{hz94}.   

Other solutions spaces, such as the space $\mathscr{M}_{\Upsilon}^{\Omega}$ for the triangle products, involve Riemann surfaces with boundary, and the solutions take value on some conormal subbundle of $T^*M$. In this case the concentration argument for proving the $C^1$ bound could produce a non-constant $J$-holomorphic disk with boundary on the given conormal subbundle. However, the Liouville one-form vanishes on conormal subbundles, so such $J$-holomorphic disks do not exist. Again we find a contradiction, leading to $C^1$ bounds and -- by elliptic bootstrap -- to $C^k$ bounds for every $k$.  

\subsection{Removal of singularities}

Removal of singularities results state that isolated singularities of
a $J$-holomorphic map with bounded energy can be removed (see for
instance \cite[Section 4.5]{ms04}). In Proposition \ref{remsing}
below, we prove a result of this sort for corner singularities. 
The fact that we are dealing with cotangent
bundles, which can be isometrically embedded into $\C^N$, allows to
reduce such a statement to the following easy linear result, where 
$\mathbb{D}_r$ is the open disk of radius $r$ in $\C$, and 
$\mathbb{H}^+$ is the upper right quadrant $\{\re z>0, \; \im z>0\}$.

\begin{lem}
\label{mnbline}
Let $V_0$ and $V_1$ be partially orthogonal linear subspaces of
$\R^n$.
Let $u:\mathrm{Cl}(\mathbb{D}_1\cap \mathbb{H}^+)
\setminus \{0\} \rightarrow \C^n$ be a smooth map such that
\[
u\in  L^p( \mathbb{D}_1 \cap \mathbb{H}^+ ,\C^n), \quad
\label{mnbstma2}
\delbar  u  \in L^p( \mathbb{D}_1 \cap \mathbb{H}^+ ,\C^n),
\]
for some $p>2$, and 
\[
u(s) \in N^* V_0 \quad \forall \; s>0, \quad u(it) \in N^* V_1 \quad
\forall \; t>0.
\]
Then $u$ extends to a continuous map on $\mathrm{Cl}(\mathbb{D}_1\cap
\mathbb{H}^+)$.
\end{lem}

\begin{proof}
Since $V_0$ and $V_1$ are partially orthogonal, by applying twice the 
Schwarz reflection argument of the proof of Lemma \ref{czl} we can
extend $u$ to a continuous map
\[
u: \mathbb{D}_1 \setminus \{0\} \rightarrow \C^n,
\]
which is smooth on $\mathbb{D}_1 \setminus ( \R \cup i\R)$, has finite
$L^p$ norm on $\mathbb{D}_1$, and satisfies
\[
\delbar  u \in L^p(\mathbb{D}_1).
\] 
Since $p>2$, the $L^2$ norm of $u$ on $\mathbb{D}_1$ is also finite,
and by the conformal change of variables $z=s+it=e^{\zeta} = e^{\rho+i
  \theta}$, this norm can be written as
\[
\int_{\mathbb{D}_1} |u(z)|^2\, ds \,dt = \int_{-\infty}^0 \int_0^{2\pi}
|u(e^{\rho + i \theta})|^2 e^{2\rho} \, d\theta d\rho.
\]
The fact that this quantity is finite implies that there is a 
sequence $\rho_h \rightarrow -\infty$ such
that, setting $\epsilon_h := e^{\rho_h}$, we have
\begin{equation}
\label{tendea0}
\lim_{h\rightarrow \infty} \epsilon_h^2 \int_0^{2\pi} |u(\epsilon_h
e^{i\theta})|^2 d\theta = \lim_{h\rightarrow \infty}
e^{2\rho_h} \int_0^{2\pi} |u(e^{\rho_h + i \theta})|^2 \, d\theta
= 0.
\end{equation}
If $\varphi\in C^{\infty}_c (\mathbb{D}_1,\C^N)$, an integration by parts 
using the Gauss formula leads to
\begin{eqnarray*}
\int_{\mathbb{D}_1} \langle u , \partial \varphi \rangle \, ds\,dt =   
\int_{\mathbb{D}_{\epsilon_h}} \langle u , \partial \varphi \rangle \, ds\,dt
+ \int_{\mathbb{D}_1 \setminus \mathbb{D}_{\epsilon_h}} \langle u ,
\partial \varphi \rangle \, ds\,dt \\ 
= \int_{\mathbb{D}_{\epsilon_h}} \langle u , \partial \varphi
\rangle \, ds\,dt - \int_{\mathbb{D}_1 \setminus
  \mathbb{D}_{\epsilon_h}} \langle \delbar  u , \varphi
\rangle \, ds\,dt + i \int_{\partial \mathbb{D}_{\epsilon_h}} \langle u,
\varphi \rangle \, dz.
\end{eqnarray*}
Since $u$ and $\delbar  u$ are integrable over
$\mathbb{D}_1$, the the first integral in the latter
expression tends to zero, while the second one tends to
\[
- \int_{\mathbb{D}_1} \langle \delbar  u , \varphi
\rangle \, ds\,dt.
\]
As for the last integral, we have
\[
\int_{\partial \mathbb{D}_{\epsilon_h}} \langle u,
\varphi \rangle \, dz = i \epsilon_h \int_0^{2\pi} \langle
u(\epsilon_h e^{i\theta}), \varphi(\epsilon_h e^{i\theta}) \rangle
e^{i\theta} \, d\theta,
\]
so by the Cauchy-Schwarz inequality,
\begin{eqnarray*}
\left|\int_{\partial \mathbb{D}_{\epsilon_h}} \langle u,
\varphi \rangle \, dz \right| \leq \epsilon_h \left( \int_0^{2\pi}
|u(\epsilon_h e^{i\theta})|^2 \, d\theta \right)^{1/2} \left( \int_0^{2\pi}
|\varphi(\epsilon_h e^{i\theta})|^2 \, d\theta \right)^{1/2} \\ \leq
\sqrt{2\pi} \, \epsilon_h \left( \int_0^{2\pi}
|u(\epsilon_h e^{i\theta})|^2 \, d\theta \right)^{1/2}
\|\varphi\|_{\infty}.
\end{eqnarray*}
Then (\ref{tendea0}) implies that the latter quantity tends to zero
for $h\rightarrow \infty$. Therefore,
\[
\int_{\mathbb{D}_1} \langle u , \partial \varphi \rangle \, ds\,dt =   
- \int_{\mathbb{D}_1} \langle \delbar  u , \varphi
\rangle \, ds\,dt,
\]
for every test function $\varphi\in
C^{\infty}_c(\mathbb{D}_1,\C^n)$. Since $\delbar  u \in L^p$,
by the regularity theory of the weak
solutions of $\delbar $ (see Theorem \ref{rws} (i)), $u$ belongs
to $W^{1,p}(\mathbb{D}_1,\C^n)$. Since $p>2$, we conclude that $u$ is
continuous at $0$.
\end{proof}

Let $R_0$ and $R_1$ be closed submanifolds of $Q$, and assume that there
is an isometric embedding $Q\hookrightarrow \R^N$ such that
\[
R_0 = Q \cap V_0, \quad R_1 \cap V_1,
\]
where $V_0$ and $V_1$ are partially orthogonal linear subspaces of $\R^N$.

\begin{prop}
\label{remsing}
Let $X: \mathbb{D}_1\cap \mathbb{H}^+ 
\times T^*Q \rightarrow TT^*Q$ be a smooth vector
field such that $X(z,q,p)$ grows at most polynomially in $p$,
uniformly in $(z,q)$. Let $u:\mathrm{Cl}(\mathbb{D}_1\cap \mathbb{H}^+)
\setminus \{0\} \rightarrow T^*Q$ be a smooth solution of the equation
\begin{equation}
\label{mnbeq}
\delbar_J (u) (z) = X(z,u(z)) \quad \forall z\in
\mathrm{Cl}(\mathbb{D}_1\cap \mathbb{H}^+)\setminus \{0\},
\end{equation}
such that
\[
u(s) \in N^* R_0 \quad \forall \; s>0, \quad u(it) \in N^* R_1 \quad
\forall \; t>0.
\]
If $u$ has finite energy,
\begin{equation}
\label{mnben}
\int_{\mathbb{D}_1 \cap \mathbb{H}^+} |\nabla u|^2 \, ds\,dt < +\infty,
\end{equation}
then $u$ extends to a continuous map on $\mathrm{Cl}(\mathbb{D}_1\cap
\mathbb{H}^+)$.
\end{prop}

\begin{proof}
By means of the above isometric embedding, we may regard $u$ as a
$\C^N$-valued map, satisfying the equation (\ref{mnbeq}) with
$\delbar_J = \delbar $, the energy estimate
(\ref{mnben}), and the boundary condition
\[
u(s) \in N^* V_0 \quad \forall \; s>0, \quad u(it) \in N^* V_1 \quad
\forall \; t>0.
\]
By the energy estimate (\ref{mnben}), $u$ belongs to $L^p(\mathbb{D}_1
\cap \mathbb{H}^+,\C^N)$ for every $p<+\infty$: for instance, 
this follows from the Poincar\'e inequality and the Sobolev embedding 
theorem on $\mathbb{D}_1$, after applying
a Schwarz reflection twice and after multiplying by a cut-off function
vanishing on $\partial \mathbb{D}_1$ and equal to $1$ on a
neighborhood of $0$. The polynomial growth of $X$ then implies that 
\begin{equation}
\label{mnblp}
X(\cdot,u(\cdot)) \in L^p(\mathbb{D}_1
\cap \mathbb{H}^+,\C^N) \quad \forall p<+\infty.
\end{equation}
Therefore, Lemma \ref{mnbline} implies that $u$ extends to a
continuous map on $\mathrm{Cl}(\mathbb{D}_1\cap \mathbb{H}^+)$.
\end{proof}

The corresponding statement for jumping conormal boundary conditions
is the following:

 \begin{prop}
\label{remsing2}
Let $X: \mathbb{D}_1\cap \mathbb{H} 
\times T^*Q \rightarrow TT^*Q$ be a smooth vector
field such that $X(z,q,p)$ grows at most polynomially in $p$,
uniformly in $(z,q)$. Let $u:\mathrm{Cl}(\mathbb{D}_1\cap \mathbb{H})
\setminus \{0\} \rightarrow T^*Q$ be a smooth solution of the equation
\[
\delbar_J (u) (z) = X(z,u(z)) \quad \forall z\in
\mathrm{Cl}(\mathbb{D}_1\cap \mathbb{H})\setminus \{0\},
\]
such that
\[
u(s) \in N^* R_0 \quad \forall \; s>0, \quad u(s) \in N^* R_1 \quad
\forall \; s<0.
\]
If $u$ has finite energy,
\[
\int_{\mathbb{D}_1 \cap \mathbb{H}} |\nabla u|^2 \, ds\,dt < +\infty,
\]
then $u$ extends to a continuous map on the closed half-disk 
$\mathrm{Cl}(\mathbb{D}_1\cap \mathbb{H})$.
\end{prop}

\begin{proof}
The energy is invariant with respect to conformal changes of
variable. Therefore, it is enough to apply Proposition \ref{remsing}
to the map $v(z)=u(z^2)$, with $z\in \mathbb{D}_1\cap \mathbb{H}^+$.
\end{proof}

\subsection{Proof of Proposition \ref{facthom}}
\label{factsec}

Let $x_1\in \mathscr{P}^{\Lambda}(H_1)$, $x_2\in \mathscr{P}^{\Lambda}(H_2)$, and $z\in \mathscr{P}^{\Lambda}(H_1 \# H_2)$ be such that
\begin{equation}
\label{dimfor}
\mu^{\Lambda}(x_1) + \mu^{\Lambda}(x_2) - \mu^{\Lambda}(z) = n,
\end{equation}
so that the manifold $\mathscr{M}_{GE}^{\Upsilon}(x_1,x_2;z)$ is one-dimensional. 
By standard arguments, Proposition \ref{facthom} is implied by the following two statements:
\begin{enumerate}
\item for every $y\in \mathscr{P}^{\Theta}(H_1 \oplus H_2)$ such that
\[
\mu^{\Theta}(y) = \mu^{\Lambda}(z) = \mu^{\Lambda}(x_1) + \mu^{\Lambda}(x_2) - n,
\]
and every pair $(u_1,u_2)$ with $u_1\in \mathscr{M}_E(x_1,x_2;y)$ and $u_2\in \mathscr{M}_G(y,z)$, there is a unique connected component of $\mathscr{M}_{GE}^{\Upsilon}(x_1,x_2;z)$ containing a curve  $\alpha \mapsto (\alpha,u_{\alpha})$ which -- modulo translations in the $s$ variable -- converges to $(+\infty,u_1)$ and to $(+\infty,u_2)$;
\item for every $u\in \mathscr{M}_{\Upsilon}^{\Lambda}(x_1,x_2;z)$, there is a unique connected component of $\mathscr{M}_{GE}^{\Upsilon}(x_1,x_2;z)$ containing a curve $\alpha \mapsto (\alpha,u_{\alpha})$ which converges to $(0,u)$. 
\end{enumerate}
The first statement follows from standard gluing arguments. Here we prove the second statement, by reducing it to an implicit function type
argument. At first the difficulty consists in the parameter dependence
of the underlying domain for the elliptic PDE. By using the special form
of the occurring conormal type boundary conditions and a suitable
localization argument, we equivalently translate this parameter
dependence into a continuous family of elliptic operators with fixed
boundary conditions.

If $(\alpha,v)\in \mathscr{M}_{GE}^{\Upsilon}(x_1,x_2;z)$, we define the map
\[
u: \Sigma = \set{z\in \C}{0\leq \im z \leq 1} \rightarrow T^* M^2, \quad 
u(z) := \bigl(\mathscr{C} v(\overline{z}),v(z)\bigr),
\]
and the Hamiltonian $K$ on $\T \times T^* M^2$ by
\[
K(t,x_1,x_2) := H_1(-t,\mathscr{C} x_1) + H_2(t,x_2).
\]
By this identification, we can view the space $\mathscr{M}_{GE}^{\Upsilon}(x_1,x_2;z)$ as the space of pairs $(\alpha,u)$, where $\alpha>0$ and 
$u:\Sigma\rightarrow T^* M^2$ solves the Floer equation 
\begin{equation}
\label{nnnfl}
\delbar_{J,K} (u)=0,
\end{equation} 
with nonlocal boundary conditions 
\begin{equation}
\label{facto}
(u(s),\mathscr{C} u(s+i)) \in \left\{ \begin{array}{ll} N^* \Delta_{M^2} & \mbox{if } s\leq 0, \\ N^* \Delta^{(4)}_M & \mbox{if } 0 \leq s \leq \alpha, \\ N^* (\Delta_M \times \Delta_M) & \mbox{if } s\geq \alpha, \end{array} \right.
\end{equation}
and asymptotics
\begin{equation}
\label{aym}
\lim_{s\rightarrow -\infty} u(s+it) = \bigl( \mathscr{C} x_1(-t),x_2(t) \bigr), \quad 
\lim_{s\rightarrow +\infty} u(s+it) = \bigl( \mathscr{C} z((1-t)/2),z((1+t)/2) \bigr).
\end{equation}
Similarly, we can view the space $\mathscr{M}^{\Lambda}_{\Upsilon}(x_1,x_2;z)$ as the space of maps $u: \Sigma \rightarrow T^* M^2$ solving the equation (\ref{nnnfl}) with asymptotics (\ref{aym}) and nonlocal boundary conditions
\begin{equation}
\label{pro}
(u(s),\mathscr{C}u(s+i)) \in \left\{ \begin{array}{ll} N^* \Delta_{M^2} & \mbox{if } s\leq 0, \\ N^* (\Delta_M \times \Delta_M) & \mbox{if } s\geq 0. \end{array} \right.
\end{equation}

\paragraph{Compactness.} We start with the following compactness results, which also clarifies the sense of the convergence in (ii):

\begin{lem}
\label{fp-cpt}
Let $(\alpha_h,u_h)$ be a sequence in $\mathscr{M}^{\Upsilon}_{GE} (x_1,x_2;z)$ with $\alpha_h \rightarrow 0$. Then there exists $u_0 \in \mathscr{M}^{\Lambda}_{\Upsilon}(x_1,x_2;z)$ such that, up to a subsequence, $u_h$ converges to $u_0$ in $C^{\infty}_{\mathrm{loc}}(\Sigma \setminus \{0,i\})$, in $C^{\infty}(\Sigma\cap \{|\re z|>1\})$, and uniformly on $\Sigma$.
\end{lem}

\begin{proof}
Since the sequence of maps $(u_h)$ has uniformly bounded energy, Proposition \ref{gencomp} implies a uniform $L^{\infty}$ bound. Then, the usual non-bubbling-off analysis for interior points and boundary points away from the jumps in the boundary condition
implies that, modulo subsequence, we have
\[
u_h\to u_0 \quad \mbox{in } C^{\infty}_{\mathrm{loc}} (\Sigma\setminus \{0,i\}, T^* M^2),
\]
where $u_0$ is a smooth solution of equation (\ref{nnnfl}) on $\Sigma\setminus \{0,i\}$ with bounded energy and satisfying the boundary conditions (\ref{pro}), except possibly at $0$ and $i$. By Proposition \ref{remsing2}, the singularities $0$ and $i$ are removable, and $u_0$ satisfies the boundary condition also at $0$ and $i$. By the index formula (\ref{dimfor}) and transversality, the sequence $u_h$ cannot split, so $u_0$ satisfies also the asymptotic conditions (\ref{aym}), and $u_h \to u_0$ in $C^{\infty}(\Sigma\cap \{|\re z|>1\})$. Therefore, $u_0$ belongs to $\mathscr{M}^{\Lambda}_{\Upsilon}(x_1,x_2;z)$, and there remains to prove that $u_h \to u$ uniformly on $\Sigma$. 

We assume by contraposition that $(u_h)$ does not converge
uniformly on $\Sigma$. By Ascoli-Arzel\`a theorem, there must be some blow-up of the gradient. That is, modulo subsequence, we can find $z_h\in \Sigma$ converging either to $0$ or to $i$ such that
\[
R_h :=|\nabla u_h(z_h)| = \|\nabla u_h\|_{\infty} \to\infty.
\]
For sake of simplicity, we only consider the case where $z_h=(s_h,0) \rightarrow 0$, $0<s_h<\alpha_h\to 0$. The general
case follows along analogous arguments using additional standard
bubbling-off arguments. For more details, see e.g. \cite[Section 6.4]{hz94}.

We now have to make a case distinction concerning the behavior of the quantity $0<R_h \alpha_h<\infty$:
\begin{enumerate}

\item[(a)] The case of a diverging subsequence $R_{h_j}
\alpha_{h_j}\to\infty$ can be handled by conformal rescaling
$v_j(s,t):=u_{h_j}(s_{h_j} + s/R_{h_j},t/R_{h_j})$ which provides us with a finite energy disk with boundary on a single Lagrangian submanifold of conormal type. This has to be constant due to the vanishing of the Liouville 1-form on conormals, contradicting the
convergence of $|\nabla v_j(0)|=1$.

\item[(b)] The case of convergence of a subsequence $R_{h_j}
\alpha_{h_j}\to 0$ can be dealt with by rescaling
$v_j(s,t):=u_{h_j}(s_{h_j} + \alpha_{h_j},\alpha_{h_j} t)$. Now $v_k$ has to converge uniformly on compact subsets towards a constant map, since $\|\nabla v_j\|_\infty=|\nabla v_j(0)|=R_{h_j} \alpha_{h_j}\to
0$. This in particular implies that $u_{h_j}(\cdot,0)_{|[0,\alpha_{h_j}]}$ converges uniformly to a point contradicting the contraposition assumption.

\item[(c)] It remains to study the case $R_h \alpha_h\to
c>0$. Again we rescale $v_h(s,t)=u_h(\alpha_h s,\alpha_h t)$,
which now has to converge to a non-constant $J$-holomorphic map $v$ on the upper half plane. After applying a suitable conformal coordinate
change and transforming the nonlocal boundary conditions into local
ones, we can view $v$ as a map on the half strip $v\colon \Sigma^+ \to T^* M^4$, satisfying the boundary conditions
\begin{eqnarray*}
     v(it) \in & N^*\Delta^{(4)}_M & \mbox{for } t\in [0,1],\\
     v(s) \in & N^* (\Delta_M \times \Delta_M) & \mbox{for } s\geq 0,\\
     v(s+i) \in & N^* \Delta_{M^2} & \mbox{for } s\geq 0. 
\end{eqnarray*}
Applying again the removal of singularities for $s\to\infty$, we obtain
$v$ as a $J$-holomorphic triangle with boundary on three conormals.
Hence, $v$ would have to be constant, contradicting again the
rescaling procedure.

\end{enumerate}

This shows the uniform convergence of a subsequence of $(u_h)$.
\end{proof}

\paragraph{Localization.} It is convenient to transform the nonlocal boundary conditions (\ref{facto}) and (\ref{pro}) into local boundary conditions, by the usual method of doubling the space: Given $u: \Sigma \rightarrow T^* M^2$ we define $\tilde{u} : \Sigma \rightarrow T^* M^4$ as
\[
\tilde{u} (z) := \bigl( u(z/2), \mathscr{C} u(i + \overline{z}/2) \bigr).
\]
Then $u$ solves $\delbar_{J,K} (u)=0$ if and only if $\tilde{u}$ solves the equation 
\begin{equation}
\label{ooofl}
\delbar_{J,\tilde{K}} (\tilde{u}),
\end{equation}
with upper boundary condition
\begin{equation}
\label{common}
\tilde{u}(s+i) \in N^* \Delta_{M^2} \quad \forall s\in \R,
\end{equation}
where the Hamiltonian $\tilde{K}:[0,1] \times T^* M^4 \to \R$ is defined by 
\[
\tilde{K}(t,x_1,x_2,x_3,x_4) := \frac{1}{2} K\left(\frac{t}{2},x_1,x_2\right) + \frac{1}{2} K\left(1-\frac{t}{2},\mathscr{C}x_3,\mathscr{C}x_4\right).
\] 
Moreover, $u$ satisfies (\ref{facto}) if and only if $\tilde{u}$ satisfies
\begin{equation}
\label{facto1}
\tilde{u}(s) \in \left\{ \begin{array}{ll} N^* \Delta_{M^2} & \mbox{if } s\leq 0, \\ N^* \Delta_{M}^{\Theta} & \mbox{if } 0\leq s \leq 2\alpha, \\ N^* (\Delta_M \times \Delta_M) & \mbox{if } s\geq 2\alpha, \end{array} \right.
\end{equation}
whereas $u$ satisfies (\ref{pro}) if and only if $\tilde{u}$ satisfies 
\begin{equation}
\label{pro1}
\tilde{u}(s) \in \left\{ \begin{array}{ll} N^* \Delta_{M^2} & \mbox{if } s\leq 0, \\  N^* (\Delta_M \times \Delta_M), & \mbox{if } s\geq 0.\end{array} \right.
\end{equation}
Finally, the asymptotic condition (\ref{aym}) is translated into
\begin{equation}
\label{aym1}
\begin{split}
\lim_{s\rightarrow -\infty} \tilde{u}(s+ti) &= \bigl(\mathscr{C} x_1(-t/2),x_2(t/2),x_1(t/2-1),\mathscr{C} x_2(1-t/2) \bigr), \\
\lim_{s\rightarrow +\infty} \tilde{u}(s+ti) &= \bigl(\mathscr{C} z(1/2-t/4),z(1/2+t/4),z(t/4),\mathscr{C} z(1-t/4)\bigr). \end{split}
\end{equation}

Let $u_0 \in \mathscr{M}^{\Lambda}_{\Upsilon}(x_1,x_2;z)$. We must prove that there exists a unique connected component of $\mathscr{M}_{GE}^{\Upsilon}(x_1,x_2;z)$ containing a curve $\alpha \mapsto (\alpha,u_{\alpha})$ which converges to $(0,u_0)$, in the sense of Lemma \ref{fp-cpt}. 

Let $\tilde{u}_0$ be the map from $\Sigma$ to $T^* M^4$ associated to $u_0$: $\tilde{u}_0$ solves (\ref{ooofl}) with boundary conditions (\ref{common}), (\ref{pro1}), and asymptotic conditions (\ref{aym1}). Since we are looking for solutions which converge to $\tilde{u}_0$ uniformly on $\Sigma$, we may localize the problem and assume that $M=\R^n$. More precisely, if the projection of $\tilde{u}_0(z)$ onto $M^4$ is $(q_1,q_2,q_3,q_4)(z)$, we construct open embeddings
\[
\Sigma \times \R^n \rightarrow \Sigma \times M, \quad (z,q) \mapsto (z,\varphi_j(z,q)), \quad j=1,\dots,4,
\]
such that $\varphi_j(z,0) = q_j(z)$ and $D_2 \varphi_j(z,0)$ is an isometry, for every $z\in \Sigma$ (for instance, by composing an isometric trivialization of $q_j^*(TM)$ by the exponential mapping). The induced open embeddings
\begin{equation*}\begin{split}
&\Sigma \times T^* \R^n \rightarrow \Sigma \times T^* M,\\ &(z,q,p) \mapsto (z,\psi_j(z,q,p)) :=
\bigl( z,\varphi_j(z,q), (D_2\varphi_j(z,q)^*)^{-1} p \bigr), \quad j=1,\dots,4,
\end{split}\end{equation*}
are the components of the open embedding
\[
\Sigma \times T^* \R^{4n} \rightarrow \Sigma \times T^* M^4, \quad (z,\xi) \mapsto (z,\psi(z,\xi)) := \bigl( z,\psi_1(z,\xi_1), \dots, \psi_4(z,\xi_4) \bigr).
\]
Such an embedding allows us to associate to any $\tilde{u}: \Sigma \rightarrow T^*M^4$ which is $C^0$-close to $\tilde{u}_0$ a map $w:\Sigma \rightarrow T^* \R^{4n} = \C^{4n}$, by setting
\[
\tilde{u}(z) = \psi(z,w(z)).
\]
Then $\tilde{u}$ solves (\ref{ooofl}) if and only if $w$ solves an equation of the form
\begin{equation}
\label{locfl}
\mathscr{D}(w) := \partial_s w (z) + J(z,w(z)) \partial_t w(z) + G(z,w(z)) = 0,
\end{equation}
where $J$ is an almost complex structure on $\C^{4n}$ parametrized on $\Sigma$ and such that $J(z,0)=J_0$ for any $z\in \Sigma$, whereas $G: \Sigma \times \C^{4n} \to \C^{4n}$ is such that $G(z,0)=0$ for any $z\in \Sigma$. Moreover, $\tilde{u}$ solves the asymptotic conditions (\ref{aym1}) if and only if $w(s,t)$ tends to $0$ for $s\to \pm \infty$.
The maps $\psi_j(z,\cdot)$ preserve the Liouville form, so they map conormals into conormals. It easily follows that the boundary condition (\ref{common}) on $\tilde{u}$ is translated into
\begin{equation}
\label{common2}
w(s+i) \in N^* \Delta_{\R^{2n}} \quad \forall s\in \R.
\end{equation}
Moreover, $\tilde{u}$ satisfies the boundary condition (\ref{facto1}) if and only if $w$ satisfies 
\begin{equation}
\label{facto2}
w(s) \in \left\{ \begin{array}{ll} N^* \Delta_{\R^{2n}} & \mbox{if } s\leq 0, \\ N^* \Delta^{(4)}_{\R^n} & \mbox{if } 0 \leq s \leq 2\alpha, \\ N^* (\Delta_{\R^n} \times \Delta_{\R^n}) & \mbox{if } s\geq 2\alpha. \end{array} \right.
\end{equation}
Similarly, $\tilde{u}$ satisfies the boundary condition (\ref{pro1}) if and only if $w$ satisfies 
\begin{equation}
\label{pro2}
w(s) \in \left\{ \begin{array}{ll} N^* \Delta_{\R^{2n}} & \mbox{if } s\leq 0, \\ N^* (\Delta_{\R^n} \times \Delta_{\R^n}) & \mbox{if } s\geq 0. \end{array} \right.
\end{equation}
The element $u_0\in \mathscr{M}^{\Lambda}_{\Upsilon}(x_1,x_2;z)$ corresponds to the solution $w_0=0$ of (\ref{locfl})-(\ref{pro2}). By using the functional setting introduced in Section \ref{lineartheory}, we can view the nonlinear operator $\mathscr{D}$ defined in (\ref{locfl}) as a continuously differentiable operator 
\[
\mathscr{D} : X^{1,p}_{\mathscr{S},\mathscr{V}, \mathscr{V}'} (\Sigma,\C^{4n}) \rightarrow X^p_{\mathscr{S}} (\Sigma,\C^{4n}) 
\]
where $\mathscr{S}:=\{0\}$, $\mathscr{V}:=(\Delta_{\R^{2n}}, \Delta_{\R^n} \times \Delta_{\R^n})$, $\mathscr{V}':=(\Delta_{\R^{2n}})$, and $p$ is some number larger than $2$. Since $J(z,0)=J_0$, the differential of $\mathscr{D}$ at $w_0=0$ is a linear operator of the kind studied in Section \ref{lineartheory}, and by the transversality assumption it is an isomorphism.

Consider the orthogonal decomposition
\[
\R^{4n} = W_1 \oplus W_2 \oplus W_3 \oplus W_4,
\]
where
\[
W_1 := \Delta^{(4)}_{\R^n} = \Delta_{\R^{2n}} \cap (\Delta_{\R^n} \times \Delta_{\R^n}), \quad \Delta_{\R^{2n}} = W_1 \oplus W_2, \quad \Delta_{\R^n} \times \Delta_{\R^n} = W_1 \oplus W_3,
\]
and denote by $P_j$ the the orthogonal projection of $\C^{4n}$ onto $N^* W_j=W_j \oplus i W_j^{\perp}$.   If $\mathscr{T}_{\alpha}$ is the translation operator mapping some $w:\Sigma \to \C^{4n}$ into
\[
(\mathscr{T}_{\alpha} w)(z) := \bigl(P_1 w(z), P_2 w(z), P_3 w(z- 2\alpha), P_4 w(z) \bigr),
\]
we easily see that $w$ satisfies the boundary conditions (\ref{pro2}) if and only if $\mathscr{T}_{\alpha} w$ satisfies the boundary conditions (\ref{facto2}). Therefore, if we define the operator
\[
\mathscr{D}_{\alpha} : X^{1,p}_{\mathscr{S},\mathscr{V}, \mathscr{V}'} (\Sigma,\C^{4n}) \rightarrow X^p_{\mathscr{S}} (\Sigma,\C^{4n}), \quad \mathscr{D}_{\alpha} = \mathscr{D} \circ \mathscr{T}_{\alpha},
\]
we have that $w\in X^{1,p}_{\mathscr{S},\mathscr{V}, \mathscr{V}'} (\Sigma,\C^{4n})$ solves $\mathscr{D}_{\alpha}(w)=0$ if and only if $\mathscr{T}_{\alpha} w$ is a solution of (\ref{locfl}) satisfying the boundary conditions (\ref{common2}) and (\ref{facto2}). The operator 
\[
[0,+\infty[ \times X^{1,p}_{\mathscr{S},\mathscr{V}, \mathscr{V}'} (\Sigma,\C^{4n}) \rightarrow X^p_{\mathscr{S}} (\Sigma,\C^{4n}), \quad (\alpha,w) \mapsto \mathscr{D}_{\alpha} (w),
\]
is continuous on the product, it is continuously differentiable with respect to the second variable, and this partial differential is continuous on the product. Moreover, $D\mathscr{D}_0(0) = D \mathscr{D}(0)$ is an isomorphism, so the parametric inverse mapping theorem implies that there are a number $\alpha_0>0$ and a neighborhood $\mathscr{U}$ of $0$ in $X^{1,p}_{\mathscr{S},\mathscr{V}, \mathscr{V}'} (\Sigma,\C^{4n})$, such that the set of zeroes in $[0,\alpha_0[ \times \mathscr{U}$ of the above operator consists of a continuous curve $[0,\alpha_0[ \ni \alpha \to (\alpha,w_{\alpha})$ starting at $w_0=0$. Then $\alpha \to(\alpha, \mathscr{T}_{\alpha} w_{\alpha})$ provides us with the unique curve in $\mathscr{M}^{\Upsilon}_{GE}(x_1,x_2;z)$ converging to $(0,u_0)$. This concludes the proof of Proposition \ref{facthom}.

\subsection{Proof of Proposition \ref{omegahom}}
\label{omegasec}

Fix some $\gamma_1\in \mathscr{P}^{\Omega}(L_1)$, $\gamma_2\in\mathscr{P}^{\Omega}(L_2)$, and $x\in \mathscr{P}^{\Omega}(H_1 \# H_2)$ such that
\[
i^{\Omega}(\gamma_1;L_1) + i^{\Omega}(\gamma_2;L_2) - \mu^{\Omega}(x;H_1 \# H_2) = 0.
\]
By a standard argument in Floer homology, the claim that $P^K_{\Upsilon}$ is a chain homotopy between $K^{\Omega}$ and $\Upsilon^{\Omega} \circ ( \Phi_{L_1}^{\Omega} \otimes \Phi_{L_2}^{\Omega})$ is implied by the following statements:
\begin{enumerate}

\item For every $(u_1,u_2)\in \mathscr{M}^{\Omega}_{\Phi}(\gamma_1,y_1) \times \mathscr{M}^{\Omega}_{\Phi}(\gamma_2,y_2)$ and every $u \in \mathscr{M}^{\Omega}_{\Upsilon}(y_1,y_2;x)$, where $(y_1,y_2)\in \mathscr{P}^{\Omega}(H_1) \times \mathscr{P}^{\Omega}(H_2)$ is such that
\[
\mu^{\Omega}(y_1;H_1) +  \mu^{\Omega}(y_2;H_2) = \mu^{\Omega}(x;H_1 \# H_2),
\]
there exists a unique connected component of $\mathscr{M}^K_{\Upsilon}(\gamma_1,\gamma_2;x)$ containing a curve $(\alpha,u_{\alpha})$ such that $\alpha\rightarrow +\infty$,
$u_{\alpha}(\cdot,\cdot-1)$ and $u_{\alpha}$ converge to $u_1$ and $u_2$ in $C^{\infty}_{\mathrm{loc}}([0,+\infty[ \times [0,1],T^*M)$, while $u_{\alpha}(\cdot+\sigma(\alpha),2\cdot-1)$ converges to $u$ in $C^{\infty}_{\mathrm{loc}}(\R \times [0,1],T^*M)$, for a suitable function $\sigma$ diverging at $+\infty$.

\item For every $u\in \mathscr{M}_K^{\Omega}(\gamma_1,\gamma_2;x)$ there exists a unique connected component of $\mathscr{M}^K_{\Upsilon}(\gamma_1,\gamma_2;x)$ containing a curve $\alpha \mapsto (\alpha,u_{\alpha})$ which converges  to $(0,u)$.

\end{enumerate}

Statement (i) can be proved by the standard gluing arguments in Floer theory. Here we prove statement (ii). 

Given $u:[0,+\infty[ \times [-1,1] \rightarrow T^*M$, 
we define the map
\[
\tilde{u} : \Sigma^+ = \set{z\in \C}{\re z\geq 0, \; 0\leq \im z\leq 1} \rightarrow T^*M^2, \quad
\tilde{u}(z) := \bigl(\mathscr{C}u(\overline{z}),u(z)\bigr).
\]
If we define $\tilde{x}:[0,1]\rightarrow T^*M^2$ and $\tilde{H}\in C^{\infty}([0,1]\times T^*M^2)$ by
\[
\tilde{x}(t) := \bigl(\mathscr{C} x((1-t)/2),x((1+t)/2)\bigr), \quad \tilde{H}(t,x_1,x_2) := H_1(1-t,\mathscr{C} x_1) + H_2(t,x_2),
\]
we see that associating $\tilde{u}$ to $u$ produces a one-to-one correspondence between $\mathscr{M}^{\Omega}_K(\gamma_1,\gamma_2 ;x)$ and the space $\widetilde{\mathscr{M}}^{\Omega}_K(\gamma_1,\gamma_2 ;x)$ consisting of the maps 
$\tilde{u}: \Sigma^+ \rightarrow T^*M^2$ solving $\delbar_{J,\tilde{H}} (\tilde{u}) = 0$ with boundary conditions
\begin{eqnarray}
\label{bbcc1}
\tilde{u}(s) \in N^* \Delta_M \quad \forall s\geq 0, \\
\label{bbcc2}
\tilde{u}(s+i) \in T^*_{q_0} M \times T^*_{q_0} M \quad \forall s\geq 0, \\
\label{bbcc3}
\pi\circ \tilde{u}_1(0,1-\cdot) \in W^u(\gamma_1), \quad 
\pi\circ \tilde{u}_2(0,\cdot) \in W^u(\gamma_2), \\
\label{bbcc4}
\lim_{s\rightarrow +\infty} \tilde{u}(s+it) = \tilde{x}(t). 
\end{eqnarray}
Similarly, we have a one-to-one correspondence between $\mathscr{M}^K_{\Upsilon}(\gamma_1,\gamma_2;x)$ and the space $\widetilde{\mathscr{M}}^K_{\Upsilon}(\gamma_1,\gamma_2;x)$ consisting of pairs $(\alpha,\tilde{u})$ where $\alpha$ is a positive number and $\tilde{u}: \Sigma^+ \rightarrow T^*M^2$  is a solution of the problem above, with (\ref{bbcc1}) replaced by
\begin{equation}
\label{bbcc5}
\tilde{u}(s) \in T^*_{q_0} M \times T^*_{q_0} M \quad \forall s\in [0,\alpha], \quad \tilde{u}(s) \in N^* \Delta_M \quad \forall s\geq \alpha.
\end{equation}

Fix some $\tilde{u}^0\in \widetilde{\mathscr{M}}_K^{\Omega}(\gamma_1,\gamma_2;x)$. Since we are looking for solutions near $\tilde{u}^0$, we can localize the problem as follows.
Let $k = i^{\Omega}(\gamma_1;L_1) + i^{\Omega}(\gamma_2;L_2)$. Let $q: \R^k \times \Sigma^+ \rightarrow M^2$ be a map such that
\[
q(0,z) = \pi\circ \tilde{u}^0(z) \;\; \forall z\in \Sigma^+, \quad
q(\lambda,s+it) \rightarrow \pi \circ \tilde{x} (t) \;\; \mbox{for } s\rightarrow +\infty, \; \forall \lambda\in \R^k,
\]
and such that the map
\[
\R^k \ni \lambda \mapsto (q_1(\lambda,0,1-\cdot),q_2(\lambda,0,\cdot)\in \Omega^1(M^2) 
\]
is a  diffeomorphism onto a neighborhood of $\pi\circ(\tilde{x}_1(-\cdot),\tilde{x}_2)$ in $W^u(\gamma_1)\times W^u(\gamma_2)$. By means of a suitable trivialization of $q^*(TM^2)$ and using the usual $W^{1,p}$ Sobolev setting with $p>2$, we can transform the problem of finding $\tilde{u}$  solving $\delbar_{J,\tilde{H}} (\tilde{u}) = 0$ together with (\ref{bbcc2}), (\ref{bbcc3}) and (\ref{bbcc4}) and being close to $\tilde{u}^0$, into the problem of finding pairs $(\lambda,u) \in \R^k \times  W^{1,p}(\Sigma^+,T^* \R^{2n})$, solving an equation of the form
\begin{equation}
\label{nf1}
\delbar  u(z) + f(\lambda,z,u(z)) = 0 \quad \forall z \in \Sigma^+,
\end{equation}
with boundary conditions
\begin{equation}
\label{nf2}
u(it) \in N^* (0) \;\; \forall t\in [0,1], \quad u(s+i) \in N^*(0) \;\; \forall s\geq 0.
\end{equation}
Then the boundary condition (\ref{bbcc1}) is translated into
\begin{equation}
\label{nf3}
u(s) \in N^* \Delta_{\R^n}\quad \forall s\geq 0,
\end{equation}
and the solution $\tilde{u}^0$ corresponds to the solution $\lambda=0$ and $u\equiv 0$ of (\ref{nf1}), (\ref{nf2}), and (\ref{nf3}). On the other hand, the problem $\widetilde{\mathscr{M}}^K_{\Upsilon} (\gamma_1,\gamma_2;x)$ of finding $(\alpha,\tilde{u}^{\alpha})$ solving  $\delbar_{J,\tilde{H}} (\tilde{u}^{\alpha}) = 0$ together with (\ref{bbcc2}), (\ref{bbcc3}), (\ref{bbcc4}) and (\ref{bbcc5}) corresponds to the problem of finding $(\lambda,u) \in \R^k \times W^{1,p}(\Sigma^+,T^*\R^{2n})$ solving (\ref{nf1}) with boundary conditions (\ref{nf2}) and
\begin{equation}
\label{nf4}
u(s) \in N^*(0) \;\; \forall s\in [0,\alpha], \quad u(s) \in N^* \Delta_{\R^n} \;\; \forall s\geq \alpha.
\end{equation} 
In order to find a common functional setting, it is convenient to turn the boundary condition (\ref{nf4}) into (\ref{nf3}) by means of a suitable conformal change of variables on the half-strip $\Sigma^+$.

The holomorphic function $z\mapsto \cos z$ maps the half strip $\{ 0 < \re z < \pi, \; \im z>0\}$ biholomorphically onto the upper half-plane $\mathbb{H} = \{\im z>0\}$. It is also a homeomorphism between the closure of these domains. We denote by $\arccos$ the determination of the arc-cosine which is the inverse of this function. Then the function $z \mapsto (1+\cos (i\pi z))/2$ is a biholomorphism from the interior of $\Sigma^+$ to $\mathbb{H}$, mapping $0$ into $1$ and $i$ into $0$. Let $\epsilon>0$. If we conjugate the linear automorphism $z \mapsto (1+\epsilon)z$ of $\mathbb{H}$ by the latter biholomorphism, we obtain the following map:  
\[
\varphi_{\epsilon}(z) = \frac{i}{\pi} \arccos \bigl( (1+\epsilon) \cos (i\pi z) + \epsilon \bigr).
\]
The map $\varphi_{\epsilon}$ is a homeomorphism of $\Sigma^+$ onto itself, it is biholomophic in the interior, it preserves the upper part of the boundary $i+ \R^+$, while it slides the left part $i[0,1]$ and the lower part $\R^+$ by moving the corner point $0$ into the real positive number
\[
\alpha(\epsilon) := - \frac{i}{\pi} \arccos (1+2\epsilon).
\]
The function $\epsilon \mapsto \alpha(\epsilon)$ is invertible, and we denote by $\alpha \mapsto \epsilon(\alpha)$ its inverse.
Moreover, $\varphi_{\epsilon}$ converges to the identity uniformly on compact subsets of $\Sigma^+$ for $\epsilon\rightarrow 0$. An explicit computation shows that
\begin{equation}
\label{converge}
\varphi_{\epsilon}' - 1 \rightarrow 0 \quad \mbox{in } L^p(\Sigma^+), \quad \mbox{if } 1<p<4.
\end{equation} 
If $u:\Sigma^+ \rightarrow T^* \R^{2n}$ and $\alpha>0$, we define
\[
v(z) := u(\varphi_{\epsilon(\alpha)} (z)).
\]
Since $\varphi_{\epsilon}$ is holomorphic, $\delbar  (u\circ \varphi_{\epsilon}) = \varphi'_{\epsilon} \cdot \delbar  u \circ \varphi_{\epsilon}$. Therefore,
$u$ solves the equation (\ref{nf1}) if and only if $v$ solves the equation
\begin{equation}
\label{nnff1}
\delbar  v (z) + \varphi'_{\epsilon(\alpha)}(z) f(\lambda,\varphi_{\epsilon(\alpha)}(z),v(z)) = 0.
\end{equation}
Given $2<p<4$, we set
\begin{eqnarray*}
W^{1,p}_*(\Sigma^+,T^* \R^{2n}) = \bigl\{v\in W^{1,p}(\Sigma^+,T^* \R^{2n}) \, | \, v(s,0)\in N^* \Delta_{\R^n} \; \forall s\geq 0, \\ v(s,1) \in N^*(0) \; \forall s\geq 0, \; v(0,t) \in N^* (0)\; \forall t\in [0,1]\bigr\},
\end{eqnarray*}
and we consider the operator
\begin{equation*}\begin{split}
&F : [0,+\infty[ \times \R^k \times W^{1,p}_*(\Sigma^+,T^* \R^{2n}) \rightarrow L^p (
\Sigma^+,T^* \R^{2n}),\\
&F(\alpha,\lambda,v) = \delbar  v +  \varphi'_{\epsilon(\alpha)} f(\lambda,\varphi_{\epsilon(\alpha)}(\cdot),v),
\end{split}\end{equation*}
where $\varphi_0=\mathrm{id}$. The problem of finding $(\alpha,\tilde{u})$ in $\widetilde{\mathscr{M}}^K_{\Upsilon}(\gamma_1,\gamma_2;x)$ with $\tilde{u}$ close to $\tilde{u}^0$  is equivalent to finding zeroes of the operator $F$ of  the form $(\alpha,\lambda,v)$ with $\alpha>0$. By (\ref{converge}), the operator $F$ is continuous, and its differential $D_{(\lambda,v)}F$ with respect to the variables $(\lambda,v)$ is continuous. The transversality assumption that $\tilde{u}^0$ is a non-degenerate solution of problem $\widetilde{\mathscr{M}}^{\Omega}_K(\gamma_1,\gamma_2;x)$ is translated into the fact that $D_{(\lambda,v)} F(0,0,0)$ is an isomorphism. Then the parametric inverse mapping theorem implies that there is a unique curve $\alpha \mapsto (\lambda(\alpha),v(\alpha))$, $0<\alpha<\alpha_0$, converging to $(0,0)$ for $\alpha\rightarrow 0$, and such that $(\lambda(\alpha),v(\alpha))$ is the unique zero of $F(\alpha,\cdot,\cdot)$ in a neighborhood of $(0,0)$. This concludes
  the proof of statement (ii). 

\subsection{Proof of Proposition \ref{coupro}}
\label{coupros} 

\paragraph{The setting.} We recall the setting of Section \ref{chlhs}. Let $\gamma_1\in \mathscr{P}^{\Lambda}(L_1)$, $\gamma_2 \in \mathscr{P}^{\Lambda}(L_2)$, and $x\in \mathscr{P}^{\Theta}(H_1 \oplus H_2)$. If $\alpha\geq 0$, $\mathscr{M}^K_{\alpha}(\gamma_1,\gamma_2;x)$ is the space of solutions $u : \Sigma^+ = \set{z\in \C}{\re z\geq 0, \; 0 \leq \im z \leq 1} \rightarrow T^*M^2$ of the equation
\begin{equation}
\label{cp0}
\delbar_{J,H_1 \oplus H_2} (u) = 0,
\end{equation}
satisfying the boundary conditions
\begin{eqnarray}
\label{cp1} \pi \circ u(\cdot i) \in W^u((\gamma_1,\gamma_2); X_{L_1 \oplus L_2}^{\Lambda}), \\
\label{cp2} (u(s),\mathscr{C}u(s+i)) \in N^* \Delta_{M^2} \mbox{ if } 0 \leq s < \alpha, \\
\label{cp3} (u(s),\mathscr{C}u(s+i)) \in N^* \Delta^{(4)}_M \mbox{ if } s\geq \alpha, \\
\label{cp4} \lim_{s\rightarrow +\infty} u(s+ ti) = x(t).
\end{eqnarray}
The energy of a solution $u\in \mathscr{M}_{\alpha}^K(\gamma_1,\gamma_2;x)$ is uniformly bounded:
\begin{equation}
\label{ube}
E(u) := \int_{\Sigma^+} |\partial_s u|^2\, ds \, dt \leq \mathbb{S}_{L_1}(\gamma_1) + \mathbb{S}_{L_2}(\gamma_2) - \mathbb{A}_{H_1 \oplus H_2} (x).
\end{equation}
Let $\alpha_0>0$. For a generic choice of $L_1$, $L_2$, and $X_{L_1 \oplus L_2}^{\Lambda}$, both $\mathscr{M}^K_0(\gamma_1,\gamma_2;x)$ and  $\mathscr{M}^K_{\alpha_0}(\gamma_1,\gamma_2;x)$ are smooth oriented manifolds of dimension
\[
i^{\Lambda}(\gamma_1,L_1) + i^{\Lambda}(\gamma_2,L_2) - \mu^{\Theta}(x) - n,
\]
for every $\gamma_1\in \mathscr{P}^{\Lambda}(L_1)$, $\gamma_2\in \mathscr{P}^{\Lambda}(L_2)$, $x\in \mathscr{P}^{\Theta}(H_1\oplus H_2)$ (see Proposition \ref{Kappa}). The usual counting process defines the chain maps
\[
K_0^{\Lambda}, \; K_{\alpha_0}^{\Lambda} : \bigl(M(\mathbb{S}_{L_1}^{\Lambda}) \otimes M(\mathbb{S}_{L_2}^{\Lambda})\bigr)_* \longrightarrow F_{*-n}^{\Theta} (H_1\oplus H_2),
\]
and we wish to prove that $K_0^{\Lambda} \otimes K_{\alpha_0}^{\Lambda}$ is chain homotopic to $K_{\alpha_0}^{\Lambda} \otimes K_0^{\Lambda}$. Since $K_{\alpha_0}^{\Lambda}$ is chain homotopic to $K_{\alpha_1}^{\Lambda}$ for $\alpha_0,\alpha_1 \in ]0,+\infty[$, we may as well assume that $\alpha_0$ is small. Moreover, since the chain maps $K_0^{\Lambda}$ and $K_{\alpha_0}^{\Lambda}$ preserve the filtrations of the Morse and Floer complexes given by the action sublevels
\[
\mathbb{S}_{L_1}^{\Lambda}(\gamma_1) + \mathbb{S}_{L_2}^{\Lambda} (\gamma_2) \leq A, \quad \mathbb{A}_{H_1 \oplus H_2}(x)
 \leq A,
 \] 
we can work with the subcomplexes corresponding to a fixed (but arbitrary) action bound $A$. We also fix our generic data  in such a way that transversality holds for the problem $\mathscr{M}^K_0$. 

The union of all the spaces of solutions $\mathscr{M}_0^K(\gamma_1,\gamma_2;x)$ where $\gamma_1\in \mathscr{P}^{\Lambda}(L_1)$, $\gamma_2 \in \mathscr{P}^{\Lambda}(L_2)$, and $x\in \mathscr{P}^{\Theta}(H_1 \oplus H_2)$ satisfy the index identity
\[
i^{\Lambda}(\gamma_1) + i^{\Lambda}(\gamma_2) - \mu^{\Theta}(x) = n,
\]
and the action estimates
\[
\mathbb{S}_{L_1}(\gamma_1) + \mathbb{S}_{L_2}(\gamma_2) \leq A, \quad \mathbb{A}_{H_1 \oplus H_2} (x) \leq A,
\]
is a compact zero-dimensional manifold, hence a finite set. We denote by $\mathscr{Q}$ the finite subset of $M$ consisting of the points $q\in M$ such that $\pi\circ u(0) = \pi\circ u(i) = (q,q)$ for some $u$ in the above finite set. We fix a positive number $\delta$ such that for every $q\in \mathscr{Q}$ the ball $B_{\delta}(q)$ is diffeomorphic to $\R^n$, and
\begin{equation}
\label{ul2}
B_{\delta}(q) \cap B_{\delta}(q') = \emptyset \quad \forall q,q'\in \mathscr{Q}, \; q\neq q'.
\end{equation}

Since real Grassmannians of a given dimension are connected, we can find a smooth path of $3n$-dimensional linear subspaces of $\Delta_{\R^{2n}} \times \Delta_{\R^{2n}} \subset \R^{8n}$ connecting $\Delta^{(4)}_{\R^n} \times \Delta_{\R^{2n}}$ to $\Delta_{\R^{2n}} \times \Delta^{(4)}_{\R^n}$, and containing the intersection of these two spaces, that is $\Delta^{(4)}_{\R^n} \times \Delta^{(4)}_{\R^n}$. Therefore, for every pair of points $q,q'\in \mathscr{Q}$ we can find a smooth isotopy of embeddings
\[
\varphi: [0,1] \times \R^{3n} \rightarrow B_{\delta}(q)^4 \times B_{\delta}(q')^4 \subset M^8,
\]
such that, setting $V^{\lambda}_{q,q'} := \varphi (\{\lambda\} \times \R^{3n})$, we have:
\begin{enumerate}
\item $V_{q,q'}^{\lambda}$ is relatively closed in $B_{\delta}(q)^4 \times B_{\delta}(q')^4$;
\item $(\Delta^{(4)}_M \times \Delta^{(4)}_M) \cap (B_{\delta}(q)^4 \times B_{\delta}(q')^4) \subset V_{q,q'}^{\lambda}$ for every $\lambda\in [0,1]$;   
\item $V_{q,q'}^{\lambda} \subset \Delta_{M^2} \times \Delta_{M^2}$ for every $\lambda\in [0,1]$;
\item $V_{q,q'}^0 = (\Delta^{(4)}_M \times \Delta_{M^2}) \cap (B_{\delta}(q)^4 \times B_{\delta}(q')^4)$;
\item $V_{q,q'}^1 = (\Delta_{M^2} \times \Delta^{(4)}_M ) \cap (B_{\delta}(q)^4 \times B_{\delta}(q')^4)$.
\end{enumerate}
We define $V^{\lambda} \subset M^8$ to be the union of $V_{q,q'}^{\lambda}$ over all pairs $(q,q')\in \mathscr{Q}\times \mathscr{Q}$. By (\ref{ul2}), $V^{\lambda}$ is an isotopy of (non-compact) $3n$-dimensional submanifolds of $M^8$.

Let $\gamma_1,\gamma_3\in \mathscr{P}^{\Lambda}(L_1)$, $\gamma_2,\gamma_4 \in \mathscr{P}^{\Lambda}(L_2)$, and $x_1,x_2\in \mathscr{P}^{\Theta}(H_1 \oplus H_2)$ satisfy the index identity 
\begin{equation}
\label{ul3} 
i^{\Lambda}(\gamma_1) + i^{\Lambda}(\gamma_2) + i^{\Lambda}(\gamma_3) + i^{\Lambda}(\gamma_4) - \mu^{\Theta}(x_1) - \mu^{\Theta}(x_2) = 2n,
\end{equation}
and the action bounds 
\begin{equation}
\label{ul4}
\mathbb{S}_{L_1}(\gamma_1) + \mathbb{S}_{L_2}(\gamma_2) + \mathbb{S}_{L_1}(\gamma_3) + \mathbb{S}_{L_2}(\gamma_4) \leq A,  \quad \mathbb{A}_{H_1\oplus H_2}(x_1) + \mathbb{A}_{H_1 \oplus H_2}(x_2) \leq A.
\end{equation}
Given $\alpha>0$, we define 
\[
\mathscr{M}^P_{\alpha}(\gamma_1,\gamma_2,\gamma_3,\gamma_4;x_1,x_2)
\]
to be the set of pairs $(\lambda,u)$ where $\lambda\in [0,1]$ and
$u : [0,+\infty[ \times [0,1] \rightarrow T^*M^4$ is a solution of the equation
\begin{equation}
\label{ul5}
\delbar_{J,H_1 \oplus H_2 \oplus H_1 \oplus H_2} (u) = 0,
\end{equation}
satisfying the boundary conditions
\begin{eqnarray}
\label{ul6} \pi \circ u(0,\cdot) \in W^u(\gamma_1) \times  W^u(\gamma_2) \times W^u(\gamma_3) \times  W^u(\gamma_4), \\
\label{ul7} (u(s),\mathscr{C} u(s+i)) \in N^* V^{\lambda}  \mbox{ if } 0 \leq s \leq \alpha, \\
\label{ul8} (u(s),\mathscr{C} u(s+i)) \in N^* (\Delta^{(4)}_M \times \Delta^{(4)}_M) \mbox{ if } s\geq \alpha, \\
\label{ul9} \lim_{s\rightarrow +\infty} u(s+ti) = (x_1(t),x_2(t)).
\end{eqnarray}
By property (ii) above, the submanifold $V^{\lambda}$ is partially orthogonal to $\Delta^{(4)}_M \times \Delta^{(4)}_M$ for every $\lambda\in [0,1]$, so the results of Section \ref{lineartheory} provide us with elliptic estimates for solutions of the above problem. 

Notice that if $(0,u)$ belongs to $\mathscr{M}^P_{\alpha}(\gamma_1,\gamma_2,\gamma_3,\gamma_4;x_1,x_2)$, then writing $u=(u_1,u_2)$ where $u_1$ and $u_2$ take values into $T^*M^2$, we have
\[
u_1 \in \mathscr{M}^K_0(\gamma_1,\gamma_2;x_1), \quad
u_2 \in \mathscr{M}^K_{\alpha}(\gamma_3,\gamma_4;x_2).
\]
If transversality holds, we deduce the index estimates
\[
i^{\Lambda}(\gamma_1) + i^{\Lambda}(\gamma_2) - \mu^{\Theta}(x_1) \geq n, \quad i^{\Lambda}(\gamma_3) + i^{\Lambda}(\gamma_4) - \mu^{\Theta}(x_2) \geq n.
\]
But then (\ref{ul3}) implies that the above inequalities are indeed identities. Similarly, if $(1,u)$ belongs to $\mathscr{M}^P_{\alpha}(\gamma_1,\gamma_2,\gamma_3,\gamma_4;x_1,x_2)$, we deduce that
\[
u_1 \in \mathscr{M}^K_{\alpha}(\gamma_1,\gamma_2;x_1), \quad
u_2 \in \mathscr{M}^K_0(\gamma_3,\gamma_4;x_2).
\]
and
\[
i^{\Lambda}(\gamma_1) + i^{\Lambda}(\gamma_2) - \mu^{\Theta}(x_1) = n, \quad i^{\Lambda}(\gamma_3) + i^{\Lambda}(\gamma_4) - \mu^{\Theta}(x_2) = n.
\]
Conversely, we would like to show that pairs of solutions in $\mathscr{M}^K_0 \times \mathscr{M}^K_{\alpha}$ (or $\mathscr{M}^K_{\alpha} \times \mathscr{M}^K_0$) correspond to elements of $\mathscr{M}^P_{\alpha}$ of the form $(0,u)$ (or $(1,u)$), at least if $\alpha$ is small. This fact follows from the following localization result:

\begin{lem}
\label{localization}
There exists a positive number $\alpha(A)$ such that for every $\alpha\in ]0,\alpha(A)]$, for every $\gamma_1,\gamma_3\in \mathscr{P}^{\Lambda}(L_1)$, $\gamma_2,\gamma_4\in \mathscr{P}^{\Lambda}(L_2)$, $x_1,x_2\in \mathscr{P}^{\Theta}(H_1 \oplus H_2)$ satisfying (\ref{ul3}) and (\ref{ul4}) and for every $(\lambda,u) \in \mathscr{M}^P_{\alpha}(\gamma_1,\gamma_2,\gamma_3,\gamma_4;x_1,x_2)$ there holds
\[
\pi\circ u([0,\alpha]) = \pi \circ u([0,\alpha] + i) \subset B_{\delta/2}(q)^2 \times B_{\delta/2}(q')^2,
\]
for suitable $q,q'\in \mathscr{Q}$.
\end{lem}

\begin{proof}
By contradiction, we assume that there are an infinitesimal sequence of positive numbers $(\alpha_h)$ and elements $(\lambda_h,u_h) \in \mathscr{M}_{\alpha_h}^P (\gamma_1,\gamma_2,\gamma_3,\gamma_4;x_1,x_2)$ where $\gamma_1,\gamma_2,\gamma_3,\gamma_4,x_1,x_2$ satisfy (\ref{ul3}), (\ref{ul4}), and
\begin{equation}
\label{ul10}
\pi\circ u_h([0,\alpha_h]) = \pi\circ u_h([0,\alpha_h] + i) \not\subset \bigcup_{(q,q') \in \mathscr{Q}\times \mathscr{Q}} B_{\delta/2} (q)^2 \times B_{\delta/2} (q')^2.
\end{equation}

\paragraph{Claim 1.} {\em Up to a subsequence, $c_h(t):= \pi\circ u_h(0,t)$ converge to some $c$ in $W^{1,2}([0,1],M^4)$ such that $c(0)=c(1)=(q,q,q',q')$ with $(q,q')\in \mathscr{Q}\times \mathscr{Q}$, and $u_h$ converges to some $u\in \mathscr{M}^K_0(\gamma_1,\gamma_2;x_1) \times \mathscr{M}^K_0(\gamma_3,\gamma_4;x_2)$ such that $
\pi \circ u(ti) = c(t)$ in $C^{\infty}_{\mathrm{loc}} (]0,+\infty[ \times [0,1], T^* M^4)$ and uniformly on compact subsets of $[0,+\infty[ \times [0,1] \setminus \{ (0,0), (0,1) \}$.}

\medskip

By (\ref{ul6}), $c_h$ is an element of $W^u(\gamma_1) \times W^u(\gamma_2) \times W^u(\gamma_3) \times W^u(\gamma_4)$. The latter space is pre-compact in $W^{1,2}([0,1],M^4)$. By the argument of breaking gradient flow lines, up to a subsequence we may assume that $(c_h)$ converges in $W^{1,2}$ to a curve $c$ belonging to $W^u(\tilde\gamma_1) \times W^u(\tilde\gamma_2) \times W^u(\tilde\gamma_3) \times W^u(\tilde\gamma_4)$, for some $\tilde\gamma_1,\tilde\gamma_3 \in \mathscr{P}^{\Lambda}(L_1)$, $\tilde\gamma_2,\tilde\gamma_4\in \mathscr{P}^{\Lambda}(L_2)$ such that
\begin{equation}
\label{ul11}
\mbox{either} \quad \sum_{j=1}^4 i^{\Lambda}(\tilde\gamma_j) < \sum_{j=1}^4 i^{\Lambda}(\gamma_j) \quad \mbox{or} \quad (\tilde\gamma_1,\tilde\gamma_2,\tilde\gamma_3,\tilde\gamma_4) = (\gamma_1,\gamma_2,\gamma_3,\gamma_4).
\end{equation}
Similarly, the upper bound (\ref{ube}) on the energy $E(u_h)$ implies that $(u_h)$ converges to some $u$ in $C^{\infty}_{\mathrm{loc}} (]0,+\infty[ \times [0,1])$, using the $L^{\infty}$ estimates of Section \ref{cocosec}, the standard argument excluding bubbling off of spheres and disks, and elliptic bootstrap. The same arguments, together with the $W^{1,2}$ convergence of $(c_h)$ to $c$, imply that $(u_h)$ converges to $u$ uniformly on compact subsets of $[0,+\infty[ \times [0,1] \setminus \{ (0,0), (0,1) \}$ (actually, the $W^{1,2}$ convergence of the boundary data implies $W^{3/2,2}$ convergence near the portion of the boundary $\{0\}\times ]0,1[$). In particular,
\begin{equation}
\label{ul12}
(u(s),\mathscr{C}u(s+i)) \in N^* (\Delta^{(4)}_M \times \Delta^{(4)}_M ) \quad \forall s>0,
\end{equation}
and
\begin{equation}
\label{ul13}
\pi\circ u (0,t) = \pi \circ u(1,t) = c(t) \quad \forall t\in ]0,1[.
\end{equation}
The limit $u$ satisfies equation (\ref{ul5}), and
\[
\lim_{s\rightarrow +\infty} u(s+ti) = (\tilde{x}_1(t),\tilde{x}_2(t)),
\]
with $(\tilde{x}_1,\tilde{x}_2) \in \mathscr{P}^{\Theta}(H_1 \oplus H_2) \times  \mathscr{P}^{\Theta}(H_1 \oplus H_2)$ such that
\begin{equation}
\label{ul14}
\mbox{either} \quad \mu^{\Theta}(\tilde{x}_1) + \mu^{\Theta}(\tilde{x}_2) > \mu^{\Theta}(x_1) + \mu^{\Theta}(x_2) \quad \mbox{or} \quad (\tilde{x}_1,\tilde{x}_2) = (x_1,x_2),
\end{equation}
by the argument of breaking Floer trajectories.
Due to finite energy,
\begin{equation*}\begin{split}
E(u) \,\leq\, &\liminf_{h\rightarrow +\infty} E(u_h)\\ \leq\, & \mathbb{S}_{L_1}(\gamma_1) + \mathbb{S}_{L_2}(\gamma_2) + \mathbb{S}_{L_1}(\gamma_3) + \mathbb{S}_{L_2}(\gamma_4) - \mathbb{A}_{H_1 \oplus H_2} (x_1) - \mathbb{A}_{H_1 \oplus H_2} (x_2),
\end{split}\end{equation*}
we find by removal singularities (Proposition \ref{remsing}) 
a continuous extension of $u$ to the
corner points $0$ and $i$. By (\ref{ul12}) and (\ref{ul13}) we have
\[
(u(0),\mathscr{C}u(i)) \in N^* (\Delta^{(4)}_M \times \Delta^{(4)}_M) , \quad \pi\circ u(0) = \pi\circ u(i) = c(0) = c(1).
\]
It follows that, setting $c=(c_1,c_2,c_3,c_4)$, there holds $c_1(0)=c_1(1)=c_2(0)=c_2(1)$ and $c_3(0)=c_3(1)=c_4(0)=c_4(1)$. So $(c_1,c_2)$ and $(c_3,c_4)$ describe two figure-8 loops, and $u$ belongs to $\mathscr{M}^K_0(\tilde\gamma_1,\tilde\gamma_2;\tilde{x}_1) \times \mathscr{M}^K_0(\tilde\gamma_3,\tilde\gamma_4;\tilde{x}_2)$. In particular, the latter space in not empty, so
\[
i^{\Lambda}(\tilde{\gamma}_1) + i^{\Lambda}(\tilde\gamma_2) - \mu^{\Theta}(\tilde{x}_1) \geq n, \quad i^{\Lambda}(\tilde{\gamma}_3) + i^{\Lambda}(\tilde\gamma_4) - \mu^{\Theta}(\tilde{x}_2) \geq n.
\]
Together with (\ref{ul3}), this implies
\[
\sum_{j=1}^4 i^{\Lambda}(\tilde\gamma_j) - \sum_{j=1}^2 \mu^{\Theta} (\tilde{x}_j) \geq 2n = \sum_{j=1}^4 i^{\Lambda}(\gamma_j) - \sum_{j=1}^2 \mu^{\Theta} (x_j).
\]
Comparing the above inequality with (\ref{ul11}) and (\ref{ul14}), we deduce that \[
(\tilde\gamma_1,\tilde\gamma_2,\tilde\gamma_3,\tilde\gamma_4) = (\gamma_1,\gamma_2,\gamma_3,\gamma_4), \quad (\tilde{x}_1,\tilde{x}_2)= (x_1,x_2).
\]
Therefore, $u$ belongs to $\mathscr{M}^K_0(\gamma_1,\gamma_2;x_1) \times \mathscr{M}^K_0(\gamma_3,\gamma_4;x_2)$ and 
\[
c(0)=c(1) = (q,q,q',q'),
\] 
for some pair $(q,q')\in \mathscr{Q}\times \mathscr{Q}$, concluding the proof of Claim 1.

\paragraph{Claim 2.} {\em Up to a subsequence, the sequences of maps
\[
v_h^0(z) := u_h(\alpha_h z), \quad 
v_h^1(z) := u_h(i + \alpha_h \overline{z}),
\]
converge to constant maps $v^0,v^1$ uniformly on compact subsets of $[0,+\infty[\times [0,+\infty[$. Moreover, $\pi(v^0) = \pi(v^1) = (q,q,q',q')$ for some pair $(q,q')\in \mathscr{Q}\times \mathscr{Q}$.}

\medskip 

It is convenient to replace the nonlocal boundary conditions (\ref{ul7}), (\ref{ul8}) by local ones. We define the sequence of maps
\[
w_h : [0,+\infty[ \times [0,1/\alpha_h] \rightarrow T^* M^8, \quad w_h(z) := (v_h^0(z), \mathscr{C} v_h^1(z)).
\]
Then $w_h$ satisfies the local boundary conditions
\begin{eqnarray}
\label{ul15}
w_h(s) = (u_h(\alpha_h s), \mathscr{C} u_h(\alpha_h s+i)) \in \left\{ \begin{array}{ll} N^* V^{\lambda_h} & \mbox{if } 0\leq s \leq 1, \\ N^* (\Delta^{(4)}_M \times \Delta^{(4)}_M) & \mbox{if } s\geq 1, \end{array} \right. \\
\label{ul16}
\pi\circ w_h(ti) = \bigl(\pi\circ u_h(0,\alpha_h t),\pi \circ u_h(0,(1-\alpha_h) t)\bigr) = \bigl(c_h(\alpha_h t),c_h(1-\alpha_h t)\bigr).
\end{eqnarray}
Up to a subsequence, we may assume that $(\lambda_h)$ converges to some $\lambda\in [0,1]$. The maps $w_h$ solve a Floer equation of the form
\[
\delbar_J (w_h) = \alpha_h J X_K, 
\]
for a suitable time-dependent Hamiltonian $K$ on $T^*M^8$. Since we have applied a conformal rescaling, the energy of $w_h$ is
uniformly bounded, so $(w_h)$ converges up to a subsequence to some
$J$-holomorphic map $w: [0,+\infty[ \times [0,+\infty[ \rightarrow T^* M^8$ uniformly on compact subsets of $[0,+\infty[ \times
[0,+\infty[$ (more precisely, we have\\
$C^{\infty}_{\mathrm{loc}}(]0,+\infty[\times [0,+\infty[)$ convergence once the domain $[0,+\infty[\times [0,+\infty[$ is transformed by a conformal mapping turning the portion near the boundary point $(1,0)$ into a neighborhood of $(0,0)$ in the upper-right quarter $\mathbb{H}^+ =
]0,+\infty[ \times ]0,+\infty[$, and we have $W^{3/2,2}$ convergence near the piece of the boundary $\{0\}\times ]0,+\infty[$). The $J$-holomorphic map $w$ has finite
energy, so by removal singularities it has a continuous extension at
$\infty$ (again, by Proposition \ref{remsing} together with a suitable
conformal change of variables). 
By (\ref{ul15}) and (\ref{ul16}) it satisfies the
boundary conditions 
\begin{eqnarray}
\label{ul17}
w(s) \in \left\{ \begin{array}{ll} N^* V^{\lambda} & \mbox{if } 0 \leq s \leq 1, \\ N^* (\Delta^{(4)}_M \times \Delta^{(4)}_M) & \mbox{if } s\geq 1, \end{array} \right. \\
\label{ul18}
\pi\circ w(0,t) = (c(0),c(1)) = (c(0),c(1)) \mbox{ for } t\geq 0.
\end{eqnarray}
Since the boundary conditions are of conormal type and the Liouville one-form $\eta$ vanishes on conormals, we have
\[
\int_{\mathbb{H}^+} |\nabla w|^2 \, ds\, dt = \int_{\mathbb{H}^+} w^*(\omega) = \int_{\mathbb{H}^+} w^*(d\eta) = \int_{\mathbb{H}^+} d w^*(\eta) = \int_{\partial \mathbb{H}^+} w^*(\eta) = 0,
\]
so $w$ is constant. By (\ref{ul17}) and (\ref{ul18}), $w$ belongs to $N^* V^{\lambda} \cap N^* (\Delta^{(4)}_M \times \Delta^{(4)}_M) \cap T^*_{(c(0),c(0))} M^8$. In particular, $\pi(w)=(c(0),c(0))$, and since $c(0)=(q,q,q',q')$ with $(q,q')\in \mathscr{Q}$, so Claim 2 follows.

\medskip

Claim 2 contradicts (\ref{ul10}), proving the lemma.
\end{proof}

\begin{rem}
If $(u_h)$ is the sequence of solutions considered in the proof of the above lemma, we cannot conclude that $(u_h)$ converges uniformly in a neighborhood of the corner points $(0,0)$ and $(0,1)$. Indeed, generically this will not happen: otherwise the limit $u$ would be an element of $\mathscr{M}_0^K(\gamma_1,\gamma_2;x_1) \times \mathscr{M}_0^K(\gamma_3,\gamma_4;x_2)$ satisfying the extra condition $u(0)=u(i)$, and this is a problem of Fredholm index $-2n$, for which generically there are no solutions.
\end{rem}

We can now built the required chain homotopy between $K_0^{\Lambda} \otimes K_{\alpha_0}^{\Lambda}$ and $K_{\alpha_0}^{\Lambda} \otimes K_0^{\Lambda}$. We fix some $\alpha_0\in ]0,\alpha(A)]$, and we choose the generic data $L_1$, $L_2$, $X_{L_1}^{\Lambda}$, $X_{L_2}^{\Lambda}$ in such a way that transversality holds for the problems $\mathscr{M}^K_0$, $\mathscr{M}^K_{\alpha_0}$, and $\mathscr{M}^P_{\alpha_0}$, and such that there are no elements of the form $(0,u)$ or $(1,u)$ in the zero-dimensional components of $\mathscr{M}^P_{\alpha_0}$.

Let $\gamma_1,\gamma_3\in \mathscr{P}^{\Lambda}(L_1)$, $\gamma_2,\gamma_4 \in \mathscr{P}^{\Lambda}(L_2)$, $x_1,x_2\in \mathscr{P}^{\Theta}(H_1\oplus H_2)$ be Hamiltonian orbits satisfying the index identity (\ref{ul3}) and the action bounds (\ref{ul4}). Consider the one-dimensional manifold 
$\mathscr{M}^P_{\alpha_0}(\gamma_1,\gamma_2,\gamma_3,\gamma_4;x_1,x_2)$, and let 
\[
\overline{\mathscr{M}}^P_{\alpha_0}(\gamma_1,\gamma_2,\gamma_3,\gamma_4;x_1,x_2)
\]
be the one-dimensional manifold with boundary obtained by attaching in the usual way elements $(\lambda,u)$ of the zero-dimensional spaces of solutions $\mathscr{M}^P_{\alpha_0}$ (necessarily, with $0<\lambda<1$).  
The submanifold $V^{\lambda}$ is not compact, but its intersection with the closure of $B_{\delta/2}(q)^4 \times B_{\delta/2}(q')^4$, $q,q'\in \mathscr{Q}$, is compact, so the localization Lemma \ref{localization} implies that $\overline{\mathscr{M}}^P_{\alpha_0}(\gamma_1,\gamma_2,\gamma_3,\gamma_4;x_1,x_2)$ is compact, and its intersections with $\{\lambda = 0 \}$ and $\{\lambda=1\}$ are
\begin{eqnarray}
\label{chom1}
\mathscr{M}^P_{\alpha_0} (\gamma_1,\gamma_2,\gamma_3,\gamma_4;x_1,x_2) \cap \{\lambda = 0\} = \mathscr{M}^K_0(\gamma_1,\gamma_2;x_1) \times \mathscr{M}^K_{\alpha_0} (\gamma_3,\gamma_4;x_2), \\
\label{chom2}
\mathscr{M}^P_{\alpha_0} (\gamma_1,\gamma_2,\gamma_3,\gamma_4;x_1,x_2) \cap \{\lambda = 1\} = \mathscr{M}^K_{\alpha_0}(\gamma_1,\gamma_2;x_1) \times \mathscr{M}^K_0 (\gamma_3,\gamma_4;x_2).
\end{eqnarray}
We denote by $M^A$ the subcomplex of
\[
M(\mathbb{S}_{L_1}^{\Lambda}) \otimes M(\mathbb{S}_{L_2}^{\Lambda})  \otimes M(\mathbb{S}_{L_1}^{\Lambda})  \otimes M(\mathbb{S}_{L_2}^{\Lambda})
\]
spanned by generators $\gamma_1\otimes \gamma_2\otimes \gamma_3 \otimes \gamma_4$ with
\[
\mathbb{S}_{L_1}(\gamma_1) + \mathbb{S}_{L_2}(\gamma_2) +  
\mathbb{S}_{L_1}(\gamma_3) + \mathbb{S}_{L_2}(\gamma_4) \leq A.
\]
Similarly, we denote by $F^A$ the subcomplex of
\[
F^{\Theta}(H_1\oplus H_2) \otimes  F^{\Theta}(H_1\oplus H_2) 
\]
spanned by generators $x_1 \otimes x_2$ with
\[
\mathbb{A}_{H_1 \oplus H_2} (x_1) + \mathbb{A}_{H_1 \oplus H_2} (x_2) \leq A.
\]
We define a homomorphism
\[
P : M^A_* \rightarrow F^A_{*-2n+1}
\]
by counting the elements of $\mathscr{M}^P_{\alpha_0}(\gamma_1,\gamma_2,\gamma_3,\gamma_4;x_1,x_2)$ in the zero-dimensional case:
\[
i^{\Lambda}(\gamma_1) + i^{\Lambda}(\gamma_2) + i^{\Lambda}(\gamma_3) + i^{\Lambda}(\gamma_4) - \mu^{\Theta}(x_1) - \mu^{\Theta}(x_2) = 2n - 1.
\]
Using the identities (\ref{chom1}) and (\ref{chom2}) we see that $P$ is a chain homotopy between the restrictions of $K^{\Lambda}_0 \otimes K^{\Lambda}_{\alpha_0}$ and $K^{\Lambda}_{\alpha_0} \otimes K^{\Lambda}_{\alpha}$ to the above subcomplexes. This concludes the proof of Proposition \ref{coupro}.

\begin{rem}
As mentioned in the introduction, the homotopy and localization argument used in the proof of Proposition \ref{coupro} is inspired by an idea of H.\ Hofer's. One could try to apply this idea directly to the proof that the pair-of-pants homomorphism $\Upsilon^{\Lambda}$ corresponds to the Morse theoretical interpretation of the loop product (see Proposition \ref{nonserve}), without passing from the factorizations through the figure-8 Floer and Morse complexes. 
This would involve showing that a Floer problem for $u: \Sigma^+ \rightarrow T^*M^2$ satisfying the boundary conditions
\[
(u(s),\mathscr{C}u(s+i)) \in N^* (\Delta_M \times \Delta_M) \mbox{ for } s\geq 0,
\]
is homotopic to a Floer problem for $u$ satisfying  
\begin{equation*}\begin{split}
&(u(s),\mathscr{C}u(s+i)) \in N^* \Delta_{M\times M} \mbox{ for } 0\leq s \leq \alpha,\\ 
&(u(s),\mathscr{C}u(s+i)) \in N^* (\Delta_M \times \Delta_M) \mbox{ for } s\geq \alpha.
\end{split}\end{equation*}
Here the trick of doubling the dimensions by using the algebraic Lemma \ref{alge} would not be needed, because the manifolds $\Delta_{M\times M}$ and $\Delta_M \times \Delta_M$ have the same dimension. However, a local isotopy between such manifolds cannot be partially orthogonal to $\Delta_M \times \Delta_M$, so the elliptic estimates of Section \ref{lineartheory} would  not be available and one would face compactness problems.
\end{rem}

\renewcommand{\thesection}{\Alph{section}}
\setcounter{section}{0} 

\section{Appendix - Morse constructions}
\label{appendixA}

The aim of the first section of this appendix is to recall the construction of the Morse complex for functions defined on an infinite-dimensional Hilbert manifold. See
\cite{ama06m} for detailed proofs.  

Many operations in singular homology can be read
on the Morse complex (see for instance \cite{sch93,fuk93,bc94,vit95f,fuk97}).
Here we are interested only in functoriality, in the exterior homology
product, and in the Umkehr map. The
corresponding constructions -- still in our infinite dimensional setting
-- are outlined in the subsequent sections.

\subsection{The Morse complex}
\label{mc}

Let $\mathcal{M}$ be a
(possibly infinite-dimensional) Hilbert manifold, that is, a Hausdorff paracompact topological space which is locally homeomorphic to a real Hilbert space and admits an atlas whose transition maps are smooth. If $X$ is a smooth 
vector field on $\mathcal{M}$ and $x\in \mathcal{M}$ is a {\em singular point} for $X$, i.e.\ $X(x)=0$, the {\em Jacobian} of $X$ at $x$ is a well-defined bounded operator on $T_x \mathcal{M}$, that we denote by $\nabla X(x)$. The singular point $x$ is said to be {\em hyperbolic} if the spectrum of $\nabla X(x)$ is disjoint from $i\R$, and the {\em Morse index} $i(x)$ of $x$ is the (possibly infinite) dimension of the $\nabla X(x)$-invariant closed linear subspace of $T_x \mathcal{M}$ associated to the part of the spectrum with positive real part. Hyperbolic singular points are isolated within the set of singular points. The vector field $X$ is said to be {\em Morse} if all its singular points are hyperbolic. 

A real function $f\in C^1(\mathcal{M})$ is said to be a {\em Lyapunov function} for $X$ if\\ $Df(p)[X(p)]<0$ for every $p\in \mathcal{M}$ which is not a singular point of $X$. If $X$ is Morse, then the set of singular points of $X$ coincides with the set of critical points of $f$, that we denote by $\crit (f)$ (in general, $\crit(f)$ is contained in the set of singular points of $X$). If $(t,p) \mapsto \phi(t,p)$ denotes the flow of the vector field $X$, the {\em stable} and {\em unstable manifolds} of the hyperbolic singular point $x$ are the subsets
\begin{eqnarray*}
W^s(x;X) & := & \set{p\in \mathcal{M}}{\phi(t,p)\rightarrow x \mbox{ for } t\rightarrow +\infty}, \\   
W^u(x;X) & := & \set{p\in \mathcal{M}}{\phi(t,p)\rightarrow x \mbox{ for } t\rightarrow -\infty}.
\end{eqnarray*}
We just write $W^s(x)$ and $W^u(x)$ when the vector field $X$ is clear from the context.
They are always injectively immersed smooth submanifolds of $\mathcal{M}$, and if $X$ admits a Lyapunov function $f$ and $i(x)<+\infty$ they are actually embedded submanifolds (if $i(x)=+\infty$, the same is true under a non-degeneracy condition on $f$, see Theorem 1.20 in \cite{ama06m}). Their tangent spaces at $x$ are the $\nabla X(x)$-invariant subspaces determined by the negative part (for the stable manifold) and the positive part (for the unstable manifold) of the spectrum of $\nabla X(x)$. In particular, $\dim W^u(x) = i(x)$. The Morse vector field $X$ is said to satisfy the {\em Morse-Smale condition} if for every pair of singular points $x,y$, the unstable manifold of $x$ is transverse to the stable manifold of $y$.

A {\em Palais-Smale sequence} for the pair $(X,f)$ is a sequence $(p_h)\subset \mathcal{M}$ such that $f(p_h)$ is bounded and $Df(p_h)[X(p_h)]$ is infinitesimal. The pair $(X,f)$ is said to satisfy the {\em Palais-Smale condition} if every Palais-Smale sequence has a converging subsequence. 

Let $\mathcal{F}(\mathcal{M})$ be the set of all $C^1$ real functions $f$ on $\mathcal{M}$ which are bounded from below 
and for which there exists a smooth vector field $X$ on $\mathcal{M}$ such that:
\begin{enumerate}
\item[(X1)] $f$ is a Lyapunov function for $X$;
\item[(X2)] $X$ is Morse and all its singular points (i.e.\ the critical points of $f$) have finite Morse index;
\item[(X3)] the pair $(X,f)$ satisfies the Palais-Smale condition;
\item[(X4)] $X$ is forward complete (i.e.\ the flow $\phi(t,p)$ is defined for every $t\geq 0$).
\end{enumerate}  
Such a vector field is called a (negative) {\em pseudo-gradient} for $f$. The Morse index $i(x)$ of a critical point $x$ of $f$ does not depend on the choice of the pseudo-gradient $X$, and if $f$ has a non-degenerate Gateaux second differential $d^2 f(x)$ at the critical point $x$, $i(x)$ coincides with the standard Morse index $i(x;f)$, that is the dimension of a maximal linear subspace of $T_x \mathcal{M}$ on which $d^2 f(x)$ is negative definite. The set of critical points of $f$ of Morse index $k$ is denoted by $\mathrm{crit}_k(f)$.

If $f$ is a smooth Morse function, bounded from below, and it satisfies the Palais-Smale condition with respect to some complete Riemannian metric $g$ on $\mathcal{M}$ (in the usual sense, see e.g.\ \cite{cha93}), then $-\mathrm{grad}_g f$ satisfies (X1)-(X4), so $f\in \mathcal{F}(\mathcal{M})$. However, it is useful to have a theory which allows also for non-regular functions $f$, as the example of the Lagrangian action functional introduced in Section \ref{laf} shows.

If $f\in \mathcal{F}(\mathcal{M})$, it is possible to perturb the smooth vector field $X$ generically, in such a way that (X1)-(X4) still hold, and:
\begin{enumerate}
\item[(X5)] $X$ satisfies the Morse-Smale condition.
\end{enumerate} 

In this paper, genericity for a pseudo-gradient $X$ is meant in a suitable complete metric space of vector fields which coincide with the original one up to order one at critical points, endowed with an adapted Whitney metric such that all the vector fields whose distance from the original one is less than 1 still satisfy (X1)-(X4). We shall not specify such metric space any further (see \cite{ama06m} for more details).

Let us fix a real number $a$. Since $f$ is bounded from below, (X2) and (X3) imply that $f$ has finitely many critical points in the sublevel $\{f<a\}$. Using also (X4) and (X5), one can find open neighborhoods $\mathcal{U}(x)$ of each critical point $x$ in $\{f<a\}$ such that, if we set
\[
\mathcal{M}_k^a := \bigcup_{\substack{x\in \crit(f)\\ f(x)<a \\ i(x)\leq k}} \phi([0,+\infty[ \times \mathcal{U}(x)), \quad \forall k\in \N \cup \{ \infty \},
\]
the (eventually constant) sequence $\{\mathcal{M}_k^a\}_{k\in \N}$ is a cellular filtration of $\mathcal{M}_{\infty}^a$. More precisely, if $M_k^a(f)$ denotes the free Abelian group generated by the critical points in $\{f<a\}$ of Morse index $k$, one has 
\[
H_j(\mathcal{M}_k^a,\mathcal{M}_{k-1}^a) \cong \left\{ \begin{array}{ll} M_k^a(f), & \mbox{if } j=k, \\ 0 , & \mbox{if } j\neq k, \end{array} \right.
\]
and this isomorphism is uniquely determined by the choice of an orientation of each unstable manifold $W^u(x)$ of $x\in \crit_k(f)\cap \{f<a\}$. 
Moreover, $\mathcal{M}_{\infty}^a$ is a deformation retract of the sublevel $\{f<a\}$. Then the cellular complex of the cellular filtration $\{\mathcal{M}_k^a\}_{k\in \N}$ induces the structure of a chain complex on the 
graded group $M_*^a(f)$, whose homology is isomorphic to the singular homology of $\mathcal{M}_{\infty}^a$, hence to the singular homology of $\{f<a\}$. 
This chain complex is called the {\em Morse complex of $(X,f)$ on the sublevel $\{f<a\}$}. If we choose different neighborhoods $\mathcal{U}(x)$, we obtain the same chain complex. Actually, the boundary homomorphism
\[
\partial_k : M_k^a(f) \rightarrow M_{k-1}^a(f)
\]
can be expressed in terms of the standard generators of $M_k^a(f)$, that is the critical points $x$ such that $f(x)<a$ and $i(x)=k$, as
\begin{equation}
\label{bordo}
\partial_k x = \sum_{\substack{y\in \crit(f) \\ i(y) = k-1}} n_{\partial}(x,y) \, y, 
\end{equation}
where $n_{\partial}$ is defined as follows. The chosen orientation of each unstable manifold $W^u(x)$ induces a co-orientation of each stable manifold. By (X5), each intersection $W^u(x)\cap W^s(y)$ is a transverse intersection of an oriented submanifold of dimension $i(x)$ and a co-oriented submanifold of codimension $i(y)$. So $W^u(x)\cap W^s(y)$ is an oriented submanifold of dimension $i(x)-i(y)$. When $i(x)-i(y)=1$, $W^u(x) \cap W^s(y)$ consists of finitely many flow orbits, and $n_{\partial}(x,y)$ is the number of those orbits on which the direction of $X$ agrees with the orientation, minus the number of the other ones.
 
The {\em Morse complex of $(X,f)$ on $\mathcal{M}$} is a chain complex on the graded free Abelian group $M_*(f)$ generated by all the critical points of $f$, and it is defined either by taking a direct limit of $\{M_*^a(f),\partial_*\}$ for $a\uparrow \infty$, or - equivalently - by formula (\ref{bordo}), which does not depend on $a$. In order to use the vector field $X$ to construct a cellular filtration of the whole $\mathcal{M}$, one would need that every critical point $x$ does not belong to the closure of the union of all the unstable manifolds of critical points $y$ of Morse index  $i(y)\leq i(x)$. The latter fact is implied by the Morse-Smale condition when there is a finite number of critical points of any given index, but it mail fail in the general case. This is why we work with sublevels and we define the Morse complex on $\mathcal{M}$ as a direct limit.
Since the homology functor commutes with direct limits, the first definition implies that the homology of $\{M_*(f),\partial_*\}$ is isomorphic to the singular homology of $\mathcal{M}$.

If we change the orientations of the unstable manifolds, we get isomorphic Morse complexes. The same is true if we choose a different pseudo-gradient vector field $X$ satisfying (X1)-(X5) with respect to the same $f$. Therefore, the {\em Morse complex of $f$} is well-defined up to isomorphism, for every $f\in \mathcal{F}(\mathcal{M})$. 

It is also useful to consider the following relative version of the Morse complex. Let $\mathcal{A}$ be an open subset of the Hilbert manifold $\mathcal{M}$. Let $\mathcal{F}(\mathcal{M},\mathcal{A})$ be the set of all $C^1$ real functions $f$ on $\mathcal{M}$  which are bounded from below on $\mathcal{M}\setminus \mathcal{A}$ and for which there is a smooth vector field $X$ on $\mathcal{M}$ such that:
\begin{enumerate}
\item[(X1')] $f$ is a Lyapunov function for $X$ on $\mathcal{M} \setminus \overline{\mathcal{A}}$;
\item[(X2')] $X$ is Morse on $\mathcal{M}\setminus \overline{\mathcal{A}}$ and all its singular points in this open set have finite Morse index;
\item[(X3')] the pair $(X,f)$ satisfies the Palais-Smale condition on $\mathcal{M} \setminus \overline{\mathcal{A}}$;
\item[(X4')] $\mathcal{A}$ is positively invariant with respect to the flow of $X$ (i.e.\ $p\in \mathcal{A}$ implies $\phi(t,p)\in \mathcal{A}$ for every $t\geq 0$ for which the flow is defined), and $X$ is forward complete relative to $\mathcal{A}$ (i.e.\ if $\phi(t,p)$ is not defined for all $t\geq 0$, then $\phi(t,p)\in \mathcal{A}$ for $t$ large enough).
\end{enumerate}  
In particular, (X3') implies that $X$ has no singular points on the boundary of $\mathcal{A}$, hence $f$ has no critical points on the boundary of $\mathcal{A}$. In fact, a sequence $(p_h)\subset \mathcal{M} \setminus \overline{\mathcal{A}}$ which converges to a singular point of $X$ on the boundary of $\mathcal{A}$ would be a Palais-Smale sequence which does not converge in $\mathcal{M} \setminus \overline{\mathcal{A}}$. By a generic perturbation, we may assume that $X$ satisfies also:
\begin{enumerate}
\item[(X5')] $X$ satisfies the Morse-Smale condition on $\mathcal{M} \setminus \overline{\mathcal{A}}$.
\end{enumerate} 
Then the {\em relative Morse complex of $(X,f)$ on $(\mathcal{M},\mathcal{A})$} is constructed as before, but taking into account only the critical points of $f$ in $\mathcal{M} \setminus \overline{\mathcal{A}}$.  The homology of this chain complex is isomorphic to the relative singular homology $H_*(\mathcal{M},\mathcal{A})$.
As before, the isomorphism class of the Morse complex does not depend on the pseudo-gradient $X$, but only on the function $f$.

\subsection{Functoriality}
\label{funct}

Let $\varphi: \mathcal{M}_1 \rightarrow \mathcal{M}_2$ be a smooth map between Hilbert manifolds. Let $f_1\in \mathcal{F}(\mathcal{M}_1)$ and $f_2\in \mathcal{F}(\mathcal{M}_2)$, and let $X_1$ and $X_2$ be corresponding Morse-Smale pseudo-gradients, i.e.\ smooth vector fields satisfying (X1)-(X5). We denote by $\phi^1$ and $\phi^2$ the corresponding flows. We assume that
\begin{eqnarray}
\label{noco1}
\mbox{each } y\in \crit(f_2) \mbox{ is a regular value of } \varphi; \\
\label{noco2}
x\in \crit(f_1), \; \varphi(x) \in \crit(f_2) \quad \implies \quad
i(x;f_1) \geq i(\varphi(x);f_2).
\end{eqnarray}
The set of critical points of $f_2$ is discrete, and in many cases
(for instance, if $\varphi$ is a Fredholm map) the set
of regular values of $\varphi$ is generic (i.e., it is a countable intersection of open and dense sets), by Sard-Smale theorem
\cite{sma65}. In such a situation,
condition (\ref{noco1}) can be achieved by arbitrary small (in
several senses) perturbations of $\varphi$ or of $f_2$. Also condition
(\ref{noco2}) can be achieved by an arbitrary small perturbation of
$\varphi$ or of $f_2$, simply by requiring
that the image of the set of critical points of $f_1$ by $\varphi$
does not meet the set of critical point of $f_2$.

By (\ref{noco1}) and (\ref{noco2}), up to the perturbation of the vector field $X_1$ and $X_2$, we may assume that
\begin{equation}
\label{tra}
\forall x\in \crit (f_1), \; \forall y\in \crit(f_2), \quad
\varphi|_{W^u(x;X_1)} \mbox{ is
transverse to } W^s(y;X_2).
\end{equation}
Indeed, by (\ref{noco1}) and (\ref{noco2}) one can perturb $X_1$ in such a
way that if $p\in W^u(x;X_1)$ and $\varphi(p)$ is a critical
point of $f_2$ then $\rank D\varphi(p) |_{T_p W^u(x)} \geq
i(\varphi(p);f_2)$. The possibility of perturbing $X_2$ so that
(\ref{tra}) holds is now a consequence of the following fact: if $W$ is
a finite dimensional manifold and $\psi:
W \rightarrow \mathcal{M}_2$ is a smooth map such that for every $p\in
W$ with $\psi(p)\in \crit (f_2)$ there holds $\rank D\psi(p) \geq
i(\varphi(p);f_2)$, then the set of pseudo-gradient vector fields $X_2$ for $f_2$ on $\mathcal{M}_2$ such that the map $\psi$ is transverse to the stable
manifold of every critical point of $f_2$ is generic.  

The transversality condition (\ref{tra}) ensures that if $x\in \crit(f_1)$ and
$y\in \crit(f_2)$, then
\[
W(x,y) := W^u(x;X_1) \cap \varphi^{-1} \bigl(W^s(y;X_2) \bigr)
\]
is a submanifold of dimension $i(x;f_1) - i(y;f_2)$. If $W^u(x;X_1)$ is
oriented and the normal bundle of $W^s(y;X_2)$ in $\mathcal{M}_2$ is
oriented, the manifold $W(x,y)$ carries a canonical orientation. In
particular, if $i(x;f_1) = i(y;f_2)$, $W(x,y)$ is a discrete set, each
of whose point carries an orientation sign $\pm 1$.
The transversality condition (\ref{tra}) and the
fact that $W^u(x;X_1)$ has compact closure in $\mathcal{M}_1$ imply that
the discrete set $W(x,y)$ is also compact, so it is a finite set and
we denote by $n_{\varphi}(x,y)\in \Z$ the algebraic sum of the
corresponding orientation signs. We can then define the homomorphism
\[
M_k \varphi : M_k (f_1) \rightarrow M_k (f_2), \quad (M_k \varphi)  x =
\sum_{y\in \mathrm{crit}_k(f_2)} n_{\varphi}(x,y) \, y,
\]
for every $x\in \mathrm{crit}_k (f_1)$.

We claim that $M_* \varphi$ is a chain map from the Morse complex
of $(f_1,X_1)$ to the Morse complex of $(f_2,X_2)$, and that the
corresponding homomorphism in homology coincides -- via the
isomorphism described in Section \ref{mc} -- with the homomorphism
$\varphi_* : H_*(\mathcal{M}_1) \rightarrow H_*(\mathcal{M}_2)$.

Indeed, let us fix a real number $a_1$, and let $a_2$ be larger than the maximum of $f_2$ on the image by $\varphi$ of the union of all unstable manifolds of critical points of $f_1$ in $\{f<a_1\}$ (the latter is a compact set). Then we can find open neighborhoods $\mathcal{U}_i(x)$, $i=1,2$, of each critical point $x\in \crit(f_i)$, such that the
sequence of open sets
\[
\mathcal{M}_{i,k}^{a_i} = \bigcup_{\substack{x\in \crit(f_i) \\ f_i(x) < a_i\\ i(x;f_i) \leq
    k}} \phi^i([0,+\infty[ \times \mathcal{U}_i(x)), \quad k\in \N \cup \{\infty\}, \; i=1,2,
\]
is a cellular filtration of $\mathcal{M}^{a_i}_{i,\infty}$, which is a deformation retract of the sublevel $\set{p\in \mathcal{M}_i}{f_i(p)<a_i}$, and 
\[
\varphi(\mathcal{M}^{a_1}_{1,\infty}) \subset \set{p\in \mathcal{M}_2}{f_2(p) < a_2}.
\]
By the transversality assumption (\ref{tra}), if $p$ belongs to $
W^u(x;X_1)$ with $f_1(x)<a_1$, then $\varphi(p)$ belongs to the stable manifold of
some critical point $y\in \crit(f_2)$ with $i(y;f_2) \leq i(x;f_1)$ and $f_2(y)<a_2$. A
standard compactness-transversality argument shows that, up to the 
replacement of the neighborhoods $\mathcal{U}_1(x)$, $x\in \crit(f_1)$, by
smaller ones, we may assume that
\begin{equation*}\begin{split}
p\in \mathcal{M}_{1,k}^{a_1} \quad \implies \quad \varphi(p) \in W^s(y; X_2) \mbox{ for some } y\in \crit(f_2) \mbox{ with } &i(y;f_2) \leq k\\ \mbox{ and } &f_2(y)<a_2.
\end{split}\end{equation*}
Since the set $\mathcal{M}_{2,k}^{a_2}$ is a $\phi^2$-positively invariant open
neighborhood of the set of the critical points of $f_2$ in $\{f_2<a_2\}$ whose Morse
index does not exceed $k$, it is easy to find a continuous function $t_0:
\mathcal{M}_{1,\infty}^{a_1} \rightarrow [0,+\infty[$ such that
\[
p\in \mathcal{M}_{1,k}^{a_1} \quad \implies \quad \phi^2(t_0(p),\varphi(p))
\in \mathcal{M}_{2,k}^{a_2}, \quad \forall k\in \N.
\]
Therefore, $\psi(p) := \phi^2 (t_0(p),\varphi(p))$ is a cellular map
from $\{ \mathcal{M}_{1,k}^{a_1} \}_{k\in \N}$ to
  $\{ \mathcal{M}_{2,k}^{a_2} \}_{k\in \N}$, and it is
    easy to check that the induced cellular homomorphism
\[
\psi_* : H_*( \mathcal{M}_{1,*}^{a_1},  \mathcal{M}_{1,*-1}^{a_1} )
\rightarrow H_*( \mathcal{M}_{2,*}^{a_2},  \mathcal{M}_{2,*-1}^{a_2} )
\]
coincides with the restriction $M_* \varphi  : M_*^{a_1}(f_1) \rightarrow M_*^{a_2}(f_2)$, once we identify the group
$H_k(\mathcal{M}_{i,k}^{a_i},\mathcal{M}_{i,k-1}^{a_i})$ with $M_k^{a_i}(f_i)$, 
by taking the orientations of the unstable manifolds into account. Then,
everything follows from the naturality of cellular homology, from
the fact that the inclusions $j_i: \mathcal{M}_{i,\infty}^{a_i} \hookrightarrow
\set{p\in \mathcal{M}_i}{f(p_i)<a_i}$ are homotopy equivalences, from the fact that
$j_2 \circ \psi$ is homotopic to $\varphi \circ j_1$, and by taking a direct limit for $a_1\uparrow +\infty$.

The above construction has an obvious extension to the case of a
smooth map $\varphi$ between two pairs
$(\mathcal{M}_1,\mathcal{A}_1)$ and $(\mathcal{M}_2,\mathcal{A}_2)$, where
$\mathcal{A}_i$ is an open subset of $\mathcal{M}_i$, for $i=1,2$.

\begin{rem}
We recall that if two chain maps between {\em free} chain
complexes induce the same homomorphism in homology, they are chain
homotopic. So from the functoriality of singular homology, we
deduce that $M_* \varphi \circ M_* \psi$ and $M_* \varphi\circ
\psi$ are chain homotopic. Actually, a chain homotopy between
these two chain maps could be constructed in a direct way.
\end{rem}

\begin{rem}
Consider the following particular but important case:
$\mathcal{M}_1=\mathcal{M}_2$ and $\varphi=\mathrm{id}$. Then
(\ref{noco1}) holds automatically, and (\ref{noco2}) means asking
that every common critical point $x$ of $f_1$ and $f_2$ must satisfy
$i(x;f_1) \geq i(x;f_2)$. In this case, the above construction produces
a chain map from $M_*(f_1)$ to $M_*(f_2)$ which induces the
identity map in homology (after the identification with singular
homology).
\end{rem}

\begin{rem}
\label{nntp}
For future reference, let us stress the fact that if it is already
known that $p\in W^u(x;X_1)$ and $\varphi(p)\in \crit(f_2)$
imply $\rank D\varphi(p) |_{T_p W^u(x;X_1)} \geq
i(\varphi(p);f_2)$, then condition (\ref{noco1}) is redundant, condition
(\ref{noco2}) holds automatically, and there is no need of perturbing
the vector field $X_1$ on $\mathcal{M}_1$.
\end{rem}

\subsection{The exterior homology product}
\label{exte}

Let $\mathcal{M}_1, \mathcal{M}_2$ be Hilbert manifolds, let $f_1\in
\mathcal{F}(\mathcal{M}_1)$, $f_2\in
\mathcal{F}(\mathcal{M}_2)$, and let $X_1$ and $X_2$ be corresponding Morse-Smale pseudo-gradients.  If we denote by $f_1 \oplus
f_2$ the function on $\mathcal{M}_1 \times \mathcal{M}_2$,
\[
f_1 \oplus f_2 \, (p_1,p_2) := f_1(p_1) + f_2(p_2),
\]
we see that $f_1 \oplus f_2$ belongs to $\mathcal{F}(\mathcal{M}_1
\times \mathcal{M}_2)$, and 
\[
X_1 \oplus X_2 (p_1,p_2) := (X_1(p_1),X_2(p_2)),
\]
satisfies (X1)-(X5) with respect to $f_1\oplus f_2$. Moreover,
\[
\mathrm{crit}_{\ell} (f_1 \oplus f_2) = \bigcup_{j+k=\ell}
\mathrm{crit}_j(f_1) \times \mathrm{crit}_k(f_2),
\]
hence
\[
M_{\ell} (f_1 \oplus f_2) = \bigoplus_{j+k=\ell} M_j (f_1) \otimes M_k (f_2).
\]
If we fix orientations for the unstable manifold of each critical point
of $f_1,f_2$, and we endow the unstable manifold of each $(x_1,x_2) \in
\crit(f_1 \oplus f_2)$,
\[
W^u((x_1, x_2)) = W^u(x_1) \times W^u(x_2),
\]
with the product orientation, we see that the boundary homomorphism in the
Morse complex of $(f_1 \oplus f_2,X_1 \oplus X_2)$ is given by
\[
\partial (x_1,x_2) = (\partial x_1,x_2) + (-1)^{i(x_1)}
(x_1,\partial x_2), \quad
\forall x_i \in \crit (f_i), \; i=1,2.
\]
We conclude that the Morse complex of $(f_1 \oplus f_2,X_1 \oplus X_2)$ is the tensor product of the Morse complexes of $(f_1,X_1)$ and
$(f_2,X_2)$. So, using the natural homomorphism from the tensor
product of the homology of two chain complexes to the homology of the
tensor product of the two complexes, we obtain the homomorphism
\begin{equation}
\label{p1}
H_j M(f_1) \otimes H_k M(f_2) \rightarrow H_{j+k} M (f_1
\oplus f_2).
\end{equation}
We claim that this homomorphism corresponds to the exterior product
homomorphism
\begin{equation}
\label{p2}
H_j (\mathcal{M}_1) \otimes H_k(\mathcal{M}_2)
\stackrel{\times}{\longrightarrow} H_{j+k} (\mathcal{M}_1 \times
\mathcal{M}_2),
\end{equation}
via the isomorphism between Morse homology and singular homology
described in Section \ref{mc}.

Indeed, the cellular filtration in $\{f_1<a_1\} \times \{ f_2 < a_2\}$ can be chosen to be generated by small product neighborhoods of the critical points,
\begin{equation*}\begin{split}
\mathcal{W}_{\ell}^{a_1,a_2} = &\bigcup_{\substack{(x_1,x_2) \in \crit(f_1 \oplus
    f_2)\\ f_1(x_1) < a_1, \; f_2(x_2)<a_2\\ i(x_1) + i(x_2) = \ell}} \hspace{-1em}\phi^1([0,+\infty[ \times
\mathcal{U}_1(x_1)) \times \phi^2([0,+\infty[ \times
\mathcal{U}_2(x_2))\\ = &\bigcup_{j+k=\ell} \mathcal{M}_{1,j}^{a_1} \times
\mathcal{M}_{2,k}^{a_2}.
\end{split}\end{equation*}
By excision and by the K\"unneth theorem, together with the fact that we
are dealing with free Abelian groups, one easily obtains that
\[
H_{\ell} (\mathcal{W}_{\ell}^{a_1,a_2},
\mathcal{W}_{\ell-1}^{a_1,a_2}) \cong \bigoplus_{j+k=\ell} H_j(\mathcal{M}_{1,j}^{a_1},
\mathcal{M}_{1,j-1}^{a_1}) \otimes H_k(\mathcal{M}_{2,k}^{a_2},
\mathcal{M}_{2,k-1}^{a_2}),
\]
and that the boundary homomorphism of the cellular filtration
$\mathcal{W}_*^{a_1,a_2}$ is the tensor product of the boundary homomorphisms of
the cellular filtrations $\mathcal{M}_{1,*}^{a_1}$ and $\mathcal{M}_{2,*}^{a_2}$. Passing
to homology, we find that (\ref{p1}) corresponds to the exterior
homology product
\[
 H_j (\mathcal{M}_{1,\infty}^{a_1}) \otimes H_k(\mathcal{M}_{2,\infty}^{a_2})
\stackrel{\times}{\longrightarrow} H_{j+k} (\mathcal{M}_{1,\infty}^{a_1} \times
\mathcal{M}_{2,\infty}^{a_2}) = H_{j+k} (\mathcal{W}_{\infty}^{a_1,a_2}),
\]
by the usual identification of the cellular complex to the Morse
complex induced by a choice of orientations for the unstable
manifolds. But using the fact that the inclusion $\mathcal{M}_{1,\infty}^{a_1}
\hookrightarrow \{f_1<a_1\}$ and
$\mathcal{M}_{2,\infty}^{a_2} \hookrightarrow \{f_2<a_2\}$ are homotopy
equivalences, and by considering a direct limit for $a_1,a_2 \uparrow +\infty$, we conclude that (\ref{p1}) corresponds to (\ref{p2}).

\subsection{Intersection products}
\label{inter}

Let $\mathcal{M}_0$ be a Hilbert manifold, and let $\pi: \mathcal{E}
\rightarrow \mathcal{M}_0$ be a smooth rank-$n$ oriented real vector
bundle over $\mathcal{M}_0$. It is easy to describe the Thom
isomorphism
\[
\tau : H_k (\mathcal{E},\mathcal{E} \setminus \mathcal{M}_0)
\stackrel{\cong}{\longrightarrow} H_{k-n}(\mathcal{M}_0), \quad \alpha
\mapsto \tau_{\mathcal{E}} \cap \alpha,
\]
in a Morse theoretical way ($\tau_{\mathcal{E}} \in
H^n(\mathcal{E}, \mathcal{E} \setminus \mathcal{M}_0)$ denotes the
Thom class of the vector bundle $\mathcal{E}$).

Indeed,  let $f_0 \in \mathcal{F}(\mathcal{M}_0)$ and let  $X_0$ be a Morse-Smale pseudo-gradient for $f_0$. The choice of a Riemannian structure on the
vector bundle $\mathcal{E}$ determines the smooth function
\[
f_1(\xi) := f_0(\pi(\xi)) - |\xi|^2, \quad \forall \xi \in
\mathcal{E}.
\]
The choice of a connection on $\mathcal{E}$ induces the horizontal-vertical splitting and the isomorphism:
\[
T_{\xi} \mathcal{E} = T_{\xi}^h \mathcal{E} \oplus T_{\xi}^v \mathcal{E} \cong T_{\pi(\xi)} \mathcal{M}_0 \oplus \pi^{-1}(\pi(\xi)).
\]
By the above identifications, we can define the tangent vector field $X_1$ on the total space $\mathcal{E}$ by 
\[
X_1(\xi) := (X_0(\pi(\xi)),  \xi).
\]
It is readily seen that $(f_1,X_1)$ satisfies conditions (X1')-(X5') of Section \ref{mc} on the pair $(\mathcal{E}, \mathcal{E}
\setminus \overline{\mathcal{U}_1})$, $\mathcal{U}_1$ being the set of
vectors $\xi$ in the total space $\mathcal{E}$ with $|\xi|<1$. Therefore, the homology of the relative Morse
complex of $(f_1,X_1)$ on $(\mathcal{E},\mathcal{E}\setminus
\overline{\mathcal{U}_1})$ is isomorphic to the singular homology of the pair
$(\mathcal{E},\mathcal{E}\setminus \overline{\mathcal{U}_1})$, that is
to the singular homology of $(\mathcal{E},\mathcal{E}\setminus
\mathcal{M}_0)$. Actually, the critical points of $f_1$ are contained in the zero-section of $\mathcal{E}$, and if we identify such a zero section with $\mathcal{M}_0$, we have
\[
\crit_k(f_1) = \crit_{k-n}(f_0), \quad
T_x W^u(x;X_1) = T_x W^u(x; X_0) \oplus \pi^{-1}(x), \quad \forall x\in \crit(f_0),
\]
so the orientation of the vector bundle $\mathcal{E}$ allows to
associate an orientation of $W^u(x;X_1)$ to each orientation of
$ W^u(x;X_0)$. Then the relative Morse complex of $(f_1,X_1)$ on
$(\mathcal{E}, \mathcal{E} \setminus \overline{\mathcal{U}_1})$
is obtained from the Morse complex of
$(f_0,X_0)$ on $\mathcal{M}_0$ by a $-n$-shift in the grading:
\[
M_k (f_1,X_1) = M_{k-n} (f_0,X_0),
\]
and one can show that the isomorphism $\tau$ -- read on the Morse
complexes by the isomorphisms described in Section \ref{mc} -- is induced by the identity mapping
\[
M_k(f_1,X_1) \stackrel{\mathrm{id}}{\longrightarrow} M_{k-n}(f_0,X_0),
\]
see \cite{cs08} for more details.

Consider now the general case of a closed embedding $e: \mathcal{M}_0
\hookrightarrow \mathcal{M}$, assumed to be of codimension $n$ and
co-oriented. The above description of the Thom isomorphism
associated to the normal bundle $N\mathcal{M}_0$
of $\mathcal{M}_0$ and the discussion about functoriality
of Section \ref{funct} provide us with a Morse theoretical description of the Umkehr map
\[
e_! : H_k(\mathcal{M}) \rightarrow H_{k-n}(\mathcal{M}_0).
\]
It is actually useful to identify an open neighborhood of
$\mathcal{M}_0$ with $N\mathcal{M}_0$ by the tubular neighborhood
theorem, to consider again the open unit ball
$\mathcal{U}_1$ around the zero section of $N\mathcal{M}_0$, and to
see the Umkehr map as the composition
\[
H_j(\mathcal{M}) \stackrel{i_*}{\longrightarrow} H_j (\mathcal{M},
\mathcal{M} \setminus \overline{\mathcal{U}_1}) \cong H_j(N
\mathcal{M}_0, N \mathcal{M}_0 \setminus \mathcal{M}_0)
\stackrel{\tau}{\longrightarrow} H_{j-n}(\mathcal{M}_0),
\]
the map $i: \mathcal{M} \hookrightarrow (\mathcal{M}, \mathcal{M}
\setminus \overline{\mathcal{U}_1})$ being the inclusion.
Let $f_0,X_0,f_1,X_1$ be as above. We use the symbols $f_1$ and $X_1$
also to denote arbitrary extensions of $f_1$ and
$X_1$ to the whole $\mathcal{M}$.
Let $f\in \mathcal{F}(\mathcal{M})$, and let $X$ be a Morse-Smale pseudo-gradient for $f$. Since we
would like to achieve transversality by perturbing $X$ and $X_0$, but
keeping $X_1$ of product-type near $\mathcal{M}_0$, we need the
condition
\begin{equation}
\label{noco3a}
x\in \crit(f) \cap \mathcal{M}_0 \quad \implies \quad i(x;f) \geq n,
\end{equation}
which implies that up to perturbing $X$ we may assume that the
unstable manifold of each critical point of $f$ is transversal to
$\mathcal{M}_0$. Assumption (\ref{noco1}) is automatically satisfied
by the triplet $(i,f,f_1)$, while (\ref{noco2}) is equivalent to asking that
\begin{equation}
\label{noco3b}
x\in \crit(f) \cap \crit(f_0) \quad \implies \quad i(x;f) \geq
i(x;f_0) + n.
\end{equation}
Conditions (\ref{noco3a}) and (\ref{noco3b}) are implied by the
generic assumption $\crit(f) \cap \mathcal{M}_0 = \emptyset$.
By the arguments of Section \ref{funct} applied to the
map $i$ (in particular, condition
(\ref{tra})), we see that up to perturbing $X$ and $X_0$ (keeping
$X_1$ of product-type near $\mathcal{M}_0$), we may assume that for
every $x\in \crit (f)$, $y\in \crit(f_0)$, the intersection
\[
W^u(x;X)\cap W^s(y;X_1) =
W^u(x;X)\cap W^s(y;X_0)
\]
is transverse in $\mathcal{M}$, hence it is a submanifold of
dimension $i(x;f) - i(y;f_0) -n$. If we fix an orientation of
the unstable manifold of each critical point of $f$ and $f_0$, these
intersections are canonically oriented. Compactness and transversality
imply that when $i(y;f_0) = i(x;f) - n$, this intersection is a finite
set of points, each of which comes with an orientation sign $\pm
1$. Denoting by $n_{e_!}(x,y)$ the
algebraic sum of these signs, we conclude that the homomorphism
\[
M_k(f) \rightarrow M_{k-n}(f_0), \quad x\mapsto
\hspace{-1em}\sum_{\substack{y\in \crit (f_0)\\ i(y;f_0) = k - n}}\hspace{-1em} n_{e_!}(x,y)\,
y, \quad \forall x\in \mathrm{crit}_k(f),
\]
is a chain map of degree $-n$, and that it induces the Umkehr map $e_!$ in homology.

Let us conclude this section by describing the Morse theoretical
interpretation of the intersection product in homology. See \cite{bc94} for the Morse theoretical description of more general cohomology operations.
Let
$M$ be a finite-dimensional oriented manifold, and consider the
diagonal embedding $e:\Delta_M \hookrightarrow M \times M$, which is
$n$-codimensional and co-oriented. The intersection product is defined
by the composition
\[
H_j(M) \otimes H_k(M) \stackrel{\times}{\longrightarrow} H_{j+k}(M
\times M) \stackrel{e_!}{\longrightarrow} H_{j+k-n}(\Delta_M) \cong
H_{j+k-n}(M),
\]
and it is denoted by
\[
\bullet : H_j(M) \otimes H_k(M) \longrightarrow H_{j+k-n}(M).
\]
The above description of $e_!$ and the description of the exterior
homology product $\times$ given in Section \ref{exte} immediately
yield the following description of $\bullet$. Let $f_i \in \mathcal{F}(M)$,
$i=1,2,3$, and let $X_i$ be corresponding Morse-Smale pseudo-gradients (here we could assume that the $f_i$'s are smooth Morse functions, and that $X_i=-\mathrm{grad}_g f_i$ with respect to three suitable complete metrics $g_i$ on $M$). The non-degeneracy 
conditions (\ref{noco3a}) and
(\ref{noco3b}), which are used to represent $e_!$, are now
\begin{eqnarray}
\label{noco4a}
x\in \crit (f_1) \cap \crit (f_2) \quad \implies \quad i(x;f_1) +
i(x;f_2) \geq n, \\ \label{noco4b}
x\in \crit(f_1) \cap \crit(f_2) \cap \crit(f_3) \quad \implies \quad
i(x;f_1) + i(x;f_2) \geq i(x;f_3) + n.
\end{eqnarray}
These conditions are implied for instance by the generic assumption that
$f_1$ and $f_2$ do not have common critical points.
We can now perturb the vector fields $X_1$, $X_2$, and $X_3$ in such a way that for every triplet $x_i\in
\crit(f_i)$, $i=1,2,3$, the intersection
\[
W^u((x_1,x_2) ; X_1 \oplus X_2)
\cap a(W^s(x_3; X_3)),
\]
$a:M \rightarrow M \times M$ being the map $a(p)=(p,p)$,
is transverse in $M \times M$, hence it is an oriented submanifold of
$\Delta_M$ of dimension $i(x_1;f_1) + i(x_2;f_2) - i(x_3;f_3) - n$.
By compactness and
transversality,  when $i(x_3;f_3) = i(x_1;f_1) + i(x_2;f_2) - n$,
this intersection, which can also be written as
\[
\set{(p,p) \in W^u(x_1;X_1) \times W^u(x_2;X_2)}{p\in
  W^s(x_3; X_3)},
\]
is a finite set of points, each of which comes with
an orientation sign $\pm 1$. Denoting by $n_{\bullet}(x_1,x_2;x_3)$ the
algebraic sum of these signs, we conclude that the homomorphism
\[
M_j(f_1,g_1) \otimes M_k(f_2,g_2) \rightarrow M_{j+k-n}(f_3,g_3),
\quad x_1 \otimes x_2 \mapsto \hspace{-1em}\sum_{\substack{x_3 \in \crit (f_3)\\
    i(x_3;f_3) = j+k - n}} \hspace{-1em}n_{\bullet}(x_1,x_2;x_3)\,
x_3,
\]
where $x_1\in \mathrm{crit}_j(f_1)$, $x_2 \in \mathrm{crit}_k(f_2)$,
is a chain map of degree $-n$ from $M(f_1) \otimes M(f_2)$
to $M(f_3)$, and that it induces the
intersection product $\bullet$ in homology.

\providecommand{\bysame}{\leavevmode\hbox to3em{\hrulefill}\thinspace}
\providecommand{\MR}{\relax\ifhmode\unskip\space\fi MR }
\providecommand{\MRhref}[2]{%
  \href{http://www.ams.org/mathscinet-getitem?mr=#1}{#2}
}
\providecommand{\href}[2]{#2}


\begin{thebibliography}{FOOO09}

\bibitem[AM06]{ama06m}
A.~Abbondandolo and P.~Majer, \emph{Lectures on the {M}orse complex for
  infinite dimensional manifolds}, Morse theoretic methods in nonlinear
  analysis and in symplectic topology (Montreal) (P.~Biran, O.~Cornea, and
  F.~Lalonde, eds.), Springer, 2006, pp.~1--74.

\bibitem[APS08]{aps08}
A.~Abbondandolo, A.~Portaluri, and M.~Schwarz, \emph{The homology of path
  spaces and {F}loer homology with conormal boundary conditions}, J. fixed
  point theory appl. \textbf{4} (2008), 263--293.

\bibitem[AS06a]{as06m}
A.~Abbondandolo and M.~Schwarz, \emph{Notes on {F}loer homology and loop space
  homology}, Morse theoretic methods in nonlinear analysis and in symplectic
  topology (Montreal) (P.~Biran, O.~Cornea, and F.~Lalonde, eds.), Springer,
  2006, pp.~75--108.

\bibitem[AS06b]{as06}
A.~Abbondandolo and M.~Schwarz, \emph{{On the Floer homology of cotangent bundles}}, Comm. Pure Appl.
  Math. \textbf{59} (2006), 254--316.

\bibitem[AS09]{as09}
A.~Abbondandolo and M.~Schwarz, \emph{A smooth pseudo-gradient for the {L}agrangian action functional}, Adv. Nonlinear Stud. \textbf{9} (2009), 597--623.
  
\bibitem[BCR06]{bcr06}
N.~A. Baas, R.~L. Cohen, and A.~Ram{\'{\i}}rez, \emph{The topology of the
  category of open and closed strings}, Recent developments in algebraic
  topology, Contemp. Math., vol. 407, Amer. Math. Soc., Providence, RI, 2006,
  pp.~11--26.

\bibitem[BC94]{bc94}
M.~Betz and R.~L. Cohen, \emph{Moduli spaces of graphs and cohomology
  operations}, Turkish J. Math. \textbf{18} (1994), 23--41.

\bibitem[Cha93]{cha93}
K.~C. Chang, \emph{Infinite-dimensional {M}orse theory and multiple solution
  problems}, Birkh\"auser, Boston, 1993.

\bibitem[CS99]{csu99}
M.~Chas and D.~Sullivan, \emph{{String topology}}, {\tt arXiv:math/9911159
  [math.GT]}, 1999.

\bibitem[CS04]{csu99b}
M.~Chas and D.~Sullivan, \emph{Closed string operators in topology leading to {L}ie bialgebras
  and higher string algebra}, The legacy of Niels Henrik Abel, Springer, 2004,
  pp.~771--788.

\bibitem[CEL09]{cel09}
K.~Cieliebak, T.~Ekholm, and J.~Latschev, \emph{Compactness for holomorphic
  curves with switching {L}agrangian boundary conditions}, {\tt arXiv:0903.2200v1
  [math.SG]}, 2009.

\bibitem[CL07]{cl07}
K.~Cieliebak and J.~Latschev, \emph{The role of string topology in symplectic
  field theory}, {\tt arXiv:0706.3284v2 [math.SG]}, 2007.

\bibitem[Coh06]{coh06}
R.~L. Cohen, \emph{Morse theory, graphs, and string topology}, Morse
  theoretical methods in nonlinear analysis and symplectic topology (Montreal)
  (P.~Biran, O.~Cornea, and F.~Lalonde, eds.), Springer, 2006, pp.~149--184.

\bibitem[CHV06]{chv06}
R.~L. Cohen, K.~Hess, and A.~A. Voronov, \emph{String topology and cyclic
  homology}, Birkh\"auser, 2006.

\bibitem[CJ02]{cj02}
R.~L. Cohen and J.~D.~S. Jones, \emph{A homotopy theoretic realization of
  string topology}, Math. Ann. \textbf{324} (2002), 773--798.

\bibitem[CJY03]{cjy03}
R.~L. Cohen, J.~D.~S. Jones, and J.~Yan, \emph{The loop homology algebra of
  spheres and projective spaces}, Categorical decomposition techniques in
  algebraic topology (Isle of Skye, 2001), Birkh\"auser, 2003, pp.~77--92.

\bibitem[CKS08]{cks08}
R.~L. Cohen, J.~Klein, and D.~Sullivan, \emph{The homotopy invariance of the
  string topology loop product and string bracket}, J. Topol. \textbf{1}
  (2008), 391--408.

\bibitem[CS09]{cs08}
R.~L. Cohen and M.~Schwarz, \emph{A {M}orse theoretic description of string
  topology}, New perspectives and challenges in symplectic field theory
  (L.~Polterovich, M.~Abreu, and F.~Lalonde, eds.), CRM Proc. Lecture Notes,
  vol.~49, Amer. Math. Soc., 2009.

\bibitem[Dui76]{dui76}
J.~J. Duistermaat, \emph{On the {M}orse index in variational calculus},
  Advances in Math. \textbf{21} (1976), 173--195.

\bibitem[Flo88a]{flo88a}
A.~Floer, \emph{A relative {M}orse index for the symplectic action}, Comm. Pure
  Appl. Math. \textbf{41} (1988), 393--407.

\bibitem[Flo88b]{flo88d}
A.~Floer, \emph{The unregularized gradient flow of the symplectic action}, Comm.
  Pure Appl. Math. \textbf{41} (1988), 775--813.

\bibitem[Flo89a]{flo89b}
A.~Floer, \emph{Symplectic fixed points and holomorphic spheres}, Comm. Math.
  Phys. \textbf{120} (1989), 575--611.

\bibitem[Flo89b]{flo89a}
A.~Floer, \emph{Witten's complex and infinite-dimensional {M}orse theory}, J.
  Differential Geom. \textbf{30} (1989), 207--221.

\bibitem[FH93]{fh93}
A.~Floer and H.~Hofer, \emph{Coherent orientations for periodic orbit problems
  in symplectic geometry}, Math. Z. \textbf{212} (1993), 13--38.

\bibitem[FHS96]{fhs96}
A.~Floer, H.~Hofer, and D.~Salamon, \emph{Transversality in elliptic {M}orse
  theory for the symplectic action}, Duke Math. J. \textbf{80} (1996),
  251--292.

\bibitem[Fuk93]{fuk93}
K.~Fukaya, \emph{Morse homotopy, {$A\sp \infty$}-category, and {F}loer
  homologies}, Proceedings of GARC Workshop on Geometry and Topology '93
  (Seoul, 1993) (Seoul), Lecture Notes Ser., vol.~18, Seoul Nat. Univ., 1993,
  pp.~1--102.

\bibitem[Fuk97]{fuk97}
K.~Fukaya, \emph{Morse homotopy and its quantization}, Geometric topology
  (Athens, GA, 1993), AMS/IP Stud. Adv. Math., vol.~2, Amer. Math. Soc.,
  Providence, RI, 1997, pp.~409--440.

\bibitem[FOOO09]{fooo09}
K.~Fukaya, Y.-G. Oh, K.~Ono, and H.~Ohta, \emph{Lagrangian intersection {F}loer
  theory -- anomaly and obstruction -- {C}hapter 9}, Preprint available on K.
  Fukaya's homepage, 2009.
  
\bibitem[FO99]{fo99}
K.~Fukaya and K.~Ono, \emph{Arnold conjecture and {G}romov-{W}itten invariant},
  Topology \textbf{38} (1999), 933--1048.

\bibitem[GH09]{gh09}
M.~Goresky and N.~Hingston, \emph{Loop products and closed geodesics}, Duke
  Math. J. \textbf{150} (2009), 117--209.

\bibitem[HS95]{hs95}
H.~Hofer and D.~Salamon, \emph{Floer homology and {N}ovikov rings}, The Floer
  memorial volume (H.~Hofer et~al., ed.), Progr. Math., vol. 133, Birkh\"auser,
  Basel, 1995, pp.~483--524.

\bibitem[HZ94]{hz94}
H.~Hofer and E.~Zehnder, \emph{Symplectic invariants and {H}amiltonian
  dynamics}, Birkh\"auser, Basel, 1994.

\bibitem[H{\"o}r71]{hor71}
L.~H{\"o}rmander, \emph{Fourier integral operators {I}}, Acta Math.
  \textbf{127} (1971), 79--183.

\bibitem[IS02]{is02}
S.~Ivashkovich and V.~Shevchishin, \emph{Reflection principle and {J}-complex curves with boundary on totally real immersions}, Commun. Contemp. Math. \textbf{4} (2002), 65--106.

\bibitem[Lan99]{lan99}
S.~Lang, \emph{Fundamentals of differential geometry}, Springer, Berlin, 1999.

\bibitem[LT98]{lt98}
G.~Liu and G.~Tian, \emph{Floer homology and {A}rnold conjecture}, J.
  Differential Geom. \textbf{49} (1998), 1--74.

\bibitem[LT99]{lt99}
G.~Liu and G.~Tian, \emph{{On the equivalence of multiplicative structures in Floer
  homology and quantum homology}}, Acta Mathematica Sinica \textbf{15} (1999),
  53--80.

\bibitem[MS04]{ms04}
D.~McDuff and D.~Salamon, \emph{{$J$}-holomorphic curves and symplectic
  topology}, Colloquium Publications, vol.~52, American Mathematical Society,
  Providence, R.I., 2004.

\bibitem[PSS96]{pss96}
S.~Piunikhin, D.~Salamon, and M.~Schwarz, \emph{{Symplectic Floer-Donaldson
  theory and quantum cohomology}}, Contact and symplectic geometry (C.~B.
  Thomas, ed.), vol.~8, Cambridge University Press, 1996, pp.~171--200.

\bibitem[Qui85]{qui85}
D.~Quillen, \emph{Determinants of {C}auchy-{R}iemann operators over a {R}iemann
  surface}, Functional Anal. Appl. \textbf{19} (1985), 31--34.

\bibitem[Ram06]{ram06}
A.~Ramirez, \emph{Open-closed string topology via fat graphs}, {\tt
  arXiv:math/0606512v1 [math.AT]}, 2006.

\bibitem[RS93]{rs93}
J.~Robbin and D.~Salamon, \emph{Maslov index theory for paths}, Topology
  \textbf{32} (1993), 827--844.

\bibitem[RS95]{rs95}
J.~Robbin and D.~Salamon, \emph{The spectral flow and the {M}aslov index}, Bull. London Math.
  Soc. \textbf{27} (1995), 1--33.

\bibitem[SW06]{sw06}
D.~Salamon and J.~Weber, \emph{Floer homology and the heat flow}, Geom. Funct.
  Anal. \textbf{16} (2006), 1050--1138.

\bibitem[Sch93]{sch93}
M.~Schwarz, \emph{Morse homology}, Birkh\"auser, Basel, 1993.

\bibitem[Sch95]{sch95}
M.~Schwarz, \emph{Cohomology operations from {$S^1$}-cobordisms in {F}loer
  homology}, Ph.D. thesis, Swiss Federal Inst.~of Techn. Zurich, Zurich,
  Diss.~ETH No.~11182, 1995.

\bibitem[Sei08]{sei06b}
P.~Seidel, \emph{A biased view of symplectic cohomology}, Current Developments
  in Mathematics, 2006, International Press, 2008, pp.~211--253.

\bibitem[Sma65]{sma65}
S.~Smale, \emph{An infinite dimensional version of {S}ard's theorem}, Amer. J.
  Math. \textbf{87} (1965), 861--866.

\bibitem[Sul04]{sul04}
D.~Sullivan, \emph{Open and closed string field theory interpreted in classical
  algebraic topology}, Topology, geometry and quantum field theory, London
  Math. Soc. Lecture Note Ser., vol. 308, Cambridge Univ. Press, Cambridge,
  2004, pp.~344--357.

\bibitem[Sul07]{sul07}
D.~Sullivan, \emph{String topology background and present state}, Current
  developments in mathematics, 2005, Int. Press, Somerville, MA, 2007,
  pp.~41--88.

\bibitem[tDKP70]{dkp70}
T.~tom Dieck, K.~H. Kamps, and D.~Puppe, \emph{Homotopietheorie}, Springer,
  Berlin-Heidlberg-New York, 1970.

\bibitem[Vit95]{vit95f}
C.~Viterbo, \emph{The cup-product on the {T}hom-{S}male-{W}itten complex, and
  {F}loer cohomology}, The Floer memorial volume, Progr. Math., vol. 133,
  Birkh\"auser, Basel, 1995, pp.~609--625.

\bibitem[Vit96]{vit96}
C.~Viterbo, \emph{Functors and computations in {F}loer homology with
  applications, {II}}, preprint (revised in 2003), 1996.

\bibitem[Web02]{web02}
J.~Weber, \emph{Perturbed closed geodesics are periodic orbits: index and
  transversality}, Math. Z. \textbf{241} (2002), 45--82.

\bibitem[Web05]{web05}
J.~Weber, \emph{Three approaches towards {F}loer homology of cotangent bundles},
  J. Symplectic Geom. \textbf{3} (2005), 671--701.

\bibitem[WW07]{ww07}
K.~Wehrheim and C.~T. Woodward, \emph{Functoriality for {L}agrangian
  correspondences in {F}loer theory}, {\tt arXiv:0708.2851v2 [math.SG]}, 2007.
  
\bibitem[WW09a]{ww09}
K.~Wehrheim and C.~T. Woodward, \emph{Floer cohomology and geometric composition of {L}agrangian correspondences}, {\tt 	arXiv:0905.1368v1 [math.SG]}, 2009.
  
\bibitem[WW09b]{ww09a}
K.~Wehrheim and C.~T. Woodward, \emph{Pseudoholomorphic quilts}, {\tt	arXiv:0905.1369v1 [math.SG]}, 2009.
  
\bibitem[WW09c]{ww09b}
K.~Wehrheim and C.~T. Woodward, \emph{Quilted {F}loer Cohomology}, {\tt 	arXiv:0905.1370v1 [math.SG]}, 2009.

\end{thebibliography}
\end{document}